\documentclass{article}[12pt]

\usepackage{amsmath,amsfonts,amssymb,amsthm}
\usepackage{graphicx}
\usepackage{enumerate}
\usepackage{bbm}
\usepackage{verbatim}
\usepackage{hyperref,color}
\usepackage[capitalize,nameinlink]{cleveref}
\usepackage[dvipsnames]{xcolor}
\hypersetup{
	colorlinks=true,
	pdfpagemode=UseNone,
    citecolor=OliveGreen,
    linkcolor=NavyBlue,
    urlcolor=Magenta,
	pdfstartview=FitW
}
\usepackage{thmtools}
\usepackage{thm-restate}

\pdfoutput=1
\usepackage{subcaption}

\numberwithin{equation}{section}
\usepackage[tt=false]{libertine}
\usepackage{mathtools}

\usepackage{amssymb,mathrsfs}
\usepackage[varbb]{newpxmath}

\let\savedbigtimes\bigtimes
\let\bigtimes\relax
\usepackage{mathabx} 
\let\bigtimes\savedbigtimes

\usepackage[margin=1in]{geometry}
\usepackage{graphicx}
\usepackage{enumerate}
\usepackage{bbm}
\usepackage{hyperref,color}
\usepackage[capitalize,nameinlink]{cleveref}
\usepackage[dvipsnames]{xcolor}
\hypersetup{
	colorlinks=true,
	pdfpagemode=UseNone,
    citecolor=OliveGreen,
    linkcolor=NavyBlue,
    urlcolor=black,
	pdfstartview=FitW
}
\usepackage{appendix}
\crefname{appsec}{Appendix}{Appendices}
\usepackage{tikz}
\usepackage{bm}
\usepackage{mathtools}
\usepackage{lmodern}

\newtheorem*{assumption*}{Assumption}
\newtheorem*{condition*}{Condition}

\newtheorem{theorem}{Theorem}[section]
\newtheorem{proposition}[theorem]{Proposition}
\newtheorem{lemma}[theorem]{Lemma}
\newtheorem{corollary}[theorem]{Corollary}

\theoremstyle{definition}
\newtheorem{definition}[theorem]{Definition}

\newtheorem{assumption}[theorem]{Assumption}
\newtheorem{remark}[theorem]{Remark}
\crefname{lemma}{Lemma}{Lemmas}
 \crefname{theorem}{Theorem}{Theorems}
 \crefname{definition}{Definition}{Definitions}
 \crefname{fact}{Fact}{Facts}
 \crefname{claim}{Claim}{Claims}
 \crefname{proposition}{Proposition}{Propositions}

\newcommand{\E}{\mathbb{E}}
\newcommand{\Var}{\mathrm{Var}}

\newcommand{\norm}[1]{\left\lVert #1 \right\rVert}

\renewcommand{\epsilon}{\varepsilon}

\newcommand{\N}{\mathbb{N}}
\newcommand{\Z}{\mathbb{Z}}

\newcommand{\R}{\mathbb{R}}

\renewcommand{\P}{\mathbb{P}}

\usepackage{algorithm}
\usepackage{algpseudocode} % <-- provides \State

\newcommand{\PP}{\mathbb{P}}
\newcommand{\EE}{\mathbb{E}}

\newcommand{\beq}{\begin{equation}}
\newcommand{\eeq}{\end{equation}}

\def\supp{{\rm supp}}

\def\de{{\rm d}}

\def\de{{\rm d}}

\def\de{{\rm d}}

\renewcommand{\d}{\textup{d}}

\newcommand{\<}{\langle}
\renewcommand{\>}{\rangle}

\newcommand{\inner}[2]{\langle #1, #2 \rangle}

\title{The monotonicity of the Franz-Parisi potential\\ is equivalent with Low-degree MMSE lower bounds}

\author{
Konstantinos Tsirkas\thanks{Department of Statistics and Data Science, Yale University.\\ Email: \texttt{\{kostas.tsirkas,leda.wang,ilias.zadik\}@yale.edu}}
\and
Leda Wang\footnotemark[1]
\and
Ilias Zadik\footnotemark[1]
}

\date{\today}

\begin{document}
\maketitle

\maketitle
\begin{abstract}
Over the last decades, two distinct approaches have been instrumental to our understanding of the computational complexity of statistical estimation. The statistical physics literature predicts algorithmic hardness through local stability and monotonicity properties of the Franz--Parisi (FP) potential \cite{franz1995recipes,franz1997phase}, while the mathematically rigorous literature characterizes hardness via the limitations of restricted algorithmic classes, most notably low-degree polynomial estimators \cite{hopkins2017efficient}. For many inference models, these two perspectives yield strikingly consistent predictions, giving rise to a long-standing open problem of establishing a precise mathematical relationship between them.

Recent works \cite{bandeira2022franz,chen2025an} addressed this question in the setting of detection, showing that for broad classes of models the success of low-degree polynomials is governed by an area criterion involving the annealed FP potential. While this provides a rigorous bridge between the two frameworks, the resulting criterion differs from the monotonicity-based conditions traditionally emphasized in the physics literature, reflecting fundamental distinctions between detection and estimation tasks.

In this work, we show that for estimation problems the power of low-degree polynomials is equivalent to the monotonicity of the annealed FP potential for a broad family of Gaussian additive models (GAMs) with signal-to-noise ratio $\lambda$. In particular, subject to a low-degree conjecture for GAMs, our results imply that the polynomial-time limits of these models are directly implied by the monotonicity of the annealed FP potential, in conceptual agreement with predictions from the physics literature dating back to the 1990s. Specifically, letting $\mathcal{F}_{\mathrm{ann},\lambda}$ denote the annealed FP potential, we prove that the optimal degree-$D$ correlation satisfies, for all $D=O(\mathrm{poly}(\log n))$, the approximate fixed point equation,
\[
\mathrm{Corr}^{\le D}_{P_0}\!\left(\lambda + \left.\frac{d}{dq}F_{\mathrm{ann},\lambda}\right|_{q=q(D)}\right)^2 \approx q(D),
\]
where $q(D)$ denotes the $e^{-D}$-quantile of the overlap between two independent draws from the prior. In particular, the condition $\left.\frac{d}{dq}F_{\mathrm{ann},\lambda}\right|_{q=q(D)} \ge 0$ is equivalent to all degree-$D$ polynomial estimators achieving correlation at most $q(D)$ at signal-to-noise ratio $\lambda$. This establishes the first rigorous equivalence between a physics-inspired monotonicity criterion and low-degree estimation lower bounds, and may be viewed as a low-degree analogue of the classical I--MMSE relationship \cite{guo2005mutual}.

\end{abstract}

\newpage

\tableofcontents

\newpage

\section{Introduction}

Over the last decades, there has been a strong effort to understand the computational complexity of Bayesian statistical estimation tasks. Due to the average-case nature of such tasks, it appears well beyond our current mathematical abilities to characterize their hardness phases based on standard complexity theory assumptions (such as $\mathcal{P} \neq \mathcal{NP}$) and for this reason researchers from various backgrounds have produced different approaches to the questions.

\paragraph{MMSE lower bounds against low-degree polynomials}One prominent direction in theoretical statistics and theoretical computer science studies the limitations of restricted classes of estimators, often referred to as unconditional lower bounds. Among these frameworks, low-degree polynomial estimators have played a particularly central role in predicting computational limits. While challenging to study in the context of estimation, the pioneering work of \cite{Schramm_2022} provided some key techniques on how to bound the minimum mean squared error (MMSE) among all low-degree polynomials for some classes of Gaussian additive models (GAM), where for some prior $P_0$ on $\mathbb{R}^{n}$ one observes for some signal-to-noise ratio $\lambda$, \begin{equation}\label{eq:GAM}
Y=\sqrt{\lambda} X+Z,    
\end{equation}
for the signal $X\sim P_0$ and independent noise $Z \sim N(0,I_N)$. Following this, more works have provided MMSE lower bounds for other families of GAMs or Bernoulli models, e.g., \cite{sohn2025sharpphasetransitionsestimation,luo2023tensorclusteringplantedstructures,luo2024computational,even2025computational}. It is important to highlight that the belief that low-degree polynomials are powerful in statistical tasks has culimanted with the an influential ``low-degree conjecture" in the similar detection (hypothesis testing) setting \cite{hopkins2017efficient}, which posits that for many sufficiently ``nice" detection problems, degree-$D=O(\log N)$ polynomial tests match the performance of the optimal polynomial-time algorithm. Although no direct analogue of this conjecture has yet been formally articulated for statistical estimation, it is widely expected that a corresponding conjecture will be posed soon, at least for GAMs.

\paragraph{Monotonicity of the Franz-Parisi potential}Another highly influential perspective on the computational complexity of statistical inference comes from the statistical physics community (see e.g., \cite{zdeborova2016statistical} for a survey). From this physics viewpoint, since the optimal estimator that achieves the minimum mean-squared error (MMSE) can be obtained by sampling from the Bayesian posterior, the optimal time-efficient algorithm should correspond to a physically natural reversible dynamics on the parameter space—such as Glauber or Langevin dynamics—whose stationary distribution is the posterior itself. Understanding the computational limits of this task is therefore reduced to analyzing whether such dynamics can efficiently sample from the posterior or become trapped in bottlenecks (or ``metastable" states).

In a seminal work, Franz and Parisi \cite{franz1995recipes,franz1997phase} introduced the Franz–Parisi (FP) potential, parameterized by an overlap variable, as a tool for predicting the behavior of such dynamics. The FP potential captures the local geometry of the posterior landscape around configurations of a given overlap, and its shape is used to predict whether local dynamics will rapidly mix or instead become trapped. In particular, the physics prediction is that when the FP potential ceases to be decreasing, the dynamics become trapped in a metastable state, and the resulting estimator achieves correlation with the signal equal to the overlap at which the FP potential attains its first local minimum.

For a GAM the FP potential takes the following form (see \cite{bandeira2022franz} for details), \[\mathcal{F}_{\lambda}(q):=-\mathbb{E}_{X \sim \mu, Z \sim N(0,I_N), Y=\sqrt{\lambda}X+Z} \log \mathbb{E}_{X' \sim \mu} 1_{\langle X,X' \rangle=q} \exp\left(-\|Y-\sqrt{\lambda}X'\|^2_2/2\right), q \in [-1,1] \label{eq:quenched}\]Due to its complicated form, physicists typically study tractable approximations of the FP potential to make concrete predictions. Two approximations play a particularly important role: the replica-symmetric (RS) approximation and the annealed approximation. The RS approximation has been shown to accurately characterize the behavior of the Approximate Message Passing (AMP) algorithm in some spiked matrix models, in the sense that the asymptotic correlation of AMP iterates converges to the location of the first local minimum of the RS potential (see, e.g., \cite[Theorem 1.1.]{montanari2024equivalenceapproximatemessagepassing} and references therein). The annealed approximation, obtained by applying Jensen’s inequality to the FP potential and exchanging the logarithm and expectation, leads to a significantly simpler quantity known as the annealed FP potential (see \cite[Proposition 2.3]{bandeira2022franz}), given by
\[
\mathcal{F}_{\mathrm{ann}, \lambda}(q):=-\log \mathbb{E}_{X,X' \sim \mu, X \perp\!\!\!\perp X'} 1_{\langle X,X' \rangle=q}\exp\left(\lambda \langle X,X'\rangle\right), q \in [-1,1].\label{eq:annealed}
\]

\paragraph{A mathematical relation? Obstacles and prior results} Strikingly, for many estimation tasks, the computational hardness predictions obtained from low-degree polynomial methods agree with those arising from the physics perspective based on the monotonicity of the Franz–Parisi (FP) potential. This empirical alignment naturally leads to the central question motivating this work, which has remained an open puzzle in the field:
\begin{quote}
\centering
\emph{Is there a precise mathematical connection between\\ the monotonicity criterion of the FP potential\\ and the MMSE achieved by low-degree polynomial estimators?}
\end{quote}

Unfortunately, pursuing such a connection using the original (or quenched) FP potential, defined in \eqref{eq:quenched}, turns out to be highly delicate and, in general, invalid without substantial restrictions. The underlying issue is that there also exist canonical estimation problems—even within the class of Gaussian additive models—for which the quenched FP prediction is provably incorrect. A notable example is Bernoulli sparse tensor PCA, where $X = x^{\otimes t},$ for some $t \geq 2,$ with $x$ drawn uniformly from $\{ v \in \{0,1/\sqrt{k}\}^N: \|v\|_0=k\}$ (see also section \ref{sec:quenched}) \footnote{It is widely believed that even simple tensor PCA with spherical prior is a counterexample, but only low or high temperature variants of the quenched FP prediction appears to be rigorously proven to fail so far \cite{arous2020algorithmic}.}. In this setting, the quenched FP potential predicts computational hardness deep inside a regime where the estimation task is known to be computationally easy, whereas the low-degree polynomial prediction is conjectured to be optimal (this is an implication of the bottleneck proven in \cite[Section 3.3]{chen2024low}). Yet, on a positive light, a motivating result appeared a few years ago in \cite{montanari2024equivalenceapproximatemessagepassing}, showing that, for the (biased, i.i.d.) spiked matrix model within the class of GAMs, the MMSE performance of Approximate Message Passing (AMP) coincides with that of $O(1)$-degree polynomial estimators. Since the asymptotic MMSE of AMP is characterized by the replica-symmetric (RS) approximation of the FP potential, this result raised the hope that a broader and more systematic theory connecting FP-based predictions and low-degree estimation limits might exist. However, despite its conceptual appeal and exact nature, this correspondence is currently limited to biased i.i.d. rank-1 spiked matrix GAMs.

\paragraph{An equivalence for detection settings by another FP criterion}An important step in this direction was taken by Bandeira et al. \cite{bandeira2022franz}, who proved that for the detection variant of GAM—where the task is to distinguish between pure noise and a GAM—the power of low-degree polynomial tests is characterized by an area criterion involving the annealed FP potential defined in \eqref{eq:annealed}\footnote{Specifically, the criterion asks whether the area under the exponential of the annealed FP potential near the origin diverges.}. Notably, the use of the annealed FP potential is essential: for problems such as sparse tensor PCA, the corresponding area criterion applied to the quenched FP potential again fails to match the predictions of low-degree methods. More recently, \cite{chen2025an} significantly extended this connection beyond the class of GAMs.

While these results provide a partial resolution of our motivating question in the setting of detection, they leave open whether the area-based criterion aligns with the monotonicity-based intuition emphasized in the physics literature. It turns out that the monotonicity of the annealed FP curve is, in fact, not the right criterion for detection hardness and that happens for a fundamental reason. Many GAMs—including again sparse tensor PCA—exhibit a detection–estimation gap, in which detection is possible in polynomial time while estimation is conjectured to be computationally hard for low-degree polynomial methods. In such regimes, routine calculations can prove that the monotonicity criterion correctly predicts hardness for estimation but incorrectly predicts hardness for detection, whereas the area-based criterion correctly captures the detection low-degree hardness threshold. It is therefore suggestive that if one wants to create a relevant theory connecting the monotonicity of the FP potential with low-degree polynomials one needs to restrict themselves in the context of estimation.

\subsection{Contributions}

The main contribution of this work is the establishment of an exact quantitative relationship between the MMSE achievable by low-degree polynomial estimators and the monotonicity of the annealed Franz--Parisi (FP) potential for a broad class of Gaussian additive models (GAMs). We name the class of GAMs our results apply \emph{``low-order cumulant-nonnegative"} GAMs. This result yields the first rigorous equivalence between two a priori distinct hardness predictions: those arising from the monotonicity criterion of the annealed FP potential in statistical physics, and those obtained from low-degree MMSE lower bounds in theoretical statistics and computer science. In particular, our results provide the first formal setting in which the physics monotonicity criterion is shown to characterize low-degree estimation limits.

To state our results, recall that for any GAM with signal-to-noise ratio $\lambda>0$, prior $P_0$, and degree parameter $D\in\mathbb{Z}_{>0}$, the degree-$D$ MMSE and degree-$D$ correlation are defined as
\[
\mathrm{MMSE}^{\le D}_{P_0}(\lambda)
:= \inf_{f_1,\dots,f_n \in \mathbb{R}[Y]_{\le D}}
\mathbb{E}\!\left[\sum_{i=1}^n \big(f_i(Y)-X_i\big)^2\right], 
\mathrm{Corr}^{\le D}_{P_0}(\lambda)
:= \sqrt{\sum_{i=1}^n \sup_{\substack{g_i \in \mathbb{R}[Y]_{\le D}\\ \mathbb{E}[g_i(Y)^2]=1}}
\mathbb{E}\!\left[g_i(Y) X_i\right]^2},
\]
which are linked by the identity (see Lemma \ref{lem:vector-mmse-corr}),
$
\mathrm{MMSE}^{\le D}_{P_0}(\lambda)
= \mathbb{E}_{X\sim P_0}[\|X\|^2]
- \big(\mathrm{Corr}^{\le D}_{P_0}(\lambda)\big)^2.$

Our main result can be informally summarized as follows. For simplicity, we state it under the assumption that the overlap $\langle X,X' \rangle$, where $X,X'$ are i.i.d.\ from $P_0$, admits a continuous distribution with a differentiable PDF, so that derivatives of the annealed FP potential are well-defined.

\begin{theorem}[Informal, see Theorems \ref{thm:mainthm}, \ref{thm:corr-lower-bound}.]\label{thm:main_inf}
For any low-order cumulant-nonnegative GAM with prior $P_0$, for any degree
$0 < D \le \mathrm{poly}(\log n)$ and any SNR $\lambda>0$, the optimal degree-$D$ correlation satisfies
\[
\big(\mathrm{Corr}^{\le D}_{P_0}\big)^2
\!\left(\lambda + \left.\frac{d}{dq}\mathcal{F}_{\mathrm{ann},\lambda}(q)\right|_{q=q(D)}\right)
\approx q(D),
\]
where the approximation hides only polylogarithmic factors in $n$, and $q(D)$ denotes the $e^{-D}$-quantile of the overlap $|\langle X,X' \rangle|$ (see Definition \ref{def:quantile_ov}). 
\end{theorem}

A direct corollary provides the precise equivalence between the physics monotonicity criterion and low-degree estimation hardness.

\begin{corollary}[Informal]\label{cor:info}
For any ``low-order cumulant-nonnegative" GAM with prior $P_0$ and SNR $\lambda$, and any $0 < D \le \mathrm{poly}(\log n)$, the following are equivalent up to polylogarithmic factors in $n$:
\[
\left.\frac{d}{dq}\mathcal{F}_{\mathrm{ann},\lambda}(q)\right|_{q=q(D)} \ge 0
\quad \Longleftrightarrow \quad
\big(\mathrm{Corr}^{\le D}_{P_0}(\lambda)\big)^2 \le q(D).
\]
\end{corollary}
We briefly describe the correspondence at a conceptual level. In the classical physics picture based on the Franz--Parisi (FP) potential, the prediction is that a time-efficient algorithm—modeled by local or reversible dynamics—becomes trapped at the overlap value \( q_1 \) corresponding to the first local minimum of the FP potential, while an algorithm with unbounded computational resources eventually reaches the global minimum at overlap \( q_2 \).

Our results \emph{prove} a precise quantitative analogue of this picture is correct in the low-degree framework (see also Figure~\ref{fig:Figures}(a)). Suppose we quantify computational power by the maximum degree \( D = 1,2,\ldots \) of polynomial estimators one can use. For any SNR \( \lambda \), we examine the monotonicity of the annealed FP potential \( \mathcal{F}_{\mathrm{ann},\lambda} \) at the overlap values \( q(D) \).

If
$\left.\frac{d}{dq}\mathcal{F}_{\mathrm{ann},\lambda}(q)\right|_{q=q(D)} \le 0$
then the physics interpretation is that the dynamics \emph{descends a hill} and hence, this is locally an ``easy" phase. From Corollary \ref{cor:info}, this phase corresponds exactly to
$\big(\mathrm{Corr}^{\le D}_{P_0}(\lambda)\big)^2 \ge q(D),$
meaning that degree-\(D\) polynomials are powerful enough to achieve squared correlation at least \( q(D) \).

Conversely, if
$\left.\frac{d}{dq}\mathcal{F}_{\mathrm{ann},\lambda}(q)\right|_{q=q(D)} \ge 0,$
the physics prediction is that the dynamics must \emph{climb a hill} and therefore becomes trapped, signaling computational ``hardness". In exact correspondence, from Corollary \ref{cor:info}, this phase maps to $\big(\mathrm{Corr}^{\le D}_{P_0}(\lambda)\big)^2 \le q(D)$,
which expresses the inability of degree-\(D\) polynomials to achieve squared correlation \( q(D) \).

In this sense, the physics monotonicity prediction at overlap \( q(D) \) admits a complete and rigorous interpretation in the language of low-degree estimation. Moreover, the correspondence is bidirectional: the behavior of the low-degree correlation function as a function of \( D \) can be translated back into precise monotonicity properties of the annealed FP potential (see Figure~\ref{fig:Figures}(b)).

\begin{figure}[htbp]
\centering

\begin{minipage}[t]{0.49\linewidth}
  
  \includegraphics[width=\linewidth]{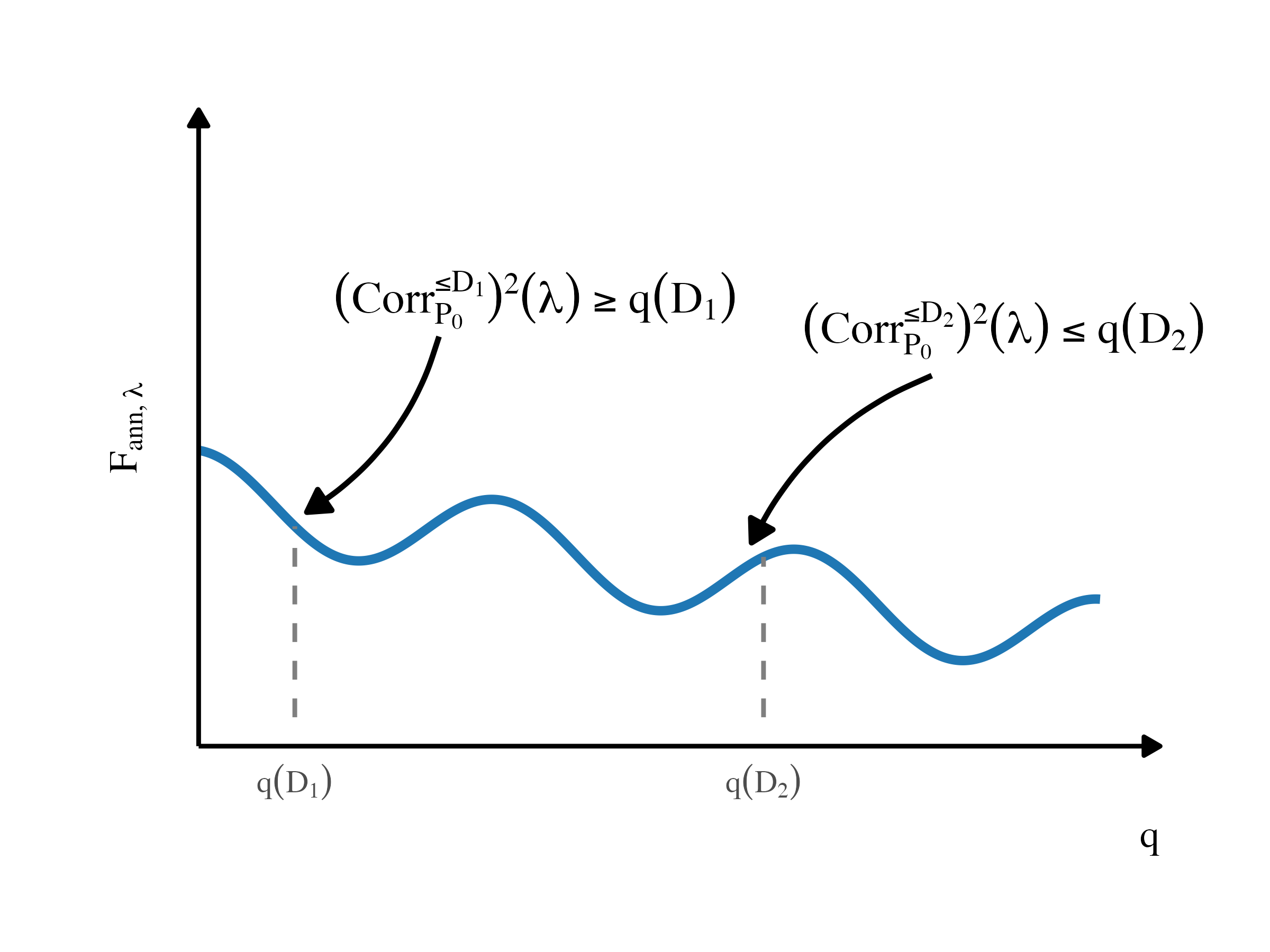}
  \par\small\textbf{(a)} This is a pictorial representation of the annealed potential $q \mapsto F_{\mathrm{ann},\lambda}(q)$ as a function of the overlap parameter $q$ (for fixed SNR $\lambda$).
  Based on Theorem \ref{thm:main_inf}, and choosing points $q=q(D_i), i=1,2$ where the curve is locally decreasing/increasing, we can conclude that $(\mathrm{Corr}_{P_0}^{\le D_i})^2(\lambda)$ is (roughly) more/less than $q(D_i)$ 
\end{minipage}\hfill
\begin{minipage}[t]{0.49\linewidth}
 
  \includegraphics[width=\linewidth]{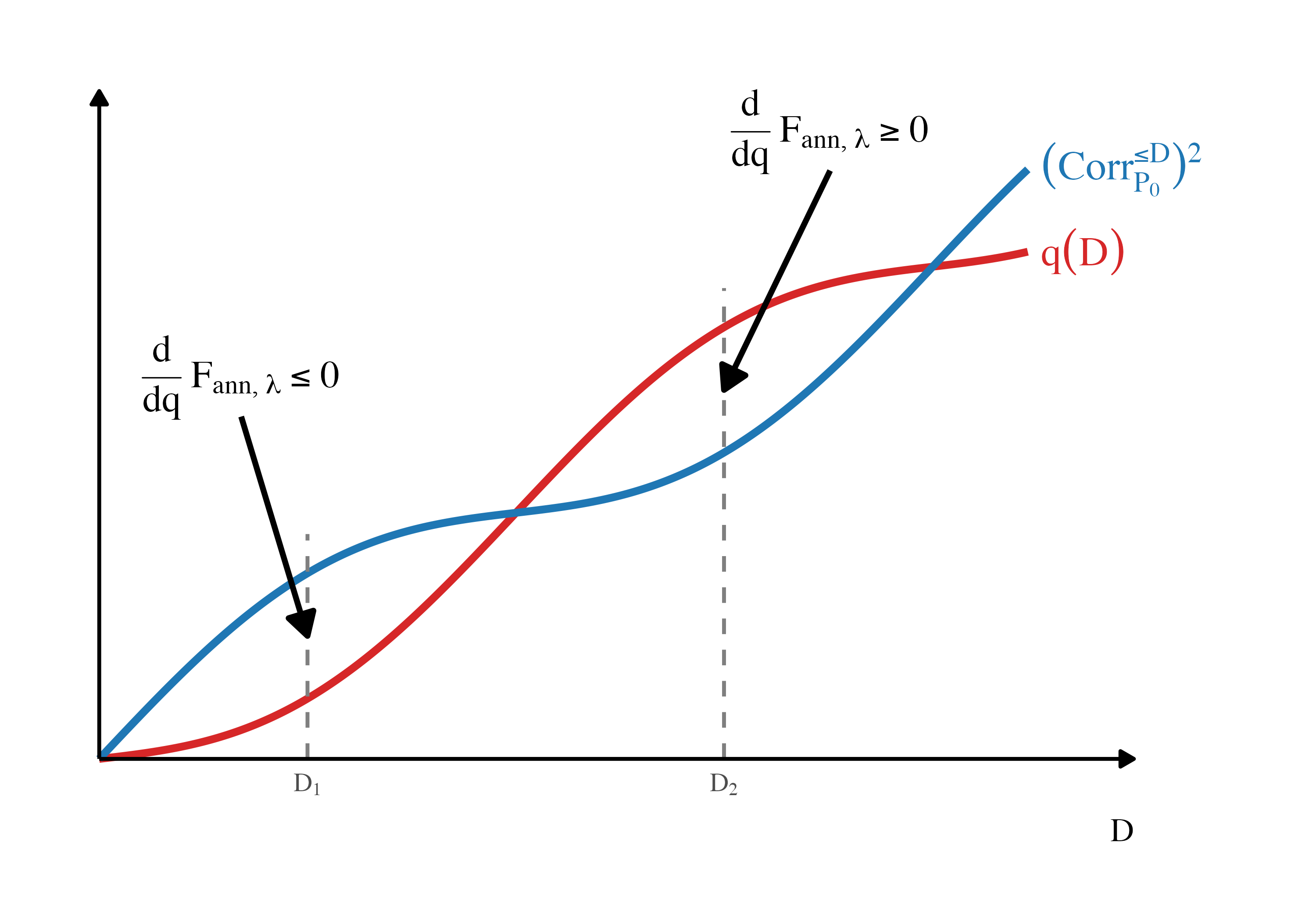}
  \par\small\textbf{(b)} This is a pictorial representation of $(\mathrm{Corr}_{P_0}^{\le D})^2(\lambda)$ and $q(D)$ as functions of $D$ (for fixed SNR $\lambda$).
   Based on Theorem \ref{thm:main_inf}, we can use the relative position of these two curves to conclude the sign of the derivative of $\mathcal{F}_{\mathrm{ann}}$. 
\end{minipage}

\caption{Pictorial representations of our equivalence.}
\label{fig:Figures}
\end{figure}

\begin{remark}
A key technical feature of our results is the role played by the $e^{-D}$-quantiles $q(D)$ of the overlap $|\langle X,X' \rangle|$. As mentioned above, these quantiles provide the exact parametrization required to map the overlap value at which the annealed FP potential is increasing/decreasing to the corresponding low-degree MMSE upper/lower bound. We emphasize that the importance of overlap quantiles in the annealed FP framework was first identified by \cite{bandeira2022franz}, where it was shown—roughly speaking—that boundedness of the area under $\exp(-\mathcal{F}_{\mathrm{ann},\lambda}(q))$ on the interval $[-q(D),q(D)]$ characterizes low-degree detection hardness.
\end{remark}

\begin{remark}
While Theorem \ref{thm:main_inf} is stated for continuous priors, we also establish an analogous result for discrete priors by replacing derivatives of the annealed FP potential with appropriate discrete difference operators.
\end{remark}

We defer the precise set of technical assumptions on the GAMs, and specifically the assumption on low-order cumulant non-negativity, to the main body of the paper. We note here, however, that these assumptions are satisfied by several canonical GAMs studied in the recent literature. As an illustration, we apply our main theorem to recover existing and prove new state-of-the-art low-degree MMSE lower bounds via direct, black-box applications of our framework. In all cases, the proofs reduce to bounding the overlap quantiles $q(D)$ of $|\langle X,X' \rangle|$.

In particular, we obtain from our method tight $O(\log n)$-degree MMSE hardness results below the conjectured algorithmic threshold for:
(a) tensor PCA with prior $X = x^{\otimes r}$, where the coordinates of $x$ are either i.i.d.\ Gaussian (Section~\ref{sec:statements_tpca}) or i.i.d.\ $\mathrm{Rad}(k/n)$ with $\widetilde{\omega}(\sqrt{n})=k=o(n)$ (Section~\ref{sec:statements_sptpca}); and
(b) the sparse clustering model with $X=\xi\mu^\top$, where $\xi$ has i.i.d.\ $\mathrm{Rad}(1/2)$ entries and $\mu$ has i.i.d.\ coordinates distributed as the product of a standard Gaussian and a $\mathrm{Ber}(s/p)$ random variable with $\widetilde{\omega}(\sqrt{p})=s=o(p)$ (Section~\ref{sec:proofs_spclust}).

Finally, we address the natural question of whether our results could be extended to the \textbf{quenched} FP potential; such an inquiry is well-motivated, because as we mentioned above the annealed FP potential is often treated in physics as a mere tractable proxy for its quenched counterpart (often the quenched FP potential is challenging to calculate). While such an extension is likely possible for many of the GAMs our work applies to, we must emphasize that there exist GAMs for which our equivalence holds, yet the monotonicity of the quenched FP potential is \textbf{not in agreement} with the behavior of the low-degree MMSE. More specifically, there are cases for which the monotonicity of the annealed FP potential successfully captures the low-degree MMSE behavior while the monotonicity of the quenched FP potential does not. We present and analyze a specific such counterexample in Section \ref{sec:quenched} (a truncated version of Rademacher sparse 3-tensor PCA). We believe that this unexpected ``computational success" of the annealed FP potential—where it outperforms the quenched potential in predicting algorithmic hardness—represents a significant open question for future work.

\section{Main Result: Equivalence between FP monotonicity and low-degree MMSE}\label{sec:equiv_cont}
As mentioned in the Introduction, of crucial importance to our work is the quantiles of the overlap between two i.i.d. copies from the prior $P_0$ of a GAM. To present our results, we start by defining the quantile function of the absolute value of the overlap $\langle X,X' \rangle$ between two i.i.d. draws of the prior.

\begin{definition}[Quantile function]\label{def:quantile_ov}
 Let $X,X'$ i.i.d. draws from the prior $P_0$. For any $D>0$, we define the quantile $q(D)$ by
$q(D) \coloneqq \inf \left\{ y \in \mathbb{R} \mid \mathbb{P}(|\langle X,X' \rangle| \leq y) \geq 1-e^{-D} \right\}.$
\end{definition}

For clarity reasons, we first present our equivalence results for the case $P_0$ admits a continuous and differentiable PDF in $\mathbb{R}^n.$ We then describe in the last subsection, Section \ref{sec:discrete}, the almost identical conclusions (via, in fact, almost identical proofs) in the case $P_0$ is a discrete distribution.

\subsection{Low Degree correlation upper bound from decreasing FP}

We prove an upper bound on the low-degree correlation from the monotonicity of annealed FP. We start by describing exactly the set of assumptions we make for the GAMs. We first need the following notion from probability theory.

\begin{definition}[Sub-Weibull$(\phi)$]\label{def:subweibull}[See e.g.,  Theorem 2.1 from  \cite{Vladimirova_2020}]
Fix $\phi>0$. A real-valued random variable $X$ is called sub-Weibull of order $\phi$ if there exist constants $K>0$ and $C_\phi\in(0,\infty)$ such that $\big(\mathbb{E}|X|^{p}\big)^{1/p}\ \le\ C_{\phi} K p^{1/\phi}$, 
for all $p\ge 1$.
\end{definition}

\begin{remark}[Examples of sub-Weibull random variables]\label{rem:weib}
The class of sub-Weibull random variables includes many familiar distributions. For instance:
(i) any bounded random variable is sub-Weibull$(\phi)$ for every $\phi>0$,
(ii) any sub-Gaussian random variable is sub-Weibull$(2)$, and
(iii) any sub-exponential random variable is sub-Weibull$(1)$.
\end{remark}

We study Gaussian additive models (GAMs) whose priors satisfy a \emph{low-order cumulant-nonnegative} property. This condition constitutes our main structural assumption, supplemented by several mild regularity assumptions.

\begin{assumption}[Low-order cumulant-nonnegatively GAMs]\label{assump:mainassumption}
Fix a universal constant \( C > 0 \). Let \( (D_n)_{n\in\mathbb{N}} \) be a sequence of positive integers satisfying \( D_n \le n^{C} \), and let \( (A_n)_{n\in\mathbb{N}} \) be a sequence of positive reals.

A GAM with prior \( P_0 \) satisfies our assumption with parameters \( (D_n)_{n\in\mathbb{N}} \), \( (A_n)_{n\in\mathbb{N}} \), and \( C>0 \) the following conditions hold.

\begin{enumerate}
\item[(1)] (\emph{$D_n$-order cumulant-nonnegative})  
For every multi-index \( \alpha \in \mathbb{N}^n \) with \( |\alpha| \le D_n \), the joint cumulant satisfies
\(
\kappa_{\alpha}(X_1,\dots,X_n) \ge 0.
\)

\item[(2)] (\emph{Polynomial growth of marginal moments})  
There exists a universal constant \( \phi>0 \) such that for all \( i \in [n] \), the marginal \( X_i \) is sub-Weibull\((\phi)\).

\item[(3)] (\emph{Controlled decay of overlap quantiles})  
There exists a universal constant \( C_1>0 \) such that for all \( n \),
$q\!\left(D_n(\log n)^2\right)
\;\ge\;
\max\!\left\{ (\log n)^{-D_n/2},\, n^{-C_1} \right\}.$

\item[(4)] (\emph{Quantile stability})  
The overlap distribution satisfies (here $ \mathbb{P}(\langle X,X'\rangle = t)$ should be understood as the PDF of $\langle X,X'\rangle$),
\begin{equation}\label{eq:quantiledecay}
\left.\frac{d}{dt}\log \mathbb{P}(\langle X,X'\rangle = t)\right|_{t=q(D_n)}
\;\ge\;
-\frac{A_n}{q(D_n(\log n)^2)\log n}.
\end{equation}
\end{enumerate}
\end{assumption}

Assumptions~(1)--(4) should be interpreted as consisting of one principal structural condition, Assumption~(1), and three auxiliary regularity conditions that are satisfied by most canonical priors considered in the GAM literature.

\begin{remark}[On Assumption~(1)]
Assumption~(1) is the main technical condition underpinning our equivalence, and in particular our low-degree correlation upper bounds. Theorem \ref{thm:mainthm} (the low-degree correlation lower bounds per Theorem \ref{thm:corr-lower-bound} apply without it). 

Many natural priors satisfy this condition. For instance, the tensor prior \( X = x^{\otimes p} \) satisfies Assumption~(1), (a) for all \( D>0 \) when \( x \) has i.i.d.\ Gaussian entries, and (b) for \( D = O(\log n) \) when \( x \) has i.i.d.\ Bernoulli--Rademacher entries with sparsity \( k/n \). In Section~\ref{sec:positivity_cumulants}, we show that a substantially broader class of priors satisfies this condition, including for example the prior used for Gaussian sparse clustering \cite{even2025computational}.

\end{remark}

Assumption~(2) is standard and mild; see Remark~\ref{rem:weib}.  
Assumption~(3) is also weak: for \( D_n = \omega(\log n / \log\log n) \), it reduces to the condition \( q(D_n) \ge n^{-C} \) for some universal constant \( C>0 \), which holds for most priors of interest.

\begin{remark}[On Assumption~(4)]\label{rem:assump_quantile}
Although Assumption~(4) may appear technical, it is mild in most relevant settings. In many GAMs, the overlap \( \langle X,X' \rangle \) is centered and composed of weakly dependent terms—for example, in tensor PCA of any finite order with i.i.d.\ coordinates, or in sparse tensor PCA.

In such cases, after appropriate normalization, the overlap converges in distribution to a polynomial of a standard Gaussian: \(
\langle X,X' \rangle / a_n \xrightarrow{d}  Z^p,
\ Z \sim \mathcal{N}(0,1),
\) for some fixed \( p \in \mathbb{N} \).
Consequently, for any \( D_n = \mathrm{polylog}(n) \), the quantiles behave as follows (explaining the name ``quantile stability" for the assumption)
\(q(D_n)=\Theta(a_n D_n^{p/2}).\)
Moreover, via a similar local central limit theorem argument, one naturally expects also $-\log \mathbb{P}(\langle X,X'\rangle = t)$ to be close to $(t/a_n)^{2/p}/2$, which together suggest
\[
-\left.\frac{d}{dt}\log \mathbb{P}(\langle X,X'\rangle = t)\right|_{t=q(D_n)}
=\Theta\left(
\frac{D_n}{q(D_n)}\right),
\]
which is consistent with Assumption~(4) (where some additional slack $A_n$, which should be treated as polylogarithmic, and other slack polylogarithmic factors are allowed).
\end{remark}

We are now in a position to state and prove our first one-sided equivalence result
\begin{theorem}[Increasing FP potential implies Low-Degree hard]\label{thm:mainthm}
Suppose we have a Gaussian Additive Model with prior $P_0$ that satisfies Assumption \ref{assump:mainassumption} and let for each $n$, $D_n, D'_n$ positive integers with $D_n'=(1+o(1))D_n\log ^2n$. 

For any $u>0$ satisfying $u +\frac{d}{dq}\mathcal{F}_{\mathrm{ann}, u}\bigg{|}_{q=q(D_n)}>0$ we have 
\begin{equation}
\mathrm{Corr}^{\leq D_n}_{P_0}\left(\frac{1}{A_n}\cdot \left(u+\frac{\d}{\d q}\mathcal{F}_{\mathrm{ann}, u}|_{q=q(D_n)}\right)\right)^2\leq  2q(D'_n).
\end{equation}In particular, for any SNR $\lambda=\lambda_n>0$, if $n$ is large enough, we have that $\frac{d}{dq}\mathcal{F}_{\mathrm{ann}, \lambda}|_{q=q(D_n)} \geq 0$ implies $\mathrm{Corr}^{\leq D_n}_{P_0}\left(\frac{1}{A_n} \lambda \right)^2 \leq 2q(D'_n).$
\end{theorem}

The proof of the theorem is deferred to Section \ref{sec:pf_UB}; see also Section \ref{sec:intuition} for a proof sketch.

\subsection{Low Degree correlation lower bound from increasing FP} 

We now turn to the other side of the equivalence. This time, this works under only mild assumptions for the GAM, and in particular, no assumption on cumulant-nonnegativity is needed here.

Before we proceed, we first state the assumptions that will be used in this section.

\begin{assumption}\label{assump:lower-bound}
Fix constants $c,c_1,c_2,C,\kappa>0$,  a sequence of positive integers $(D_n)_{n\in \N}$ such that $D_n = \omega(1)$ and a sequence of positive numbers $ (B_n)_{n\in \N}$ such that $B_n = \omega(\max\{ n^{-C}, e^{-C D_n}\})$. A GAM  with prior $P_0$ satisfies the assumption for parameters $(D_n)_{n \in \mathbb{N}}, (B_n)_{n \in \mathbb{N}}$ and $c,c_1,c_2,C,\kappa>0$ if the following conditions hold.

\begin{enumerate}
    \item[(1)] (Fixed quantile order) For all $D>0,$ $c_1 B_n D^{\kappa} \leq q(D) \leq c_2 B_n D^{\kappa}$.
    %\item $\E_{X,X'\sim P_0, X\perp\!\!\!\perp X'}[\langle X,X'\rangle] \ge 0$; \textcolor{red}{KT: that should be removed right? }
    \item[(2)] (Quantile stability) The PDF of the overlap satisfies,
    $
    \left.\frac{d}{dt}\log \P (\< X,X'\>=t)\right|_{t=q(D_n)}\leq -\frac{c}{q(D_n)}.$
\end{enumerate}

\end{assumption}

Both the assumptions (1), (2) are rather mild and satisfied by most GAMs in the literature. The motivation behind them in fact lies entirely on what is described in Remark \ref{rem:assump_quantile} on Assumption (4) from Assumption \ref{assump:mainassumption} and we direct the reader there for the details. Interestingly, notice that for this section we need to assume the other inequality direction from Assumption (4). 

We are now ready to state the main result of this direction.
\begin{theorem}[Decreasing FP potential implies Low-Degree easy]\label{thm:corr-lower-bound}
For any GAM satisfying Assumption \ref{assump:lower-bound} for sequence $ (B_n)_{n\in \N}$, $ (D_n)_{n\in \N}$ and constants $c,c_1,c_2,C,\kappa>0$, there exists a constant $C'>0$, such that the following holds. For any $u>0$ satisfying $u +\frac{d}{dq}\mathcal{F}_{\mathrm{ann}, u}\bigg{|}_{q=q(D_n)}>0$ we have 
\[(\mathrm{Corr}^{\leq D_n}_{P_0})^2\left({u +\frac{d}{dq}\mathcal{F}_{\mathrm{ann}, u}\bigg{|}_{q=q(D_n)}}\right) \geq \frac{C'}{D_n^{2\kappa}} {q(D_n)}.\]In particular, for any SNR $\lambda>0$, $(\mathrm{Corr}^{\leq D_n}_{P_0})^2(\lambda) \leq  \frac{C'}{D_n^{2\kappa}} {q(D_n)}$ implies $\frac{d}{dq}\mathcal{F}_{\mathrm{ann}, \lambda}'\bigg{|}_{q=q(D_n)} \geq 0.$
\end{theorem}

The proof of the theorem is deferred to Section \ref{sec:pf_LB}; see also Section \ref{sec:intuition} for a proof sketch.

\subsection{The equivalence for discrete priors}\label{sec:discrete}

Here we state the equivalence results for discrete prior $P_0$. The sole difference is that the derivative of $\log(\Pr(\< X,X'\>=q))$ is not well-defined anymore, where now $\Pr(\< X,X'\>=q)$ is a probability mass function (PMF) and instead consider a discrete derivative/difference operator. Importantly, modulo this difference, all statements, assumptions, and proofs of this section are identical to the ones in Section \ref{sec:equiv_cont}. More specifically, the sole distinction in the proofs is that to apply these Theorems, one needs to check the assumptions in terms of this discrete now derivative operator.
 
Let us denote for simplicity $f_{\mathrm{ov}}(q)=\Pr(\< X,X'\>=q)$ the PMF of the overlap for two independent draws from our prior. To define the discrete derivative operator we need to be careful as $f_{\mathrm{ov}}(q)$ is supported only on a set of discrete values of $q$. Let $\mathcal{S}_{\mathrm{ov}}=\{q \in \mathbb{R}: f_{\mathrm{ov}}(q)>0\}$.  Now, let any sequence of functions $a_n: \mathcal{S}_{\mathrm{ov}} \rightarrow \mathbb{R}_{>0}, n \in \mathbb{N}$ such that for all $q \in \mathcal{S}_{\mathrm{ov}}$, $q+a_n(q) \in \mathcal{S}_{\mathrm{ov}}$.  Given this choice of $a_n, n \in \mathbb{N}$ we define the discrete derivative operator $\mathbf{\Delta}_{a_n}$ on functions $f: \mathcal{S}_{\mathrm{ov}} \rightarrow \mathbb{R}$ such that \[\mathbf{\Delta}_{a_n}(f)(q)=\frac{f(q+\alpha_n(q))-f(q)}{a_n(q)}, q \in \mathcal{S}_{\mathrm{ov}}. \]

\subsubsection{Low Degree correlation upper bound from decreasing FP (discrete case)}

Similar to the continuous case to prove this result we need some assumptions, which is identical to Assumption \ref{assump:mainassumption} by switching the derivative with the discrete derivative operator.

\begin{assumption}[Low-order cumulant-nonnegatively GAMs (discrete)]\label{assump:mainassumption_discrete}
Fix a universal constant \( C > 0 \). Let \( (D_n)_{n\in\mathbb{N}} \) be a sequence of positive integers satisfying \( D_n \le n^{C} \), let \( (A_n)_{n\in\mathbb{N}} \) be a sequence of positive reals, and let any sequence of ``speed" functions $a_n: \mathcal{S}_{\mathrm{ov}} \rightarrow \mathbb{R}_{>0}, n \in \mathbb{N}$ such that for all $q \in \mathcal{S}_{\mathrm{ov}}$, $q+a_n(q) \in \mathcal{S}_{\mathrm{ov}}$. 

A GAM with prior \( P_0 \) is assumed to satisfy with parameters \( (D_n)_{n\in\mathbb{N}} \), \( (A_n)_{n\in\mathbb{N}} \), \( (a_n)_{n\in\mathbb{N}} \) and \( C>0 \) the following conditions hold.

\begin{enumerate}
\item[(1)] (\emph{$D_n$-order cumulant-nonnegative})  
For every multi-index \( \alpha \in \mathbb{N}^n \) with \( |\alpha| \le D_n \), the joint cumulant satisfies
\begin{equation}\label{assump:positivity of cumulants}
\kappa_{\alpha}(X_1,\dots,X_n) \ge 0.
\end{equation}

\item[(2)] (\emph{Polynomial growth of marginal moments})  
There exists a universal constant \( \phi>0 \) such that for all \( i \in [n] \), the marginal \( X_i \) is sub-Weibull\((\phi)\).

\item[(3)] (\emph{Controlled decay of overlap quantiles})  
There exists a universal constant \( C_1>0 \) such that for all \( n \),
$q\!\left(D_n(\log n)^2\right)
\;\ge\;
\max\!\left\{ (\log n)^{-D_n/2},\, n^{-C_1} \right\}.$

\item[(4)] (\emph{Quantile stability})  
The overlap distribution satisfies (here $ \mathbb{P}(\langle X,X'\rangle = t)$ should be understood as the PMF of $\langle X,X'\rangle$),
\begin{equation}\label{eq:quantiledecay-disc}
\left.\mathbf{\Delta}_{a_n}\log \mathbb{P}(\langle X,X'\rangle = t)\right|_{t=q(D_n)}
\;\ge\;
-\frac{A_n}{q(D_n(\log n)^2)\log n}.
\end{equation}
\end{enumerate}
\end{assumption}

Given this assumption, we obtain the following result in the discrete case.
\begin{theorem}[Increasing FP potential implies Low-Degree hard (discrete)]\label{thm:mainthm_disc}
Suppose we have a Gaussian Additive Model with prior $P_0$ that satisfies Assumption \ref{assump:mainassumption_discrete} and let for each $n$, $D_n, D'_n$ positive integers with $D_n'=(1+o(1))D_n\log ^2n$. 

For any $u>0$ satisfying $u +\mathbf{\Delta}_{a_n}\mathcal{F}_{\mathrm{ann}, u}\bigg{|}_{q=q(D_n)}>0$ we have 
\begin{equation}
\mathrm{Corr}^{\leq D_n}_{P_0}\left(\frac{1}{A_n}\cdot \left(u+\mathbf{\Delta}_{a_n}\mathcal{F}_{\mathrm{ann}, u}|_{q=q(D_n)}\right)\right)^2\leq  2q(D'_n).
\end{equation}In particular, for any SNR $\lambda=\lambda_n>0$, if $n$ is large enough, we have that $\mathbf{\Delta}_{a_n}\mathcal{F}_{\mathrm{ann}, \lambda}|_{q=q(D_n)} \geq 0$ implies $\mathrm{Corr}^{\leq D_n}_{P_0}\left(\frac{1}{A_n} \lambda \right)^2 \leq 2q(D'_n).$
\end{theorem}
The proof is identical to the proof of Theorem \ref{thm:mainthm} with swapping the derivative operator and the discrete derivative operator, and is omitted.

\subsubsection{Low Degree correlation lower bound from increasing FP (discrete case)} 

Now the assumption for the positive result is changed as follows.

\begin{assumption}\label{assump:lower-bound_disc}
Fix constants $c,c_1,c_2,C,\kappa>0$, a sequence of positive integers $(D_n)_{n\in \N}$ such that $D_n = \omega(1)$, a sequence of positive numbers $ (B_n)_{n\in \N}$ such that $B_n = \omega(\max\{ n^{-C}, e^{-C D_n}\})$,  and let any sequence of ``speed" functions $a_n: \mathcal{S}_{\mathrm{ov}} \rightarrow \mathbb{R}_{>0}, n \in \mathbb{N}$ such that for all $q \in \mathcal{S}_{\mathrm{ov}}$, $q+a_n(q) \in \mathcal{S}_{\mathrm{ov}}$. A GAM  with prior $P_0$ is assumed to satisfy for parameters $(D_n)_{n \in \mathbb{N}}, (B_n)_{n \in \mathbb{N}}, (a_n)_{n \in \mathbb{N}}$ and $c,c_1,c_2,C,\kappa>0$ if the following conditions hold.

\begin{enumerate}
    \item[(1)] (Fixed quantile order) For all $D>0,$ $c_1 B_n D^{\kappa} \leq q(D) \leq c_2 B_n D^{\kappa}$.
    %\item $\E_{X,X'\sim P_0, X\perp\!\!\!\perp X'}[\langle X,X'\rangle] \ge 0$; \textcolor{red}{KT: that should be removed right? }
    \item[(2)] (Quantile stability) The PMF of the overlap satisfies,
    
    \[\left.\mathbf{\Delta}_{a_n}\log \P (\< X,X'\>=t)\right|_{t=q(D_n)}\leq -\frac{c}{q(D_n)}.\]
\end{enumerate}

\end{assumption}

We are now ready to state the main result of this direction in the discrete case as well.
\begin{theorem}[Decreasing FP potential implies Low-Degree easy (discrete)]\label{thm:corr-lower-bound_discrete}
For any GAM satisfying Assumption \ref{assump:lower-bound_disc} for sequence $ (B_n)_{n\in \N}$, $ (D_n)_{n\in \N}$, $ (a_n)_{n\in \N}$ and constants $c,c_1,c_2,C,\kappa>0$, there exists a constant $C'>0$, such that the following holds. For any $u>0$ satisfying $u +\mathbf{\Delta}_{a_n}\mathcal{F}_{\mathrm{ann}, u}\bigg{|}_{q=q(D_n)}>0$ we have 
\[(\mathrm{Corr}^{\leq D_n}_{P_0})^2\left({u +\mathbf{\Delta}_{a_n}\mathcal{F}_{\mathrm{ann}, u}\bigg{|}_{q=q(D_n)}}\right) \geq \frac{C'}{D_n^{2\kappa}} {q(D_n)}.\]In particular, for any SNR $\lambda>0$, $(\mathrm{Corr}^{\leq D_n}_{P_0})^2(\lambda) \leq  \frac{C'}{D_n^{2\kappa}} {q(D_n)}$ implies $\mathbf{\Delta}_{a_n}\mathcal{F}_{\mathrm{ann}, \lambda}'\bigg{|}_{q=q(D_n)} \geq 0.$
\end{theorem}Similar to above, the proof is identical to the proof of Theorem \ref{thm:corr-lower-bound} with swapping the derivative operator and the discrete derivative operator, and is omitted.

\subsection{Proof ideas}\label{sec:intuition}

While the proofs of the equivalence Theorems are deferred to later sections, we briefly outline here the main ideas underlying them, assuming for simplicity that the overlap distribution admits a continuous and differentiable density.

A key starting point is the following elementary identity, which significantly clarifies the role of the annealed FP potential in GAMs: for any SNR \( \lambda > 0 \) and all \( q \in [-1,1] \), we have that
$\lambda q + \mathcal{F}_{\mathrm{ann},\lambda}(q)
= - \log \mathbb{P}(\langle X,X' \rangle = q).$
In particular, differentiating at \( q = q(D) \) yields
\begin{equation}\label{eq:fann_gams_main}
\lambda + \left.\frac{d}{dq}\mathcal{F}_{\mathrm{ann},\lambda}(q)\right|_{q=q(D)}
=
- \left.\frac{d}{dq}\log \mathbb{P}(\langle X,X' \rangle = q)\right|_{q=q(D)} .
\end{equation}

Under Assumption~\ref{assump:mainassumption} together with Assumption~\ref{assump:lower-bound}, and for degrees \( D = \mathrm{polylog}(n) \), the right-hand side of~\eqref{eq:fann_gams_main} can be shown to satisfy $ - \left.\frac{d}{dq}\log \mathbb{P}(\langle X,X' \rangle = q)\right|_{q=q(D)}
\;\approx\; \frac{1}{q(D)}.$
Combining the last two steps, establishing our main theorems in fact interestingly reduces to showing that, for all such GAMs of interest and degrees \( D=\mathrm{polylog}(n) \), $\big(\mathrm{Corr}_{P_0}^{\le D}\big)^2\!\left(\frac{1}{q(D)}\right) \;\approx\; q(D),$
that is, that the low-degree correlation satisfies an approximate fixed-point relation at the scale \( q(D) \).

We prove the two inequalities separately. For the upper bound
$
\big(\mathrm{Corr}_{P_0}^{\le D}\big)^2\!\left(\frac{1}{q(D)}\right) \le q(D)$,
we rely on a recent result of~\cite{Schramm_2022}, often referred to as \emph{Jensen's trick}, which allows one to upper bound the low-degree correlation in terms of the cumulants of the prior \( P_0 \). This reduces the problem to verifying an inequality relating cumulants of \( P_0 \) to quantiles of the overlap \( \langle X,X' \rangle \). The proof exploits the fact that overlap quantiles can be controlled via the log-moment generating function, which in turn admits a Taylor expansion in terms of cumulants. The details are rather delicate and the whole proof is provided in Section \ref{sec:pf_UB}.

For the matching lower bound,
$\left(\mathrm{Corr}_{P_0}^{\le D}\right)^2\left(\frac{1}{q(D)}\right) \ge q(D),$
we explicitly construct a degree-$D$ estimator achieving a squared correlation of at least $q(D)$ under mild conditions on the GAM. A natural approach is to analyze the optimal degree-$D$ polynomial, which corresponds to the projection of the posterior mean onto the linear space of degree-$D$ polynomials under the marginal measure of $Y$.

Since $Y$ does not follow a product measure, and therefore no canonical basis of orthonormal polynomials is known, the analysis of this projection is considered often in the literature as ``technically prohibitive"—even for simple GAMs—which is why researchers frequently analyze the low-degree MMSE via alternative means. In this work, we design a close proxy to this optimal low-degree projection of the posterior mean which is tractable to define and analyze \emph{for all GAMs} we consider in this work, and could be of independent interest. This analysis yields the desired tight lower bound on the low-degree correlation. 

The construction is based on an importance-sampling scheme: we draw $M$ independent samples $X_1, \ldots, X_M$ from the prior $P_0$ and consider a polynomial estimator of the form
\begin{equation}
    p(Y) = \sum_{i=1}^M W(Y \mid X_i) X_i,
\end{equation}
where the weights $W(Y \mid X_i)=W_D(Y \mid X_i)$ are carefully chosen degree-$D$ polynomials designed to maximize correlation with the signal. The motivation behind this choice of weights is for $p(Y)$ to approximate the low-degree projection of the posterior mean. The key insight is to treat the Hermite polynomials as if they form an orthogonal basis under the measure of $Y$ and project onto them accordingly. This leads to a canonical choice of the weights $W=W_D$ (see \eqref{eq:weights}) and enables a tractable analysis that establishes the desired lower bound, completing the proof of the approximate fixed-point relation. We note that the idea of using the Hermite/Fourier-Walsh  basis as almost orthonormal basis under the ``planted" $Y$ for certain models has recently appeared in the low-degree MMSE literature \cite{carpentier2025low} (in \cite{carpentier2025low} the focus is not on GAMs but on some random graph models). The technical details of this construction are subtle, also in this case, and are provided in Section \ref{sec:pf_LB}.

\section{Applications}

To present the applicability of our equivalence we prove that our results can prove some old and new state-of-the-art MMSE lower bounds in the recent literature by (1) directly calculating the quantiles of the overlap of the prior, (2) proving the derivative of the annealed FP potential is nonnegative (hence, ``physics-hard"), and (3) applying Theorem \ref{thm:mainthm} (or, Theorem \ref{thm:mainthm_disc} in the discrete case) to conclude its low-degree MMSE hard.

We highlight that in these proof the technical work lies on checking Assumption \ref{assump:mainassumption} for Theorem \ref{thm:mainthm}. Among them, all parts are relatively easy to check with the sole exception of Assumption (4) which sometimes is more challenging. To verify Assumption (4), we prove several local CLT theorems for the density of the overlap which can be seen as a technical contribution of potential independent interest. All proofs are deferred to Section \ref{sec:pf_applications}.
\subsection{Tensor PCA models}
In this model, the prior $P_0$ is of the form $X=\mathrm{vec}(v^{\otimes r})$ for some constant $r>0$, some distribution on $v \in \mathbb{R}^N$ and $n=N^r$. In particular, notice that under such prior for any $X=\mathrm{vec}(v^{\otimes r})$, $X'=\mathrm{vec}((v')^{\otimes r})$, it holds $\<X,X'\>=\<\mathrm{vec}(v^{\otimes r}), \mathrm{vec}(v'^{\otimes r})\>=\<v,v'\>^r.$

\subsubsection{Tensor PCA (Gaussian prior)}\label{sec:statements_tpca}

In this model $v\in \R^n$ has i.i.d. $\mathcal{N}(0,1)$ entries. The statistical threshold for this model is $\lambda_{\mathrm{IT}}=\widetilde{\Theta}(n^{1-r})$ \cite{montanari2014statisticalmodeltensorpca} and the algorithmic threshold is conjectured to be  $\lambda_{\mathrm{ALG}}=\widetilde{\Theta}(n^{-r/2})$ (see e.g., \cite{Hopkins_t_PCA} and references therein). 

It turns out that Theorem \ref{thm:mainthm} applies for the tensor PCA model for any $r \geq 1$, as the prior satisfies Assumption \ref{assump:mainassumption} for any $D_n= \widetilde{o}(\sqrt{n})$, Weibull constant $\phi=2$ and $A_n=\Theta(D_n^{3/2}\log^{r+1} n)$. In particular, Theorem \ref{thm:mainthm} implies the following result. 

\begin{theorem}\label{thm:tpca}
Fix $r \geq 2$. For the Gaussian prior in the rank one tensor spike model we prove that for all $\lambda >0$ and for any sequence of integers $D_n=o(\sqrt{n}/\log^2 n)$: 
\begin{align*}
\mathrm{MMSE}^{\le D_n}_X  \left(\frac{1}{D_n^{3/2}\log ^{r+1}n}\big(\lambda+\frac{\d}{\d q}\mathcal{F}_{\mathrm{ann}, \lambda))}\bigg{|}_{q=q(D_n)}\big)\right)= n^r - \widetilde{O}\big((D_n n)^{r/2}\big)= (1+o(1))\mathrm{MMSE}^{\mathrm{trivial}}_X. 
\end{align*}
In particular, for some $\lambda=\widetilde{\Theta}(n^{-r/2})=\widetilde{\Theta}(\lambda_{\mathrm{ALG}}),$ we have $\frac{\d}{\d q}\mathcal{F}_{\mathrm{ann}, \lambda}\bigg{|}_{q(D_n)} \geq 0$, and therefore\\$\mathrm{MMSE}^{\le D_n}_X  \left(\widetilde{\Theta}\Big(\frac{\lambda_{\mathrm{ALG}}}{D_n^{r/2}}\Big)\right) \geq (1+o(1))\mathrm{MMSE}^{\mathrm{trivial}}_X.$
\end{theorem}The proof of the theorem is deferred to Section \ref{sec:tpca}. 

A low-degree MMSE lower bound for the tensor PCA setting (with i.i.d. ``dense" prior) has also been attained in \cite{kunisky2024tensorcumulantsstatisticalinference} for the Rademacher prior, using techniques from free probability theory, but only for $r$ odd. Our approach is able to prove the low-degree MMSE lower bound for all $r$ for the Gaussian prior.

\subsubsection{Sparse Tensor PCA (Rademacher sparse prior)}\label{sec:statements_sptpca}

We now move to another Tensor PCA setting with a very well-studied discrete prior. Let $k \in \mathbb{N}$ with $k=n^{\beta+o(1)}$ for some $\beta \in (1/2,1)$. Then, the sparse (Rademacher) tensor PCA model, is the tensor PCA model where $v\in \R^n$ has i.i.d. entries $\mathrm{Rad}(k/n),$ meaning that for all $i=1,\ldots,N$, $v_i=1$ with probability $k/(2n)$, $v_i=-1$ with probability $k/(2n)$ and $v_i=0$, otherwise. The statistical threshold for this model is $\lambda_{\mathrm{IT}}=\widetilde{\Theta}(k^{1-r})$ \cite{luo2023tensorclusteringplantedstructures} and the algorithmic threshold of this model is widely conjectured to be  $\lambda_{\mathrm{ALG}}=\widetilde{\Theta}(\min \{ 1, n^{r/2}/k^r\})$ \cite{luo2023tensorclusteringplantedstructures}.

It turns out that the discrete version of Theorem \ref{thm:mainthm} (Theorem \ref{thm:mainthm_disc} in the Appendix) applies for the sparse tensor PCA model for any $r \geq 1, \beta \in (0,1/2)\cup (1/2,1)$, as the sparse Rademacher prior satisfies the discrete version of Assumption \ref{assump:mainassumption} (specifically, Assumption \ref{assump:mainassumption_discrete}) for $D_n = \lfloor (1-\beta) \log n \rfloor$, and $A_n=\Theta(D_n^{2-r}\log^{(3r+2)/2}n)$. In particular, the discrete version of Theorem \ref{thm:mainthm} implies the following result (below $\mathbf{\Delta}_{a_n}$ corresponds to the discrete derivative operator, see Section \ref{sec:discrete} for details). 

\begin{theorem}\label{thm:sptpca}
Fix $r \geq 2$. For the sparse tensor PCA model for all SNR $\lambda >0$ and $D_n = \lfloor (1-\beta) \log n \rfloor$, we have the following.

If $\beta\in (\frac{1}{2},1)$, then
\begin{align*}
 \mathrm{MMSE}^{\le D_n}  \left(\frac{1}{D_n^{r-2}\log^{(3r+2)/2}n}\cdot \big(\lambda+\mathbf{\Delta}_{a_n}\mathcal{F}_{\mathrm{ann}, \lambda}\bigg{|}_{q=q(D_n)}\big)\right)&=k^r- \widetilde{O}\bigg(\frac{k^r}{n^{r/2}}\bigg)=(1+o(1))\mathrm{MMSE}^{\mathrm{trivial}} .   
\end{align*}In particular, since for some $\lambda=\widetilde{\Theta}( n^{r/2}/k^r)=\widetilde{\Theta}(\lambda_{\mathrm{ALG}}),$ $\mathbf{\Delta}_{a_n}\mathcal{F}_{\mathrm{ann}, \lambda}\bigg{|}_{q(D_n)} \geq 0$, we conclude 

\[\mathrm{MMSE}^{\le D_n} \left(\widetilde{\Theta}(\lambda_{\mathrm{ALG}})\right) \geq (1+o(1))\mathrm{MMSE}^{\mathrm{trivial}}.\] 
If $\beta\in (0, \frac{1}{2})$, then

\begin{align*}
 \mathrm{MMSE}^{\le D_n}  \left(\frac{1}{\log^{3}n}\cdot \big(\lambda+\mathbf{\Delta}_{a_n}\mathcal{F}_{\mathrm{ann}, \lambda}\bigg{|}_{q=q(D_n)}\big)\right)&=k^r- \widetilde{O}(1)=(1+o(1))\mathrm{MMSE}^{\mathrm{trivial}} .   
\end{align*}In particular, since for some $\lambda=\widetilde{\Theta}( 1)=\widetilde{\Theta}(\lambda_{\mathrm{ALG}}),$ $\mathbf{\Delta}_{a_n}\mathcal{F}_{\mathrm{ann}, \lambda}\bigg{|}_{q(D_n)} \geq 0$, we conclude 

\[\mathrm{MMSE}^{\le D_n} \left(\widetilde{\Theta}(\lambda_{\mathrm{ALG}})\right) \geq (1+o(1))\mathrm{MMSE}^{\mathrm{trivial}}.\] 
  
\end{theorem}The proof of the theorem deferred to Section \ref{sec:sptpca}.

To the best of our knowledge, this is the first tight low-degree MMSE for this setting. For this model, there exists some estimation lower bounds against the support recovery question \cite{luo2023tensorclusteringplantedstructures} which concerns the recovery of the support of $v,$ but (a) this appears not applicable to recovering the signed vector $x$ and (b) it is coupled with a reduction-based hardness argument rather than a direct low-degree MMSE lower bound.

\subsection{Gaussian-Mixture model}
In the Gaussian Mixture Model with two centers we observe $n$ data points in $\mathbb{R}^p$ each one of which is either a draw from $\mathcal{N}(\mu_1,\sigma^2I_p)$ or $\mathcal{N}(\mu_2,\sigma^2I_p)$ with equal probability, where $\mu_1,\mu_2\in \mathbb{R}^p$ are the two ``centers". We collect these $n$ observations as the rows of a matrix
$Y\in\mathbb{R}^{n\times p}$. The canonical parameter for this model $\Delta^2$ is proportional to the minimum separation between the centers, i.e., $ \Delta^2:=\frac{\|\mu_1-\mu_2\|^2_2}{2\sigma^2}.$

We focus on the balanced case, where the two centers are $\mu$ and $-\mu$, for some $\mu \in \mathbb{R}^p,$ and without loss of generality assume $\sigma=1$. Equivalently,
let latent labels $\xi_i\in\{+1,-1\}$ be i.i.d.\ with $\mathbb{P}(\xi_i=+1)=\mathbb{P}(\xi_i=-1)=1/2$,
and assume $Y_i = \xi_i\,\mu + \, z_i$, for $z_i\stackrel{i.i.d.}{\sim}\mathcal{N}(0,I_p)$. In that case this maps exactly to the GAM setting in $n \times p$-dimensions $Y=X+Z$ setting for $\xi=(\xi_1,\dots,\xi_n)^\top$ and $X=\mathrm{vec}(\xi \mu^{\top})$ and
$Y$ is the vectorized $n\times p$ matrix with rows $Y_i\in\mathbb{R}^p, i=1,\ldots,n$. Moreover, in this case $\Delta^2=2\|\mu\|^2_2.$

\subsubsection{Sparse Clustering}\label{sec:proofs_spclust}
In sparse clustering, the mean $\mu$ is assumed to be sparse. Specifically, 
there exists an unknown set $J\subseteq[p]$ with $|J|\le s$ such that
$\mu_j=0$ for all $j\notin J$. We will focus on the case $s=p^{\alpha+o(1)}$, where $\alpha \in (1/2,1).$ In the regime we are interested in the model is conjectured to have an algorithmic threshold, in our scaling, at $\Delta^2=\Delta^2_{\mathrm{ALG}}=\Theta(p/n)$ \cite{löffler2021computationallyefficientsparseclustering}.

We now prove low-degree MMSE lower bound near the predicted algorithmic threshold. Consider the following prior on our centers $\mu$ and $-\mu$, borrowed from \cite{even2025computationallowerboundslatent}.
Let $\Delta>0$ and i.i.d. $ \  b_j\sim\mathrm{Ber}(s/p), j=1,\ldots,p$ and i.i.d. $g_j\sim N(0,\frac{\Delta}{s}), j=1,\ldots,n$ for which we set $\mu_j:=b_j g_j, j=1,\ldots,p.$ Note that for this choice of the mean of $\|\mu\|^2_2$ equals $\Theta(\Delta^2)$ with high probability as $n$ grows, so we would like a low-degree MMSE lower bound at $\Delta^2=\widetilde{\Theta}(p/n).$ We rescale the GAM so that the prior satisfies our assumptions to $Y=\sqrt{\lambda} X'+Z,$ for $X'=X/\sqrt{\Delta/s}$ and $\lambda=\Delta/s.$ In this scaling, $\lambda_{\mathrm{ALG}}=\widetilde{\Theta}(\sqrt{p/(s^2n)})$ and our goal then is to prove low-degree MMSE hardness at $\lambda=\widetilde{\Theta}(\lambda_{\mathrm{ALG}}).$

It turns out that Theorem \ref{thm:mainthm} applies for this Sparse Clustering model, as this prior also satisfies Assumption \ref{assump:mainassumption}, this time for $D_n = \lfloor (1-\alpha)\log n \rfloor$, Weibull constant $\phi=2$ and $A_n=\Theta(\log ^4 n)$. In particular, Theorem \ref{thm:mainthm} implies the following result.

\begin{theorem}\label{thm:spclust}
Let $p,n$ be positive integers such that $p=n^{c+o(1)}$ for some constant $c>0$. For the Sparse Clustering model and the prior discussed above we prove that for  $\lambda =\Delta/s$, $s=p^{a+o(1)}$ where $a\in (1/2,1)$ and for $D_n=\lfloor (1-\alpha)\log n \rfloor$: 
\begin{align*}
\mathrm{MMSE}^{\le D_n}_X  \left(\frac{1}{\log^{4}n}\cdot \big(\lambda+\frac{\d}{\d q}\mathcal{F}_{\mathrm{ann}, \lambda}\bigg{|}_{q=q(D_n)}\big)\right)= ns- \widetilde{O}\left(s\sqrt{n/p}\right)= (1+o(1))\mathrm{MMSE}^{\mathrm{trivial}}_X . 
\end{align*}
In particular, for $\lambda=\widetilde{\Theta}(\sqrt{p/(s^2n)})=\widetilde{\Theta}(\lambda_{\mathrm{ALG}}),$ we have $\frac{\d}{\d q}\mathcal{F}_{\mathrm{ann}, \lambda}\bigg{|}_{q=q(D_n)} \geq 0$, and therefore \\ $\mathrm{MMSE}^{\le D_n}_X  \left(\widetilde{\Theta}(\lambda_{\mathrm{ALG}})\right) \geq (1+o(1))\mathrm{MMSE}^{\mathrm{trivial}}_X. $
\end{theorem}The proof is deferred to Section \ref{sec:spclust}.

This tight low-degree lower bound for estimation in this model has also been established very recently in~\cite{even2025computationallowerboundslatent}. In fact, that work proves a slightly stronger result than what follows from our general equivalence, as by exploiting the specific latent-variable structure of the clustering problem they obtain the $D$-degree MMSE lower bound to degrees $D_n$ that grow slightly faster than logarithmic.

%\appendix

%\appendix

%\newpage

% \section{Preliminaries}
% We include here a few technical results and observations that will be useful in what follows.

\section{Conclusion}

In this work, we establish a rigorous mathematical equivalence between the monotonicity of the annealed Franz-Parisi (FP) potential and low-degree MMSE lower bounds for a broad class of Gaussian Additive Models (GAMs). We show that the sign of the derivative of the annealed FP potential is equivalent to the low-degree MMSE surpassing a specific benchmark defined by the quantiles of the overlap of two i.i.d. draws from the prior. Our results bridge these two seemingly distinct formulations of computational hardness, which are distinct both in terms of their scientific origins and in terms of their mathematical formulation. Furthermore, our findings open several promising avenues for future research:

\begin{itemize}
    \item \textbf{Annealed vs. Quenched Potential:} Perhaps the most significant departure from classical statistical physics methodology in our work is the focus on the annealed, rather than quenched, FP potential. As discussed in Section \ref{sec:quenched}, this appears to some extent fundamental; for certain GAMs, the quenched potential remains increasing (indicating a ``physics-hard" regime) even deep into the ``low-degree easy" regime. In contrast, the monotonicity of the annealed potential tightly captures the low-degree MMSE phase transition from trivial to non-trivial. This suggests that the annealed FP potential may be more naturally suited for characterizing computational hardness in statistical settings, raising the fundamental question of whether this phenomenon has a deeper interpretation in statistical physics.

    \item \textbf{Link between Potentials and Algorithmic classes:} It is somewhat folklore in estimation tasks that the quenched FP potential relates to the performance of MCMC methods, see for example \cite{arous2020algorithmic,arous2020free} where the non-monotonicity of the quenched FP potential termed as a free energy barrier is proven to imply MCMC lower bounds. Moreover, the replica-symmetric FP potential is known to relate to the performance of Approximate Message Passing (AMP) in spiked rank-1 models (see e.g., \cite{montanari2024equivalenceapproximatemessagepassing} for a discussion). Our work (alongside \cite{bandeira2022franz,chen2025an} for detection) establishes a new connection; the annealed FP potential characterizes the performance of low-degree polynomials. Deepening our understanding of this mapping between physics potentials and classes of algorithms is a very interesting direction for future study.

    \item \textbf{Obtaining more Low-Degree Information from the Annealed FP Potential:} We prove that the monotonicity of the annealed FP potential determines whether the low-degree MMSE outperforms a specific benchmark. In most applications, this benchmark is of the order of the ``trivial MMSE," meaning an increasing annealed FP potential implies the failure of low-degree polynomials to beat the trivial performance. While we utilize a lot this ``one-sided" implication of our results in our applications, a compelling open question is whether the annealed potential can provide a more refined ``success" result for the low-degree MMSE—specifically, whether it can indicate exactly when low-degree polynomials achieve low-degree MMSE that beats \emph{in order} the trivial MMSE. Our current results suggest that the sign of the derivative alone is not sufficient to find this information.

    \textbf{Low-degree I-MMSE relation:} A compelling way to frame the success of our technique is as follows. One of the most fundamental identities in information theory is the I-MMSE relation \cite{guo2005mutual}, which links for any GAM the derivative of the mutual information (i.e., in physics jargon, the derivative of the ``free energy" of the system) to the MMSE. This identity has been instrumental in characterizing the asymptotic MMSE for numerous high-dimensional models. Specifically, it has served as a key step in proving the ``replica-symmetric" MMSE formulas for compressed sensing \cite{barbier2019optimal,reeves2016replica} and low-rank matrix estimation \cite{lelarge2017fundamental}, as well as establishing ``all-or-nothing" phase transitions in various settings \cite{reeves2019all,niles2020all}. The technical advantage of this approach rather than analyzing the MMSE directly is that researchers can instead study the more tractable free energy and simply differentiate it to extract the MMSE.

Given the recent surge of interest in the low-degree MMSE—which is typically analyzed via direct combinatorial methods and ``cumulant" bounds—it is natural to ask whether a \textit{low-degree I-MMSE relation} exists. Such a relation would ideally reduce the low-degree MMSE analysis to the study of a (hopefully simpler) ``low-degree free energy." Our main result (Theorem \ref{thm:main_inf}) establishes exactly such an approximate identity for a large class of GAMs: we prove that for these models the derivative of the annealed FP potential is linked to the low-degree MMSE. As the annealed FP potential is a much easier object to compute, this leads to relatively easier analysis of the low-degree MMSE of a series of interesting models. 

Since the classical I-MMSE relation holds for all GAMs and extends to other settings such as Poisson \cite{atar2012mutual} and Bernoulli channels \cite{mossel2023sharp}, generalizing our key identity beyond the Gaussian framework is a natural and promising direction for future research.

\end{itemize}

\newpage

\section{Additional notation for the proofs}
We write $\mathbb{N}=\{1,2,\dots\}$ and $[N]=\{1,\dots,N\}$.
All asymptotic notation ($O,o,\Omega,\omega,\Theta$) is for $n\to\infty$ unless stated otherwise, and
$\widetilde{O}(\cdot),\widetilde{\Omega}(\cdot),\widetilde{\Theta}(\cdot)$ hide $\text{polylog}(n)=(\log n)^{O(1)}$ factors.
For an event $A$, $\mathbf{1}_{A}$ denotes its indicator. All logarithms are natural.

For $\alpha\in\mathbb{N}^N$, define $|\alpha|=\sum_{i=1}^N\alpha_i$, $\alpha!=\prod_{i=1}^N \alpha_i!$, and for $X\in\mathbb{R}^N$,
$X^\alpha=\prod_{i=1}^N X_i^{\alpha_i}$. We write $\alpha\ge\beta$ for coordinate wise inequality, and define
$\binom{\alpha}{\beta}=\prod_{i=1}^N \binom{\alpha_i}{\beta_i}$ when $\alpha\ge\beta$.
We write $\beta\prec\alpha$ to mean $\beta\le\alpha$ and $\beta\neq\alpha$. Also, $\supp(\alpha)=\{i\in[N]:\alpha_i\neq 0\}$.

For $v\in\mathbb{R}^n$ and $r\ge 1$, $v^{\otimes r}\in\mathbb{R}^{n^r}$ denotes the $r$-fold tensor power with
$(v^{\otimes r})_{i_1\ldots i_r}=v_{i_1}\cdots v_{i_r}$. 

For two random variables $X,Y$, we write $X\stackrel{d}{=}Y$ for equality in distribution, $X\perp\!\!\!\perp Y$ for independence and we write $X\sim P$ to mean that the random variable $X$ has distribution $P$.

Lastly, for any $D\in\N$, we define $\exp_{\leq D}(v) = \sum_{k=0}^D \frac{v^k}{k!}$ and we use $\Gamma(t)$ to denote the Gamma function, which is define as $\Gamma(t)= \int_0^\infty x^{t-1} e^{-x} \d x$, $t>0$.

\section{Proof that increasing FP implies MMSE lower bounds (Proof of Theorem \ref{thm:mainthm})}\label{sec:pf_UB}

\subsection{Cumulants background}\label{sec:cumulants}
We give some background on cumulants together with some identities that they satisfy.
\begin{definition}[Cumulants]
Let $X_1,\ldots , X_n$ be jointly-distributed random variables. Their cumulant generating function (CGF) is the function
\begin{equation*}
K(t_1,\ldots,t_n)=\log\left(\E\left[\exp\left(\sum_{i=1}^{n}t_i X_i\right)\right]\right),
\end{equation*}
for all $t_1, \dots , t_n \in \R$ such that the expectation is finite. The joint cumulants are defined as 
\[ 
\kappa(X_1,\ldots,X_n)=\left.\prod_{i=1}^{n}\frac{\de}{\de t_i}K(t_1,\ldots,t_n) \right|_{t_1=\ldots=t_n=0}.
\]
\end{definition}
\noindent Using the Taylor-Lagrange multivariate theorem, we have the following expansion into a series for all $t=(t_1, \dots , t_n)\in \R^n$ for which the following Taylor series converges, 
\[
K(t_1,\ldots,t_n)=\sum_{\alpha\in \N^{n}}^{+\infty}\frac{\left.\prod_{i=1}^{n}\frac{\de}{\de t_i^{\alpha_i}}K(t_1,\ldots,t_n) \right|_{t_1=\ldots=t_n=0}}{\alpha !}\prod_{i=1}^{n}t_{i}^{\alpha_i}.
\]

The proof for the following Proposition is standard,  see e.g. \cite{Schramm_2022}, Appendix D.
\begin{proposition}[Vanishing under independence]\label{prop:cumulant-vanishes-indep}
Let $a,b\ge 1$, and let $X_1,\dots,X_a,Y_1,\dots,Y_b$ be random variables such that
$\{X_i\}_{i=1}^a$ is independent of $\{Y_j\}_{j=1}^b$. Then
\[
\kappa(X_1,\dots,X_a,Y_1,\dots,Y_b)=0.
\]
\end{proposition}

The proof of the following Proposition is deferred to Section \ref{lem:MMSE_decomposition}.
\begin{proposition}[Linearity in one argument]\label{prop:cumulant-linearity-one-arg}
Fix integers $n\ge 1$, $n\geq t\geq 1$ and let $X,Y$, 
$Z_1,\dots,Z_{t-1},Z_{t+1},\dots,Z_n$ be jointly distributed random variables.
Then for any scalars $a,b\in\R$,
\begin{align*}
\kappa\big(Z_1,\dots,Z_{t-1},&\, aX+bY,\, Z_{t+1},\dots,Z_n\big)
\\&=
a\,\kappa\big(Z_1,\dots,Z_{t-1},X,Z_{t+1},\dots,Z_n\big)
+
b\,\kappa\big(Z_1,\dots,Z_{t-1},Y,Z_{t+1},\dots,Z_n\big).   
\end{align*}
\end{proposition}
 
The following is a standard formula for cumulants. For example, it can be found in \cite{McCullagh1987TensorMI}, Eq. $(2.9)$.
\begin{proposition}[Moment-cumulant partition formulas]\label{prop:moment-cumulant-partitions}
Let $\mathcal{P}([n])$ denote the set of all set partitions of $[n]:=\{1,\dots,n\}$.
For $\pi\in\mathcal{P}([n])$, write $|\pi|$ for the number of blocks of $\pi$.
Assume $\E|X_B|<\infty$ for every nonempty $B\subseteq [n]$. The joint cumulant $\kappa(X_1,\dots,X_n)$ satisfies  the partition formula
\begin{equation}\label{eq:cumulant-from-moments}
\kappa(X_1,\dots,X_n)
=
\sum_{\pi\in\mathcal{P}([n])} (|\pi|-1)!\,(-1)^{|\pi|-1}\,
\prod_{B\in\pi}\E\Big[\prod_{i\in B}X_i\Big].
\end{equation}
Moreover, the moments satisfy the  moment - cumulant relation:
\begin{equation}\label{eq:moments-from-cumulants}
\E\Big[\prod_{i=1}^nX_i\Big]
=
\sum_{\pi\in\mathcal{P}([n])}\ \prod_{B\in\pi}\kappa(X_i:i\in B).
\end{equation}

\end{proposition}

The proof for the following Proposition can be found in \cite[Appendix D]{Schramm_2022}.
\begin{proposition}[Cumulant recursion by splitting]\label{prop:cumulant-recursion}
Let $Y_1,\dots,Y_n$ be random variables with $\E\big[\prod_{i\in T}|Y_i|\big]<\infty$
for every $T\subseteq [n]$. Write $\kappa(Y_i:i\in T)$ for the joint cumulant of the subfamily
$\{Y_i\}_{i\in T}$, and use the convention $\prod_{i\in\varnothing} Y_i = 1$.
Then
\begin{equation}\label{eq:cumulant-recursion}
\kappa(Y_1,\dots,Y_n)
=
\E\Big[\prod_{i=1}^n Y_i\Big]
-
\sum_{\emptyset\neq S\subseteq [n]\setminus\{1\}}
\kappa\big(Y_i: i\notin S\big)\,
\E\Big[\prod_{i\in S} Y_i\Big].
\end{equation}
\end{proposition}

\begin{lemma}[Diagonal slice identity]
\label{lem:diagonal-slice}
Let $X = \sum_{i=1}^N X_i$ and let $\kappa_m(X)$ be its $m$-th cumulant. Then
\begin{equation}
\label{eq:diagonal-slice}
\frac{\kappa_m(X)}{m!}
= \sum_{|\gamma|=m} \frac{\kappa_{\gamma}(X_1, \dots X_N)}{\gamma!}.
\end{equation}
\end{lemma}

\begin{proof}
By definition,
\[
\kappa_m(X)
:= \kappa(\underbrace{X,\dots,X}_{m\ \text{times}}).
\]
Using multilinearity of cumulants, i.e.  Proposition \ref{prop:cumulant-linearity-one-arg}, and the definition $X = \sum_{i=1}^N X_i$,
\begin{align*}
\kappa_m(X)
= \kappa\Big(\sum_{i_1=1}^N X_{i_1},\dots,\sum_{i_m=1}^N X_{i_m}\Big) = \sum_{i_1,\dots,i_m=1}^N \kappa\big(X_{i_1},\dots,X_{i_m}\big).
\end{align*}
For each $m$-tuple $(i_1,\dots,i_m)$, define its multiplicity vector
\[
\gamma = (\gamma_1,\dots,\gamma_N),\qquad
\gamma_k := \#\{j : i_j = k\}.
\]
Then $|\gamma|=\sum_{k=1}^N \gamma_k=m$. For all $m$-tuples with the same multiplicity vector $\gamma$, the joint cumulant
\(\kappa(X_{i_1},\dots,X_{i_m})\)
is the same by symmetry and equals $\kappa_{\gamma}$ by definition. The number of such $m$-tuples is the number of permutations of a multiset with $\gamma_k$ copies of $k$, namely
$\frac{m!}{\gamma!},$ where $ \gamma! := \prod_{k=1}^N (\gamma_k)!. $ Hence,
\[
\kappa_m(X)
= \sum_{|\gamma|=m} \frac{m!}{\gamma!} \kappa_{\gamma},
\]
which, dividing by $m!$, is \eqref{eq:diagonal-slice}.  
\end{proof}

% \begin{remark}
% If we plug the diagonal $t_1=\dots=t_N=t$ into the multivariate expansion \eqref{eq:multivarite_cgf}, we obtain
% \[
% K_n(t,\dots,t)
% = \sum_{\gamma\neq 0} \frac{\kappa_{n,\gamma}}{\gamma!} t^{|\gamma|}
% = \sum_{m\ge 1} t^m \sum_{|\gamma|=m} \frac{\kappa_{n,\gamma}}{\gamma!}.
% \]
% By Lemma \ref{lem:diagonal-slice}, the coefficient of $t^m$ is exactly $\kappa_m(X)/m!$, i.e.\ the usual scalar cumulant of $X$. Thus the diagonal multi-cumulant expansion coincides with the standard cumulant expansion of the sum.
% \end{remark}

\subsubsection{Key Lemmas}

Some key lemmas are needed for us to prove the low-degree correlation upper bound, also describing the logic of the proof along the way. The proofs of these Lemmas are deferred to Section \ref{sec:key_lem_upper}

First, we will make use of an easy lemma, which follows from an important result upper bounding the low-degree correlation from \cite{Schramm_2022} via the cumulants of the prior. In particular, to produce our low-degree correlation upper bounds it suffices to bound the resulting weighted sum of the squared cumulants.
\begin{lemma}[Correlation bound via cumulants]\label{thm:sw22thm}
For any GAM with SNR $\lambda$, 
\[ 
(\mathrm{Corr}_{P_0}^{\leq D})^{2}(\lambda)\leq \sum_{i=1}^n \sum_{0\leq |\alpha|\leq D}\frac{\kappa_{\alpha}^{2}(X_i, \sqrt{\lambda}X_1,\sqrt{\lambda}X_2,\ldots,\sqrt{\lambda}X_n)}{\alpha!},
\]
where the quantity $\kappa_{\alpha}$ is equal, for fixed $i\in [n]$, to the joint cumulant of the following collection of dependent random
variables: one instance of $X_i$, and $\alpha_j$ copies of $X_j$ for each $j\in[n]$.   
\end{lemma}

We will also make use of this Lemma that gives two decompositions of the low degree MMSE for Gaussian Additive Models which we use interchangably. We start with some notation. Fix an integer $D\ge 0$ and let
\(\mathbb{R}[Y]_{\le D} \)
denote the space of polynomials in $Y$ of total degree at most $D$. Define the degree-$D$ MMSE of the  vector $X$  to be:
\[
\mathrm{MMSE}^{\le D}_X
:=\inf_{f_1,\dots,f_n\in \mathbb{R}[Y]_{\le D}}
 \mathbb{E}\Big[\sum_{i=1}^n \big(f_i(Y)-X_i\big)^2\Big].
\]
We also define the degree-$D$  coordinatewise correlations for all $i\in [n]$ and the correlation for the vector $X$ to be respectively:
\[
\mathrm{Corr}^{\le D}_{P_0,i}
:=\sup_{\substack{g\in \mathbb{R}[Y]_{\le D}\\  \mathbb{E}[g(Y)^2]=1}} \mathbb{E}\big[g(Y)X_i\big],
\qquad 
\mathrm{Corr}^{\le D}_{P_0}
:=\sup_{\substack{f_1,\dots,f_n\in \mathbb{R}[Y]_{\le D}\\ \E\sum_{i=1}^n f_i(Y)^2=1}}
 \EE\Big[\sum_{i=1}^n f_i(Y) X_i\Big].
\]
Then, we have the following convenient formulas. The proof of this Lemma is deferred to Section \ref{lem:MMSE_decomposition}.
\begin{lemma}[low-degree MMSE formula for the vector $X$]\label{lem:vector-mmse-corr}
Suppose we have a Gaussian Additive Model as described above with $X=(X_1,\dots,X_n)\in\mathbb{R}^n$. Then:
\begin{align}
\mathrm{MMSE}^{\le D}_X
&=\sum_{i=1}^n  \mathbb{E}[X_i^2]-\sum_{i=1}^n\big(\mathrm{Corr}^{\le D}_{P_0,i}\big)^2\\
&= \mathbb{E}\|X\|^2-\big(\mathrm{Corr}^{\le D}_{P_0}\big)^2.
\end{align}
\end{lemma}

Next, we make use of the following easy result regarding the derivative of the log-MGF of a bounded random variable. This will be a key first step to bound the sum of the squared cumulants.

\begin{lemma}[Bounded derivative of a truncated log-MGF]\label{lem:boundtrunc_logmgf}
Let $(\Omega,\mathcal{F},\mathbb{P})$ be a probability space and $Y_n$ be
a real-valued random variable such that $|Y_n|\leq M_n$, almost surely, where $M_n>0$. 
If the log-MGF of $Y_n$,
\(
\phi_{n}(\theta) := \log\mathbb{E}\big[e^{\theta Y_n}\big]
\)
exists on an open interval $I\subset\mathbb{R}$, 
Then $\phi_{n}(\theta)$ is differentiable on $I$ and its derivative satisfies
\[
(\phi_{n})'(\theta) \le M_n,
\qquad\text{for all }\theta\in I.
\]
\end{lemma}

To leverage the Lemma above, we need the following slightly more involved proposition regarding the log-MGF Taylor expansion of a random variable around $0$ and its relation with its cumulants. This proposition will now be crucial as combined with the previous Lemma allows us in the proof to obtain bounds on a weighted sum of some \emph{similar-looking} cumulants.

\begin{proposition}[Taylor expansion of log-MGF]
\label{prop:radius-Mn}
Let $T=\sum_{i=1}^nT_i$ be the sum of $n$ random variables $T_i$ such that $|T|\leq M_n$, almost surely, for some $M_n>0$. If 
\(
\phi_n(\theta)\) denotes the log MGF of $T$ then the Taylor series of $\phi_n(\theta)$ at $\theta=0$ is
\[
\phi_n(t)
= \sum_{m\ge 1} \frac{\kappa_m(T)}{m!} t^m= \sum_{m\ge 1} t^m \sum_{|\gamma|=m} \frac{\kappa_{n,\gamma}(T_1, \dots, T_n)}{\gamma!},
\]
and this series has radius of convergence at least
\(
R_n  \ge  1/eM_n.
\)
Furthermore, for $|t|<1/(eM_n)$ the degree-$m$ slice satisfies
\begin{equation}
\label{eq:slice-bound}
\left|\sum_{|\gamma|=m} \frac{\kappa_{n,\gamma}}{\gamma!} t^{|\gamma|}\right|
\le \big(eM_n|t|\big)^m.
\end{equation}

\end{proposition}

Lastly, our next two lemmas allows to translate our cumulant bounds from the previous two results, to the desired sum of the squared cumulants. To do this an important step would be to compare the cumulants of truncated random variables to the cumulants of original random variables. The following Lemma guarantees such a bound for sub-Weibull random variables.
\begin{lemma}\label{cor:trunccum_originalcum}
Fix $m\ge1$, indices $i_1,\dots,i_m\in\{1,\dots,n\}$ and $M_n>0$.
Let $T_i, i\in [n]$ be random variables such that for some $\phi>0$, $T_i$ are sub-Weibull$(\phi)$ distributions for all $i\in [n]$.
If $T=\sum_{i=1}^nT_i$, $T_i^{\mathrm{tr}}=T_i1_{A_n}$ and $\varepsilon_n=\P(|T|\ge M_n)$  
then,
\begin{equation}\label{eq:global-main-crude}
\big|\kappa_m(T_{i_1}^{\mathrm{tr}},\dots,T_{i_m}^{\mathrm{tr}})
- \kappa_m(T_{i_1},\dots,T_{i_m})\big|
\le m^{C(\phi)m}\sqrt{\varepsilon_n}.
\end{equation}  
\end{lemma}

Finally, the last step would be to compare the cumulants produced by our application of the Taylor expansion as described in Proposition \ref{prop:radius-Mn} with the squared cumulants obtained in \cite{Schramm_2022}.

We introduce some notation. Let $X=(X_1,\ldots,X_N)\in\R^N$ be a random vector, and let $X'=(X_1',\ldots,X_N')$
be an independent copy of $X$. For a multi-index $\alpha=(\alpha_1,\ldots,\alpha_N)\in\N^N$, we define
\begin{equation}\label{def:ktilde}
\widetilde{\kappa}_{\alpha}(X_1, \dots , X_N)
\;:=\;
\kappa_{\alpha}\big(X_1X_1',\,X_2X_2',\,\ldots,\,X_NX_N'\big).
\end{equation}
\begin{lemma}\label{cor:ktildegeqksquared}
Fix $\alpha \in \N^{N}$. Suppose that for some random variables $X_1, \dots , X_N$ it holds that for any  multi-index $\alpha \in \N^N$ with cardinality $|\alpha|\leq D_n$ it holds that:
    \begin{equation}
    \kappa_{\alpha}(X_1, \dots , X_N)\geq 0.
    \end{equation} Then, we can prove the following: for all $\alpha \in \N^N$ such that $|\alpha|\leq D_n$:
\[
\widetilde{\kappa}_{\alpha}(X_1, \dots, X_N)\geq \kappa_{\alpha}^2(X_1, \dots ,X_N).
\]    
\end{lemma}

\subsection{Proof of the upper bound}

\begin{proof}[Proof of Theorem \ref{thm:mainthm}]
We denote by $\phi_n$ the log-moment generating function (log-MGF) of $\<X,X'\>$, where $X,X'$ are two i.i.d. draws from $P_0$. We also denote by $\phi_{n}^{\mathrm{tr}}$ the log MGF of the ``truncated" random variable $\<X,X'\>1_{(|\<X,X'\>|\leq q(D_n'))} $.

Then, notice that using Lemma \ref{lem:boundtrunc_logmgf} applied for $M_n=q(D_n')$ we get that for all $\theta\in I:=\{t \in \mathbb{R}: \phi_{n}^{\mathrm{tr}}(t)<\infty\}$, it holds that: 
\begin{align}\label{eq:uniform_bd}
    (\phi_{n}^{\mathrm{tr}})'(\theta) \leq q(D_n').
\end{align}

Let $R_n$ be the Taylor expansion radius of $\phi_{n}^{\mathrm{tr}}$ around $0$. By Proposition \ref{prop:radius-Mn}, $(q(D_n'))^{-1}\leq R_n$ and therefore by the same Proposition we get that for all $|\theta| \leq (q(D_n'))^{-1}$:  
\[
\phi_{n}^{tr}(\theta)=
\sum_{m=1}^{\infty} \theta^m \sum_{|\gamma|=m} \frac{\kappa_{n,\gamma}^{\mathrm{tr}}(X_1X_1', \dots , X_nX_n')}{\gamma!}.
\]Taking derivatives 
\[
(\phi_{n}^{tr})'(\theta)=
\sum_{m=1}^{\infty} m\theta^{m-1} \sum_{|\gamma|=m} \frac{\kappa_{n,\gamma}^{tr}(X_1X_1', \dots , X_nX_n')}{\gamma!}.
\]
Now, again using Proposition \ref{prop:radius-Mn}, if $r_n=\log n$, we get that for $|\theta|\leq (q(D_n')r_n)^{-1}$, we have $q:=q(D_n')|\theta|\leq 1/r_n$, and the tail beyond degree $D_n+1$ satisfies:
\[
\left|\sum_{m\ge D_n+2} mt^{m-1}\sum_{|\gamma|=m} \frac{\kappa^{tr}_{n,\gamma}}{\gamma!}\right|
\le \sum_{m\ge D_n+2} (q(D_n')|\theta|)^m
\leq \frac{q^{D_n+2}}{1-q} \leq  \frac{(1/r_n)^{D_n+2}}{1-1/r_n}.
\]
Therefore, for all $|\theta|\leq (q(D_n')r_n)^{-1}$ we conclude,
\begin{equation}\label{eq:remainderfromexpansion}
\left|(\phi_{n}^{tr})'(\theta)-
\sum_{m=1}^{D_n+1} m\theta^{m-1} \sum_{|\gamma|=m} \frac{\kappa^{tr}_{n,\gamma}(X_1X_1', \dots , X_nX_n')}{\gamma!}\right|\leq \frac{(1/r_n)^{D_n+2}}{1-1/r_n}.
\end{equation}which combined with \eqref{eq:uniform_bd} implies 

\begin{equation}\label{eq:remainderfromexpansion_2}
\sum_{m=1}^{D_n+1} m\theta^{m-1} \sum_{|\gamma|=m} \frac{\kappa^{tr}_{n,\gamma}(X_1X_1', \dots , X_nX_n')}{\gamma!}-\frac{(1/r_n)^{D_n+2}}{1-1/r_n}\leq q(D_n').
\end{equation}

But, using Lemma \ref{cor:trunccum_originalcum} we get that  for all $\gamma$ such that $|\gamma|\leq D_n$
\begin{equation}\label{eq:trunc_gap}
\big|\kappa_{\gamma}^{\mathrm{tr}}(X_1,\dots,X_n)
- \kappa_{\gamma}(X_1,\dots,X_n)\big|
\le |\gamma|^{C(\phi)|\gamma|}e^{-D_n'}\le e^{C(\phi)D_n \log D_n-D_n'},
\end{equation}
for some $C(\phi)>0$ depending on the sub-Weibull constant $\phi>0$ from Item 3 in Assumption \ref{assump:mainassumption}. 
Finally, notice that there are $n^{D_n+1}$ such multi-indices $\gamma$ such that $|\gamma|\leq D_n+1$. Applying \eqref{eq:trunc_gap} to each term  we get by direct algebra,
\begin{align*}
    \sum_{|\gamma|\leq D_n+1} &\frac{\kappa^{tr}_{n,\gamma}(X_1X_1', \dots , X_nX_n')-\kappa_{n,\gamma}(X_1X_1', \dots , X_nX_n')}{\gamma!}|\gamma|\theta^{|\gamma|-1}  \\& \geq -e^{C(\phi)D_n\log D_n-D_n'}\cdot \sum_{|\gamma|\leq D_n+1} |\gamma|\theta^{|\gamma|-1}(\gamma!)^{-1}\\
    &\geq -e^{C(\phi)D_n\log D_n-D_n'}\cdot \sum_{0<|\gamma|\leq D_n+1} \theta^{|\gamma|-1}\\
    &\geq -e^{C(\phi)D_n\log D_n+D_n\log n-D_n'} \cdot \max\{ \theta, \theta^{D_n}\}.
\end{align*}
Assuming furthermore that $0<\theta<n^{C_1}$, for some positive universal constant $C_1>0$ this allows to conclude for all $|\theta|\leq \min\{(q(D_n')r_n)^{-1},n^{C_1}\},$
\begin{equation}\label{eq:trcum_to_origcum}
  \sum_{|\gamma|\leq D_n+1} \frac{\kappa^{tr}_{n,\gamma}(X_1X_1', \dots , X_nX_n')-\kappa_{n,\gamma}(X_1X_1', \dots , X_nX_n')}{\gamma!}|\gamma|\theta^{|\gamma|-1} \geq   -e^{C(\phi)D_n\log D_n+(C_1+1)D_n\log n-D_n'}
\end{equation}

Hence, \eqref{eq:remainderfromexpansion_2} together with \eqref{eq:trcum_to_origcum} imply for all $|\theta|\leq \min\{(q(D_n')r_n)^{-1},n^{C_1}\},$
\begin{equation}\label{eq:remainderfromexpansion_3}
\sum_{m=1}^{D_n+1} m\theta^{m-1} \sum_{|\gamma|=m} \frac{\kappa_{n,\gamma}(X_1X_1', \dots , X_nX_n')}{\gamma!}-\frac{(1/r_n)^{D_n+2}}{1-1/r_n}-e^{C(\phi)D_n\log D_n+(C_1+1)D_n\log n-D_n'}\leq q(D_n').
\end{equation}

But now, notice that our prior on the signal $X$ satisfies Item $2$ from Assumption \ref{assump:mainassumption} and therefore Lemma \ref{cor:ktildegeqksquared} together with elementary algebraic manipulations imply, 
\begin{align}
\sum_{m=1}^{D_n+1} m\theta^{m-1} \sum_{|\gamma|=m} \frac{\kappa_{n,\gamma}(X_1X_1', \dots , X_nX_n')}{\gamma!}&=\sum_{|\gamma|\leq D_n+1} \frac{\kappa_{n,\gamma}(X_1X_1', \dots , X_nX_n')}{\gamma!}|\gamma|\theta^{|\gamma|-1}\\
&\geq \sum_{|\gamma|\leq D_n+1} \frac{\kappa^2_{n,\gamma}(X_1, \dots , X_n)}{\gamma!}|\gamma|\theta^{|\gamma|-1} \\
& = \sum_{i=1}^n\sum_{|\gamma|\leq D_n} \frac{\kappa^2_{n,\gamma}(X_i, \sqrt{\theta}X_1 \dots , \sqrt{\theta}X_n)}{(\gamma+e_i)!}(|\gamma|+1)\\
& \geq \sum_{i=1}^n\sum_{|\gamma|\leq D_n} \frac{\kappa^2_{n,\gamma}(X_i, \sqrt{\theta}X_1 \dots , \sqrt{\theta}X_n)}{\gamma!}, \label{eq:cum_squared_app}
\end{align}where for the last inequality we used that $\gamma!(|\gamma|+1)\geq (\gamma+e_i)!$ for all $i\in \N$. 

Now by Theorem \ref{thm:sw22thm} we have for each $i=1,\ldots,n$, \[(\mathrm{Corr}_{P_0,i}^{\leq D})^2(\theta)_{\leq D_n} \leq \sum_{|\gamma|\leq D_n} \frac{\kappa^2_{n,\gamma}(X_i, \sqrt{\theta}X_1 \dots , \sqrt{\theta}X_n)}{\gamma!}.\] Summing then over $i$, and combining the last displayed inequality with \eqref{eq:cum_squared_app} and with \eqref{eq:remainderfromexpansion_3} we conclude for all $|\theta|\leq \min\{(q(D_n')r_n)^{-1},n^{C_1}\},$
\begin{align}
    (\mathrm{Corr}^{\le D}_{P_0})^2(\theta)-\frac{(1/r_n)^{D_n+2}}{1-1/r_n}-e^{C(\phi)D_n\log D_n+(C_1+1)D_n\log n-D_n'}\leq q(D_n').
\end{align}

 Using Item 3 from Assumption \ref{assump:mainassumption}, substituing $r_n=\log n$ and for $n$ large enough we know that: 
\begin{align*}
e^{(C(\phi)+C+1)D_n\log (nD_n)-D_n'}&\leq e^{-D_n'/2}\leq n^{-\frac{D_n}{2}\log n}\leq (\log n)^{-D_n/2}=o( q(D_n')) \\ \frac{(1/r_n)^{D_n+2}}{1-1/r_n}&\leq 2(\log n)^{-D_n}=o( q(D_n')). 
\end{align*}and therefore for all $|\theta|\leq \min\{(q(D_n')r_n)^{-1},n^{C_1}\},$ if $n$ is large enough,
\begin{align}\label{eq:corr_ub}
    (\mathrm{Corr}_{P_0}^{\le D})^2(\theta)\leq 2q(D_n').
\end{align}
But, now observe the identity that for all $u>0,$ 
\[u +\frac{d}{dq}\mathcal{F}_{\mathrm{ann}, u}|_{q=q(D_n)}=-\frac{d}{dq}\log \P (\< X,X'\>=q)|_{q=q(D_n)}.\]
which combined with Condition 2 from  Assumption \ref{assump:mainassumption} implies
\[
\frac{1}{A_n}\left(u +\frac{d}{dq}\mathcal{F}_{\mathrm{ann}, u}|_{q=q(D_n)}\right) \leq (q(D_n')r_n)^{-1}.
\]
Notice next that using Item 3 from Assumption \ref{assump:mainassumption} this quantity is upper bounded by $n^C/\log n$ for some universal constant $C>0$.
Hence, we can pick $C_1=C$ and then plug in $\theta=\frac{1}{A_n}\left(u +\frac{d}{dq}\mathcal{F}_{\mathrm{ann}, u}|_{q=q(D_n)}\right)\leq n^{C_1}$  to \eqref{eq:corr_ub} concluding the proof. 
\end{proof}

\subsection{Proofs of Key Lemmas}\label{sec:key_lem_upper}

\subsubsection{Proof of Lemma \ref{lem:boundtrunc_logmgf}}
% \begin{lemma}\label{lem:boundtrunc_logmgf}
% This will say that for all $\theta \in \mathbb{R}$
% \[
% (\phi_{n,\mathrm{tr}})'(\theta) \leq q(D_n').
% \] 
% \end{lemma}
 
Since $Y_n$ is bounded by $M_n$,
\[
|Y_n(\omega)| \le M_n
\qquad\text{for all }\omega\in\Omega.
\]
Next, $\phi_n(\theta)$ is finite on $I$ and $Y_n$ is bounded, by the dominated convergence theorem,
$\phi_{n}$
is differentiable on $I$ and for each $\theta\in I$ we have
\[
(\phi_{n})'(\theta)
= \frac{\d}{\d\theta}\log \mathbb{E}\big[e^{\theta Y_n}\big]
= \mathbb{E}_\theta\big[Y_n\big],
\]
where $\mathbb{E}_\theta$ denotes expectation with respect to the \emph{tilted
measure} $\mathbb{P}_\theta$ defined by
\[
\frac{d\mathbb{P}_\theta}{d\mathbb{P}}
= \frac{e^{\theta Y_n^{\mathrm{tr}}}}{\mathbb{E}\big[e^{\theta Y_n^{\mathrm{tr}}}\big]}.
\]
Using the pointwise bound on $Y_n$ and the fact that $\mathbb{P}_\theta$
is a probability measure, we obtain
\[
\big|(\phi_{n})'(\theta)\big|
= \big|\mathbb{E}_\theta[Y_n]\big|
\le \mathbb{E}_\theta\big[|Y_n|\big]
\le M_n,
\]
for all $\theta\in I$, as claimed.

\subsubsection{Proof of Proposition \ref{prop:radius-Mn}}

% We will denote throughout this proof the multivariate log MGF of $T_i$ and the one dimensional log MGF of $T$ as:
% \[
% \psi_n(t_1,\dots,t_n)
% := \log \E\exp\Big(\sum_{i=1}^n t_i T_i\Big) \qquad \text{and} \qquad  \phi_n(t)
% := \log \E\exp\Big( tT\Big)
% \]
% respectively, 
% whenever the expectations exist. 

% Then for $t=(t_1,\dots,t_n)$ near $0$ using Theorem \ref{thm:multivariate_cgf_taylor} we know that:
% \begin{equation}
% \label{eq:multi-cgf-expansion}
% \psi_n(t_1, \dots , t_n)
% = \sum_{\gamma\in\mathbb{N}^n\setminus\{0\}}
%    \frac{\kappa_{n,\gamma}(T_1, \dots , T_n)}{\gamma!} t^\gamma,
% \quad \text{where} \quad
% t^\gamma := \prod_{i=1}^n t_i^{\gamma_i},\quad
% \gamma! := \prod_{i=1}^n (\gamma_i)!.
% \end{equation}
% Convergence holds at least in some neighborhood of the origin. The main statement we want to prove shows that this diagonal series has radius of convergence of order at least $M_n^{-1}$, with a geometric tail, uniformly in $n$. 
 
Before moving to the proof we state the following Lemma the proof of which is deferred to Section \ref{sec:aux_for_keylemmas}.
\begin{lemma}[Cumulant bound for bounded variables]
\label{lem:bounded-cumulants}
Let $Z$ be a real-valued random variable with $|Z|\le B$ almost surely for some $B>0$. Let $\kappa_m(Z)$ denote its $m$-th cumulant for $m\ge 1$. Then
\begin{equation}
\label{eq:scalar-cumulant-bound}
|\kappa_m(Z)|   \le   m! (eB)^m,\qquad m=1,2,\dots
\end{equation}

\end{lemma}

We now combine Lemma \ref{lem:diagonal-slice} and Lemma \ref{lem:bounded-cumulants} to prove Proposition \ref{prop:radius-Mn}.

\begin{proof}[Proof of Proposition \ref{prop:radius-Mn}]
Since $|T|\le M_n$, the moment generating function (MGF)
\[
\Phi_n(t) := \E e^{t T}
\]
is an analytic function for all $t\in\mathbb{R}$. Moreover, $\Phi_n(0)=1$, so $\Phi_n(t)\neq 0$ for $|t|$ sufficiently small, and hence the log-MGF of $T$
\[
\phi_n(t) := \log \Phi_n(t)
\]
is analytic in some neighborhood of the origin. Its Taylor expansion at $0$ is
\[
\phi_n(t)
= \sum_{m\ge 1} \frac{\phi_n^{(m)}(0)}{m!} t^m
= \sum_{m\ge 1} \frac{\kappa_m(T)}{m!} t^m,
\]
where $\phi_n^{(m)}(0)$ denotes the $m$-th derivative and equals the $m$-th cumulant of $T$. By Lemma \ref{lem:bounded-cumulants} applied for $Z=Y$ and $B=M_n$, we have
\[
\big|\kappa_m(T)\big|
\le m! \ (eM_n)^m.
\]
Hence the Taylor coefficients $a_m := \kappa_m(T)/m!$ satisfy
\(
|a_m|
\le (eM_n)^m.
\)
By the root test,
\[
\limsup_{m\to\infty} |a_m|^{1/m} \le eM_n,
\]
and therefore the radius of convergence $R_n$ of the Taylor series is bounded below by
\[
R_n \ge \frac{1}{eM_n}.
\]
In particular, the series converges absolutely for all $|t|<1/(eM_n)$.
 We now consider the following series
 % which corresponds to the mulativariate Taylor expansion of the log-MGF of $(T_1, \dots , T_n)$ restricted to the diagnal $t=t_1=\dots =t_n$ (one can obtain that by applying Theorem \ref{thm:multivariate_cgf_taylor} for $t_i=t$ for all $i$)
\[
 \sum_{\gamma\neq 0} \frac{\kappa_{n,\gamma}(T_1, \dots ,T_n)}{\gamma!} t^{|\gamma|}
= \sum_{m\ge 1} t^m \sum_{|\gamma|=m} \frac{\kappa_{n,\gamma}(T_1, \dots ,T_n)}{\gamma!}.
\]
By Lemma \ref{lem:diagonal-slice}, the coefficient of $t^m$ equals
\[
\sum_{|\gamma|=m} \frac{\kappa_{n,\gamma}}{\gamma!}
= \frac{\kappa_m(T)}{m!}.
\]
Therefore
\[
\phi_n(t)
= \sum_{m\ge 1} \frac{\kappa_m(T)}{m!} t^m
= \sum_{m\ge 1} t^m \sum_{|\gamma|=m} \frac{\kappa_{n,\gamma}}{\gamma!}.
\]
Moreover, by Lemma \ref{lem:bounded-cumulants},
\[
\left|\sum_{|\gamma|=m} \frac{\kappa_{n,\gamma}}{\gamma!} t^{|\gamma|}\right|
= \left|\frac{\kappa_m(T)}{m!} t^m\right|
\le (eM_n)^m |t|^m
= (eM_n|t|)^m,
\]
which is \eqref{eq:slice-bound} completing the proof.
\end{proof}

% Thus, if $|t|<1/(eM_n)$, we have $q:=eM_n|t|<1$, and the tail beyond degree $M$ satisfies
% \[
% \left|\sum_{m\ge M+1} t^m \sum_{|\gamma|=m} \frac{\kappa_{n,\gamma}}{\gamma!}\right|
% \le \sum_{m\ge M+1} (eM_n|t|)^m
% = \frac{q^{M+1}}{1-q},
% \]
% which is geometric in $M$. This shows absolute convergence and a geometric tail in the ball $|t|<1/(eM_n)$.
% 

\subsubsection{Proof of Lemma \ref{cor:trunccum_originalcum}}
In this section we aim to prove Lemma \ref{cor:trunccum_originalcum}. To do this we start with some notation and by proving two results (Lemmas \ref{lem:block-global} and \ref{thm:global-truncation}) that will help us. Fix an integer $n\ge 1$, let $
T_1,\dots,T_n$ be any real-valued random variables defined on a common probability space and define $
T := \sum_{i=1}^nT_i.$ Fix $m\ge1$ and indices $i_1,\dots,i_m\in\{1,\dots,N\}$. We are interested in the $m$-th joint cumulant of the random variables
\[
T_{i_1},\dots,T_{i_m}
\quad\text{and}\quad
T_{i_1}^{\mathrm{tr}},\dots,T_{i_m}^{\mathrm{tr}}, \qquad \textrm{for all} \qquad i_1, \dots ,i_m\in [N]
\]
and in bounding their difference. For convenience define 
\begin{equation}\label{eq:L-U-def}
L := \max_{1\le t\le m} \big(\E|T_{i_t}|^m\big)^{1/m},
\qquad
U := \max_{1\le t\le m} \big(\E|T_{i_t}|^{2m}\big)^{1/(2m)}.
\end{equation} 
To simplify notation we set
$
Z_t := T_{i_t},\ \ 
Z_t^{\mathrm{tr}} := T_{i_t}^{\mathrm{tr}} = T_{i_t}1_{A_n},\ \ t=1,\dots,m,
$
so that
\[
\kappa_m(T_{i_1},\dots,T_{i_m}) = \kappa_m(Z_1,\dots,Z_m), \qquad \text{and} \qquad
\kappa_m(T_{i_1}^{\mathrm{tr}},\dots,T_{i_m}^{\mathrm{tr}}) = \kappa_m(Z_1^{\mathrm{tr}},\dots,Z_m^{\mathrm{tr}}).
\]
Also, $L,U$ will now be: 
\begin{equation}
L := \max_{1\le t\le m} \big(\E|Z_t|^m\big)^{1/m},
\qquad
U := \max_{1\le t\le m} \big(\E|Z_t|^{2m}\big)^{1/(2m)},
\end{equation}
assuming that $\E|Z_t|^{2m}<\infty$ for all $t$ so that everything is well defined. Next, given a block $B\subseteq\{1,\dots,m\}$, define
\[
Z_B := \prod_{t\in B} Z_t,
\qquad s := |B|.
\]
We will use the following Lemma the proof of which is deferred to Sextion \ref{sec:aux_for_keylemmas}. 
\begin{lemma}[Block bounds under truncation by $A_n$]\label{lem:block-global}
Let $B\subseteq\{1,\dots,m\}$ be a block of size $s=|B|$.
With $L,U$ as in \eqref{eq:L-U-def} and $\varepsilon_n=\P(A_n^c)$, we have
\begin{align}\label{eq:block-global-1}
\E\big|Z_B\big| \le L^{ s} \qquad \text{and} \qquad
\E\big[ |Z_B| 1_{A_n^c}\big] \le U^{ s} \sqrt{\varepsilon_n}. 
\end{align}
\end{lemma}

Using this Lemma we are ready to prove the following result that relates the multivariate cumulants of truncated random variables back to the original ones.
\begin{lemma}[Cumulants under truncation]\label{thm:global-truncation}
Fix $m\ge1$ and indices $i_1,\dots,i_m\in\{1,\dots,N\}$.
Assume that $\E|T_{i_t}|^{2m}<\infty$ for all $t$.
Let $L,U$ be given by \eqref{eq:L-U-def} and $\varepsilon_n=\P(|Y|\ge M_n)$ as above.
Then,

\begin{equation}\label{eq:global-main-fixed}
\big|\kappa_m(Z_1^{\mathrm{tr}},\dots,Z_m^{\mathrm{tr}})
- \kappa_m(Z_1,\dots,Z_m)\big|
\le
(m-1)!\,m^m\,2^m\,(U\vee L)^m\,\sqrt{\varepsilon_n}.
\end{equation}

\end{lemma}

\begin{proof}
Recall the partition formula \eqref{eq:cumulant-from-moments} applied to the truncated variables
$Z_t^{\mathrm{tr}} = X_{i_t}^{\mathrm{tr}}$ and the original ones $Z_t=X_{i_t}$:
\[
\kappa_m(Z_1^{\mathrm{tr}},\dots,Z_m^{\mathrm{tr}})
- \kappa_m(Z_1,\dots,Z_m)
= \sum_{\pi\in\Pi_m} (|\pi|-1)! (-1)^{|\pi|-1} \Delta_\pi,
\]
where
\[
\Delta_\pi
:= \prod_{B\in\pi}\E\Big[\prod_{t\in B} Z_t^{\mathrm{tr}}\Big]
   - \prod_{B\in\pi}\E\Big[\prod_{t\in B} Z_t\Big].
\]
Fix $\pi\in\Pi_m$. For a block $B\in\pi$, we have denoted $Z_B := \prod_{t\in B} Z_t$.
Then
\[
\E\Big[\prod_{t\in B} Z_t^{\mathrm{tr}}\Big]
= \E\big[Z_B 1_{A_n}\big]
= \E[Z_B] - \E\big[Z_B 1_{A_n^c}\big].
\]
Thus,
\[
\prod_{B\in\pi}\E\Big[\prod_{t\in B} Z_t^{\mathrm{tr}}\Big]
= \prod_{B\in\pi}\Big(\E[Z_B] - \E\big[Z_B 1_{A_n^c}\big]\Big).
\]
Subtracting $\prod_{B\in\pi}\E[Z_B]$ and expanding the product, we obtain
\[
\Delta_\pi
= \sum_{\emptyset\neq S\subseteq\pi}
(-1)^{|S|}
\prod_{B\in S}\E\big[Z_B 1_{A_n^c}\big]
\prod_{B\notin S}\E[Z_B].
\]
Taking absolute values and using Lemma \ref{lem:block-global} gives, for every nonempty
$S\subseteq\pi$,
\[
\prod_{B\in S}\big|\E[Z_B1_{A_n^c}]\big|\,
\prod_{B\notin S}\big|\E[Z_B]\big|
\le
\prod_{B\in S}\E\big[|Z_B|1_{A_n^c}\big]\,
\prod_{B\notin S}\E|Z_B|
\le
\prod_{B\in S} U^{|B|}\sqrt{\varepsilon_n}\;
\prod_{B\notin S} L^{|B|}.
\]
Write $r(S):=\sum_{B\in S}|B|$ and note that $\sum_{B\in\pi}|B|=m$. Then
\[
\prod_{B\in S} U^{|B|}\sqrt{\varepsilon_n}\;
\prod_{B\notin S} L^{|B|}
=
(\sqrt{\varepsilon_n})^{|S|}\,U^{r(S)}\,L^{m-r(S)}
\le
\sqrt{\varepsilon_n}\,(U\vee L)^m,
\]
since $|S|\ge 1$ implies $(\sqrt{\varepsilon_n})^{|S|}\le \sqrt{\varepsilon_n}$
(we always have $\varepsilon_n\le 1$) and $U^{r}L^{m-r}\le (U\vee L)^m$.
Therefore,
\[
|\Delta_\pi|
\le
\sum_{\emptyset\neq S\subseteq\pi}
\sqrt{\varepsilon_n}\,(U\vee L)^m
=
(2^{|\pi|}-1)\,(U\vee L)^m\,\sqrt{\varepsilon_n}
\le
2^{m}\,(U\vee L)^m\,\sqrt{\varepsilon_n},
\]
using $|\pi|\le m$ for any partition $\pi$ of $[m]$. Plugging this estimate into the partition formula,
\[
\big|\kappa_m(Z_1^{\mathrm{tr}},\dots,Z_m^{\mathrm{tr}})
- \kappa_m(Z_1,\dots,Z_m)\big|
\le
\sum_{\pi\in\Pi_m} (|\pi|-1)!\,|\Delta_\pi|
\le
2^{m}(U\vee L)^m\sqrt{\varepsilon_n}\,
\sum_{\pi\in\Pi_m} (|\pi|-1)!,
\]
which almost yields our first inequality. To get \eqref{eq:global-main-fixed}, observe that there are at most $m^m$ set partitions of $\{1,\dots,m\}$,
and $(|\pi|-1)!\le(m-1)!$ for all $\pi\in\Pi_m$, so
\[
\sum_{\pi\in\Pi_m} (|\pi|-1)!
\le (m-1)! m^m.
\]
$\sum_{\pi\in\Pi_m} (|\pi|-1)!\le (m-1)!\,m^m$, hence
\begin{equation}
\big|\kappa_m(Z_1^{\mathrm{tr}},\dots,Z_m^{\mathrm{tr}})
- \kappa_m(Z_1,\dots,Z_m)\big|
\le
(m-1)!\,m^m\,2^m\,(U\vee L)^m\,\sqrt{\varepsilon_n}.
\end{equation}
which is our desired inequality and the proof is complete. 
\end{proof}

If we want a bound that is uniform in the choice of indices $(i_1,\dots,i_m)$, we may simply replace \eqref{eq:L-U-def} by
\begin{equation}\label{eq:L_N-U_N-def}
L_n := \max_{1\le i\le n} \big(\E|T_i|^m\big)^{1/m},
\qquad
U_n := \max_{1\le i\le n} \big(\E|T_i|^{2m}\big)^{1/(2m)}.    
\end{equation}
Then $L\le L_n$ and $U\le U_n$ for all $(i_1,\dots,i_m)$, and Lemma \ref{thm:global-truncation} holds with $L,U$ replaced by $L_n,U_n$.

Now we are ready to prove  Lemma     \ref{cor:trunccum_originalcum} that we need for our Main Theorem.

\begin{proof}[Proof of Lemma \ref{cor:trunccum_originalcum}]
Note that $2^mm!\leq m^{m\log 2}$ and $\sqrt{\varepsilon}\leq 1$ trivially. Therefore, for our result to hold it suffices to prove that it holds that 
\[
\max \{L_n,U_n\}\leq m^{c_1(\phi)m},
\] for all sufficiently large $n$ and for a constant $c_1(\phi)>0$ where $\phi$ is the sub-Weibul constant of $X_i$.  We will show that this is true for all priors that satisfy Assumption \ref{assump:mainassumption}.  Indeed, item 3 from Assumption \ref{assump:mainassumption} implies that all coordinates $(T_i)_{i=1}^N$ are uniformly sub-Weibull$(\phi/2)$ (since $T_i=X_iX_i'$). The fact that $T_i$ are sub-Weibull$(\phi/2)$ for some $\phi>0$ is equivalent, from Definition \ref{def:subweibull}, with the following:
for a constant $K$ independent of $i$ and $n$:
\[
(\E|T_i|^{p})^{1/p} \leq C_{\phi} K  p^{2/\phi}
\qquad\text{for all } p\ge1,
\]
and therefore
\[
L \leq C_{\phi} K  m^{2/\phi},
\qquad
U \leq C_{\phi}K  (2m)^{2/\phi}.
\]
Hence $\max \{U_n, L_n\}^m\leq (2^{1/\phi}C_{\phi}K)^m m^{2m/\phi}$ 
and therefore our proof is complete. 
    
\end{proof}

\subsubsection{Proof of Lemma \ref{cor:ktildegeqksquared}}
In this section we aim to prove Lemma \ref{cor:ktildegeqksquared}. We start with some notation and a definition.  For $N$ random variables $X_1, \dots ,X_N$ remember we denoted
\[
\widetilde{\kappa}_{\alpha}(X_1, \dots, X_n):=\kappa_{\alpha}(X_1X_1', \dots, X_nX_n'),
\]
where the vector $X'=(X_1', \dots , X_n')$ is an i.i.d. copy of $X=(X_1, \dots , X_n)$.  In this subsection we give a closed form for the cumulants $\widetilde{\kappa}_{\alpha}$ that we defined above. Our results are based on an important result from \cite{McCullagh1987TensorMI}. We start with the following definition: 
\begin{definition}
  For any partition $\pi \in \mathcal{P}([n])$ we denote by $E(\pi)$ to be the ``induced" graph from the partition $\pi$, i.e., the graph $G$ on the vertices $[n],$ which is the union of $|\pi|$ cliques $(C_S)_{S \in \pi}$, i.e., $G=\bigcup_{S \in \pi} C_S,$ where for each part of the partition $S\in \pi$ we denote by $C_S$ the clique on the vertices $i \in S$.  
\end{definition} 
\begin{theorem}[Eq. $(3.3)$ from \cite{McCullagh1987TensorMI}]\label{thm:cumofpolynom}
 Suppose we have $n$ random variables $X_1, \dots , X_n$. Then, for any partition $\pi \in \mathcal{P}([n])$, with parts $S_1, \dots , S_{|\pi|}$, the following holds: 
 \begin{align}\label{eq:polynomialofcumulants}
     \kappa(\prod_{i\in S_1}X_i, \dots , \prod_{i\in S_{|\pi|}}X_i)=\sum_{\pi'\in G}\prod _{S'\in \pi'}\kappa(S'). 
 \end{align}
 where $G:=\{ \pi'\in \mathcal{P}([n]) \ \text{s.t.} \ E(\pi)\cup E(\pi') \text{ is a connected graph } \}$.
\end{theorem}
Using this result we can prove for the cumulants $\widetilde{\kappa}$ that we are interested in the following result. Define, for any $N\in \N$, the partition $\pi_{\mathrm{pair}}^N=\{\{1,2\},\{3,4\}, \dots , \{2N-1,2N\}\}$. 
\begin{lemma}\label{lem:k_tilde_formula}
Consider $X_1, \dots , X_N$ be a sequence of identically distributed but not necessarily independent random variables. Fix any $\alpha \in \mathbb{N}^N$ and let $X'=(X_1', \dots , X_N')$ be an i.i.d. copy of $X=(X_1, \dots , X_N)$. Denote $\widetilde{X}\in \R^{2|\alpha|}$ the vector defined as \[
\widetilde{X}:=(\underbrace{X_1, X_1' \dots , X_1, X_1'}_{\alpha_1-\text{times}}, \dots , \underbrace{X_N, X_N' \dots ,X_N, X_N'}_{\alpha_N-\text{times}}).\] Then, their joint cumulant $\widetilde{\kappa}_{\alpha}:=\kappa_{\alpha}(X_1X_1', \dots , X_NX_N')$,   satisfies,
\begin{align}\label{eq:ktildeformula}
\widetilde{\kappa}_{\alpha}:=\kappa_{\alpha}(X_1X_1', \dots , X_NX_N')=\kappa(\widetilde{X}_1\widetilde{X}_2, \dots , \widetilde{X}_{2|\alpha|-1}\widetilde{X}_{2|\alpha|})=\sum_{(\pi_1, \pi_2)\in J}\prod_{S\in \pi_1, \ S\in \pi_2}\kappa(S)\kappa(S')   
\end{align}
where   $J:=\{(\pi_1, \pi_2)\in \mathcal{P}(\{1,3, \dots , 2|\alpha|-1\} )\times \mathcal{P}(\{2,,4, \dots , 2|\alpha|\}) \ \text{s.t.} \ E(\pi_1)\cup E(\pi_2) \cup E(\pi_{\mathrm{pair}}^{|\alpha|}) \text{ is a connected graph} \}$. 
\end{lemma}
\begin{proof}
Applying Theorem \ref{thm:cumofpolynom} for the $\pi_{\mathrm{pair}}^{|\alpha|}=\{\{1,2\},\{3,4\}, \dots , \{2|\alpha|-1,2|\alpha|\}\}$ partition for the coordinates of $\widetilde{X}$ we have 
\[
\widetilde{\kappa}_{\alpha}:=\kappa_{\alpha}(X_1X_1', \dots , X_NX_N')=\sum_{\pi'\in G}\prod _{S'\in \pi'}\kappa(S')
\]
where $G:=\{\pi'\in \mathcal{P}([2|\alpha|]) \ \text{s.t.} \ E(\pi_{\mathrm{pair}}^{|\alpha|})\cup E(\pi') \text{ is a connected graph} \}$. In words, the first summation is over all partitions $\pi'$ such that when combined with the partition 
$\pi_{\mathrm{pair}}^{|\alpha|}$
they produce a connected graph. 

Our proof readily follows by establishing the following two steps. First, we show that every partition pair in $J$ appears in the summation defined by $G$. Second we argue that no other partition contributes a non-zero value to our summation.

For the first part, let $(\pi_1, \pi_2) \in J$. Since $\pi_1$ partitions the odd indices (from the vector $X$) and $\pi_2$ partitions the even indices (from the vector $X'$), their union $\pi = \pi_1 \cup \pi_2$ is a partition on $[2|\alpha|]$. Moreover, since $E(\pi)=E(\pi_1) \cup E(\pi_2)$, $\pi$ belongs to $G$, completing this step.

%By the definition of $J$, the edges formed by these partitions, along with the edges of $\pi_{\mathrm{pair}}$ that connect for all $i\in[N]$ the variables $X_i$ and $X_i'$, create a connected graph. Therefore, $ E(\pi_{\mathrm{pair}})\cup E(\pi_1 \cup \pi_2)$ is a connected graph with $[2N]$ vertices and we conclude that $\pi \in G$.

For the second step, consider a partition $\pi^*$ that is not in $J$ but in $G$. Then $\pi^*$ cannot be decomposed into two separate partitions of the $X$ variables and the $X'$ variables. Then we argue $\prod _{S'\in \pi^*}\kappa(S')=0$. Indeed, in this case $\pi^*$ contains a block $S^*$, and indices $i, j$ such that $X_i \in S^*$ and $X'_j \in S^*$. But the vectors $X,X'$ are independent and therefore $\kappa(S^*)=0$, from Proposition \ref{prop:cumulant-vanishes-indep}, concluding the second step.

 %For any choices though of distinct $i_1, i_2, \dots i_t \in [N]$ and any choice of distinct $j_1, \dots , j_s\in[N]$ the cumulant of the form 
%\[
%\kappa(X_{i_1}, \dots X_{i_t}, X_{j_1}', \dots , X_{j_s}')=0 
%\]
%as long as $t,s\geq 1.$   
%Consequently, the entire product for this partition vanishes:
%\[
%\prod_{S \in \pi^*} \kappa_{\alpha}(S) = 0.
%\]
%Thus, only partitions belonging to $J$ yield non-zero terms, completing the proof.
% Then, using the definition of the set $J$ we claim that this partition $\pi^*$ will have at least one part $S^*\in \pi^*$ such that there exist $i_1, j_1\in [N] $ with $X_{i_1}, X_{j_1}'\in S^*$. Combining that and the fact we mentioned above about cumulants we get that: 
% \[
% \prod _{S\in \pi^*}\kappa(S)=0.
% \]
% Now it remains to prove that there exists such a part $S^*$. Indeed, notice that form the definition of $J$ a partition on $[2N]$ doesn't belong in this set if and only if it can't be decomposed in a partition of the $X_i$ variables and the $X_i'$ variables. 
\end{proof}

We are now ready to prove Lemma \ref{cor:ktildegeqksquared}.

\begin{proof}[Proof of Lemma \ref{cor:ktildegeqksquared}]
We use formula \eqref{eq:ktildeformula} for $\widetilde{\kappa}$ and notice two things about it. First, if $\psi=\{1,2, \dots , |\alpha|\}$ then $(\psi, \psi)\in J$. Therefore, \[\widetilde{\kappa}_{\alpha}(X_1, \dots ,X_N)= \kappa_{\alpha}^2(X_1, \dots ,X_N)+\sum_{(\pi_1, \pi_2)\in J\setminus \{(\psi, \psi)\}}\prod_{S\in \pi_1, \ S\in \pi_2}\kappa(S)\kappa(S').\]  Moreover, any other term in the remaining sum is a product that only contains term which joint cumulants of the random variables $X_1, \dots ,X_N$, possibly repeated (recall $X'$ is i.i.d. with $X$). But all such cumulants are nonnegative according to the assumption of our Corollary, and therefore,
\[
\widetilde{\kappa}_{\alpha}(X_1, \dots ,X_N)\geq \kappa_{\alpha}^2(X_1, \dots ,X_N).
\] 
\end{proof}

\subsection{Proofs of Auxiliary Lemmas}\label{sec:aux_for_keylemmas}
 In this Section we prove Lemmas \ref{lem:bounded-cumulants} and \ref{lem:block-global}. 

\begin{proof}[Proof of Lemma \ref{lem:bounded-cumulants}]

Recall the moment - cumulant formula for a single variable \ref{prop:moment-cumulant-partitions}:
\[
\kappa_m(Z)
= \sum_{\pi\in\mathcal{P}([m])}
   (|\pi|-1)!(-1)^{|\pi|-1}
   \prod_{B\in\pi} \E\Big[Z^{|B|}\Big].
\]
 Taking absolute values and using $|Z|\le B$,
\[
|\kappa_m(Z)|
\le \sum_{\pi\in\mathcal{P}([m])}
     (|\pi|-1)!\prod_{B\in\pi} \E|Z|^{|B|}
\le \sum_{\pi\in\mathcal{P}([m])}
     (|\pi|-1)!\prod_{B\in\pi} B^{|B|}
= B^m \sum_{\pi\in\mathcal{P}([m])} (|\pi|-1)!.
\]
Let $|\pi|=k$ be the number of blocks. If $S(m,k)$ is the number of ways to partition a set with $m$ objects into $k$ non-empty subsets, then
\[
\sum_{\pi\in\mathcal{P}([m])} (|\pi|-1)!
= \sum_{k=1}^m (k-1)! S(m,k).
\]
Using the standard bound $S(m,k)\le k^m/k!$, we obtain
\[
(k-1)! S(m,k)
\le (k-1)! \frac{k^m}{k!}
= k^{m-1},
\]
and hence
\[
\sum_{k=1}^m (k-1)! S(m,k)
\le \sum_{k=1}^m k^{m-1}
\le m^m.
\]
Finally, the crude bound $m^m\le e^m m!$ yields
\[
\sum_{\pi\in\mathcal{P}([m])} (|\pi|-1)! \le e^m m!,
\]
so altogether
\[
|\kappa_m(Z)|
\le B^m e^m m!
= m! (eB)^m.
\]
This proves \eqref{eq:scalar-cumulant-bound}.

\end{proof}

\begin{proof}[Proof of Lemma \ref{lem:block-global}]
We use H\"older's inequality with equal exponents.
For \eqref{eq:block-global-1}, we apply H\"older with weights $p_t$ for every $t\in B$ where we choose $p_t=s$ for all $t\in B$. Then, $\sum_{t\in B} 1/p_t = s\cdot (1/s)=1$, so
\[
\E\big|Z_B\big|
= \E\prod_{t\in B}|Z_t|
\le \prod_{t\in B}\big(\E|Z_t|^s\big)^{1/s}.
\]
Since $s\le m$ and $L^p$-norms are nondecreasing in $p$,
\[
\big(\E|Z_t|^s\big)^{1/s} \le \big(\E|Z_t|^m\big)^{1/m} \le L
\]
for each $t$. Hence
\[
\E|Z_B| \le \prod_{t\in B} L = L^{ s}.
\]
For the second equation, apply Cauchy--Schwarz inequality:
\[
\E\big[ |Z_B| 1_{A_n^c}\big]
\le \big(\E|Z_B|^2\big)^{1/2} \sqrt{\varepsilon_n}.
\]
Now $|Z_B|^2 = \prod_{t\in B}|Z_t|^2$, so applying H\"older again with exponents $s$,
\[
\E|Z_B|^2 = \E\prod_{t\in B}|Z_t|^2
\le \prod_{t\in B}\big(\E|Z_t|^{2s}\big)^{1/s}.
\]
Since $2s\le 2m$,
\[
\big(\E|Z_t|^{2s}\big)^{1/(2s)}
\le \big(\E|Z_t|^{2m}\big)^{1/(2m)}
\le U,
\]
and therefore
\[
\big(\E|Z_t|^{2s}\big)^{1/s}
= \Big(\big(\E|Z_t|^{2s}\big)^{1/(2s)}\Big)^2
\le U^2.
\]
Thus
\[
\E|Z_B|^2 \le \prod_{t\in B} U^2 = U^{2s},
\]
and hence
\[
\E\big[ |Z_B| 1_{A_n^c}\big]
\le (U^{2s})^{1/2}\sqrt{\varepsilon_n}
= U^{ s}\sqrt{\varepsilon_n}.
\]
\end{proof}

\section{Proof that decreasing FP implies MMSE upper bounds (Proof of Theorem \ref{thm:corr-lower-bound})}\label{sec:pf_LB}

In this section and only by we study the quantiles of random variables beyond the overlap between two draws from the prior. For this reason, we provide a more general definition to Definition \ref{def:quantile_ov}.

\begin{definition}[Quantile function]\label{def:quantile}
    Let $Y_n$ be a real-valued random variable. For any $D>0$, we define the quantile $q(D)$ by
$q(D) \coloneqq \inf \left\{ y \in \mathbb{R} \mid \mathbb{P}(|Y_n| \leq y) \geq 1-e^{-D} \right\}.$
\end{definition}
We also make a note of the following remark.
\begin{remark}
    Notice that another way to define the quantile function $q(D)$, which is equivalent to Definition \ref{def:quantile}, is as the time-rescaled version of the generalized inverse function of the cumulative distribution function of $|Y_n|$, i.e, $q(D) = F_{|Y_n|}^{-1}(1-e^{-D})$ (we direct the reader to e.g., \cite[Chap 21.1]{van2000asymptotic} for more details).
\end{remark}

\subsection{Hermite Background}
We consider the \textit{probabilist’s Hermite polynomials} on $\mathbb{R}^N$. For a multi-index $\alpha \in \mathbb{N}^N$ and a vector $x \in \mathbb{R}^N$, the multivariate Hermite polynomial $H_{\alpha}(x)$ is defined via the tensor product:
\begin{equation}
    H_{\alpha}(x) = \prod_{j=1}^N h_{\alpha_j}(x_j),
\end{equation}
where $h_k(z) = (-1)^k e^{z^2/2} \frac{d^k}{dz^k} e^{-z^2/2}$ denotes the univariate Hermite polynomial of degree $k$.

These polynomials form an orthogonal basis for the Hilbert space $L^2(\mathbb{R}^N, \gamma)$, where $\gamma$ is the standard Gaussian measure. The inner product is defined as $\langle f, g \rangle_{L^2} = \mathbb{E}_{Z}[f(Z)g(Z)]$ with $Z \sim \mathcal{N}(0, I_N)$. The orthogonality relation is given by:
\begin{equation} \label{eq:orthogonality}
    \mathbb{E}_{Z}[H_{\alpha}(Z) H_{\beta}(Z)] = \alpha! \delta_{\alpha, \beta},
\end{equation}
where $\delta_{\alpha, \beta}$ is the Kronecker delta and $\alpha! = \prod_{j=1}^d \alpha_j!$.

Similarly, for a multi-index $\alpha \in \N^N$ and two vectors $x,y\in \R^N$, abusing notation, we denote \[x^\alpha = \prod_{i=1}^n x_i^{\alpha_i}\quad \text{ and } \quad (xy)^\alpha = \prod_{i=1}^n (x_iy_i)^{\alpha_i}.\]
To derive the properties required for the proof, we rely on the following fundamental {Translation Identity} given by Proposition \ref{prop:translation}, which expresses a shifted Hermite polynomial in terms of the basis centered at the origin. We give the proof for completeness in Section \ref{lem:MMSE_decomposition}.

\begin{proposition}[Translation Identity]
\label{prop:translation}
For any $z, \mu \in \mathbb{R}^N$ and $\alpha \in \mathbb{N}^N$:
\begin{equation}
    H_{\alpha}(z + \mu) = \sum_{\gamma \le \alpha} \binom{\alpha}{\gamma} \mu^{\alpha - \gamma} H_{\gamma}(z).
\end{equation}
\end{proposition}

Using this expansion, we prove the following lemmas.

\begin{lemma}[Gaussian Mean Shift]
\label{lem:mean_shift}
Let $Z \sim \mathcal{N}(0, I_N)$ and $\mu \in \mathbb{R}^N$. Then:
\begin{equation}
    \mathbb{E}_{Z}[H_{\alpha}(\mu + Z)] = \mu^{\alpha}.
\end{equation}
\end{lemma}

\begin{proof}
Applying Proposition \ref{prop:translation} to $H_{\alpha}(Z + \mu)$, we have:
\[
    \mathbb{E}_{Z}[H_{\alpha}(Z+\mu)] = \sum_{\gamma \le \alpha} \binom{\alpha}{\gamma} \mu^{\alpha - \gamma} \mathbb{E}_{Z}[H_{\gamma}(Z)].
\]
By the orthogonality of Hermite polynomials, $\mathbb{E}[H_{\gamma}(Z)] = \langle H_{\gamma}, H_{0} \rangle_{L^2} = 0$ for all $\gamma \ne 0$. The only non-vanishing term corresponds to $\gamma = 0$ (where $H_0 \equiv 1$). Thus:
\[
    \mathbb{E}_{Z}[H_{\alpha}(Z+\mu)] = \binom{\alpha}{0} \mu^{\alpha} \cdot 1 = \mu^{\alpha}.
\]
\end{proof}

\begin{lemma}[Expectation of Shifted Products]
\label{lem:shifted_product}
Let $Z \sim \mathcal{N}(0, I_N)$ and $\mu \in \mathbb{R}^N$. Then:
\begin{equation}
    \mathbb{E}_{Z} \left[ H_{\alpha}(\mu + Z) H_{\beta}(\mu + Z) \right] = \sum_{r \le \min(\alpha, \beta)} r! \binom{\alpha}{r} \binom{\beta}{r} \mu^{\alpha + \beta - 2r}.
\end{equation}
\end{lemma}

\begin{proof}
We expand both polynomials using the Translation Identity (Proposition \ref{prop:translation}):
\[
    \mathbb{E}_{Z} \left[ H_{\alpha}(Z+\mu) H_{\beta}(Z+\mu) \right] = \mathbb{E}_{Z} \left[ \left( \sum_{\gamma \le \alpha} \binom{\alpha}{\gamma} \mu^{\alpha - \gamma} H_{\gamma}(Z) \right) \left( \sum_{\delta \le \beta} \binom{\beta}{\delta} \mu^{\beta - \delta} H_{\delta}(Z) \right) \right].
\]
Using the linearity of expectation, we can interchange the sum and expectation. Due to the orthogonality relation \eqref{eq:orthogonality}, the cross-terms where $\gamma \ne \delta$ vanish. We sum only over indices $r$ common to both expansions ($r \le \alpha$ and $r \le \beta$):
\begin{align*}
    \mathbb{E}_{Z} \left[ H_{\alpha}(\mu + Z) H_{\beta}(\mu + Z) \right] &= \sum_{r \le \min(\alpha, \beta)} \binom{\alpha}{r} \mu^{\alpha - r} \binom{\beta}{r} \mu^{\beta - r} \mathbb{E}_{Z}[H_{r}(Z)^2] \\
    &= \sum_{r \le \min(\alpha, \beta)} \binom{\alpha}{r} \mu^{\alpha - r} \binom{\beta}{r} \mu^{\beta - r} (r!) \\
    &= \sum_{r \le \min(\alpha, \beta)} r! \binom{\alpha}{r} \binom{\beta}{r} \mu^{\alpha + \beta - 2r}.
\end{align*}
\end{proof}

\subsection{Key Lemmas}

To prove our low-degree correlation lower bound, we construct and analyze a nearly-optimal low-degree polynomial. To do this we need a few key lemmas which we state here and defer their proofs to Section \ref{sec:keylem_lb}. First, one simple observation is made by the following:

\begin{lemma}[Correlation is non-decreasing]\label{lem:corr-non-dec}
$\mathrm{Corr}^{\leq D_n}_{P_0}(h)$ is a non-decreasing function in $h>0$.
\end{lemma}

Then we need a quantitative error bound between the degree-$D$ Taylor expansion of the exponential function and the exponential function restricted near zero. This lemma is a fundamental tool behind our construction of the optimal low-degree polynomial, and in particular how the quantiles of the overlap of the prior appear in its analysis.

\begin{lemma}[Quantitative Error Bound for Truncated Exponential Moments]\label{lem:exptrun-error}
Let $V$ be a random variable with finite moments up to order $2D_n+2$, where $D_n \in \mathbb{N}$ is a degree, $C > 0$ be a constant, $(q(D'))_{D'>0}$ the quantiles of $|V|$ defined in Definition \ref{def:quantile}. 
Then, 
\begin{equation}
\left| \E[V \exp_{\leq D_n}(V)] - \E[V e^V 1_{|V| \leq q(CD_n)}] \right| \le \frac{e^{q(CD_n)} \norm{V}_{D_n+2}^{D_n+2}}{(D_n+1)!} + e^{-CD_n/2} \sum_{k=0}^{D_n} \frac{\norm{V}_{2k+2}^{k+1}}{k!}
\end{equation}
where $\norm{V}_p = (\E[|V|^p])^{1/p}$ denotes the $L_p$-norm of $V$. Furthermore, if there exists constants $C_1,C_2>0$, such that $\sup_{p \in [2D_n+2]}\norm{V}_p \le C_1$ and $q(C D_n) \leq C_2$, then
\begin{equation}
\left| \E[V \exp_{\leq D_n}(V)] - \E[V e^V 1_{|V| \leq q(CD_n)}] \right| = o(e^{-CD_n/2}).
\end{equation}
\end{lemma}

Next, we need a lemma comparing the quantiles of the sum of three identically distributed random variables to the quantiles of each one of them.

\begin{lemma}[Quantile Domination]
\label{lem:quantile}
Let $X,X',X''$ be three independent samples from the prior $P_0$,  $\widetilde{q}(D)$ be the quantile function of the sum $S=|\inner{X'}{X''}| + |\inner{X}{X'}| + |\inner{X}{X''}|$ using Definition \ref{def:quantile}, and $q(D)$ be the quantile function of $\inner{X}{X'}$ accordingly. There exists a universal constant $C \ge 1$ such that for all $D\ge \ln 3$,
\begin{equation}
    q(D) \le \widetilde{q}(D) \le C q(D).
\end{equation}
\end{lemma}

We also use an interesting easy fact that one can express the expectation of any function against any probability measure, with respect to the quantiles of the measure.
\begin{lemma}[Change of variable]\label{lem:change-of-variable}
    Suppose a random variable $X\ge 0$, and its quantile function $q$ is defined as in Definition \ref{def:quantile}.  Then for any integrable function $g$, that is, $\E |g(X)|<\infty$, we have
    \begin{equation}\label{eq:change-of-variable}
    \E g(X) = \int_0^\infty g(q(t)) e^{-t} \d t.
    \end{equation}
\end{lemma}
Finally, we will also need a general $L_p$ norm control of a random variable using the growth rate of its quantiles.
\begin{lemma}[Moment Growth Bound]\label{lem:moment-growth-bound}
Let $V$ be a non-negative random variable. Assume its quantile function $q(t)$ satisfies the upper growth condition $q(t) \le C B_n t^\kappa$ for all $t \ge t_0$. Then, there exists a constant $K$, such that for any $p\in \N$, the $L_p$ norm is bounded by:
\begin{equation}
    \norm{V}_p \le K B_n p^\kappa
\end{equation}
for some constant $K$ depending only on $C$ and $\kappa$.
\end{lemma}

\subsection{Proof of the lower bound}
\begin{proof}[Proof of Theorem \ref{thm:corr-lower-bound}]
According to Lemma \ref{lem:corr-non-dec}, $\mathrm{Corr}^{\leq D_n}_{P_0}(h)$ is a monotonically non-decreasing function in $h>0$,
and by \eqref{eq:fann_gams_main} together with Assumption \ref{assump:lower-bound}, \[\lambda +\frac{d}{dq}\mathcal{F}_{\mathrm{ann}, \lambda} \bigg{|}_{q=q(D_n)} = -\left.\frac{d}{dq}\log \P (\< X,X'\>=t)\right|_{q=q(D_n)} \ge \frac{c }{q(D_n)},\]
it suffices to prove 
\begin{align} \label{eq:goal1}
\mathrm{Corr}^{\leq D_n}_{P_0 }\left(\frac{c}{ {q(D_n)}}\right) \ge \frac{C}{D_n^\kappa} \sqrt{q(D_n)}.
\end{align}
In order to show this, we use the probabilistic method and construct a polynomial estimator based on a reference set drawn from the prior, which will satisfy with high probability \eqref{eq:goal1}.

Let $\alpha$ denote a multi-index element in $\N^n$ and $\lambda = \sqrt{c/q(D_n)}$ in the following. We define the function $W(Y|X)$ via a truncated Hermite expansion of degree $D_n$:
\begin{equation}\label{eq:weights}
    W(Y|X) = \sum_{|\alpha| \leq D_n} \frac{1}{\alpha!} X^{\alpha} H_{\alpha}(Y).
\end{equation}
Let $M \in \N$ be a sufficiently large number which we will choose later and $\mathcal{S} = \{X_1, \dots, X_M\}$ be a set of $M$ independent samples from the prior distribution $P_0$. Consider the degree-$D_n$ polynomial estimator defined by  $$p(Y) = \frac{1}{M} \sum_{k=1}^M W(Y|\sqrt\lambda X_k) \sqrt\lambda X_k,$$
%Assuming $a$ has the same distribution as $\langle X, X' \rangle$ and the noise $z \sim \mathcal{N}(0,1)$. 
%By our assumption, $a\sim\inner{X}{X'}$ follows a sub-Weibull($\alpha$/2) distribution. 
Conditional on $X$, 
we utilize  Proposition \ref{prop:translation} to get
\begin{equation*}
    \E_Z [H_{\alpha}(Y)|X] = \E_Z [H_{\alpha}(\sqrt\lambda X+Z)|X] = (\sqrt\lambda X)^{\alpha}.
\end{equation*}
Then direct algebra gives,
\begin{align*}
\frac{1}{\sqrt\lambda}\E \langle p(Y), X \rangle &= \frac{1}{M} \E_{Z,X}\left[\sum_{k=1}^M W(Y|\sqrt\lambda X_k)  \langle X_k, X \rangle\right] = \frac{1}{M} \sum_{k=1}^M  \sum_{|\alpha| \leq D_n} \frac{1}{\alpha!}  (\sqrt\lambda X_k)^{\alpha} \E_{Z,X}[\langle X_k, X \rangle  H_{\alpha}(Y)] \\
&= \frac{1}{M} \sum_{k=1}^M  \sum_{|\alpha| \leq D_n} \frac{1}{\alpha!} \E_X[\langle X_k, X \rangle (\sqrt\lambda X_k)^{\alpha}(\sqrt\lambda X)^{\alpha} ] = \frac{1}{M} \sum_{k=1}^M  \sum_{|\alpha| \leq D_n} \frac{1}{\alpha!} \E_X[\langle X_k, X \rangle (\lambda X_k X)^{\alpha} ] \\
&= \frac{1}{M} \sum_{k=1}^M   \E_X[\langle X_k, X \rangle \exp_{\leq D_n}(\lambda \langle X_k, X \rangle) ].
\end{align*}
Since $X_k, k\in[M]$ is independently sampled from the same prior $P_0$, by the strong law of large numbers, with high probability, as long as $|\E_{X,X'}[\langle X, X' \rangle \exp_{\leq D_n}(\lambda \langle X, X' \rangle) ]| < + \infty$,
\begin{equation}\label{eq:slln-1}
\left|\frac{1}{M} \sum_{k=1}^M   \E_X[\langle X_k, X \rangle \exp_{\leq D_n}(\lambda \langle X_k, X \rangle) ] - \E_{X,X'}[\langle X, X' \rangle \exp_{\leq D_n}(\lambda \langle X, X' \rangle) ]\right| = o_M(1).
\end{equation}
Combining the above, we conclude,
\begin{equation}\label{eq:numer}
    \frac{1}{\sqrt\lambda}\E \langle p(Y), X \rangle = \E[\inner{X'}{X} \exp_{\leq D_n}{(\lambda  \inner{X'}{X})}] + o_M(1).
\end{equation}

We now invoke the identity given in Lemma \ref{lem:shifted_product} to get
\begin{equation*}
    \E_Z [H_{\alpha}(\sqrt\lambda X+Z)H_{\beta}(\sqrt\lambda X+Z)]|X = \sum_{r \leq \min(\alpha, \beta)} r! \binom{\alpha}{r} \binom{\beta}{r} (\sqrt\lambda X)^{\alpha +\beta - 2r}.
\end{equation*}
Therefore, by direct algebra 
\begin{align*}
\frac{1}{\lambda}{\E[\|p(Y)\|^2]}&= {\frac{1}{M^2}\sum_{i,j=1}^M \E_{X,Z}[W(Y|\sqrt\lambda X_i)W(Y|\sqrt\lambda X_j) \langle X_i, X_j \rangle]}\\
&= {\frac{1}{M^2}\sum_{i,j=1}^M  \langle X_i, X_j \rangle \sum_{\substack{|\alpha| \leq D_n , \, |\beta| \leq D_n}} \frac{(\sqrt\lambda X_i)^{\alpha} (\sqrt\lambda X_j)^{\beta}}{\alpha! \beta!} \E_{X,Z} [H_{\alpha}(Y)H_{\beta}(Y)]}\\
&={\frac{1}{M^2}\sum_{i,j=1}^M  \langle X_i, X_j \rangle \sum_{\substack{\alpha, \beta, r\\|\alpha| \leq D_n , |\beta| \leq D_n}} \frac{1}{r! (\alpha-r)! (\beta-r)!} (\sqrt\lambda X_i)^{\alpha} (\sqrt\lambda X_j)^{\beta} \E_X[(\sqrt\lambda X)^{\alpha+\beta-2r}]}\\
&={\frac{1}{M^2}\sum_{i,j=1}^M  \langle X_i, X_j \rangle \sum_{\substack{\alpha, \beta, r\\|\alpha| \leq D_n , |\beta| \leq D_n}} \frac{1}{r! (\alpha-r)! (\beta-r)!} \E_X[(\lambda X_i X)^{\alpha-r} (\lambda X_j X)^{\beta-r} (\lambda X_i X_j)^{r}]}\\
&={\frac{1}{M^2}\sum_{i,j=1}^M  \langle X_i, X_j \rangle \sum_{\substack{ |r| \leq D_n}} \frac{1}{r!} \E_X[ \exp_{\leq D_n-|r|}(\lambda \inner{X_i}{X})\exp_{\leq D_n-|r|}(\lambda \inner{X_j}{X})] (\lambda X_i X_j)^r}\\
&={\frac{1}{M^2}\sum_{i,j=1}^M  \langle X_i, X_j \rangle \sum_{\substack{ k \leq D_n}} \frac{1}{k!} \E_X[ \exp_{\leq D_n-k}(\lambda \inner{X_i}{X})\exp_{\leq D_n-k}(\lambda \inner{X_j}{X})] (\lambda \inner{X_i}{ X_j})^k}.
\end{align*}
We use three observations: for a given $D\in\N$, $\exp_{\leq D}(x)$ is monotone increasing for all $x\in \R$; for all $x\ge 0$, $\exp_{\leq D}(x)$ is monotone increasing for all $D\in \N$; and for any $x,y\ge 0$, any $D_1, D_2\in \N$, $\exp_{\leq D_1}(x)\exp_{\leq D_2}(y) \leq \exp_{\leq D_1+D_2}(x+y)$. Define \[h(x,y) = |\langle x, y \rangle|\E_X[ \exp_{\leq 3D_n}(\lambda (|\inner{x}{X}|+|\inner{y}{X}|+|\inner{x}{y}|))],\] 
then the above quantity can be bounded as follows,
\begin{align*}
    &{\frac{1}{M^2}\sum_{i,j=1}^M  |\langle X_i, X_j \rangle| \sum_{\substack{ k \leq D_n}} \frac{1}{k!} \E_X[ \exp_{\leq D_n-k}(\lambda |\inner{X_i}{X}|)\exp_{\leq D_n-k}(\lambda |\inner{X_j}{X}|)] |\lambda \inner{X_i} {X_j}|^k}\\
    &\leq {\frac{1}{M^2}\sum_{i,j=1}^M  |\langle X_i, X_j \rangle| \E_X[ \exp_{\leq D_n}(\lambda |\inner{X_i}{X}|)\exp_{\leq D_n}(\lambda |\inner{X_j}{X}|)]\sum_{\substack{ k \leq D_n}} \frac{1}{k!}  |\lambda \inner{X_i} {X_j}|^k}\\
    &={\frac{1}{M^2}\sum_{i,j=1}^M  |\langle X_i, X_j \rangle| \exp_{\leq D_n}(|\lambda \inner{X_i} {X_j}|) \E_X[ \exp_{\leq D_n}(\lambda |\inner{X_i}{X}|)\exp_{\leq D_n}(\lambda |\inner{X_j}{X}|)]  }
    \\
    &\leq \frac{1}{M^2}\sum_{i,j=1}^M  |\langle X_i, X_j \rangle|\E_X[ \exp_{\leq 3D_n}(\lambda (|\inner{X_i}{X}|+|\inner{X_j}{X}|+|\inner{X_i}{X_j}|))] = \frac{1}{M^2}\sum_{i,j=1}^M h(X_i,X_j).
\end{align*}

For the same reason, since $X_k, k\in[M]$ is independently sampled from the same prior $P_0$, applying the strong law of large numbers for U-statistics \cite[Problem 12.15]{van2000asymptotic}, as long as $|\E_{X',X''}[h(X',X'')]| < + \infty$, 
\begin{equation}\label{eq:slln-2}
\left| {\frac{1}{M^2}\sum_{i,j=1}^M h(X_i,X_j)} - 
{\E_{X',X''}[h(X',X'')]}\right|
=o_M(1) .
\end{equation}
In short, we established that with high probability, 
\[
\frac{1}{\lambda}{\E[\|p(Y)\|^2]} \leq \E_{X,X',X''}[ |\inner{X'}{X''}| \exp_{\leq 3D_n}\left({\lambda( |\inner{X'}{X''}| + |\inner{X}{X'}| + |\inner{X}{X''}|)}\right)] + o_M(1).
\]
Denoting the sum $S = |\inner{X'}{X''}| + |\inner{X}{X'}| + |\inner{X}{X''}|$, and the overlap $A = \inner{X}{X'}$, then the above is controlled by the degree-$3D_n$ truncation of the exponential of the sum of inner products:
\begin{equation}
    |\inner{X'}{X''}| \exp_{\leq 3D_n}\left({\lambda( |\inner{X'}{X''}| + |\inner{X}{X'}| + |\inner{X}{X''}|)}\right) \leq S\exp_{\leq 3D_n}(\lambda   S
    ).
\end{equation}

Therefore, combining with \eqref{eq:numer}, there exists a  sufficiently large $M$, such that the correlation ratio satisfies the following lower bound with high probability,
\[
\mathrm{Corr}^{\le D_n}_{P_0}(\lambda) \ge \frac{ \E \langle p(Y), X \rangle }{ \sqrt{\E[\|p(Y)\|^2]} } \ge \frac{\E[A \exp_{\leq D_n}{(\lambda  A)}]}{2\sqrt{\E [S\exp_{\leq 3D_n}(\lambda   S)]}}.
\]
We apply Lemma \ref{lem:exptrun-error} to the numerator $\E[A \exp_{\leq D_n}{(\lambda A)}]$, so define $\epsilon_1(D_n)>0$, such that
\[
\left|\E[A \exp_{\leq D_n}{(\lambda A)}] -
\E\left[ A e^{\lambda A} 1_{|A| \leq q(C_t D_n)} \right]\right| = \epsilon_1(D_n).
\]
Similarly, since $\lambda = c/ {q(D_n)}$, define $\epsilon_2(D_n)>0$ such that
$|\E [S\exp_{\leq 3D_n}(\lambda  S)] - \E [S e^{\lambda  S} 1_{S \leq \widetilde q(C_t D_n)} ]| = \epsilon_2(D_n)$
. Then applying Lemmas \ref{lem:exptrun-error} and \ref{lem:quantile}, and according to the Assumption \ref{assump:lower-bound}(1), there exists some constant $C'>0$, such that
\[
\E [S\exp_{\leq 3D_n}(\lambda  S)] \le \E [S e^{\lambda  S} 1_{S \leq \widetilde q(C_t D_n)} ] + \epsilon_2(D_n) \le e^{\lambda \widetilde{q}(C_t D_n)} \widetilde{q}(C_t D_n) + \epsilon_2(D_n)  \le e^{C'} C' \cdot q(D_n) + \epsilon_2(D_n).
\]

We will verify $\epsilon_1(D_n) = \epsilon_2(D_n) = o_n(q(D_n)) $ such that they are negligible comparing to $q(D_n)$ for large $n$.
Starting with $\epsilon_1(D_n)$, we use the scaling $\lambda = c/{q(D_n)}$. The scaled variable is $Z = \lambda A$.
%We analyze the first error term $T_1$ from Lemma \ref{lem:exptrun-error}. Note that the truncation threshold is exactly $ \lambda q(D_n) = c$. Then b
By the Assumption \ref{assump:lower-bound}(1) and Lemma \ref{lem:moment-growth-bound}, for any $p \in \N$ such that $p\leq 2D_n+2$,
\begin{equation*}
    \norm{Z}_{p} = \frac{c\norm{A}_{p}}{q(D_n)} \le \frac{cK B_n p^{\kappa}}{c B_n  D_n^{\kappa}} \le C_1 \left( \frac{p}{D_n} \right)^{\kappa} \le C_1 3^{\kappa}.
\end{equation*}
Therefore, $\sup_{p \in [2D_n+2]}\norm{Z}_p \le C_1 3^{\kappa}$. Since $Z$ is just a rescale of the random variable $A$, the quantile function of $|Z|$ is $q_{|Z|}(t) = \lambda q(t)$, we have
\[
q_{|Z|}(C_t D_n) = \frac{cq(C_t D_n)}{q(D_n)} \leq \frac{c C B_n (C_t D_n)^{\kappa}}{c B_n D_n^\kappa} = {C C_t^{\kappa}}.
\]
By Lemma \ref{lem:exptrun-error}, $\epsilon_1(D_n) = o(e^{-C_t D_n / 2})$.  According to our assumption, $D_n = \omega(1)$, $q(D_n) \ge c B_n D_n^{\kappa} = \omega (e^{-C D_n})$. Choosing the truncation constant $C_t$ (in the definition of the integral range) such that $C_t = 20 C > 0$ sufficiently large, then this verifies $\epsilon_1(D_n) = o(q(D_n))$.

Similarly, Lemma \ref{lem:quantile} gives a matching upper bound on $\widetilde{q}(D)$, i.e.\ there exists some new constant $\widetilde C > \widetilde c > 0$, such that $\widetilde c B_n t^\kappa \leq\widetilde{q}(t) \leq \widetilde C B_n t^\kappa$, where $B_n,\kappa$ are the same as given in Assumption \ref{assump:lower-bound}(1). Following the same logic, $Z = \lambda S$ also satisfies $\sup_{p \in [2D_n+2]}\norm{Z}_p \le \widetilde C_1$ and $q_{|Z|}(C_t D_n)\le \widetilde C_2$ for some constants $\widetilde C_1 , \widetilde C_2 > 0$. By Lemma \ref{lem:exptrun-error}, $\epsilon_2(D_n) = o(e^{-C_t D_n / 2}) = o(q{ (D_n) })$.

Since $|\E[A \exp_{\leq D_n}{(\lambda A)}]| \leq  
|\E[ A e^{\lambda A} 1_{|A| \leq q(C_t D_n)} ]| + \epsilon_1(D_n) \leq q(C_t D_n) e^{\lambda q(C_t D_n)} + \epsilon_1(D_n) < +\infty$, we have verified that
$|\E_{X,X'}[\langle X, X' \rangle \exp_{\leq D_n}(\lambda \langle X, X' \rangle) ]| < + \infty$. It is the same to verify $|\E_{X',X''}[h(X',X'')]| < + \infty$. Therefore, \eqref{eq:slln-1} and \eqref{eq:slln-2} are valid.

Therefore, the lower bound simplifies to 
\[
\mathrm{Corr}^{\le D_n}_{P_0}(\lambda) \ge c \frac{\E\left[ A e^{\lambda A} 1_{|A| \leq q(D_n)} \right]}{\sqrt{q(D_n)}} - o(1).
\]
% Where in the last step we use $e^x-e^{-x}\ge 2x$ for all $x\ge 0$.
Since $x e^x \ge x + \frac{1}{2}x^2$ holds for all $x \in [-1, 1]$, substituting $x = \lambda A$ and truncating on the event $|x| \le \lambda q(D_n) \le 1$, we have
\[
\E\left[ A e^{\lambda A} 1_{|A| \leq q(D_n)} \right] \ge \E [A 1_{|A| \leq q(D_n)}] + \frac{1}{2} \lambda \E [A^2 1_{|A| \leq q(D_n)}]\ge \E [A] - |\E [A1_{|A| \geq q(D_n)}]| + \frac{1}{2} \lambda \E [A^2 1_{|A| \leq q(D_n)}].
\]
%Let $f_{|a|}(x)$ be the PDF of $|a|$. We introduce the change of variables $x = q(t)$ such that $f_{|a|}(q(t)) q'(t) = e^{-t}$ according to \eqref{eq:pdf}. 

Plugging in $g(|A|) = |A|^2 1_{|A| \leq q(D_n)}$ in Lemma \ref{lem:change-of-variable}, we have
\begin{align*}
    \E [A^2 1_{|A| \leq q(D_n)}] &= \int_0^{\infty} q(t)^2 1_{q(t) \leq q(D_n)} e^{-t} \d t = \int_0^{D_n} q(t)^2  e^{-t}\d t.
\end{align*}
% \begin{align}\label{eq:change-of-variable}
%     \E\left[ a e^{\lambda^2 a} 1_{|a| \leq q(D)} \right] &= \int_0^{q(D)} |a| (e^{\lambda^2 |a|}-e^{-\lambda^2 |a|}) f_{|a|}(|a|)\d |a| = \int_0^{D} q(t) (e^{\lambda^2 q(t)}-e^{-\lambda^2 q(t)}) f_{|a|}(q(t)) q'(t)\d t\nonumber\\
%     &= \int_0^{D} q(t)(e^{\lambda^2 q(t)}-e^{-\lambda^2 q(t)}) e^{ - t} \d t \ge 2\lambda^2 \int_0^{D} q(t)^2 e^{ - t} \d t.
% \end{align}

By Assumption \ref{assump:lower-bound}(1), $c_1 B_n t^{\kappa} \leq q(t)$ for all $t\leq D_n$, there exists a new absolute constant $c'>0$, such that 
\begin{align*}
\int_0^{D_n} q(t)^2 e^{ - t} \d t &\ge c_1 B_n^2 \int_0^{D_n} t^{2\kappa} e^{-t} \d t
= c_1 B_n^2 \left(\Gamma(2\kappa + 1) - \int_{D_n}^\infty t^{2\kappa} e^{-t} \d t\right) \\
&\ge c_1 B_n^2 \left(\Gamma(2\kappa + 1) - CD_n^{2\kappa + 1} e^{-D_n}\right) \ge c' B_n^2.
\end{align*}
By our definition of $A$, $\E[A] = \E_{X,X'\sim P_0, X\perp X'} [\langle X,X'\rangle] = \|\E X\|^2 \ge 0$, and using Lemma \ref{lem:moment-growth-bound}, $|\E [A1_{|A| \geq q(D_n)}]| \leq \sqrt{\E[A^2] \P(|A| \geq q(D_n))} \leq e^{-D_n/2} \E[A^2] \leq C e^{-D_n/2} B_n^2$. To sum up, we get
\[
\E\left[ A e^{\lambda  A} 1_{|A| \leq q(D_n)} \right] \ge c' B_n^2 (\lambda  - C' e^{-D_n/2}).
\]
 Since $\lambda = c/ {q(D_n)} \ge cB_n^{-1 } D_n^{-\kappa }=\omega(e^{-D_n/2})$, finally, there exists a constant $C>0$ such that
\[
\mathrm{Corr}^{\le D_n}_{P_0}(\lambda) \ge \frac{c'}{2} \cdot \frac{B_n D_n^{-\kappa}}{\sqrt{q(D_n)}} \ge C D_n^{-2\kappa} \sqrt{q(D_n)}.
\]
\end{proof}

\subsection{Proofs of Key Lemmas}\label{sec:keylem_lb}

\begin{proof}[Proof of Lemma \ref{lem:corr-non-dec}]
We will prove for any $0<h_1<h_2$, $\mathrm{Corr}^{\leq D_n}_{P_0}(h_1^2) \leq \mathrm{Corr}^{\leq D_n}_{P_0}(h_2^2)$.  
Denote $Y_1 = h_1 X + Z$, $Y_2 = h_2 X + Z$. 
Since $X\perp\!\!\!\perp Z\sim N(0,I_N)$, then there exists a random vector $W\perp\!\!\!\perp Y_2$, $W\sim N(0,I_N)$, 
\[
Y_1 \stackrel{d}{=} \frac{h_1}{h_2} Y_2 + \sqrt{1 - \frac{h_1^2}{h_2^2}} W\,.
\]
According to \cite[Claim A.2]{Schramm_2022}, adding noise can only make the MMSE larger, which leads to 
\[\mathrm{MMSE}^{\le D}_{P_0}\left( Y_1\right) = \mathrm{MMSE}^{\le D}_{P_0}\left( \frac{h_1}{h_2} Y_2 + \sqrt{1 - \frac{h_1^2}{h_2^2}} W\right) \ge \mathrm{MMSE}^{\le D}_{P_0}\left( \frac{h_1}{h_2} Y_2\right) = \mathrm{MMSE}^{\le D}_{P_0}( Y_2).\]
Therefore, simple algebra yields
\[
(\mathrm{Corr}_{P_0}^{\le D})^2(h_1^2)   =\mathbb{E}_{X \sim P_0}[\|X\|^2_2]-\mathrm{MMSE}^{\le D}_{P_0}(Y_1)\leq \mathbb{E}_{X \sim P_0}[\|X\|^2_2]-\mathrm{MMSE}^{\le D}_{P_0}(Y_2)=  (\mathrm{Corr}_{P_0}^{\le D})^2(h_2^2).
\]
\end{proof}

\begin{proof}[Proof of Lemma \ref{lem:exptrun-error}]
Let $\Delta$ denote the quantity of interest:
\[
\Delta = \left| \E[V \exp_{\leq D_n}(V)] - \E[V e^V 1_{|V| \leq q(CD_n)}] \right|.
\]
We decompose the expectation of the polynomial term over the event $E = \{|V| \leq q(CD_n)\}$ and its complement $E^c = \{|V| > q(CD_n)\}$:
\[
\E[V \exp_{\leq D_n}(V)] = \E[V \exp_{\leq D_n}(V) 1_{E}] + \E[V \exp_{\leq D_n}(V) 1_{E^c}].
\]
Substituting this into $\Delta$ and applying the triangle inequality:
\[
\Delta \le \underbrace{\E\left[ |V| \cdot \left| \exp_{\leq D_n}(V) - e^V \right| 1_{E} \right]}_{T_1} + \underbrace{\E\left[ |V| \cdot \left| \exp_{\leq D_n}(V) \right| 1_{E^c} \right]}_{T_2}.
\]

\paragraph{Bounding $T_1$:}
By Taylor's Theorem with Lagrange remainder, for any $v$, there exists $\xi$ between $0$ and $v$ such that $e^v - \exp_{\leq D_n}(v) = \frac{e^{\xi}}{(D_n+1)!}v^{D_n+1}$.
On the event $E = \{|V| \leq q(CD_n)\}$, we have $|V| \le q(CD_n)$. Since $q(CD_n) > 0$, $\xi \le \max(0, V) \le q(CD_n)$, which implies 
\[
\left| \exp_{\leq D_n}(V) - e^V \right| 1_{E} \le \frac{e^{q(CD_n)}}{(D_n+1)!} |V|^{D_n+1} 1_{E}.
\]
Substituting this into $T_1$ and relaxing the indicator $1_E \le 1$:
\[
T_1 \le \E\left[ |V| \frac{e^{q(CD_n)}}{(D_n+1)!} |V|^{D_n+1} \right] = \frac{e^{q(CD_n)}}{(D_n+1)!} \norm{V}_{D_n+2}^{D_n+2}.
\]

\paragraph{Bounding $T_2$:}
We expand the polynomial and use the triangle inequality:
\[
T_2 = \E\left[ \left| \sum_{k=0}^{D_n} \frac{V^{k+1}}{k!} \right| 1_{|V| > q(CD_n)} \right] \le \sum_{k=0}^{D_n} \frac{1}{k!} \E\left[ |V|^{k+1} 1_{|V| > q(CD_n)} \right].
\]
We apply the Cauchy--Schwarz inequality to the expectation term:
\[
\E\left[ |V|^{k+1} 1_{|V| > q(CD_n)} \right] \le \sqrt{\E[|V|^{2k+2}]} \sqrt{\E[1_{|V| > q(CD_n)}^2]} = \norm{V}_{2k+2}^{k+1} \sqrt{\mathbb{P}(|V| > q(CD_n))}.
\]
Using the tail assumption $\mathbb{P}(|V| > q(CD_n)) \le e^{-CD_n}$, we have $\sqrt{\mathbb{P}(|V| > q(CD_n))} \le e^{-CD_n/2}$. Thus:
\[
T_2 \le e^{-CD_n/2} \sum_{k=0}^{D_n} \frac{\norm{V}_{2k+2}^{k+1}}{k!}.
\]
% Combining the bounds for $T_1$ and $T_2$ yields the first result.
Furthermore, if $\sup_{p \in [2D_n+2]}\norm{V}_p \le C_1$ and {$q(C D_n) \leq C_2$}, then using the bound $(D_n+1)! \ge (D_n/e)^{D_n}$, there exists a sufficiently large constant $C_3>0$, 
\begin{align*}
    T_1 &\leq \frac{e^{C_2} \norm{V}_{D_n+2}^{D_n+2}}{(D_n+1)!} \leq e^{C_2} C_1^2 \left(\frac{e C_1}{D_n}\right)^{D_n} = o\left(e^{-C_3 D_n}\right); \\
    T_2 &\leq e^{-CD_n/2} \sum_{k=0}^{D_n} \frac{\norm{V}_{2k+2}^{k+1}}{k!} \leq e^{-CD_n/2} \sum_{k=0}^{D_n} \frac{C_1^{k+1}}{k!}\leq C_1 e^{C_1} e^{-CD_n/2} = o(e^{-CD_n/2}).
\end{align*}
So in this case, $\Delta = o(e^{-CD_n/2})$.
\end{proof}
\begin{proof}[Proof of Lemma \ref{lem:quantile}]
Since $S = |\inner{X'}{X''}| + |\inner{X}{X'}| + |\inner{X}{X''}|$. Let $V_1, V_2, V_3$ be these identically distributed terms, which shares the same distribution as $|\inner{X}{X'}|$. Using a union bound:
\begin{equation}
    \mathbb{P}(S > 3q(D)) \le \sum_{k=1}^3 \mathbb{P}(V_k > q(D)) = 3e^{-D} = e^{-(D - \ln 3)}.
\end{equation}
Since $q(t)$ is increasing, $q(D - \ln 3) \le  q(D)$. Thus, we can choose $C$ such that the event $\{S \le C q(D)\}$ holds with probability at least $1 - e^{-D}$, which implies $\widetilde{q}(D) \leq C q(D)$.
\end{proof}
\begin{proof}[Proof of Lemma \ref{lem:change-of-variable}]
By the definition of the quantile function given by Definition \ref{def:quantile}, we know that $q(t) = F_{X}^{-1}(1-e^{-t})$ for all $t\ge 0$.  Using then \cite[Lemma 21.1]{van2000asymptotic}, $F_{X}^{-1}(U)\sim X$ for a $U\sim\mathrm{Unif}[0,1]$. Hence, for any integrable function $g$ we have
\begin{equation*}
\E g(X) = \E g(F_{X}^{-1}(U)) = \int_0^1 g(F_{X}^{-1}(u)) \d u = \int_0^\infty g(F_{X}^{-1}(1-e^{-t})) e^{-t} \d t = \int_0^\infty g(q(t)) e^{-t} \d t .
\end{equation*}
\end{proof}
\begin{proof}[Proof of Lemma \ref{lem:moment-growth-bound}]
We express the $p$-th moment using the tail integral representation and apply the change of variables formula given by \eqref{eq:change-of-variable}, the integral transforms to:
\begin{equation*}
    \E[V^p] = \int_0^\infty (q(t))^p e^{-t} \d t.
\end{equation*}
Since $q(t) \le C B_n t^\kappa$ and handling small $t$ by a constant $M$,:
\begin{equation}
    \E[V^p] \le \int_0^\infty (C B_n t^\kappa)^p e^{-t} \d  t = (C B_n)^p \int_0^\infty t^{\kappa p} e^{-t} \d  t = (C B_n)^p \Gamma(\kappa p + 1).
\end{equation}
Taking the $p$-th root, and using the upper bound for Gamma function given in \cite[Pg.\ 26]{vershynin-HDP}:
\begin{equation}
    (\Gamma(\kappa p + 1))^{1/p} \le  \left(C  (\kappa p/e)^{\kappa p} \right)^{1/p} = C^{1/p}(\kappa p/e)^\kappa .
\end{equation}
Thus,
\begin{equation}
    \norm{V}_p = (\E[V^p])^{1/p}\le C B_n (\kappa/e)^\kappa p^\kappa.
\end{equation}
Setting $K = C (\kappa/e)^\kappa$, we obtain the desired bound $\norm{V}_p \le K B_n p^\kappa$.
\end{proof}

\section{Two broad classes of GAMs with positive cumulants}\label{sec:positivity_cumulants}

In this section we present two key Lemmas that help us verify Item 1 from Assumption \ref{assump:mainassumption}, i.e. that the priors on the GAMs of interest have low-order nonnegative cumulants. In fact, we will establish that two broad families of GAMs satisfy the condition, which will include as particular cases the models of interest in our applications. We state them here to a degree of generality for independent interest.

More specifically, all GAMs considered in this section satisfy the following.

Fix integer $t\geq 1$. Suppose we have $v_1, \dots ,v_n \sim P_1$ i.i.d. draws from some distribution $P_1\in \mathcal{P}(\R)$ that has finite moments and $Z_1, \dots ,Z_t$ some i.i.d.  random variables for all $j\in [t]$ that they are either symmetric around $0$ or satisfy $Z_j\geq 0, j\in [t]$ a.s. Furthermore, $Z_j$ are independent from $v_i$ for all $i\in [n]$ and $j\in [t]$. 
Then we assume that for all $1\leq i\leq N$, 
\begin{align}\label{eq:classofgams}
X_i=Z_{\alpha(i)}\cdot f_i(v_1, \dots ,v_n)    
\end{align}
where $f_i:\mathbb{R}^n\rightarrow \mathbb{R}$ are monomials with respect to $v_1, \dots ,v_n$, i.e., $f(v_1, ... ,v_n)=v_1^{k_1}\dots v_n^{k_n} $ for some positive integers $k_1, \dots , k_n$ and $a(i)\in [t]$ for all $i\in [n]$. Of course all tensor PCA models \cite{montanari2014statisticalmodeltensorpca} fall into this framework.

\subsection{The two GAM classes}

\subsubsection{$P_1$ has nonnegative and  supermultiplicative moments}
Consider first the following assumption on the prior $P_1$. 

\begin{assumption}\label{ass:constrainprior}
Suppose that for some $c>0$ the prior $P_1$ satisfies the following assumptions.
\begin{itemize}
    \item (``Nonnegative moments") For all $k\in \N$ and $v\sim P_1$ it holds that $\mathbb{E}_{P_1}{[ v^k  ]}\geq 0$.
    \item (``Supermultiplicative moments") For $\rho=n^{-c}$, it holds $\mathbb{E}_{P_1}[v^k]\mathbb{E}_{P_1}[v^t]\leq \rho \cdot \mathbb{E}_{P_1}[v^{k+t}]$ for any $k,t \in \mathbb{N}$. 
     
\end{itemize}
\end{assumption}

For all GAMs where $P_1$ satisfies Assumption \ref{ass:constrainprior}, we can prove that they have low-order nonnegative cumulants.

The result is as follows, and the proof of it is deferred to Section \ref{sec:spar_distr}.
\begin{lemma}\label{lem:positivecumforsomegams}
 Suppose we observe a GAM, of the form \eqref{eq:classofgams},
with a prior distribution $P_1$ satisfying Assumption \ref{ass:constrainprior}
 for some $c>0$. Then, for any $ \alpha \in \mathbb{N}^n$ such that $1\leq |\alpha|\leq \log_2(1/\rho)-1$, where $\rho=n^{-c}$, we have
% \[
% 0 \leq \lambda^{|\alpha|/2}\E[X_1X^{\alpha}]/2 \leq \kappa_{\alpha}(X_1, \textbf{X}) \leq \lambda^{|\alpha|/2}\E[X_1X^{\alpha}],\widetilde{\kappa}_{\alpha}(X_1,\textbf{X})\geq 
% \]
%  and 

\[
\kappa_{\alpha}(X_1,X_2,\ldots,X_n)\geq 0,
\]

% \textcolor{red}{Change the proof to include this. Delete below.}
% \[
% \kappa_{\alpha}(X_1,\textbf{X})\geq 0,
% \] 
% where the cumulant $\widetilde{\kappa}(\textbf{X})$ for a vector $\textbf{X}\in \R^N$ was defined in \eqref{def:ktilde}.
\end{lemma}

Now, interestingly, Assumption \ref{ass:constrainprior} for the prior $P_1$ is satisfied by many standard priors.

\begin{lemma}\label{lem:satisfing_positivecum}
Let $\rho=n^{-c}$ for some $c>0$. Then consider $B_n \sim \text{Ber}(\rho)$ and $T$ any real-valued random variable that is independent of $B_n$ and either has a symmetric law around 0 or is non-negative almost surely. Then if $P_1$ is the law of $T\times B_n$, $P_1$ satisfies Assumption \ref{ass:constrainprior} for $c>0.$

 In particular,  $B_n \sim  \text{Ber}(\rho)$
   and $R_n \sim \text{Rad}(\rho)$ satisfy the Assumption \ref{ass:constrainprior}.
\end{lemma}

\begin{proof} 
Let $v\sim P_1$. In the setting of the lemma we may write
\( v = T_n := B_n T, \)
where $B_n\sim \mathrm{Ber}(\rho)$ is independent of $T$.

First, it's easy to see that $\mathbb{E}[T^m]\ge 0$ for every $m\in\mathbb{N}$ under either condition on $T$. Also $B_n\geq 0$ almost surely so  $\mathbb{E}[v^m]\ge 0$ for every $m\in\mathbb{N}$ and the first condition of the assumption is satisfied.

Then, we only need to show the second condition of the Assumption. Fix any $k,t\in\mathbb{N}$ and any $m\in\mathbb{N}$. Since $B_n\in\{0,1\}$ we have
$B_n^m=B_n$, and therefore by independence,
\begin{equation}\label{eq:moments_factorize}
\mathbb{E}_{P_1}\left[v^m\right]
= \mathbb{E}\left[(B_n T)^m\right]
= \mathbb{E}\left[B_n^m\right]\mathbb{E}\left[T^m\right]
= \mathbb{E}[B_n]\mathbb{E}[T^m]
= \rho\,\mathbb{E}[T^m].
\end{equation} Using this we have
\[
\mathbb{E}_{P_1}[v^k]\mathbb{E}_{P_1}[v^t]
=\rho^2\ \mathbb{E}[T^k]\mathbb{E}[T^t],
\qquad
  \mathbb{E}_{P_1}[v^{k+t}]
=\rho \ \mathbb{E}[T^{k+t}].
\]
Thus the desired inequality
\[
\mathbb{E}_{P_1}[v^k]\mathbb{E}_{P_1}[v^t]\le \rho\,\mathbb{E}_{P_1}[v^{k+t}]
\]
is equivalent to
\begin{equation}\label{eq:reduce_to_Z}
\mathbb{E}[T^k]\mathbb{E}[T^t]\le \mathbb{E}[T^{k+t}].
\end{equation}
If $T\ge 0$ almost surely, then by H\"older's inequality with exponents $\frac{k+t}{k}$ and $\frac{k+t}{t}$,
\[
\mathbb{E}[T^k]
=\mathbb{E}\left[(T^{k+t})^{\frac{k}{k+t}}\right]
\le \mathbb{E}[T^{k+t}]^{\frac{k}{k+t}},
\qquad
\mathbb{E}[T^t]
=\mathbb{E}\left[(T^{k+t})^{\frac{t}{k+t}}\right]
\le \mathbb{E}[T^{k+t}]^{\frac{t}{k+t}}.
\]
Multiplying yields \eqref{eq:reduce_to_Z}.

If $T$ has a law symmetric around $0$, then whenever at least one of $k,t$ is odd we have
$\mathbb{E}[T^k]\mathbb{E}[T^t]=0$. Also, $\mathbb{E}[T^{k+t}]=0$ if $k+t$ is odd, while
$\mathbb{E}[T^{k+t}]\ge 0$ if $k+t$ is even, since then $T^{k+t}\ge 0$ almost surely. Hence
\eqref{eq:reduce_to_Z} holds. If $k$ and $t$ are both even, set $Y:=T^2\ge 0$. Then
\[
\mathbb{E}[T^k]\mathbb{E}[T^t]=\mathbb{E}[Y^{k/2}]\,\mathbb{E}[Y^{t/2}]
\le \mathbb{E}[Y^{(k+t)/2}]
=\mathbb{E}[T^{k+t}].
\]
Thus \eqref{eq:reduce_to_Z} holds in all cases, proving the condition for $P_1$.

In particular,
taking $T= 1$ a.s.  gives $v=B_n$
and
taking $T$ to be a $\mathrm{Rad}(1/2)$ random variable independent of $B_n$ gives $v=B_n T$,
which has law $\mathrm{Rad}(\rho)$, so $B_n$ and $R_n\sim \mathrm{Rad}(\rho)$ satisfy Assumption \ref{ass:constrainprior}.
\end{proof}

\subsubsection{$P_1$ has nonnegative cumulants}

 We now turn to the second family of GAMs. 
 
 First, for a distribution $\mu \in \mathcal{P}(\R)$ and any $r\in  \N$ we denote the $r$-cumulant of $\mu$ by
\[
\kappa_r(X):=\kappa(\underbrace{X, X, \dots ,X}_{r-\text{times}}),
\] 
where $X\sim \mu$. Now given that the second assumption asks for all $r$-cumulants of $P_1$ to be nonnegative.
\begin{assumption}[Nonnegative cumulants]\label{ass:positivecum}
 Suppose that for a random variable $X\sim P_1$ for any $r\in \N$,
\[
\kappa_r(X)\geq 0.
\]
\end{assumption}
 Note that Assumption \ref{ass:positivecum} is satisfied by multiple standard distributions, such as the Gaussian $\mathcal{N}(\mu, \sigma^2)$ with $\mu\geq 0$, the Poisson $\text{Pois}(\lambda)$ for any $\lambda>0$, and the Exponential $\text{Exp}(\lambda)$ for any $\lambda>0$, see e.g.~\cite[Chap 2]{McCullagh1987TensorMI} for details.

Now, suppose we have a GAM of the form \eqref{eq:classofgams}, with $Z_j=1$ for all $j\in [t]$, i.e.
\begin{align}\label{eq:subclassofgams}
X_i= f_i(v_1, \dots ,v_n)    
\end{align}
where $f_i:\mathbb{R}^n\rightarrow \mathbb{R}$ are monomials with respect to the i.i.d. random variables $v_1, \dots ,v_n$. If the distribution $P_1$, that $v_i, i\in [n]$ follow, satisfies  Assumption \ref{ass:positivecum} then the following holds. Notice that the lemma directly implies the low-order cumulant nonnegative property, alongside a technical useful second condition.
\begin{lemma}\label{lem:all_positive_cumulants}
For any GAM of the form \eqref{eq:subclassofgams} with distribution $P_1 \in \mathcal{P}(\R)$ satisfies Assumption \ref{ass:positivecum} the following holds. For all $\alpha\in \N^N$,  \[\kappa_{\alpha}(X_1,\dots , X_N)\geq 0,\] and $\widetilde{\kappa}_{\alpha}(X_1,\dots , X_N)\geq\kappa_{\alpha}^2(X_1,\dots , X_N)$. 
\end{lemma}
The proof of this Lemma is deferred to Section \ref{sec:non_cum}

\subsection{Proof of Lemma \ref{lem:positivecumforsomegams}}\label{sec:spar_distr}

For the remaining of this section we remind the reader that for convenience we have defined,
\[
\kappa_{\alpha}(X_1, \textbf{X})=\kappa_{\alpha}(X_1, \underbrace{X_1, \dots , X_1}_{\alpha_1-\text{times}}, \dots , \underbrace{X_N, \dots ,X_N}_{\alpha_N-\text{times}}).
\]

Before we move on with the proof a definition is in order. 
\begin{definition}[The dependency multigraph]
For the class of GAMs of interest notice that for any $\kappa_{\alpha}( X_1, \dots , X_N)$ the multi-index $\alpha \in \mathbb{R}^N$ naturally corresponds to a ``dependency" multigraph with vertices $\{1, \dots , n\}$ where for any $i \in [N]$ if $\alpha_i\neq 0$ then the multigraph has a clique between the vertices $j\in [n]$ s.t. $v_j\in \text{supp}(f_i)$. 
\end{definition}

We now state a useful corollary for the proof which is immediate from Proposition \ref{prop:cumulant-vanishes-indep}.

\begin{corollary}
For any $\alpha\in \mathbb{R}^N$, if it's dependency multigraph is not connected then $\kappa_{\alpha}(X_1, \dots , X_N)=0$.
\end{corollary}
Therefore, we may restrict our attention for multi-indices $\alpha \in \N^n$ for which their dependency multigraph is connected.

Now we proceed with the proof.

\begin{proof}[Proof of Lemma \ref{lem:positivecumforsomegams}]

For convenience, we assume $\lambda=1$ as all conditions are homogeneous with respect to $\lambda$. 

Also, note that we will prove for convenience that $\kappa_{\alpha}(X_i,\textbf{X})\geq 0$ for some fixed $i\in [n]$ and for all multindices $\alpha$  with cardinality $1\leq |\alpha|\leq \log_2(1/\rho)-1.$ Combining this with the fact that $\kappa(X_i)\geq 0$ for all $i\in [n]$, since $X_i$ are for all $i\in [n]$ the product of powers of the independent random variables $v_i, i\in [n]$ that have nonnegative moments, we will have $\kappa_{\alpha}(X_1,\dots , X_n)\geq 0$  for all $\alpha$ such that  $1\leq |\alpha|\leq \log_2(1/\rho)-1$ as desired. 

Without loss of generality for the remaining of the proof we fix $i=1$ and we show that $\kappa_{\alpha}(X_1,\textbf{X})\geq 0$ to prove the result.

Consider any GAM, of the form \eqref{eq:classofgams},
with a prior distribution $P_1$ satisfying Assumption \ref{ass:constrainprior}
 for some $c>0$. We will start by proving via induction on $|\alpha|$ the following:
\begin{equation}\label{eq:induc.poscum}
\frac{1}{2}\E[X_1X^{\alpha}]\leq \kappa_{\alpha}\leq \E[X_1X^{\alpha}].
\end{equation} Notice that for $\alpha=\textbf{0}$,  $\kappa_0=\mathbb{E}[X_1]\geq 0$ by Assumption \ref{ass:constrainprior}.

Now assume that the induction hypothesis  holds for any $\beta$ with $|\beta|<|\alpha|$. Using Proposition \ref{prop:cumulant-recursion} we know that for all multindices $\alpha \in \N^N$  the cumulant $\kappa_{\alpha}$ satisfies the following recursive formula,
\begin{align}
\kappa_{\alpha}=\mathbb{E}[X_1X^{\alpha}]-\sum_{0\leq \beta \prec\alpha} \kappa_{\beta} \binom{\alpha}{\beta}\mathbb{E}[X^{\alpha-\beta}].  
\end{align}
The induction hypothesis, though, implies that for all $P_1$ satisfying Assumption \ref{ass:constrainprior} it holds that $ \kappa_{\beta}\geq 0$ for all $\beta \prec\alpha$, since $X_1X^{\alpha}$ can be written as a product of powers of the independent $v_i$'s and the moments of $v_i$ are nonnegative. This implies that the second term is non-positive  and thus we obtain $\kappa_{\alpha} \leq \mathbb{E}[X_1X^{\alpha}]$. To finish the induction step, we need to show that  $\kappa_{\alpha} \geq \frac{1}{2}\mathbb{E}[X_1X^{\alpha}]$. For the lower bound we upper bound the second term of the recursive formula for $\kappa_{\alpha}$ by using Assumption \ref{ass:constrainprior} in the following way: first
\begin{align*}
\sum_{0\leq \beta <\alpha} \kappa_{\beta} \binom{\alpha}{\beta}\mathbb{E}[X^{\alpha-\beta}]\leq \sum_{0\leq \beta \prec\alpha}  \binom{\alpha}{\beta}\mathbb{E}[X_1X^{\beta}]\mathbb{E}[X^{\alpha-\beta}]
\end{align*}
where we used the induction hypothesis once again. We want to show that for every multi-index $0\leq \beta\prec \alpha$ it holds that 
\begin{equation}\label{eq:upperbound_allbeta}
\mathbb{E}[X_1X^{\beta}]\mathbb{E}[X^{\alpha-\beta}]\leq \E[X_1X^{\alpha}]\rho.
\end{equation}
Note though that $X_i$ are of the form \ref{eq:classofgams} and $Z_j, j\in [t]$ are independent from $v_i, i\in [n]$. Therefore, to prove our desired inequality it suffices to show that for all suh $\beta$ it holds that \[\mathbb{E}[Z_i^k]\mathbb{E}[Z_i^t]\le \mathbb{E}[Z_i^{k+t}] \quad \text{and} \quad  \E[v_j^{\beta_j}]\E[v_j^{\alpha_j-\beta_j}]\leq \rho \E[v_j^{a_j}]\] for all $i\in [t]$ and at least one $j\in [n]$ (we have a product of i.i.d. random variables so we can treat each coordinate separately). For the first inequality notice that this is an implication of the fact that $Z_j, j\in [t]$ are symmetric or non-negative a.s. --  we proved this also in Lemma \ref{lem:satisfing_positivecum}. For the second inequality we do the following. Since $\alpha$ is a connected multi graph then for any $0<\beta<\alpha$ s.t. $\kappa_{\beta}>0$ it should hold that $\beta$ is also connected. Furthermore, 
\begin{equation}\label{eq:vmultigraphs_vertcies}
|V(\beta)|+|V(\alpha-\beta)|\geq |V(\alpha)|+1.
\end{equation}
Indeed, if that wasn't the case then the vertices of $V(\alpha-\beta)$ and $V(\alpha)$ would form two components of $\alpha$ that are not connected which leads to a contradiction. Equation \eqref{eq:vmultigraphs_vertcies} implies that there will always exist a coordinate $j \in [n]$ s.t. $j\in (\{1\}\cup\text{supp}(\beta))\cap \text{supp}(\alpha)$. For this coordinate the expression
\begin{align*}
\mathbb{E}[X_1X^{\beta}]\mathbb{E}[X^{\alpha-\beta}] 
\end{align*}
contains the term $\E[v_j^{\beta_j}]\E[v_j^{\alpha_j-\beta_j}]$ with $\beta_j>0$. Using  Assumption \ref{ass:constrainprior} we can upper bound  $\E[v_j^{\beta_j}]\E[v_j^{\alpha_j-\beta_j}]\leq\E[v^{a_j}]\rho$ and the proof of \eqref{eq:upperbound_allbeta} is complete. Using this equation, we get the following:
\begin{align*}
\sum_{0\leq \beta \prec\alpha}  \binom{\alpha}{\beta}\mathbb{E}[X_1X^{\beta}]\mathbb{E}[X^{\alpha-\beta}]& \leq \sum_{0\leq \beta \prec\alpha}  \binom{\alpha}{\beta}\mathbb{E}[X_1X^{\alpha}]\rho\\
& \leq 2^{|\alpha|}\mathbb{E}[X_1X^{\alpha}]\rho\\
& \leq \frac{1}{2}\mathbb{E}[X_1X^{\alpha}]. 
\end{align*}
where we used the fact that $|\alpha|\leq \log_2(1/\rho)-1$ to obtain the last inequality. The upper bound implies $\kappa_{\alpha} \geq \frac{1}{2}\mathbb{E}[X_1X^{\alpha}]$ and our proof for the induction is complete.

Using now \ref{eq:induc.poscum},  we know that for any prior $P_1$ satisfying  Assumption \ref{ass:constrainprior}  for all $\alpha$ such that $|\alpha|\leq C\log n$ and for some universal constant $C:=C(\rho)$ the cumulants $\kappa(X_1, \textbf{X})$ of the prior satisfy $\kappa(X_1, \textbf{X})\geq 0$ and the proof is complete. 

% Therefore, for all such $\alpha$ using corollary \ref{cor:ktildegeqksquared} we get that $\widetilde{\kappa}_{\alpha}(X_1, \dots ,X_N)\geq\kappa_{\alpha}^2(X_1, \dots ,X_N)$ and the proof is complete. 
\end{proof}

 \subsection{Proof of Lemma \ref{lem:all_positive_cumulants}}\label{sec:non_cum}

\begin{proof}[Proof of Lemma \ref{lem:all_positive_cumulants}]
We will make use again of Theorem \ref{thm:cumofpolynom}.

Given any GAM of the form \eqref{eq:classofgams} we want to decompose any cumulant of the form 
\(
\kappa_{\alpha}(f_1, \dots ,f_N)
\), where $\alpha\in \N^N$.
Since $f_i$ are for all $i\in [N],$ are equal to $f_i(v_1, \dots ,v_n)=v_1^{k^i_1}\dots v_n^{k^i_n}$ for some positive integers $k^i_1, \dots, k^i_n$ we  apply Theorem \ref{thm:cumofpolynom} with respect to the \[\sum_{j=1}^N\alpha_j\Big(\sum_{i=1}^n k_i^j\Big)\] random variables that appear in this cumulant (note that we have many of the $v_i$'s repeated multiple times in this family of random variables, but we yet consider them as different random variables for this argument). To explain this further consider the following: for $i=1$, $f_1(v_1, \dots , v_n)$ is the product of $S_1=\sum_{i=1}^nk_i^1$ random variables. Relabeling these variables to $R_1^1, \dots , R_{S_1}^1$ and repeating the same process for all $f_i, i\in [n]$ we can rewrite the cumulant of interest as:
\[
\kappa_{\alpha}(f_1, \dots , f_n) =\kappa_{\alpha} (R_1^1\cdot \dots \cdot R_{S_1}^1, \dots , R_1^N \cdot \dots \cdot R_{S_n}^N)
\]

Then, Theorem \ref{thm:cumofpolynom} gives a decomposition of the form
\[
\kappa_{\alpha}(f_1, \dots ,f_N)=\sum_{\pi'\in G}\prod_{S'\in \pi'}\kappa(S')
\]
where importantly, 
all the cumulants that appear on the right hand side take the form 
 \[
 \kappa(v_{i_1}, \dots v_{i_t}), \qquad \text{for some} \quad i_{1}, \dots , i_{t}\in [n], t\geq 1,
 \] 
 where $i_{1}, \dots , i_{t}$ are not necessarily distinct. We argue that the only terms that are non-zero are of the form $\kappa(v_i, \dots , v_i)$ for some $i\in [n]$. Indeed, if we had a term corresponding to a part $S'$ that contained at least two $v_i,v_j\in S'$ for two distinct integers $i,j\in [n]$ then  $\kappa(S')=0$, since $v_i,v_j$ are independent. Using the notation we introduced above, this implies that the sum will only involve $r-$th cumulants of the prior distribution on $v_i$ for $r\in\N$. Assumption \ref{ass:positivecum} guarantees that these terms are all nonnegative therefore for all $\alpha \in \N^N$
\[
\kappa_{\alpha}(f_1, \dots ,f_N)\geq 0.
\]
Lastly, using Lemma \ref{lem:k_tilde_formula} we also conclude that $\widetilde{\kappa}_{\alpha}(X_1, \dots X_N)\geq\kappa_{\alpha}^2(X_1, \dots X_N)$ and our proof is complete.
\end{proof}

\section{Background on modified Bessel functions}\label{app:bessel}

Throughout this appendix, let $d\ge 1$ be an integer and let
$G,H\in\mathbb R^d$ be independent $N(0,I_d)$ vectors.
Define
\[
W:=\langle G,H\rangle=\sum_{i=1}^d G_iH_i,
\]
and let $f_d$ denote the density of $W$.
For background on modified Bessel functions, see
\cite{yang_approximating_2017,abramowitz1965handbook}.

For $\nu\in\mathbb R$, the modified Bessel function of the first kind is defined by
\[
I_\nu(x)
:=
\sum_{m=0}^\infty
\frac{1}{m!\,\Gamma(m+\nu+1)}
\left(\frac{x}{2}\right)^{2m+\nu},
\qquad x>0.
\]
Especially, for $\nu=0$, we have the integral representation
$$I_0(x) = \frac{1}{\pi} \int_0^\pi \exp(x \cos \theta) d\theta.$$
The equivalence of the integral representation to the series definition for $\nu=0$ follows directly from the Taylor expansion of the exponential function. Expanding $\exp(x \cos \theta)$ and exchanging the sum and integral by uniform convergence yields $\frac{1}{\pi} \sum_{k=0}^\infty \frac{x^k}{k!} \int_0^\pi \cos^k \theta d\theta$. The integral of $\cos^k \theta$ over $[0, \pi]$ vanishes by symmetry for odd $k$. For even powers $k=2m$, standard trigonometric integration yields $\int_0^\pi \cos^{2m} \theta d\theta = \pi \frac{(2m)!}{2^{2m}(m!)^2}$. Substituting this evaluation into the summation, the $\pi$ and $(2m)!$ terms cancel, leaving exactly $\sum_{m=0}^\infty \frac{1}{(m!)^2} \left(\frac{x}{2}\right)^{2m}$. Since $\Gamma(m+1) = m!$, this perfectly recovers the series definition of $I_0(x)$. 

\subsection{Modified Bessel functions of the second kind}
The modified Bessel function of the second kind is then defined, for
$\nu\notin \mathbb Z$, by
\[
K_\nu(x):=\frac{\pi}{2\sin(\pi \nu)}\bigl(I_{-\nu}(x)-I_\nu(x)\bigr),
\qquad x>0,
\]
with the integer-order case obtained by continuity, see
\cite[(1.1)]{yang_approximating_2017}.

It also admits the integral representation
\[
K_\nu(x)=\int_0^\infty e^{-x\cosh u}\cosh(\nu u)\,du,
\qquad x>0,
\]
see \cite[(1.2)]{yang_approximating_2017} and \cite[Chap.~9]{abramowitz1965handbook}.
In particular, $K_\nu(x)>0$ for all $x>0$, and $K_{-\nu}(x)=K_\nu(x)$.

We shall use the following standard identities.

\begin{lemma}[Recurrences and derivative formulas for $K_\nu$]\label{lem:bessel-basic}
For every $x>0$ and every real $\nu$,
\begin{align}
K_{\nu+1}(x) &= K_{\nu-1}(x)+\frac{2\nu}{x}K_\nu(x), \label{eq:bessel-rec}\\
\frac{d}{dx}\bigl(x^\nu K_\nu(x)\bigr) &= -x^\nu K_{\nu-1}(x), \label{eq:bessel-der-1}\\
K_\nu'(x) &= -K_{\nu-1}(x)-\frac{\nu}{x}K_\nu(x)
= -K_{\nu+1}(x)+\frac{\nu}{x}K_\nu(x), \label{eq:bessel-der-2}\\
\frac{d}{dx}\log K_\nu(x)
&=
-\frac{\nu}{x}-\frac{K_{\nu-1}(x)}{K_\nu(x)}
=
\frac{\nu}{x}-\frac{K_{\nu+1}(x)}{K_\nu(x)}. \label{eq:bessel-logder}
\end{align}
Equivalently,
\begin{equation}\label{eq:bessel-logder-shifted}
\frac{d}{dx}\log \bigl(x^\nu K_\nu(x)\bigr)
=
-\frac{K_{\nu-1}(x)}{K_\nu(x)}.
\end{equation}
\end{lemma}

\begin{proof}
The recurrence \eqref{eq:bessel-rec} and derivative identity \eqref{eq:bessel-der-1}
are standard, see \cite[(9.6.26)]{abramowitz1965handbook}.
Expanding \eqref{eq:bessel-der-1} gives the first formula in \eqref{eq:bessel-der-2},
while the second follows by combining the first with \eqref{eq:bessel-rec}.
Dividing by $K_\nu(x)$ yields \eqref{eq:bessel-logder}, and
\eqref{eq:bessel-logder-shifted} is just \eqref{eq:bessel-der-1}
after taking logarithmic derivatives.
\end{proof}

The small and large argument asymptotics we use are the following.

\begin{lemma}[Basic asymptotics]\label{lem:bessel-asympt}
Fix $\nu>0$. As $x\downarrow 0$,
\begin{equation}\label{eq:bessel-small}
K_\nu(x)=\frac{(1+o(1))}{2}  \Gamma(\nu)\Bigl(\frac{x}{2}\Bigr)^{-\nu},
\end{equation}
while
\(
K_0(x)= -(1+o(1))\log x.\) As $x\to\infty$,
\begin{equation}\label{eq:bessel-large}
K_\nu(x)=(1+o(1))
\sqrt{\frac{\pi}{2x}}e^{-x}
\left(
1+\frac{4\nu^2-1}{8x}+O_\nu(x^{-2})
\right).
\end{equation}
Consequently, for any $\nu_1,\nu_2=O(1)$,
\[
\frac{K_{\nu_1}(x)}{K_{\nu_2}(x)}\to 1
\qquad\text{as }x\to\infty.
\]
\end{lemma}

\begin{proof}
These are standard, see \cite[(1.3)--(1.5)]{yang_approximating_2017}
and \cite[Chap.~9]{abramowitz1965handbook}. The ratio limit follows immediately
from the common leading term $\sqrt{\pi/(2x)}e^{-x}$.
\end{proof}

For later use we also record two ratio estimates.

\begin{lemma}[Ratio inequalities]\label{lem:bessel-ratio}
For every $\nu\ge 0$ and $x>0$,
\begin{equation}\label{eq:yang-ratio}
\frac{\nu+\sqrt{x^2+\nu^2}}{x}
<
\frac{K_{\nu+1}(x)}{K_\nu(x)}
<
\frac{\nu+\frac12+\sqrt{x^2+(\nu+\frac12)^2}}{x}.
\end{equation}
% Moreover, for every $x>0$,
% \begin{equation}\label{eq:yang-K1K0}
% 1+\frac{1}{2(x+\frac14)}
% <
% \frac{K_1(x)}{K_0(x)}
% <
% 1+\frac{1}{2x}.
% \end{equation}
\end{lemma}

\begin{proof}
The inequality \eqref{eq:yang-ratio} is exactly \cite[(1.10)]{yang_approximating_2017}.
% The bound \eqref{eq:yang-K1K0} is stated in  \textcolor{red}{Abstract? We need to cite theorems/proofs. Yes will try tofind a citation for that }.
\end{proof}

\begin{corollary}[Large-order consequence of the ratio bound]\label{cor:bessel-ratio-large-order}
Let $\nu\to\infty$ and let $x=x_\nu>0$ satisfy $x=o(\nu)$. Then
\begin{equation}\label{eq:bessel-ratio-large-order}
\frac{K_{\nu-1}(x)}{K_\nu(x)}
=
\frac{x}{2\nu}\bigl(1+o(1)\bigr).
\end{equation}
\end{corollary}

\begin{proof}
Apply \eqref{eq:yang-ratio} with $\nu$ replaced by $\nu-1$:
\[
\frac{\nu-1+\sqrt{x^2+(\nu-1)^2}}{x}
<
\frac{K_\nu(x)}{K_{\nu-1}(x)}
<
\frac{\nu-\frac12+\sqrt{x^2+(\nu-\frac12)^2}}{x}.
\]
Since $x=o(\nu)$,
\[
\nu-1+\sqrt{x^2+(\nu-1)^2}=2\nu+o(\nu),
\qquad
\nu-\frac12+\sqrt{x^2+(\nu-\frac12)^2}=2\nu+o(\nu).
\]
Hence
\[
\frac{K_\nu(x)}{K_{\nu-1}(x)}
=
\frac{2\nu}{x}\bigl(1+o(1)\bigr),
\]
and taking reciprocals gives \eqref{eq:bessel-ratio-large-order}.
\end{proof}

\subsection{Density of a Gaussian inner product}

The distribution of $W=\langle G,H\rangle$ can be written explicitly in terms of $K_\nu$.

\begin{lemma}[Explicit density of a Gaussian inner product]\label{lem:innerprod-bessel}
Let \( W=\langle G,H\rangle=\sum_{i=1}^d G_iH_i.
\) Then \(
W=\frac{X-Y}{2}, \) where $X,Y$ are independent $\chi^2_d$ random variables. Moreover, $W$ has density
\begin{equation}\label{eq:fd-bessel}
f_d(x)
=
\frac{1}{\sqrt{\pi}\,\Gamma(d/2)\,2^{\,d/2-1/2}}
\,|x|^{\,d/2-1/2}\,
K_{\,d/2-1/2}(|x|),
\qquad x\in\mathbb{R}\setminus\{0\}.
\end{equation}
For $d\ge 2$, the formula extends continuously to $x=0$. For $d=1$,
\[
f_1(x)=\frac{1}{\pi}K_0(|x|),
\]
and $f_1(x)=-\frac{(1+o(1))}{\pi}\log|x|$ as $x\to0$.
\end{lemma}

The proof of this Lemma is deferred to Section \ref{sec:aux_for_applic}.

\section{Proofs for applications}\label{sec:pf_applications}

Before we prove Theorems \ref{thm:tpca}, \ref{thm:sptpca} and \ref{thm:spclust} we introduce  state an Auxiliary Lemma the proof of which is deferred to Section \ref{sec:aux_for_applic}.

% Throughout this section, let $d\ge 1$ be an integer and let $G,H\in\mathbb{R}^d$ be independent $N(0,I_d)$ vectors.
% Define the random inner product
% \[
% W:=\langle G,H\rangle=\sum_{i=1}^d G_iH_i.
% \]
% Let $f_d$ denote the density of $W_d$ on $\mathbb{R}$ and $K_\nu$, where $\nu>0,$ denote the modified Bessel function of the second kind with parameter $\nu>0$. For background in Bessel Functions we encourage the reader to read \cite{yang_approximating_2017} and \cite{abramowitz1965handbook}.

% \begin{lemma}[Explicit density of a Gaussian inner product]\label{lem:innerprod-bessel}
% For each integer $d\ge 1$, the random variable $W=\langle G,H\rangle=(X-Y)/2,$ where $X,Y$ are i.i.d. $\chi^2_d$ random variables and $W$ has a continuous density $f_d$ given by
% \begin{equation}\label{eq:fd-bessel}
% f_d(x)
% =
% \frac{1}{\sqrt{\pi}\,\Gamma(d/2)\,2^{\,d/2-1/2}}
% \,|x|^{\,d/2-1/2}\,
% K_{\,d/2-1/2}(|x|),
% \qquad x\in\mathbb{R}.
% \end{equation}
% \end{lemma}

% We are also going to need the following result for the PMF of a Rademacher sum. 

\begin{lemma}[Bounds for Rademacher sum]\label{lem:rad-master-two-regimes}
Let $S_n=\sum_{i=1}^n \varepsilon_i$ with $\varepsilon_i\stackrel{iid}{\sim}\mathrm{Rad}(\pm1)$ and
$\phi(x)=(2\pi)^{-1/2}e^{-x^2/2}$.
For integers $s\equiv n\pmod 2$, the following holds.

Let $T=T_n\ge 1$ satisfy $T^3=o(\sqrt n)$. Then, uniformly for integers $s\equiv n\pmod 2$ with
$|s|\le \sqrt n\,T$,
\begin{equation}\label{eq:rad-master-lclt}
\P(S_n=s)
=
\frac{2}{\sqrt n}\,
\phi\Big(\frac{s}{\sqrt n}\Big)\,
\big(1+o(1)\big).
\end{equation}

Moreover, there exist absolute constants $c,C>0$ and $n_\star\ge 2$ such that for all $n\ge n_\star$
and all integers $s\equiv n\pmod 2$ with $0\le s\le n/2$, it holds
\begin{equation}\label{eq:rad-master-lb}
\P(S_n=s)\ \ge\ \frac{c}{\sqrt n}\exp\Big(-\frac{s^2}{2n}-C\frac{s^3}{n^2}\Big).
\end{equation}
\end{lemma}

\subsection{Proofs for Sparse Tensor PCA}

We start with some notation. Suppose $v,v'\in \mathbb{R}^n$ and \( k=n^{\beta + o(1)}, \beta \in (0,1)\) have i.i.d. $\text{Rad}(k/n)$ entries, meaning each entry is $1$ with probability $k/n$, $-1$ with probability $k/n$ and $0$ otherwise.  Define
\(
p:=(k/n)^2.
\)
Then, if $X_1=v_1v_1',\dots,X_n=v_nv_n'$ these random variables are i.i.d. with
\(
 \mathbb{P}(X_i=1)= \mathbb{P}(X_i=-1)=p/2,\  \mathbb{P}(X_i=0)=1-p.
\)
Set \(S:=\sum_{i=1}^n X_i\) and \(
T:=\sum_{i=1}^n 1\{|X_i|=1\}.
\) We denote
\begin{equation}\label{eq:defmu}
\mu:= \mathbb{E}[T]=np=\frac{k^2}{n}.
\end{equation}
Note also that
\[
\Var(X_i)= \mathbb{E}[X_i^2]=p,
\qquad
\Var(S)=n\Var(X_i)=np=\mu,
\qquad
\sigma=\sigma_S=  \frac{k}{\sqrt{n}}.
\]
Based on all the above, the inner product $\<X,X'\>=\<v^{\otimes r},v'^{\otimes r}\>$ of two i.i.d. draw from the prior takes the form $S^r$.

Before proving the main Theorem of this section we state some Lemmas that we are going to use. For each of these models there are two key ingredients that we need. Finding the order of the derivative that appears in \eqref{eq:fann_gams_main} evaluated at order $q(D_n)$ and then finding the leading order of the quantile $q(D_n)$. In Sparse Tensor PCA we have two sparsity regimes $k=\omega(\sqrt{n})$ and $k=o(\sqrt{n})$ which we treat separately.

In order to analyze the algorithmic threshold in the moderate sparsity regime where $\beta \in (1/2, 1)$, we first need to carefully control the scale of the overlap quantile $q(D)$. The following lemma establishes tight (up to constant) upper and lower bounds for this quantile function as $D$ varies.

\begin{lemma}\label{lem:stpca_quantiles}
Fix any constant $c_1>0$, an integer $r\geq 2$ and a sparsity level $k =n^{\beta + o(1)}$ where $\beta \in (1/2,1)$. Then there exist absolute constants $0<c<C<\infty$ and $\mu_0\ge 1$
(depending only on $c_1$) such that whenever $\mu$, defined in \eqref{eq:defmu}, satisfies $\mu\ge \mu_0$, and for all $D$ satisfying \(
c_1\log \mu  \le D \le \mu \)
we have
\[
\big(c\,\sigma\sqrt{D}\big)^r  \le  q(D)  \le  \big(C\,\sigma\sqrt{D}\big)^r.
\]
\end{lemma}
The proof of Lemma \ref{lem:stpca_quantiles} is deferred to Section \ref{sec:proof_quantile_sptpca}.

Building on the quantile bounds established above, we now characterize the decay rate of the discrete derivative of the log-probability of the overlap.
\begin{lemma}\label{le:fppred_sparserad}
Fix $k =n^{\beta+o(1)}$ where $\beta \in (1/2,1)$ and $D_n\leq C\log n$, where $C>0$ is any positive constant. Then, for $a_n(s)=(s+2)^r-s^r$, $s\in 2\Z$ evaluated at $s=q(D_n)$ the following holds:
\[
-\Delta_{a_n}\log \P(\<X,X'\>=q(D_n))=\Theta\left( \frac{n^{\frac{r}{2}}}{k^{r}}D_n^{2-r}\right)=\Theta\left(\lambda_{\mathrm{ALG}}\cdot D_n^{2-r}\right),
\]
\end{lemma}
%  In the sparse regime $\omega(1) = k =o( \sqrt{n})$ we get
% \[
% \lambda+\Delta\mathcal{F}_{\mathrm{ann},\lambda}(q(D_n))\asymp \log n.
% \]
% Therefore, we get once more up to a log factor the correct algorithmic threshold which in this regime is $\lambda_{\mathrm{ALG}}=1$.

The proof of Lemma  \ref{le:fppred_sparserad} is deferred to Section \ref{sec:proof_derivative_sptpca}.

We next shift our focus to the highly sparse regime, characterized by $\beta \in (0, 1/2)$. In this setting, the following lemma dictates the asymptotic order of the quantile function $q(D_n)$.
\begin{lemma}\label{lem:q_clogn_theta1_short}
Fix $\beta\in(0,1/2)$ and set $k=n^{\beta + o(1)}$, so $\mu:=k^2/n=n^{-(1-2a)}\to 0$. Assume  $D_n=n^{o(1)}$ and $D_n\ge c_0\log n$ for some fixed $c_0>0$.
Then there exist constants $c,C>0$ (depending only on $a$ and $c_0$) such that for all sufficiently large $n$,
\[
c \frac{D_n}{\log n} \le q(D_n) \le C \frac{D_n}{\log n}.
\]
\end{lemma}

% \begin{lemma}\label{lem:q_clogn_theta1_short}
% Fix $a\in(0,1/2)$ and let $k=n^a$. Fix $C>1-2a$ and set $D_n:=C\log n$. Then,
% \[
% q(D_n)=\Theta(1)
% \qquad\text{and}\qquad
% q(D_n)\ge 1
% \]
% for all sufficiently large $n$.
% \end{lemma}

The proof of Lemma \ref{lem:q_clogn_theta1_short} is deferred to Section \ref{sec:proof_qclogntheta1}.

The final lemma in this section estimates the discrete derivative of the log-probability for small $k$, completing the necessary technical estimates for the highly sparse regime.
\begin{lemma}\label{le:fppred_sparserad_small_k}
Fix an integer $r\ge 2$, $\beta \in (0,1/2)$ and let $k = n^{\beta+o(1)}$, and $a_n(s):=(s+2)^r-s^r,\  s\in \mathbb Z_{\ge 0}$.

Then for any 
% \(
% 1 \leq s_n =n^{o(1)}
% \),
% \[
% -\Delta_{a_n}\log \P(\<X,X'\>=s_n^r)
% =
% \Theta \left(\frac{(\log n)}{s_n^{\,r-1}}\right).
% \]
% In particular, if we assume 
$c_0\log n\le D_n=n^{o(1)}$ for some fixed $c_0>0$,
and \(
s_n:=q_n(D_n),
\) it holds 
\[
-\Delta_{a_n}\log \P(\<X,X'\>=s_n^r)
=
\Theta \left(\frac{(\log n)^r}{D_n^{\,r-1}}\right).
\]
% In particular, if $D_n\le C\log n$ for some fixed constant $C>0$, then
% \[
% -\Delta_{a_n}\log \P(\<X,X'\>=s_n^r)=\Theta(\log n).
% \]
\end{lemma}
%\begin{lemma}\label{le:fppred_sparserad_small_k}
% Fix constants $C,C'>0$, $\beta \in (0,1/2)$ and let $k = n^{\beta+o(1)}$. Then, for $D_n\leq C\log n$ and for $a_n=C',$ the following holds:
% \[
% -\Delta_{a_n}\log \P(\<X,X'\>=q(D_n))=\Theta\left( \log n\right)=\widetilde{\Theta}\left(\lambda_{\mathrm{ALG}}\right),
% \]
% \end{lemma}
The proof of Lemma \ref{le:fppred_sparserad_small_k} is deferred to Section \ref{sec:proof_fppred_sparserad_small_k}.

\subsubsection{Proof of Theorem \ref{thm:sptpca}}\label{sec:sptpca}

We prove Theorem \ref{thm:sptpca} by directly applying Theorem \ref{thm:mainthm}. To do this, we first prove that the sparse Rademacher prior satisfies Assumption \ref{assump:mainassumption_discrete}. Specifically, we verify the four conditions.  

We start with the case $k=n^{\beta+o(1)}, \beta \in (1/2,1)$.

\medskip
\noindent \textbf{Condition 1:} Since this model is of the form \eqref{eq:classofgams} (for $Z_j\equiv 1$ for all $j\in[t]$) and the prior $P_0$ we have assigned satisfies Assumption \ref{ass:constrainprior} using Lemma \ref{lem:positivecumforsomegams} we know that for all multi-indices $\alpha \in \N^n$ with $|\alpha|\leq D_n$ for $D_n=\lfloor\log_2(n/k)-1\rfloor$ it holds that 
$\kappa_{\alpha}(X_i)\geq 0$.

\noindent \textbf{Condition 2:} This is trivially true since $X_i$ are bounded by $1$.

\noindent \textbf{Condition 3:} This condition is straightforward. For $D_n'=D_n\log^2 n$ and using Lemma \ref{lem:stpca_quantiles} for $k=n^{\beta+o(1)}$, $\beta\in(1/2,1)$,  
\[
q(D_n')\geq (D_n')^{r/2}(k/\sqrt{n})^r=\omega (\max \{\log^{-D_n} n, n^{-C}\}),
\]
for any constant $C>0$.  

\noindent \textbf{Condition 4:}  Using Lemma \ref{le:fppred_sparserad} we know that the derivative at $q(D_n)$ is of order 
\(
n^{\frac{r}{2}}D_n^{2-r}/k^{r}
\) for $k=\omega(\sqrt{n})$.
 Therefore, Assumption 4 is equivalent to the following condition:
\[
\frac{D_n^{2-r}n^{\frac{r}{2}}/k^{r}}{A_n}\leq \frac{1}{q(\log^3n)\log n}.
\]
Using Lemma \ref{lem:stpca_quantiles} we know that for 
\begin{equation}\label{eq:dnprimechoice}
D_n'=\log^3n,    
\end{equation}
which satisfies $\omega(\log \mu)=D_n'=o(\mu),$  it holds that 
\[
q(D_n')\leq C_1\frac{k^{r}\log^{3r/2}n}{n^{\frac{r}{2}}}.
\] Plugging this inequality to the one above we get that Assumption 4 holds  for $A_n=C_1D_n^{2-r}\log^{(3r+2)/2}n$, where $C_1>0$ is a positive constant.

From all the above, combined with \eqref{eq:fann_gams_main}, we have shown that for the sequence $A_n=C_1D_n^{2-r}\log^{(3r+2)/2}n$ it holds that
\[
\mathrm{Corr}^{\leq D_n}_{P_0 }\left(\frac{1}{A_n}\cdot \big(\lambda+\Delta_{a_n}\mathcal{F}_{\mathrm{ann}, \lambda}(q(D_n))\big)\right)^2\leq q(D_n')
\]
where $D_n'$ was defined in \eqref{eq:dnprimechoice} and $a_n$ represents the sequence defined in Lemmas \ref{le:fppred_sparserad}. Now notice that 
\begin{align*}
\mathrm{MMSE}^{\mathrm{trivial}}_X&=\E\|X\|^2-(\E\|X\|)^2=k^r.
\end{align*}
Next,   combining  Lemma \ref{le:fppred_sparserad} with \eqref{eq:fann_gams_main}, implies that for all $k=n^{\beta}, \beta\in (1/2,1)$,  $ \lambda+\Delta_{a_n}\mathcal{F}_{\mathrm{ann}, \lambda}(q(D_n))=\Theta\left( \frac{n^{\frac{r}{2}}}{k^{r}}D_n^{2-r}\right)$. Using this and choosing $\lambda=\frac{n^{\frac{r}{2}}}{k^{r}}D_n^{2-r}/\log n$ we get that for this $\lambda$,  $\Delta_{a_n}\mathcal{F}_{\mathrm{ann}, \lambda}(q(D_n))\geq 0$. Therefore, using Theorem \ref{thm:mainthm}, $\mathrm{Corr}(\lambda/A_n)^2\leq 2q(D_n'),$  where $A_n=C_1\log^{3} n$. Equivalently, using Lemmas \ref{lem:stpca_quantiles} and \ref{lem:vector-mmse-corr} and substituting $D_n=\Theta(\log n)$
\[
\mathrm{MMSE}^{\le D_n}_X  \left(\frac{\lambda}{\log^{3}n}\right)\geq k^r-C_1\frac{k^{r}\log^rn}{n^{\frac{r}{2}}}.
\]
Lastly, since in this regime, $\lambda_{\mathrm{ALG}}=\widetilde \Theta(n^{r/2}/k^r)$: 
\[
\mathrm{MMSE}^{\le D_n}_X  \left(\frac{\lambda_{\mathrm{ALG}}}{\log^{(3r+4)/2}n}\right)\geq k^r-C_1\frac{k^{r}\log^rn}{n^{\frac{r}{2}}}.
\]
Since $k=o(n)$ we have proven the desired inequality for the dense regime $k=n^{\beta+o(1)}, \beta \in (1/2,1)$ and all $r\geq 2$.

We know prove the result for $k=n^{\beta+o(1)}, \beta \in (0,1/2)$. 

\noindent \textbf{Condition 1:} Again, since this model is of the form \eqref{eq:classofgams} (for $Z_j\equiv 1$ for all $j\in[t]$) and the prior $P_0$ we have assigned satisfies \ref{ass:constrainprior} using Lemma \ref{lem:positivecumforsomegams} we know that for all multi-indices $\alpha \in \N^n$ with $|\alpha|\leq D_n$ for $D_n=\lfloor\log_2(n/k)-1\rfloor$ it holds that 
$\kappa_{\alpha}(X_i)\geq 0$.

\noindent \textbf{Condition 2:} This is true since $X_i$ are bounded by $1$.

\noindent \textbf{Condition 3:}
Using Lemma \ref{le:fppred_sparserad_small_k} for $k=n^{\beta+o(1)}$, $\beta\in(0,1/2)$, 
\[
q(D_n')\geq \Theta(1)=\omega (\max \{\log^{-D_n} n, n^{-C}\}),
\]
for any constant $C>0$.

\noindent \textbf{Condition 4} Using Lemma  \ref{le:fppred_sparserad_small_k} we know that the derivative at $q(D_n)$ is of order \(\Theta(1)\). Therefore, Assumption $4$ reduces to: 
\[
\Theta\left(\frac{1}{A_n}\right)\leq \frac{1}{q(\log^3n)\log n},
\]
which holds for $A_n=C_1\log^3n$, for some positive constant $C_1>0$.

From all the above, combined with \eqref{eq:fann_gams_main}, we have shown that for the sequence $A_n=C_1D_n^{2-r}\log^{(3r+2)/2}n$ it holds that
\[
\mathrm{Corr}^{\leq D_n}_{P_0 }\left(\frac{1}{A_n}\cdot \big(\lambda+\Delta_{a_n}\mathcal{F}_{\mathrm{ann}, \lambda}(q(D_n))\big)\right)^2\leq q(D_n')
\]
where $a_n$ represents the sequences defined in Lemmas \ref{le:fppred_sparserad_small_k}.

Next,  using Lemma \ref{le:fppred_sparserad_small_k}  and \eqref{eq:fann_gams_main}, we have that for all $k=n^{\beta}, \beta\in (0,1/2)$, $  \lambda+\Delta_{a_n}\mathcal{F}_{\mathrm{ann}, \lambda}(q(D_n))=\Theta( \log n)$. Using this and choosing $\lambda=1$ we get that for this $\lambda$,  $\Delta_{a_n}\mathcal{F}_{\mathrm{ann}, \lambda}(q(D_n))\geq 0$. Lastly, since in this regime, $\lambda_{\mathrm{ALG}}=\Theta(1)$ we get, applying Theorem \ref{thm:mainthm}:
\[
\mathrm{MMSE}^{\le D_n}_X  \left(\frac{\lambda_{\mathrm{ALG}}}{\log^{3}n}\right)\geq k^r-C_1\log^2n.
\]
Since $k=\omega(\log^2 n)$ we have proven the desired inequality for the sparse regime $k=n^{\beta+o(1)}, \beta\in (0,1/2)$ and all $r\geq 2$.

\subsubsection{Proof of Lemma \ref{lem:stpca_quantiles}}\label{sec:proof_quantile_sptpca}

\textbf{Upper Bound:}
Since $|X_i|\le 1$, $\E X_i=0$ and $\Var(S)=\mu$, Bernstein inequality yields for all $t\ge 0$,
\[
\mathbb{P}(|S|\ge t)\le 2\exp \Big(-\frac{t^2}{2(\mu+t/3)}\Big).
\]
Taking $t=C(\sqrt{\mu D}+D)$ gives $\mathbb{P}(|S|\ge t)\le e^{-D}$ for large enough absolute $C$.
Since $D\le \mu$, then $\sqrt{\mu D}+D\le 2\sqrt{\mu D}$, so
\[
q(D)\le (2C\sqrt{\mu D})^r=(2C\,\sigma\sqrt D)^r.
\]
\textbf{Lower Bound:}
Let $A_t=\sum_{j=1}^t \varepsilon_j$ with $\varepsilon_j\stackrel{iid}{\sim}\text{Rad}(\pm1)$.
Fix $c_1>0$ as in the statement. Using Lemma \ref{lem:rad-master-two-regimes} we know that there exist constants $c'>0$ and $t_0\ge 2$
(depending only on $c_1$) such that for all $t\ge t_0$ and all $D$ with \(
c_1\log t \le D \le t,
\) we have
\begin{equation}\label{eq:rad-tail-lower-log}
\mathbb{P}\big(|A_t|\ge c'\sqrt{tD}\big)\ \ge\ e^{-D}.
\end{equation}
Write $X_i=\varepsilon_i Z_i$ where $\varepsilon_i\sim \text{Rad}(\pm1)$ and
$Z_i:=\mathbf 1_{\{|X_i|=1\}}\sim \text{Ber}(\rho)$ are independent. Then
$T=\sum_i Z_i\sim\text{Bin}(n,\rho)$ with $\E T=\mu$, and conditionally on $T=t$,
we have $S\stackrel d=A_t$. By Chernoff inequality, \[
\P(T\le \mu/2)\le e^{-\mu/8}\quad \text{and} \quad \P(T\ge 2\mu)\le e^{-\mu/3},
\] hence for $\mu$ large
$\P(\mu/2\le T\le 2\mu)\ge 3/4$.
On this event, which we denote by $E$, $D\le \mu\le 2T$, so  $D/2\le T$ and $e^{- D/2}\ge e^{-D}$. Also, $T\le 2\mu$ implies
$\log T\le \log(2\mu)\le 2\log\mu$ for $\mu$ large enough, so $D/2\ge (c_1/4)\log T$.
Thus Lemma \ref{lem:rad-master-two-regimes} (with constant $c_1/4$) applies for $t=T$ and $D/2$.
Thus we may apply \eqref{eq:rad-tail-lower-log} with $t=T$ to get
\[
\mathbb{P}\Big(|S|\ge c'\sqrt{TD}\ \Big|\ T, E\Big)\geq\P \Big(|A_t|\geq c'\sqrt{TD} \Big)\ge\ e^{-D},
\]
where we used again that $S$ conditionally on $T=t$ has the same distribution as $A_t$.
Using $T\ge \mu/2$ gives $\sqrt{TD}\ge \sqrt{\mu D/2}$, hence
\[
\mathbb{P}\big(|S|\ge c\sqrt{\mu D}\big)
\ \ge\
\mathbb{P}\Big(2\mu \geq T\ge \frac{\mu}{2}\Big)e^{-D}
\ \ge\
\frac{3}{4}e^{-D}
\]
after adjusting $c$ (and absorbing constants into $\mu_0$).
Therefore the $\frac{3}{4}e^{-D}$ upper quantile of $|S|$ is at least $c\sqrt{\mu D}$. Raising to the $r$-th power yields $q(D+3/4)\ge (c\sqrt{\mu D})^r=(c\,\sigma\sqrt D)^r$. Equivalently, adjusting some constants, $q(D)\ge (c\sqrt{\mu (D-3/4)})^r=(c\,\sigma\sqrt{D-3/4})^r\geq (C\,\sigma\sqrt{D})^r,$ for $C=c/2$ since $D=\omega(1)$.  Combining with Step 1 the proof is complete.

\subsubsection{Proof of Lemma \ref{le:fppred_sparserad}}\label{sec:proof_derivative_sptpca}

We have denoted $S=\<v, v'\>$ and therefore the quantity $\< X,X'\>$ from two i.i.d. draws  $X,X' \sim P$ is $S^r$. Using that, the quantity $$-\Delta_{a_n}\log \P(\<X,X'\>=q(D_n)) $$
is equal to, substituting $S^r=\<X,X'\>$ and $a_n(s)=(s+2)^r-s^r$ for integers $s=q(D_n),$ $D_n\leq C\log n$ where $C>0$,
\[
\Lambda_r(s):
=\frac{\log \mathbb{P}(S=s+2)-\log \mathbb{P}(S=s)}{(s+2)^r-s^r},
\]
We proceed with some notation and some Auxiliary lemmas.

As we said, we will evaluate $\Lambda_r(s)$ at $s=q(D_n)$ or using Lemma \ref{lem:stpca_quantiles} this is equivalent to evaluating at $s=\Theta( D_n\sigma_S)$. Absorbing the leading constant of $s$ into $D_n$ we only need to evaluate at $s=D_n\sigma_S$. Fix any deterministic sequence $D_n\ge 1$ and define
\[
s_n:= 2\left\lfloor \frac{D_n\sigma_S}{2}\right\rceil
\quad\text{so that}\quad
|s_n-D_n\sigma_S|\le 1 \ \text{ and }\ s_n\equiv 0\pmod 2.
\]
Let $\varepsilon_1,\dots,\varepsilon_t$ be i.i.d.\ Rademacher, $ \mathbb{P}(\varepsilon_i=\pm 1)=1/2$, and define
\[
A_t:=\sum_{i=1}^t \varepsilon_i.
\]
Notice $ \mathbb{P}(A_t=s)=0$ unless $s\equiv t\mod 2$.

We first compute the step-$2$ ratio for a fixed simple random walk $A_t$. This will later serve as the main term for an expansion of $S$ once we show that the relevant values of $t$ are concentrated near $\mu$.

\begin{lemma}[Expansion of $A_t$]\label{lem:At-ratio-expand}
Uniformly over integers $t\ge 1$ and $s$ with $s\equiv t  \pmod 2$ and $s\le t-2$,
\[
\log\frac{ \mathbb{P}(A_t=s+2)}{ \mathbb{P}(A_t=s)}
=
-\frac{2s+2}{t}
+O \left(\frac{(s+1)^2}{t^2}\right).
\]
\end{lemma}
\begin{proof}
If $s\equiv t\mod 2$ for $j:=(t+s)/2\in\{0,\dots,t\}$, we have
\[
 \mathbb{P}(A_t=s)=2^{-t}\binom{t}{j}.
\]
Otherwise $ \mathbb{P}(A_t=s)=0$. Indeed, if $B_t:=\#\{i:\varepsilon_i=+1\}\sim\mathrm{Bin}(t,1/2)$. Then $A_t=2B_t-t$ and the event $\{A_t=s\}$
is equivalent to $\{B_t=(t+s)/2\}$. 

Now for $s\in\mathbb{Z}$ that satisfies $s\equiv t \mod 2$ and $ s\le t-2$ as above $j=(t+s)/2$.  Then $(t+s+2)/2=j+1$ and
\[
\frac{ \mathbb{P}(A_t=s+2)}{ \mathbb{P}(A_t=s)}
=
\frac{\binom{t}{j+1}}{\binom{t}{j}}
=
\frac{t-j}{j+1}
=
\frac{t-\frac{t+s}{2}}{\frac{t+s}{2}+1}
=
\frac{t-s}{t+s+2}.
\]
By taking logarithms,
\[
\log\frac{ \mathbb{P}(A_t=s+2)}{ \mathbb{P}(A_t=s)}
=
\log\Big(\frac{t-s}{t+s+2}\Big)
=
\log(1-u),
\qquad
u:=\frac{2s+2}{t+s+2}.
\]
If $s\le t-2$ then $0<u\le 1/2$. For $u\in[0,1/2]$ we have $|\log(1-u)+u|\le 2u^2$,
so $\log(1-u)=-u+O(u^2)$. Also
\[
u=\frac{2s+2}{t}\cdot\frac{1}{1+\frac{s+2}{t}}
=\frac{2s+2}{t}+O \left(\frac{(s+1)^2}{t^2}\right),
\qquad
u^2=O \left(\frac{(s+1)^2}{t^2}\right).
\]
Combining yields the claim.
\end{proof}

The next lemma rewrites the law of $S$ as a mixture of the laws of the random walks $A_t$, where the mixing variable is the random support overlap $T$. This decomposition is the key reduction that allows us to transfer estimates for $A_t$ to estimates for $S$.

\begin{lemma}[Binomial-mixture representation]\label{lem:mixture}
We have $T=\sum_{i=1}^n 1\{|X_i|=1\}\sim\mathrm{Bin}(n,q)$ with $q=(k/n)^2$ and $ \mathbb{E}[T]=\mu=k^2/n$. Moreover, conditional on $T=t$,
\[
S\mid(T=t)\ \stackrel{d}{=}\ A_t.
\]
Consequently, for every $s\in\mathbb{Z}$,
\[
 \mathbb{P}(S=s)=\sum_{t=0}^n  \mathbb{P}(T=t) \mathbb{P}(A_t=s).
\]
\end{lemma}

\begin{proof}
Each $X_i$ is nonzero with probability $q$, so $T\sim\mathrm{Bin}(n,q)$. Given the size of the support of set (hence given $T=t$),
the signs are i.i.d.\ symmetric, so the sum is a length-$t$ Rademacher walk $A_t$. Averaging over $T$ gives the formula.
\end{proof}

For fixed $s$, define weights
\[
w_t(s):= \mathbb{P}(T=t) \mathbb{P}(A_t=s),\qquad \text{for all} \quad  t=0,1,\dots,n.
\]
Then $ \mathbb{P}(S=s)=\sum_t w_t(s)$, and note that $w_t(s)=0$ automatically unless $t\equiv s\mod 2$.
 
 Next, assume $ \mathbb{P}(S=s)>0$ for some fixed value $s\in \R$ and define
\begin{equation}\label{eq:defrt}
r_t(s):=
\begin{cases}
\displaystyle \frac{ \mathbb{P}(A_t=s+2)}{ \mathbb{P}(A_t=s)} & \text{if } \mathbb{P}(A_t=s)>0,\\[0.8em]
0 & \text{if } \mathbb{P}(A_t=s)=0.
\end{cases}
\end{equation}

Using the mixture representation, we now express the ratio $\mathbb P(S=s+2)/\mathbb P(S=s)$ as a weighted average of the corresponding conditional ratios for $A_t$. Thus, the problem reduces to understanding which values of $t$ carry most of the mass and how much $r_t(s)$ varies across those values.

\begin{lemma}[Step-$2$ ratio as a weighted average]\label{lem:ratio-weights}
Assume $ \mathbb{P}(S=s)>0$ for some fixed value $s\in \R$. Then, 
\[
\frac{ \mathbb{P}(S=s+2)}{ \mathbb{P}(S=s)}=\frac{\sum_{t=0}^n w_t(s)  r_t(s)}{\sum_{t=0}^n w_t(s)}.
\]
Equivalently, if $u_t(s):=w_t(s)/\sum_u w_u(s)$, for all $t\geq 0$ (so $u_t(s)\ge 0$ and $\sum_tu_t(s)=1$), then
\[
\frac{ \mathbb{P}(S=s+2)}{ \mathbb{P}(S=s)}=\sum_{t=0}^n u_t(s)  r_t(s), \quad \text{for all} \quad t=0,1, \dots ,n.
\]
\end{lemma}

\begin{proof}
By Lemma \ref{lem:mixture},
\[
 \mathbb{P}(S=s+2)=\sum_t  \mathbb{P}(T=t) \mathbb{P}(A_t=s+2).
\]
If $ \mathbb{P}(A_t=s)=0$ then also $ \mathbb{P}(A_t=s+2)=0$, so the identity
$ \mathbb{P}(A_t=s+2)=r_t(s) \mathbb{P}(A_t=s)$ holds for all $t$.
Hence
\[
 \mathbb{P}(S=s+2)=\sum_t  \mathbb{P}(T=t)  r_t(s)   \mathbb{P}(A_t=s)=\sum_t w_t(s)  r_t(s).
\]
Divide by $ \mathbb{P}(S=s)=\sum_t w_t(s)$. 
\end{proof}

To use the weighted-average representation effectively, we need to know that the conditional ratio $r_t(s)$ does not change much when $t$ stays in the typical window around $\mu$. The following lemma provides precisely this stability estimate.

\begin{lemma}\label{lem:lipschitz-mu}
Assume $k=n^{\beta+o(1)}$, $\beta \in (1/2,1)$ and in particular $\mu=k^2/n=\omega(1)$. Let $s\le D_n\sqrt{\mu}$ and define
\[
W:=C'\sqrt{\mu}\log\mu,\qquad \mathcal I:=\{t\in\{0,\dots,n\}:\ |t-\mu|\le W\},
\]
where $C'>0$ is a fixed constant.
Then for all $t,t'\in\mathcal I$ with $t,t'\ge  s+2$, 
\[
|r_t(s)-r_{t'}(s)|=O\left( \frac{  D_n\log\mu}{\mu}\right),
\]
where $r_t(s),s\in \R$ was defined in \eqref{eq:defrt}.
\end{lemma}

\begin{proof}
For real $t> s$ set $f(t):=(t-s)/(t+s+2)$. Then
\[
f'(t)=\frac{2s+2}{(t+s+2)^2}.
\]
On $\mathcal I$ we have $t\ge \mu-W=\Theta(\mu)$ and $s\le D_n\sqrt{\mu}=o(\mu)$, since $D_n=O(\log n)$ and $\mu=k^2/n$. Hence, $(t+s+2)^2\ge c\mu^2$
for large $\mu$, so for some universal constant $C>0$
\[
|f'(t)|\le \frac{C(s+1)}{\mu^2}\le \frac{C D_n\sqrt{\mu}}{\mu^2}=\frac{C D_n}{\mu^{3/2}}.
\]
By the mean value theorem and $|t-t'|\le 2W$,
\[
|r_t(s)-r_{t'}(s)|
\le \sup_ {\xi\in\mathcal I}|f'(\xi)|  |t-t'|
\le \frac{C D_n}{\mu^{3/2}}\cdot (2W)
=O\left(\frac{ D_n\log\mu}{\mu}\right).
\]
\end{proof}

In order to show that the atypical values $t\notin \mathcal I$ contribute negligibly, we also need a lower bound on the denominator $\mathbb P(S=s_n)$. The next lemma gives such a bound at the relevant scale $s_n= D_n\sqrt{\mu}$.

\begin{lemma}\label{lem:Ps-lower-mu}
Assume $k=n^{\beta+o(1)}, \beta \in (1/2,1)$ and $s_n=\Theta\left( D_n\sqrt{\mu}\right)$ with $D_n\leq C\log n, C>0$ and $\mu=\omega(\log ^6n)$.
Then for all sufficiently large $n$,
\[
 \mathbb{P}(S=s_n)\ge \frac{c_1}{\mu}\exp\Big(-\frac{s_n^2(1+o(1))}{2\mu}\Big)
 \ge 
\frac{c_1}{\mu}\exp \Big(-\Theta(D_n^2)\Big),
\]
for a universal constant $c_1>0$.
\end{lemma}

\begin{proof}
Set $J:=\{t:\ |t-\mu|\le 4\sqrt{\mu}\}$. Since $\Var(T)=\mu(1-q)\le \mu$,
Chebyshev inequality implies $\P(T\in J)\ge 15/16$. Also $\E[(-1)^T]=(1-2q)^n=(1-2\mu/n)^n\le e^{-2\mu}$,
so for either parity,
\[
\P(T\equiv s_n \!\!\!\!\pmod 2)=\frac12\Big(1+(-1)^{s_n}\E[(-1)^T]\Big)\ge \frac12-\frac{e^{-2\mu}}{2}.
\]
Hence, by a union bound,
\[
\P(T\in J,\ T\equiv s_n \!\!\!\!\pmod 2)\ge 1-\frac1{16}-\Big(\frac12+\frac{e^{-2\mu}}{2}\Big)\ge \frac13
\]
for all large $\mu$. Set $J':=\{t\in J: t\equiv s_n \pmod 2\}$ and notice that among integers in $J$, at most $4\sqrt{\mu}+1$ have parity $s_n$, so 
there exists $t_{s_n}\in J'$ with
\[
\P(T=t_{s_n})\ge \frac{1/3}{|J'|}\ge  \frac{1/3}{4\sqrt{\mu}+1}=\Omega(\mu^{-1/2}).
\]
Now apply Lemma \ref{lem:mixture} and keep only the term $t=t_{s_n}$:
\[
 \mathbb{P}(S=s_n)\ge  \mathbb{P}(T=t_{s_n}) \mathbb{P}(A_{t_{s_n}}=s_n).
\]
Since $t_{s_n}=\mu+O(\sqrt{\mu})$, we have $s_n\le D_n\sqrt{\mu}\le 2D_n\sqrt{t_{s_n}}$ for large $\mu$.

Recall now that $A_t=\sum_{i=1}^t \varepsilon_i$ with $\varepsilon_i\stackrel{iid}{\sim}\mathrm{Rad}(\pm1)$, so Lemma \ref{lem:rad-master-two-regimes} applies with $n=t$, $T=D_n$ and  $|s|\le D_n\sqrt{t}$ since $D_n=o(t^{1/6})$. Indeed, we have $|s|=o(t^{2/3})$, hence in particular $|s|\le t/2$ for all sufficiently large $t$.
Therefore, for all large $t$,
\begin{equation}\label{eq:lowebound_A_t}
\mathbb{P}(A_t=s)\ge \frac{c_0}{\sqrt{t}}\exp\Big(-\frac{s^2}{2t}\Big)    
\end{equation}
for a universal constant $c_0>0$. Equation \eqref{eq:lowebound_A_t} yields
\[
 \mathbb{P}(A_{t_{s_n}}=s_n)
\ge \frac{c_0}{\sqrt{t_{s_n}}}\exp \Big(-\frac{s_n^2}{2t_{s_n}}\Big)
\ge \frac{c_0'}{\sqrt{\mu}}\exp \Big(-\frac{s_n^2}{2\mu}(1+o(1))\Big).
\]
Multiplying with $ \mathbb{P}(T=t_{s_n})\ge c/\sqrt{\mu}$ gives the result. 
\end{proof}
For a fixed integer $s\in\N$ define an integer $t_\star=t_\star(s)$ to be any integer satisfying
\begin{equation}\label{eq:deftstar}
t_\star\equiv s\mod 2,\qquad |t_\star-\mu|\le 1.
\end{equation}

We can now combine concentration of the mixing variable $T$, the stability of $r_t(s)$ on the typical window, and the lower bound on $\mathbb P(S=s_n)$. This allows us to approximate the logarithmic step-$2$ ratio for $S$.

\begin{lemma}\label{lem:rad_main_numerator}
Assume $k=n^{\beta+o(1)}, \beta \in (1/2,1)$ and $D_n\leq C\log n$, $C>0$.
Let $s=\Theta( D_n\sqrt{\mu})$.
Then
\begin{align*}
\log \mathbb{P}(S=s_n+2)-\log \mathbb{P}(S=s_n)
=
\log\Big(\frac{t_\star-s_n}{t_\star+s_n+2}\Big)
+ O\left(\frac{D_n\log n}{\mu}\right).
\end{align*}
\end{lemma}

\begin{proof}
By Lemma \ref{lem:ratio-weights},
\[
\frac{ \mathbb{P}(S=s_n+2)}{ \mathbb{P}(S=s_n)}=\sum_{t=0}^n u_t(s_n)  r_t(s_n),
\qquad
u_t(s_n):=\frac{w_t(s_n)}{\sum_u w_u(s_n)}.
\]
Split the sum into $t\in\mathcal I$ and $t\notin\mathcal I$, where $\mathcal I=\{|t-\mu|\le W\}$ with $W=C'\sqrt{\mu}\log\mu$ for some fixed $C'>0$.

\medskip\noindent
\textbf{Step 1 (outside $\mathcal I$):}
Since $0\le  \mathbb{P}(A_t=s_n)\le 1$,
\[
\sum_{t\notin\mathcal I} w_t(s_n)\le  \mathbb{P}(T\notin\mathcal I)
\quad \text{and therefore} \quad
\sum_{t\notin\mathcal I} u_t(s_n)\le \frac{ \mathbb{P}(T\notin\mathcal I)}{ \mathbb{P}(S=s_n)}.
\]
By a standard Chernoff bound, for $u\in(0,\mu)$,
$ \mathbb{P}(|T-\mu|>u)\le 2\exp(-u^2/(3\mu))$. Plugging $u=C'\sqrt{\mu}\log \mu$ gives
$u^2/(3\mu)=\frac{C'^2}{3}\log^2 \mu$. Therefore, $ \mathbb{P}(T\notin\mathcal I)\le 2e^{-c\log^2 \mu}$, with $c=C'^2/3$, and by Lemma \ref{lem:Ps-lower-mu},
$ \mathbb{P}(S=s_n)\ge (c_1/\mu)\exp(-s_n^2/(2\mu))$. Hence
\[
\sum_{t\notin\mathcal I} u_t(s_n)
\le
\mu\cdot 2e^{-c\log^2\mu}\cdot \exp \Big(\frac{s_n^2}{2\mu}\Big)\cdot \frac{1}{c_1}
=
\exp \Big(-c\log^2 \mu+\frac{s_n^2}{2\mu}+O(\log n)\Big).
\]
Since $s_n^2/\mu=\Theta  (D_n^2)$, $\mu=n^{\Theta(1)}$ and $D_n\leq C \log n$, choosing $C'$ to be large enough:
\begin{equation}\label{eq:utnotinI}
\sum_{t\notin\mathcal I} u_t(s_n)
=O(e^{-\log ^2 n}).
\end{equation}

\medskip\noindent
\textbf{Step 2 (inside $\mathcal I$):}
By Lemma \ref{lem:lipschitz-mu},
\[
\sup_{t\in\mathcal I}|r_t(s_n)-r_{t_\star}(s_n)|=O\left( \frac{C D_n\log\mu}{\mu}\right).
\]

\medskip\noindent
\textbf{Step 3 (conclude):}
For every integer $s\ge 0$ and every $t$ with $\mathbb P(A_t=s)>0$ (equivalently $s\equiv t \pmod 2$ and $|s|\le t$),
\[
r_t(s)
=
\frac{\mathbb P(A_t=s+2)}{\mathbb P(A_t=s)}
=
\frac{\binom{t}{(t+s+2)/2}}{\binom{t}{(t+s)/2}}
=
\frac{t-s}{t+s+2}.
\] Therefore, for $s\geq 0$ it holds that $0\le r_t(s_n)\le 1$, and thus
\begin{align*}
\left|\sum_t u_t(s_n)r_t(s_n)-r_{t_\star}(s_n)\right|
&\le \sup_{t\in\mathcal I}|r_t(s_n)-r_{t_\star}(s_n)| + 2\sum_{t\notin\mathcal I}u_t(s_n)\\
&\le \frac{C D_n\log n}{\mu}+o \left(\frac{D_n\log n}{\mu}\right),
\end{align*}
where for the last inequality we used Step $2$ and \eqref{eq:utnotinI}. Thus,
\[
\frac{ \mathbb{P}(S=s_n+2)}{ \mathbb{P}(S=s_n)}
=
r_{t_\star}(s_n)+O \left(\frac{D_n\log n}{\mu}\right)
=
\frac{t_\star-s_n}{t_\star+s_n+2}+O \left(\frac{D_n\log n}{\mu}\right).
\]
Write $R:=\frac{t_\star-s_n}{t_\star+s_n+2}$ and $\delta=O \left(\frac{D_n\log n}{\mu}\right)$.
Since $t_\star=\mu+O(1)$ and $s_n=o(\mu)$, we have $R=\Theta(1)$, so
\[
\log \frac{ \mathbb{P}(S=s_n+2)}{ \mathbb{P}(S=s_n)}=\log(R+\delta)=\log R+\log \Big(1+\frac{\delta}{R}\Big)
=\log R+O\Big(\frac{\delta}{R}\Big)
=\log R+O\left(\frac{D_n\log n}{\mu}\right),
\]proving our desired result.
Taking logarithms (using that the main ratio is bounded away from $0$ and $\infty$ since $s_n=o(\mu)$)
gives the claimed form. 
\end{proof}

With the numerator now identified up to a controlled error, it remains to divide by the increment $(s_n+2)^r-s_n^r$. The proposition below performs this final expansion and yields the claimed order of $\Lambda_r(s_n)$.

\begin{proposition}\label{thm:Lambda-dense}
Fix an integer $r\ge 2$. Assume $k=n^{\beta+o(1)}$, $\beta\in (1/2,1)$ and $D_n\leq C\log n$.
Let $ s_n=\Theta( D_n\sqrt{\mu})$.
Then,
\[
\Lambda_r(s_n)= -\Theta \left(\frac{1}{D_n^{  r-2}  \mu^{r/2}}\right).
\]

\end{proposition}

\begin{proof}
By Lemma \ref{lem:rad_main_numerator},
\[
\log \mathbb{P}(S=s+2)-\log \mathbb{P}(S=s)=\log\Big(\frac{t_\star-s_n}{t_\star+s_n+2}\Big)+O\Big(\frac{D_n\log n}{\mu}\Big).
\]
Apply Lemma \ref{lem:At-ratio-expand} with $t=t_\star$, defined in \eqref{eq:deftstar}, and $s=s_n$:
\[
\log\Big(\frac{t_\star-s_n}{t_\star+s_n+2}\Big)
=
-\frac{2s_n+2}{t_\star}
+O \left(\frac{(s_n+1)^2}{t_\star^2}\right).
\]
Combining with $t_\star=\mu+O(1)$, 
\begin{equation}\label{eq:logprobsptpca}
\log \mathbb{P}(S=s+2)-\log \mathbb{P}(S=s)
=
-\frac{2s_n+2}{\mu}
+O \left(\frac{(s_n+1)^2}{\mu^2}\right)
+O \left(\frac{D_n\log n}{\mu}\right).
\end{equation}
Now divide by $(s_n+2)^r-s_n^r$. Since $s_n\to\infty$, as $D_n\sqrt{\mu}\to\infty$, using the binomial theorem:
\[
(s+2)^r-s^r=\sum_{j=1}^r \binom{r}{j}2^j s^{r-j}
=2r s^{r-1}+\sum_{j=2}^r \binom{r}{j}2^j s^{r-j}
=2r s^{r-1}+O_r(s^{r-2}).
\]
Therefore,
\[
(s_n+2)^r-s_n^r = 2r s_n^{r-1}\Big(1+O_r\big(\tfrac{1}{s_n}\big)\Big) \]
which implies
\begin{equation}\label{eq:logprobsptpca2}
\frac{1}{(s_n+2)^r-s_n^r}
=
\frac{1}{2r s_n^{r-1}}\Big(1+O_r\big(\tfrac{1}{s_n}\big)\Big).
\end{equation}
Multiplying \eqref{eq:logprobsptpca}, \eqref{eq:logprobsptpca2}  implies
\begin{equation}\label{eq:sptpcaderivativeprop}
\Lambda_r(s_n)
=
-\frac{1}{r  \mu}  s_n^{-(r-2)}\Big(1+O_r\big(\tfrac{1}{ s_n}\big)\Big)
 +
O\left(\frac{D_n\log n}{\mu  s_n^{r-1}}\right)
 +
O\left(\frac{(s_n+1)^2}{\mu^2  s_n^{r-1}}\right).
\end{equation}
In particular, since $s_n=\Theta( D_n\sqrt{\mu})$ this implies the desired inequality and the proof is complete.
\end{proof}

% We have denoted $X_i:=v_i v_i'\in\{-1,0,+1\}$ and 
% \(
% T:=\sum_{i=1}^n \mathbf 1\{|X_i|=1\}.
% \)
% Then $T\sim \mathrm{Bin}(n,q)$ with $q=\mathbb P(|X_i|=1)=(k/n)^2$.
% Conditionally on $\{T=t\}$, the $t$ nonzero $X_i$'s are i.i.d.\ uniform on $\{\pm1\}$, hence
% $S_n\mid(T=t)\stackrel{d}{=}\big|\sum_{j=1}^t \xi_j\big|$ with $\xi_j\stackrel{iid}{\sim}\mathrm{Rad}(\pm1)$.
% In particular, $S_n\le T_n$ always, and on the event $\{T=1\}$ we have $S_n=1$ surely. Therefore
% \begin{equation}\label{eq:Sn_ge1_lb}
% \mathbb P(S_n\ge 1)\ \ge\ \mathbb P(T_n=1)=nq(1-q)^{n-1}.
% \end{equation}
% Since $\mu=k^2/n\to 0$, we have $(1-q)^n\ge e^{-2nq}=e^{-2\mu}\ge 1/2$ for all large $n$, and thus
% \[
% \mathbb P(S_n\ge 1)\ \ge\ \tfrac12\,nq=\tfrac12\,\mu.
% \]
% Because $C>1-2a$, we have $\mu=n^{-(1-2a)}\omega ( n^{-C})=\omega(e^{-D_n})$, hence for all large $n$,
% \[
% \mathbb P(S_n\ge 1)\ >\ e^{-D_n}.
% \]
% Since $S_n$ is integer-valued, $\mathbb P(S_n\ge t)=\mathbb P(S_n\ge 1)$ for every $t\in(0,1]$; thus the above
% forces $q_n(D_n)\ge 1$.

% For the upper bound, fix any integer $M$ large enough so that $M(1-2a)>C$.
% Using $S_n\le T_n$ and the standard binomial tail bound
% \[
% \mathbb P(T_n\ge M)\ \le\ \binom{n}{M}q^M\ \le\ \frac{(nq)^M}{M!}=\frac{\mu^M}{M!},
% \]
% we get
% \[
% \mathbb P(S_n\ge M)\le \mathbb P(T_n\ge M)\le \frac{\mu^M}{M!}.
% \]
% Since $\mu^M=n^{-M(1-2a)}=o(n^{-C})$, it follows that $\mathbb P(S_n\ge M)\le e^{-D_n}$ for all large $n$.
% By the definition of $q_n$, this implies $q_n(D_n)\le M$.

% Combining the two bounds yields $1\le q(D_n)\le M$ for all sufficiently large $n$.

\subsubsection{Proof of Lemma \ref{lem:q_clogn_theta1_short}}\label{sec:proof_qclogntheta1}

Let $X_i:=v_i v_i'\in\{-1,0,+1\}, i\in [n]$ and $T:=\sum_{i=1}^n \mathbf 1\{|X_i|=1\}$.
Then $T\sim\mathrm{Bin}(n,q)$ with $q=(k/n)^2$ and $\E[T]=\mu=nq$.
Conditionally on $\{T=t\}$, the $t$ nonzero $X_i$'s are i.i.d.\ uniform on $\{\pm1\}$, hence
$S_n\mid(T=t)\stackrel{d}=\sum_{j=1}^t \xi_j$ with $\xi_j\stackrel{iid}{\sim}\mathrm{Rad}(\pm1)$.
In particular $S_n\le T$ always, and on $\{T=m\}$ we have
$\mathbb P(S_n\ge m\mid T=m)=\mathbb P(|\sum_{j=1}^m \xi_j|=m)=2^{1-m}$.
Therefore, for every integer $m\ge 1$,
\begin{equation}\label{eq:sandwich_ST}
2^{1-m}\,\mathbb P(T=m)\ \le\ \mathbb P(S_n\ge m)\ \le\ \mathbb P(T\ge m).
\end{equation}
Notice that
\begin{equation}\label{eq:bin_tail_union}
\mathbb P(T\ge m)
\le
\binom{n}{m}\,q^m
\le
\frac{n^m}{m!}\,q^m
=
\frac{(nq)^m}{m!}
=
\frac{\mu^m}{m!}.
\end{equation}
Next, for $m=n^{o(1)}$ and since $\mu=o(1)$,
\[
\mathbb P(T=m)=\binom{n}{m}q^m(1-q)^{n-m}
=
\frac{\mu^m}{m!}\,(1+o(1)),
\]
since $\binom{n}{m}=n^m/m!\cdot(1+o(1))$ and $(1-q)^{n-m}=\exp(-(n-m)q+o(1))=1+o(1)$. Combining this with \eqref{eq:sandwich_ST} and \eqref{eq:bin_tail_union} yields, uniformly for $m=n^{o(1)}$,
\[
\frac{\mu^m}{m!}\,2^{-m}
\ \le\
\mathbb P(S_n\ge m)
\ \le\
\frac{\mu^m}{m!}.
\]
Taking logs and using $\log m!=m\log m-m+O(\log m)$ gives
\[
-\log \mathbb P(S_n\ge m)= m\log\frac{1}{\mu}\ +\ O(m\log m).
\]
Since $\log(1/\mu)=(1-2a)\log n$, the term $m\log m$ is of lower order 
relative to $m\log(1/\mu)=\Theta(m\log n)$ and therefore there exist constants $c,C>0$ such that for all integers $c_0\log n\leq D_n\leq n^{o(1)}$,
\[
-\log \mathbb P(S_n\ge m)\ge D_n\ \ \text{for all}\ \ m\ge C\,\frac{D_n}{\log n},
\qquad
-\log \mathbb P(S_n\ge m)\le D_n\ \ \text{for all}\ \ m\le c\,\frac{D_n}{\log n}.
\]
By the definition of $q(D_n)$ and monotonicity of $m\mapsto \mathbb P(S_n\ge m)$, this implies
\[
c\,\frac{D_n}{\log n}\ \le\ q(D_n)\ \le\ C\,\frac{D_n}{\log n},
\]
as claimed.

\subsubsection{Proof of Lemma \ref{le:fppred_sparserad_small_k}}\label{sec:proof_fppred_sparserad_small_k}

Let $S=\<v,v'\>$, so that $\langle X,X'\rangle=S^r$.

We first claim that uniformly for all integers $m\ge 1$ with $m=n^{o(1)}$,
\begin{equation}\label{eq:psm_asymp_smallk}
\mathbb P(S=m)=\frac{\mu^m}{2^m m!}\,(1+o(1)),
\end{equation}
where, as before, $\mu=k^2/n\to 0$. Indeed, by Lemma~\ref{lem:mixture},
\[
\mathbb P(S=m)
=
\sum_{j\ge 0}\mathbb P(T=m+2j)\,\mathbb P(A_{m+2j}=m),
\]
where $T\sim \mathrm{Bin}(n,q)$ and $A_t=\sum_{i=1}^t \xi_i$ with
$\xi_i\stackrel{iid}{\sim}\mathrm{Rad}(\pm1)$.
For the $j=0$ term,
\[
\mathbb P(T=m)\,\mathbb P(A_m=m)
=
\mathbb P(T=m)\,2^{-m}.
\]
Since $m=n^{o(1)}$, we have uniformly
\[
\mathbb P(T=m)
=
\binom{n}{m}q^m(1-q)^{n-m}
=
\frac{\mu^m}{m!}\,(1+o(1)).
\]
Hence
\[
\mathbb P(T=m)\,\mathbb P(A_m=m)
=
\frac{\mu^m}{2^m m!}\,(1+o(1)).
\]

It remains to show that the contribution of $j\ge 1$ is negligible.
Using $\mathbb P(T=t)\le \mu^t/t!$ for all $t\ge 0$, we get
\begin{align*}
\sum_{j\ge 1}\mathbb P(T=m+2j)\,\mathbb P(A_{m+2j}=m)
&\le
\sum_{j\ge 1}\frac{\mu^{m+2j}}{(m+2j)!}\binom{m+2j}{j}2^{-(m+2j)}\\
&=
\frac{\mu^m}{2^m}\sum_{j\ge 1}\frac{\mu^{2j}}{4^j\,j!\,(m+j)!}\\
&\le
\frac{\mu^m}{2^m m!}\sum_{j\ge 1}\frac{\mu^{2j}}{4^j\,j!}
=
O\!\left(\frac{\mu^{m+2}}{2^m m!}\right).
\end{align*}
Since $\mu\to 0$, this is $o\!\left(\mu^m/(2^m m!)\right)$ uniformly in $m=n^{o(1)}$.
This proves \eqref{eq:psm_asymp_smallk}.

Applying \eqref{eq:psm_asymp_smallk} with $m=s_n$ and $m=s_n+2$, we obtain
\[
\frac{\mathbb P(S=s_n+2)}{\mathbb P(S=s_n)}
=
\frac{\mu}{2(s_n+2)}\,(1+o(1)),
\]
and therefore
\begin{equation}\label{eq:num_smallk}
\log \mathbb P(S=s_n)-\log \mathbb P(S=s_n+1)
=
\log\frac1\mu+\log(2(s_n+1))+o(1).
\end{equation}
Now
\(
\mu=n^{-(1-2\beta)+o(1)},
\ \text{so}\ 
\log\mu=(2\beta-1+o(1))\log n.
\)
Also, since $s_n=n^{o(1)}$,
\(
\log(s_n+2)=o(\log n).
\)
Thus \eqref{eq:num_smallk} gives
\begin{equation}\label{eq:num_theta_logn}
\log \mathbb P(S=s_n)-\log \mathbb P(S=s_n+2)=\Theta(\log n).
\end{equation}
Next, since $s_n\ge 1$ and $r\ge 2$, we have \(
(s_n+1)^r-s_n^r=\Theta(s_n^{\,r-1}).
\) 

 Therefore,
\[
-\Delta_{a_n}\log \P(\<X,X'\>=s_n^r)
=
\frac{\log \mathbb P(S=s_n)-\log \mathbb P(S=s_n+2)}{(s_n+2)^r-s_n^r}=\Theta\left(\left(\frac{\log n}{(s_n)^{r-1}}\right)\right),
\]proving the first part of the lemma.

Set now
\(
s_n:=q_n(D_n).
\) By Lemma \ref{lem:q_clogn_theta1_short}, \(
s_n=\Theta \left(D_n/\log n\right).
\) Since $D_n=n^{o(1)}$ and $D_n\ge c_0\log n$, it follows that \(
1\le s_n=n^{o(1)} \) for all sufficiently large $n$.
Moreover,
\begin{equation}\label{eq:den_smallk}
(s_n+1)^r-s_n^r
=
\Theta\left(\left(\frac{D_n}{\log n}\right)^{r-1}\right).
\end{equation}
 Therefore, combining the above for this choice of $s_n,$

\[
-\Delta_{a_n}\log \P(\<X,X'\>=s_n^r)
=
\Theta\!\left(
\frac{\log n}{(D_n/\log n)^{r-1}}
\right)
=
\Theta\!\left(\frac{(\log n)^r}{D_n^{\,r-1}}\right).
\]
This proves the lemma.

\subsection{Proofs for Tensor PCA}\label{sec:tpca}

Suppose $v,v'\sim\mathcal{N}(0,I_n)$ independently drawn and let $S=\<v,v'\>$. Then, $X=v^{\otimes r}, X'=(v')^{\otimes r}$ are two i.i.d. draws from Gaussian Tensor PCA and $S^r=\<X,X'\>$. Also, denote $f_S$ the density of $S$. We state the two Lemmas we are going to use for the proof of Theorem \ref{thm:tpca}. 

As we mentioned before these Lemmas are doing the following for each model: evaluate the leading order of the quantile $q(D_n)$ and then we find the order of the derivative that appears in \eqref{eq:fann_gams_main} evaluated at order $q(D_n)$. These steps allow us to apply Theorem \ref{thm:mainthm} for Tensor PCA.

\begin{lemma}\label{lem:gauss}
Fix an integer $r\ge 2$. Then there exist absolute constants $0<c<C<\infty$ and an absolute $n_0\ge 1$ such that
for all $n\ge n_0$ and all $\Omega(1)= D\le n$,
\[
\big(c(\sqrt{nD}+D)\big)^r \le q(D) \le \big(C(\sqrt{nD}+D)\big)^r.
\]
\end{lemma}

The proof of this Lemma is deferred to Section \ref{sec:proof_quantile_tpca}.

\begin{lemma}\label{lem:fppred_tensorpca}
For the Tensor PCA model with the Gaussian prior and for any $1\leq D_n =o(\sqrt{n})$ the following derivative is of order:
\[
-\left.\frac{d}{ds}\log f_S(s)\right|_{s=(D_n n)^{r/2}}
=\Theta( n^{-r/2}D_n^{(3-r)/2})=\Theta(\lambda_{\mathrm{ALG}}\cdot D_n^{(3-r)/2}).
\]
\end{lemma}
The proof of this Lemma is deferred to Section \ref{sec:proof_quantile_tpca}.
 
\subsubsection{Proof of Theorem \ref{thm:tpca}}

We prove this by directly applying Theorem \ref{thm:mainthm}. To do this we first prove that the prior satisfies the main assumption \ref{assump:mainassumption}. Specifically, we verify the four conditions. 

\medskip 
    \noindent \textbf{Condition 1:} Since the standard Gaussian distribution has non-negative cumulants of all order and Tensor PCA is of the form \eqref{eq:classofgams}, Lemma \ref{lem:all_positive_cumulants} implies that for all multi-indices $\alpha \in \N^n$ it holds that
$\kappa_{\alpha}(X_i)\geq 0.$

\noindent \textbf{Condition 2:} This follows directly form the fact that each $X_i$ is a  products of $r$ Gaussian random variables.

\noindent \textbf{Condition 3:} This condition is straightforward since, for $D_n'=\Theta(D_n\log^2n)$ and using Lemma \ref{lem:gauss} 
\[
q(D_n')\geq (cD_n')^r=\omega (\max\{\log^{-D_n} n, n^{-C}\})
\]
for any $C>0$.

\noindent \textbf{Condition 4:} Using Lemma \ref{lem:gauss} we know that for $D_n=o(\sqrt{n})$ the derivative at $q(D_n)$ is of order 
\(
(nD_n)^{r/2}.
\)
 Therefore Item 4 from Assumption \ref{assump:mainassumption} is equivalent to proving that there exists a constant $C>0$ such that:
\[
\frac{D_n^{(3-r)/2}n^{-r/2}}{A_n}\leq \frac{1}{q(D_n')\log n}.
\]
Using Lemma \ref{lem:gauss} we know that for $D_n'\leq \sqrt{n},$ it holds that 
\[
q(D_n')\leq C_1(D_n'n) ^{r/2}.
\] Plugging this inequality to the one above we get that Assumption 4 holds for all  $D_n'=o(\sqrt{n})$ if:
\[
\frac{D_n^{(3-r)/2}n^{-r/2}}{A_n}\leq \frac{1}{C_1'n^{r/2}(D_n')^{r/2}\log n} \quad \text{or equivalently} \quad C_1'D_n^{3/2}\log^{r+1} n\leq A_n
\]
where we used $D_n'=D_n\log^2n$ to get the last inequality. Therefore, for $A_n=C_1'(D_n)^{3/2}\log^{r+1} n$ the inequality holds.

Now combining the above with \eqref{eq:fann_gams_main} we conclude for all $\lambda>0$
\[
\mathrm{Corr}^{\leq D_n}_{P_0 }\left(\frac{1}{C_1'D_n^{3/2}\log^{r+1} n}\cdot \big(\lambda+\frac{\d}{\d q}\mathcal{F}_{\mathrm{ann}, \lambda}\bigg{|}_{q=q(D_n)}\big)\right)^2\leq q(D_n')\leq C_1(D_n'n)^{r/2}. 
\]
 Now notice that 
\begin{align}\label{eq:trmmsetpca}
\mathrm{MMSE}^{\mathrm{trivial}}_X&=\E\|X\|^2-(\E\|X\|)^2=n^r.
\end{align}
Also, from Lemma \ref{lem:fppred_tensorpca}  and \eqref{eq:fann_gams_main} for all $\lambda>0$, $\lambda+\frac{\d}{\d q}\mathcal{F}_{\mathrm{ann}, \lambda}\bigg{|}_{q=q(D_n)}=\Theta(D_n^{(3-r)/2}\lambda_{\mathrm{ALG}})$. C
Choosing $\lambda=D_n^{(3-r)/2}n^{-r/2}/\log n$ now, implies $\frac{\d}{\d q}\mathcal{F}_{\mathrm{ann}, \lambda}\bigg{|}_{q=q(D_n)}\geq 0$.  Therefore, using Theorem \ref{thm:mainthm}, $\mathrm{Corr}(\lambda/A_n)^2\leq 2q(D_n'),$ where $A_n=C_1'(D_n)^{3/2}\log^{r+1} n$. Equivalently, 
\[
\mathrm{MMSE}^{\le D_n}_X  \left(\frac{\lambda}{C_1'D_n^{3/2}\log^{r+1} n}\right)\geq n^r-C_1(D_n'n)^{r/2}.
\]
Finally, substituting $\lambda_{\mathrm{ALG}}=n^{-r/2}$ 
\[
\mathrm{MMSE}^{\le D_n}_X  \left(\frac{\lambda_{\mathrm{ALG}}}{C_1'D_n^{r/2}\log^{r+2} n}\right)\geq n^r-C_1(D_n'n)^{r/2}
.\] 
Combining this with \eqref{eq:trmmsetpca} we have proven that for any $D_n'=o(\sqrt{n})$ and $D_n=\frac{D_n'}{\log^2 n}=o(\sqrt{n}/\log^2n)$ 
\[
\mathrm{MMSE}^{\le D_n}_X  \left(\frac{\lambda_{\mathrm{ALG}}}{C_1'D_n^{r/2}\log^{r+2} n}\right)\geq \mathrm{MMSE}^{\mathrm{trivial}}_X (1+o(1)),
\]
completing the proof.

\subsubsection{Proof of Lemma \ref{lem:gauss}}\label{sec:proof_quantile_tpca}

\textbf{Upper bound:}
$X_i:=v_i v_i'$, $i\in [n]$ are i.i.d.\ mean-zero sub-exponential, hence for some
absolute $c_0>0$,
\[
\mathbb{P}(|S|\ge t)\le 2\exp \Big(-c_0\min\Big\{\frac{t^2}{n},\,t\Big\}\Big)
\qquad t\ge 0,
\]
which implies $\mathbb{P}(|S|\ge C(\sqrt{nD}+D))\le e^{-D}$ for large enough absolute $C$.
Therefore $q(D)\le (C(\sqrt{nD}+D))^r$.

\medskip
\noindent \textbf{Lower bound:}
 Condition on $v$. Since $v'\sim\mathcal N(0,I_n)$
is independent of $v$,
\[
S=\langle v,v'\rangle \mid v \sim \mathcal N\big(0,\|v\|_2^2\big).
\]
Let $U:=\|v\|_2^2\sim\chi^2_n$ and let $Z\sim\mathcal N(0,1)$ independent of $U$. Then
$S \stackrel{d}{=} \sqrt{U}Z$.
Let $E:=\{U\ge n/2\}$. Since $\E U=n$ and $\Var(U)=2n$, Chebyshev inequality gives
\[
\mathbb{P}(E) \ge 1-\frac{\Var(U)}{n^2} = 1-\frac{8}{n} \ge \frac12,
\]
for all $n\ge 8$.
On $E$, we have $\sqrt{U}\ge \sqrt{n/2}$, hence for any $t>0$,
\[
\mathbb{P}(|S|\ge t) \ge \mathbb{P}(E)\mathbb{P} \left(|Z|\ge \frac{t}{\sqrt{n/2}}\right)
 \ge \frac12\mathbb{P} \left(|Z|\ge \frac{t}{\sqrt{n/2}}\right).
\]
Choose $t=c\sqrt{nD}$ with $c>0$ small, so $t/\sqrt{n/2}=c\sqrt{D/2}$.
Using the standard Gaussian tail lower bound, there exists a constant $c_1>0$, such that for $x\ge c_1$,
\[
\mathbb{P}(|Z|\ge x) \ge \frac{c_1^2}{3(c_1^2 + 1)x}e^{-x^2/2},
\]
Since $D\ge 1$, we have $x=\Theta( \sqrt D)\ge c_1$.
Thus for an absolute constant $c_2$ only depends on $c_1$,
\[
\mathbb{P}(|S|\ge c\sqrt{nD})
 \ge \frac{c_2}{\sqrt D}\exp \Big(-\frac{c^2}{4}D\Big)
 \ge e^{-D},
\]
for all $\Omega(1)= D\le n$ by taking $0<c<2$ and $D\ge c_2^2$ sufficiently large.
Hence the $e^{-D}$ quantile of $|S|$ is at least $c\sqrt{nD}$ in this regime and therefore $q(D)\ge (c(\sqrt{nD}+D))^r$.

\subsubsection{Proof of Lemma \ref{lem:fppred_tensorpca}}\label{proof_derivative_tpca}

Suppose $v,v'\sim\mathcal{N}(0,I_n)$ independent. Then, as proved in Lemma \ref{lem:innerprod-bessel}, 
the density of $T=\langle v, v' \rangle$ can be written explicitly in terms of a modified Bessel function of the second kind $K_\nu$:
$$
 f_T(t) = \frac{1}{\sqrt{\pi}\,\Gamma\!\big(\tfrac n2\big)} \left(\frac{|t|}{2}\right)^{\frac n2 - \frac12} K_{\frac n2 - \frac12}(|t|), \qquad t\in\mathbb R. 
$$
% For $t>0$, by symmetry, the density of $|T|$ is $2f_T(t)$, so by a slight abuse of notation we write
% $\P(|\langle v,v'\rangle|=t)$ for this density. 
Therefore,
\[
\log f_T(t)= C + \Big(\frac n2-\frac12\Big)\log t + \log K_{\frac n2-\frac12}(t),
\]
for a constant $C$ independent of $t$, Lemma \ref{lem:bessel-basic} gives
$$
\frac{\d}{\d t}\log\P(\langle v, v' \rangle  = t) = \frac{\d}{\d t}\log f_T(t) = \frac{n-1}{t} - \frac{K_{(n+1)/2}(t)}{K_{(n-1)/2}(t)} .
$$
Let $\nu = \frac{n-1}{2}$. We recall the identity for the modified Bessel function of the second kind, stated in Lemma \ref{lem:bessel-basic} and see also \cite[Chap 9]{abramowitz1965handbook},
$$
K_{\nu+1}(t) = \frac{2\nu}{t}K_\nu(t) + K_{\nu-1}(t), \qquad t>0.
$$
Therefore,
$$
\frac{\d}{\d t}\log\P(\langle v, v' \rangle  = t) = \frac{2\nu}{t} - \frac{K_{\nu+1}(t)}{K_{\nu}(t)} = - \frac{K_{\nu-1}(t)}{K_{\nu}(t)},
$$
for all $t>0$.
According to Corollary \ref{cor:bessel-ratio-large-order}, which is derived from \cite[(1.10)]{yang_approximating_2017}, for
$t = \sqrt{nD_n}$ and $D_n=o(\sqrt{n})$,
$$
\frac{K_{\nu-1}(t)}{K_\nu(t)} = \frac{t}{2\nu} (1 + o(1)) = \sqrt{\frac{D_n}{n}} (1+ o(1)).
$$
Then, by the chain rule, if $t=s^{1/r}$,
\[
\frac{\d}{\d s}\log\P(\langle X, X' \rangle  = s) = \frac{\d}{\d s}\log\P(\langle v, v' \rangle^r  = s) = \frac{\d}{\d t}\log\P(\langle v, v' \rangle  = t) \cdot \frac{t^{1-r}}{r},
\]
so plugging in $t = \sqrt{nD_n}$ (that is, $s = (nD_n)^{r/2}$) gives
\[
-\frac{\d}{\d s}\log\P(\langle X, X' \rangle  = s)\Bigg|_{s = (n D_n)^{r/2}}
= \Theta\!\left(\sqrt{\frac{D_n}{n}}\right)\cdot \frac{(n D_n)^{(1-r)/2}}{r}
= \Theta\!\left(n^{-r/2} D_n^{\,1-r/2}\right).
\]

\subsection{Proofs for Sparse Clustering}\label{sec:spclust}

Let $\xi=(\xi_1,\dots,\xi_n)^\top$ have i.i.d. $\mathrm{Rad}(1/2)$ entries and note that the signal matrix in this model can be written as $X:=\xi \mu^\top\in\mathbb{R}^{n\times p}$, i.e.,
$X_{ij}=\xi_i \mu_j$. As we have discussed, the SNR $\lambda$ when we write it as a Gaussian Additive Model of the form \eqref{eq:GAM}, will be $\lambda=\Delta/s$.
Clustering amounts to recovering the labels $(\xi_i)_{i=1}^n$ (up to a global sign flip) from $Y$. 
\subsubsection{Notation}\label{sec:notspclust}

We will start with some notation. Let $n,p\ge 1$ where $p=n^{c+o(1)}$, for some constant $c>0$, and $s=s_{n,p}$ satisfy $\widetilde{\omega}(\sqrt{p})= s= o(p)$. 
Let also $\xi_i,\xi_i'\stackrel{iid}{\sim}\mathrm{Rad}(\pm1)$ for $i\in[n]$.
Define \[ A:=\sum_{i=1}^n \xi_i \xi_i'.\] Then $\mathbb{E}[A]=0$ and $\mathrm{Var}(A)=n$.
Let $b_j,b_j'\stackrel{iid}{\sim}\mathrm{Ber}(s/p)$ and $g_j,g_j'\stackrel{iid}{\sim}N(0,1)$, all independent.
Then, we can define $\mu, \mu'$ to be two independent draws from the priors 
\[
    \mu_j:=b_j g_j,\ \mu_j':=b_j'g_j'\  
\]and we denote 
\( Z_j:=\mu_j\mu_j'=(b_jb_j')(g_jg_j')\),
$B:=\sum_{j=1}^p Z_j$.
Let also $\delta_j:=b_jb_j'\sim \mathrm{Ber}((s/p)^2)$ and $Y_j:=g_jg_j'$. Then, $Z_j=\delta_jY_j$ and \(
U:=\sum_{j=1}^p \delta_j\sim\mathrm{Bin}(p,(s/p)^2).
\)
With this notation, if we set $S=\langle X,X'\rangle$ then  \[S=A\cdot B.\] In particular, $\Var(S)=n\sigma_B=\sigma_S^2$. Since $B$ is continuous on $\mathbb{R}\setminus\{0\}$, $S$ has an absolutely continuous density on $\mathbb{R}\setminus\{0\}$.
We write $f_S$ for this density on $\mathbb{R}\setminus\{0\}$.
Finally, denote \(\phi(x):=(2\pi)^{-1/2}e^{-x^2/2}, x\in \R.\)

\subsubsection{Auxilary lemmas}
Before proving Theorem \ref{thm:spclust} we state two Lemmas that we are going to use.

Once again, these lemmas will be doing two things.  Find the order of the derivative that appears in \eqref{eq:fann_gams_main} evaluated at order $q(D_n)$ and calculate the leading order of the quantile $q(D_n)$.

\begin{lemma}\label{lem:quantile-AB} There exist absolute constants $0<c<C<\infty$ and absolute $n_0,\sigma_0\ge 1$ such that
whenever $n\ge n_0$ and $\sigma_B\ge \sigma_0$, for all $
1\le D \le c\,\min\{n,\sigma_B\}
$ it holds
\[
c\sigma_SD  \le q(D) \le  C\sigma_SD,
\]
where $\sigma_S^2=\Var(S)$.
\end{lemma}
The proof of this Lemma is deferred to Section \ref{sec:proof_quantile_sparseclust}.

\begin{lemma}\label{lem:derivative_clustering}
Assume $p\geq s=\widetilde{\omega}( \sqrt{p})$. The log-derivative of the density of $S$
satisfies for all $D_n=o(\max\{n^{1/3}, \sigma_B^{1/6}\})$:
\[
\left.\frac{d}{ds}\log f_S(s)\right|_{s=D_n\sigma_S}
=-\frac{1}{s}\sqrt{\frac{p}{n}}\cdot \big(1+o(1)\big).
\]

% Equivalently, 
% \[
% \Delta+\frac{\d}{\d q}\mathcal{F}_{\mathrm{ann}, \Delta}\bigg{|}_{q=q(D_n)}=\frac{s}{\sigma_S}
% \] which corresponds to separation $\Delta^2_{\mathrm{ALG}}=\frac{p}{n}$, matching the algorithmic threshold.
\end{lemma}
The proof of this Lemma is deferred to Section \ref{sec:proof_derivative_sparseclust}.

\subsubsection{Proof of Theorem \ref{thm:spclust}}

Once again we prove this by directly applying Theorem \ref{thm:mainthm}. To do this we first prove that the prior we chose satisfies the main assumption \ref{assump:mainassumption}. Specifically, we verify the four conditions. 

\medskip 
\noindent \textbf{Condition 1:}  Since this model is of the form \eqref{eq:classofgams} (here $Z_j=\xi_j$, $t=n$ and $a(i)=$ the row that $X_i$ belongs to) and the prior $P_0$ we have assigned satisfies Assumption \ref{ass:constrainprior} (it's a product of a standard Gaussian, so centered, with a sparse Bernoulli) using Lemma \ref{lem:positivecumforsomegams} we know that for all multi-indices $\alpha \in \N^n$ with $|\alpha|\leq D_n$ for $D_n=\lfloor\log_2(p/s)-1\rfloor$ it holds that:
$\kappa_{\alpha}(X_i)\geq 0.$

\noindent \textbf{Condition 2:} This follows directly from the fact that $X_i$ are subGaussian.

\noindent \textbf{Condition 3:}
The third condition of the Assumption is straightforward since, for $D_n'=D_n\log^2n$ and using Lemma \ref{lem:quantile-AB}
\[
q(D_n')\geq cD_n's\sqrt{n/p}=\omega (\max \{\log^{-D_n} n, n^{-C}\}),
\]
for some universal positive constant $C>0$, since $p=n^{c+o(1)}$, for some $c>0$.

\noindent \textbf{Condition 4:} Using Lemma \ref{lem:quantile-AB} we know that the derivative at $q(D_n)$ is of order 
\(
\frac{D_n}{s}\sqrt{\frac{p}{n}}.
\) Therefore, Item 4 from Assumption \ref{assump:mainassumption} is equivalent to the following condition:
\[
\frac{\frac{1}{s}\sqrt{\frac{p}{n}}}{A_n}\leq \frac{1}{q(\log^2n)\log n}.
\]
Using Lemma \ref{lem:quantile-AB} we know that for $D_n'=\Theta(\log^3n)\leq \mu,$ it holds that 
\[
q(D_n')\leq C_1D_n's\sqrt{\frac{n}{p}}.
\] Plugging this inequality to the one above we get that Assumption 4 holds for $A_n=C_1'\log^{4}n$, where $C_1'>0$ is an absolute constant.

Now combining the above with \eqref{eq:fann_gams_main} concludes for all $\lambda>0$:
\[
\mathrm{Corr}^{\leq D_n}_{P_0 }\left(\frac{1}{C_1'\log^{4}n}\cdot \big(\lambda+\frac{\d}{\d q}\mathcal{F}_{\mathrm{ann}, \lambda}\bigg{|}_{q=q(D_n)}\big)\right)^2\leq q(D_n')\leq C_1\frac{D_n'}{s}\sqrt{\frac{p}{n}}. 
\]
where for the last inequality we used again the inequality $q(D_n')\leq C_1D_n's\sqrt{\frac{n}{p}}$. Now notice that 
\begin{align*}
\mathrm{MMSE}^{\mathrm{trivial}}_X&=\E\|X\|^2-(\E\|X\|)^2=ns.
\end{align*}
Also, from Lemma \ref{lem:derivative_clustering}  and \eqref{eq:fann_gams_main}, for all $\lambda>0$, $\lambda+\frac{\d}{\d q}\mathcal{F}_{\mathrm{ann}, \lambda}\bigg{|}_{q=q(D_n)}=\Theta(\frac{1}{s}\sqrt{p/n})$. Choosing now $\lambda=\frac{1}{s\log n}\sqrt{p/n}$ implies,  $\frac{\d}{\d q}\mathcal{F}_{\mathrm{ann}, \lambda}\bigg{|}_{q=q(D_n)}\geq 0$.  Therefore, using Theorem \ref{thm:mainthm}, $\mathrm{Corr}(\lambda/A_n)^2\leq 2q(D_n'),$ where $A_n=C_1'\log^4n$. Equivalently, 
\[
\mathrm{MMSE}^{\le D_n}_X  \left(\frac{\lambda}{C_1'\log^{4}n}\right)\geq ns-C_1s\log^4n\sqrt{\frac{n}{p}}.
\]
Using that $\lambda_{\mathrm{ALG}}=\sqrt{p/(s^2n)}$ (or equivalently $\Delta_{\mathrm{ALG}}=\sqrt{p/n}$) implies
\[
\mathrm{MMSE}^{\le D_n}_X  \left(\frac{\lambda_{\mathrm{ALG}}}{C_1'\log^{5}n}\right)\geq ns-C_1s\log^4n\sqrt{\frac{n}{p}},
\]

Combining the last two inequalities with the fact that $n=\omega(\log^{8}n)$ and $p\geq 1$ we have proven 
\[
\mathrm{MMSE}^{\le D_n}_X  \left(\frac{\lambda_{\mathrm{ALG}}}{C_1'\log^{5}n}\right)\geq (1+o(1))\mathrm{MMSE}^{\mathrm{trivial}}_X,
\] completing the proof.

\subsubsection{Proof of Lemma \ref{lem:quantile-AB}}\label{sec:proof_quantile_sparseclust}

We prove matching upper and lower bounds on $q(D)$ in the range $1\le D\le c\min\{n,\sigma_B\}$.

\noindent\textbf{Step 1:}
By Hoeffding for all $t\ge 0$,
\[
\mathbb{P}(|A|\ge t)\le 2\exp\Big(-\frac{t^2}{2n}\Big).
\]
Conversely, by Lemma \ref{lem:rad-master-two-regimes} there exists a universal constant $c_A>0$ such that for all
$n$ large enough and all $1\le D\le c n$,
\[
\mathbb{P}\big(|A|\ge c_A\sqrt{nD}\big) \ge e^{-D},
\]
and similarly there exists a universal constant $C_A>0$ such that
\[
\mathbb{P}\big(|A|\ge C_A\sqrt{nD}\big) \le e^{-D}.
\]
We absorb any fixed constant in the exponent since we can adjust $c_A,C_A$.

\noindent\textbf{Step 2:}
Now we condition on $(b,b',g')$. 
\(
B=\sum_{j=1}^p (b_jb_j' g_j')\,g_j
\)
is a linear combination of independent $\mathcal N(0,1)$ variables, hence
\[
B \big| (b,b',g') \sim \mathcal N(0,V),
\qquad
V:=\sum_{j=1}^p (b_jb_j' g_j')^2=\sum_{j=1}^p b_jb_j'(g_j')^2.
\]
In particular, with $Z\sim\mathcal N(0,1)$ independent of $V$, we have the exact representation
\(
B \stackrel{d}{=} \sqrt{V}\,Z.
\)
Also $\E[V]=\sigma_B$ since $\E[b_jb_j']= (s/p)^2$ and $\E[(g_j')^2]=1$.

\noindent\textbf{Step 3:}
Let $Y_j:=b_jb_j'(g_j')^2\ge 0$, so $V=\sum_{j=1}^p Y_j$ with i.i.d. terms.
Since $(g_j')^2$ has MGF $\E e^{\theta (g_j')^2}=(1-2\theta)^{-1/2}$ for $\theta<1/2$, we get for $\theta\in(0,1/4)$
\[
\E e^{\theta Y_j}= (1-(s/p)^2)+(s/p)^2(1-2\theta)^{-1/2}\le
\exp\!\Big((s/p)^2\big((1-2\theta)^{-1/2}-1\big)\Big),
\]
which is finite and bounded uniformly in $p$ for fixed $\theta<1/2$.
For any \(\theta\in(0,1/2)\), by Chernoff inequality,
\[
\mathbb P(V\ge 2\sigma_B)
\le
e^{-2\theta\sigma_B}\big(\mathbb E[e^{\theta Y_1}]\big)^p.
\]
Taking \(\theta=\tfrac14\) gives
\(
\mathbb P(V\ge 2\sigma_B)\le e^{-c_+\sigma_B}
\)
for some universal \(c_+>0\).
Similarly, for any \(\lambda>0\),
\[
\mathbb P\Big(V\le \frac12\sigma_B\Big)
=
\mathbb P\Big(e^{-\lambda V}\ge e^{-\lambda\sigma_B/2}\Big)
\le
e^{\lambda\sigma_B/2}\big(\mathbb E[e^{-\lambda Y_1}]\big)^p,
\]
% \[
% \mathbb{P}\Big(V\ge 2\sigma_B\Big)\le e^{-c_V\sigma_B},
% \qquad
% \mathbb{P}\Big(V\le \tfrac12\sigma_B\Big)\le e^{-c_V\sigma_B}.
% \]
hence
\[
\mathbb P\Big(V\le \frac12\sigma_B\Big)
\le
\exp\!\left(
-\sigma_B\Big[1-(1+2\lambda)^{-1/2}-\lambda/2\Big]
\right).
\]
Taking \(\lambda=\tfrac12\) gives
\[
\mathbb P\Big(V\le \frac12\sigma_B\Big)\le e^{-c_-\sigma_B}
\]
for some absolute \(c_->0\).

Therefore, for some universal \(c_V>0\),
\(
\mathbb P(V\ge 2\sigma_B)\le e^{-c_V\sigma_B}, \ 
\mathbb P\Big(V\le \frac12\sigma_B\Big)\le e^{-c_V\sigma_B},
\)
and so
\begin{equation}\label{eq:V-good}
\mathbb P\Big(\frac12\sigma_B\le V\le 2\sigma_B\Big)\ge 1-2e^{-c_V\sigma_B}.
\end{equation}

% In particular, for any $D\le c\,\sigma_B$ with $c$ small enough, then \textcolor{red}{where is D here? }
% \begin{equation}\label{eq:V-good}
% \mathbb{P}\Big(\tfrac12\sigma_B\le V\le 2\sigma_B\Big) \ge 1-2e^{-c_V\sigma_B} \ge \frac{9}{10},
% \qquad
% \mathbb{P}\Big(V\ge \tfrac12\sigma_B\Big)\ge \frac{9}{10}.
% \end{equation}

\noindent\textbf{Step 4:}
Using $B\stackrel d=\sqrt V\,Z$ and \eqref{eq:V-good}, we compare to Gaussian tails. For the upper tail, for any $t>0$,
\[
\mathbb{P}(|B|\ge t)
=
\E\big[\mathbb{P}(|Z|\ge t/\sqrt V\mid V)\big]
\le
\mathbb{P}(V>2\sigma_B)+\mathbb{P}\big(|Z|\ge t/\sqrt{2\sigma_B}\big).
\]
Choose $t:=C_B\sqrt{\sigma_B D}$. Then $t/\sqrt{2\sigma_B}=C_B\sqrt{D/2}$, and the Gaussian tail bound
$\mathbb{P}(|Z|\ge x)\le 2e^{-x^2/2}$ gives
\[
\mathbb{P}(|B|\ge C_B\sqrt{\sigma_B D})
\le
e^{-c_V\sigma_B}+2\exp \Big(-\frac{C_B^2}{4}D\Big)
\le e^{-D}
\]
for all $D\le c\sigma_B$ after taking $C_B$ large enough and $c$ small enough.
 Similarly, for the lower tail
\[
\mathbb{P}(|B|\ge t)
\ge
\mathbb{P}\big(V\ge \tfrac12\sigma_B\big)\,\mathbb{P}\big(|Z|\ge t/\sqrt{\sigma_B/2}\big)
\ge
\frac{9}{10}\,\mathbb{P}\big(|Z|\ge \sqrt{2}\,t/\sqrt{\sigma_B}\big).
\]
Take $t:=c_B\sqrt{\sigma_B D}$. Then $\sqrt{2}\,t/\sqrt{\sigma_B}=\sqrt{2}\,c_B\sqrt{D}$.
By the standard Gaussian lower tail, for all $x\ge 1$,
$\mathbb{P}(|Z|\ge x)\ge (c_0/x)e^{-x^2/2}$.
Thus for $D\ge 1$ and $c_B$ small enough,
\[
\mathbb{P}(|B|\ge c_B\sqrt{\sigma_B D})
 \ge\
\frac{c_1}{\sqrt D}\exp \Big(-c_2 D\Big)
 \ge\
e^{-D},
\]
again after adjusting constants. So in the regime $1\le D\le c\sigma_B$ we have constants $0<c_B<C_B$ with
\[
\mathbb{P}\big(|B|\ge c_B\sqrt{\sigma_B D}\big)\ge e^{-D},
\qquad
\mathbb{P}\big(|B|\ge C_B\sqrt{\sigma_B D}\big)\le e^{-D}.
\]

\noindent\textbf{Step 5:}
We are now ready to prove the two bounds. Fix $D$ with $1\le D\le c\min\{n,\sigma_B\}$. 

\textbf{Upper Bound:} Let $a:=C_A\sqrt{nD}$ and $b:=C_B\sqrt{\sigma_B D}$. Then
\[
\mathbb{P}(|S|\ge ab)
\le
\mathbb{P}(|A|\ge a)+\mathbb{P}(|B|\ge b)
\le e^{-D}+e^{-D}\le 2e^{-D},
\]
so (absorbing the factor $2$ into constants) we get
\[
q(D) \le ab = C_A C_B\sqrt{n\sigma_B} D = C\sigma_S D.
\]

\textbf{Lower Bound:} For a lower bound on $q(D)$,
let $a:=c_A\sqrt{nD}$ and $b:=c_B\sqrt{\sigma_B D}$. By independence,
\[
\mathbb{P}(|S|\ge ab)
\ge
\mathbb{P}(|A|\ge a)\,\mathbb{P}(|B|\ge b)
\ge
e^{-D}\cdot e^{-D}
=
e^{-2D}.
\]
Now replace $D$ by $D/2$ in the construction (which stays within the same regime up to constants) to obtain
\[
\mathbb{P}\big(|S|\ge c\sqrt{n\sigma_B} D\big) \ge e^{-D}.
\]
Hence, \(
q(D) \ge  c \sqrt{n\sigma_B}\,D = c\sigma_S D.
\) Combining the two bounds yields $c\,\sigma_S D\le q(D)\le C\,\sigma_S D$ for all
$1\le D\le c\min\{n,\sigma_B\}$, as claimed.

\subsubsection{Proof of Lemma \ref{lem:derivative_clustering}}\label{sec:proof_derivative_sparseclust}
The proof of this Lemma is quite technical. We break it down by stating and proving some Lemmas that lead us to the result.  

We first prove a local central limit theorem (CLT) for the $\chi^2$ distribution. 
\begin{lemma}[Density CLT for $\chi^2_u$ at the $\sqrt{u}$ scale]\label{lem:chisq-density-clt}
Let $u\in \N$ such that $u=\omega(1)$ and consider $g_u$ to be the density of a $\chi^2_u$ distribution. Assume $D_n=o(u^{1/6})$. Then, uniformly for $|x|\le D_n$, as $n\to \infty$
\begin{align}
\sup_{|x|\le D_n}\left| \sqrt{2u} g_u\big(u+x\sqrt{2u}\big) - \phi(x)\right| &\rightarrow 0,\label{eq:chisq-clt}\\
\sup_{|x|\le D_n}\left| (2u) g_u'\big(u+x\sqrt{2u}\big) - \phi'(x)\right| &\rightarrow 0.\label{eq:chisq-clt-der}
\end{align}
\end{lemma}

\begin{proof}
Write $u=2\nu$ so $\nu=u/2=\omega(1)$. The $\chi^2_u$ density is, for $t>0$,
\[
g_u(t)=\frac{1}{2^\nu\Gamma(\nu)} t^{\nu-1}e^{-t/2}.
\]
Fix $|x|\le D_n$ and set
\(
t=u+x\sqrt{2u}=2\nu+2x\sqrt{\nu}
\ =2\nu\Big(1+\frac{x}{\sqrt{\nu}}\Big).
\)
Consider the log-density
\[
\log g_u(t)=-(\nu\log2+\log\Gamma(\nu))+(\nu-1)\log t-\frac t2.
\]
Using Stirling's formula,
\(
\Gamma(\nu)
=
\sqrt{2\pi}\,\nu^{\nu-\frac12}e^{-\nu}\big(1+o(1)\big),
\  \nu\to\infty,
\)
which implies
\(
\log\Gamma(\nu)=\Big(\nu-\tfrac12\Big)\log\nu-\nu+\tfrac12\log(2\pi)+o(1),
\)
uniformly as $\nu\to\infty$. Also,
\(
\log t=\log(2\nu)+\log\Big(1+\frac{x}{\sqrt{\nu}}\Big)
\)
and uniformly for $|x|\le D_n$,
\[
\log\Big(1+\frac{x}{\sqrt{\nu}}\Big)
=\frac{x}{\sqrt{\nu}}-\frac{x^2}{2\nu}+O\Big(\frac{|x|^3}{\nu^{3/2}}\Big).
\]
Since $D_n=o(u^{1/6})$ and $u=2\nu$, we have
\[
\sup_{|x|\le D_n}\frac{|x|^3}{\nu^{3/2}} = o\Big(\frac{1}{\nu}\Big)
\qquad\text{and in particular}\qquad
\sup_{|x|\le D_n}\frac{|x|^3}{\sqrt{\nu}} \to 0.
\]

Now plug into $\log g_u(t)$ and note that \(
\frac{t}{2}=\nu+x\sqrt{\nu}.
\), to get
\begin{align*}
\log g_u(t)
&=-\nu\log2-\Big[\Big(\nu-\tfrac12\Big)\log\nu-\nu+\tfrac12\log(2\pi)\Big]\\
&\quad+(\nu-1)\Big[\log(2\nu)+\frac{x}{\sqrt{\nu}}-\frac{x^2}{2\nu}
+O\Big(\frac{|x|^3}{\nu^{3/2}}\Big)\Big]
-\nu-x\sqrt{\nu}+o(1)\\
&=-\log 2-\tfrac12\log\nu-\tfrac12\log(2\pi)
 + (\nu-1)\frac{x}{\sqrt{\nu}}-x\sqrt{\nu}
 - (\nu-1)\frac{x^2}{2\nu}\\
&\quad+ O\Big(\frac{|x|^3}{\sqrt{\nu}}\Big)+O\Big(\frac{1+x^2}{\nu}\Big)+o(1)\\
&=-\log 2-\tfrac12\log\nu-\tfrac12\log(2\pi)
 - \frac{x^2}{2}
 + r_u(x),
\end{align*}
where
\[
r_u(x)=-\frac{x}{\sqrt{\nu}}
+O\Big(\frac{1+x^2}{\nu}\Big)+O\Big(\frac{|x|^3}{\sqrt{\nu}}\Big)+o(1).
\]
Now add the normalizing factor $\sqrt{2u}=\sqrt{4\nu}$:
\begin{align*}
\log\Big(\sqrt{2u}\,g_u(t)\Big)
&=\log g_u(t)+\tfrac12\log(4\nu)\\
&=\Big(-\log 2-\tfrac12\log\nu-\tfrac12\log(2\pi)\Big)
+\Big(\log 2+\tfrac12\log\nu\Big)
-\frac{x^2}{2}+r_u(x)\\
&=-\tfrac12\log(2\pi)-\frac{x^2}{2}+r_u(x)
=\log\phi(x)+r_u(x).
\end{align*}
Therefore,
\(
\sqrt{2u}\,g_u(u+x\sqrt{2u})=\phi(x)e^{r_u(x)}.
\) Thus,
\[
\big|\sqrt{2u}\,g_u(u+x\sqrt{2u})-\phi(x)\big|
=\phi(x)\,|e^{r_u(x)}-1|
\le \phi(0)\,|e^{r_u(x)}-1|,
\]
and taking $\sup_{|x|\le D_n}$ proves the first claim.

\smallskip
For the derivative statement,  define
\[
 h_u(x):=\sqrt{2u}\,g_u(u+x\sqrt{2u})=\phi(x)e^{r_u(x)}.
 \]
We differentiate $g_u$ to get
\[
g_u'(t)=g_u(t)\left(\frac{\nu-1}{t}-\frac12\right).
\]
By the chain rule,
\(
h_u'(x)=(2u)\,g_u'(u+x\sqrt{2u}).
\) Hence
\[
h_u'(x)=h_u(x)\,A_u(x),
\qquad
A_u(x):=\sqrt{2u}\left(\frac{\nu-1}{t}-\frac12\right).
\]
Using $u=2\nu$ and $t=2\nu(1+y)$,
\begin{align*}
A_u(x)
=2\sqrt{\nu}\left(\frac{\nu-1}{2\nu(1+y)}-\frac12\right)
=\sqrt{\nu}\left(\frac{1-1/\nu}{1+y}-1\right)
=-\frac{x+\nu^{-1/2}}{1+x/\sqrt{\nu}}.
\end{align*}
Therefore
\[
A_u(x)+x
=
\frac{x^2-1}{\sqrt{\nu}\,(1+x/\sqrt{\nu})}.
\]
Since $D_n/\sqrt{\nu}\to0$, for all sufficiently large $\nu$ we have
\(
1+x/\nu\ge \frac12,
\ \text{for all}\  |x|\le D_n,
\)
and hence
\[
\sup_{|x|\le D_n}|A_u(x)+x|
\le
\frac{2(D_n^2+1)}{\sqrt{\nu}}
\to0.
\]
Now recall that $h_u(x)=\phi(x)e^{r_u(x)}$ and $\phi'(x)=-x\phi(x)$. Thus
\begin{align*}
h_u'(x)-\phi'(x)
&=h_u(x)A_u(x)+x\phi(x)\\
&=\phi(x)e^{r_u(x)}A_u(x)+x\phi(x).
\end{align*}
Add and subtract $x\phi(x)e^{r_u(x)}$ to get
\[
h_u'(x)-\phi'(x)
=
\phi(x)e^{r_u(x)}(A_u(x)+x)
+
x\phi(x)\big(1-e^{r_u(x)}\big).
\]
Hence
\[
|h_u'(x)-\phi'(x)|
\le
\phi(x)e^{|r_u(x)|}\,|A_u(x)+x|
+
|x|\phi(x)\,|1-e^{r_u(x)}|.
\]
Taking the supremum over $|x|\le D_n$, we obtain
\begin{align*}
\sup_{|x|\le D_n}|h_u'(x)-\phi'(x)|
&\le
\Big(\sup_{|x|\le D_n}\phi(x)\Big)
\Big(\sup_{|x|\le D_n}e^{|r_u(x)|}\Big)
\Big(\sup_{|x|\le D_n}|A_u(x)+x|\Big)\\
&\qquad
+
\Big(\sup_{|x|\le D_n}|x|\phi(x)\Big)
\Big(\sup_{|x|\le D_n}|1-e^{r_u(x)}|\Big).
\end{align*}
Now
\(
\sup_{|x|\le D_n}\phi(x)\le \frac{1}{\sqrt{2\pi}}.
\)
Also,
\(
|x|\phi(x)=\frac{|x|}{\sqrt{2\pi}}e^{-x^2/2}.
\)
For $x\ge0$, the function $x\mapsto xe^{-x^2/2}$ has derivative
\(
\frac{d}{dx}\big(xe^{-x^2/2}\big)=e^{-x^2/2}(1-x^2),
\)
so its maximum is attained at $x=1$. Hence
\[
\sup_{|x|\le D_n}|x|\phi(x)\le \frac{e^{-1/2}}{\sqrt{2\pi}}.
\]
Since $\sup_{|x|\le D_n}|r_u(x)|\to0$, we have
\[
\sup_{|x|\le D_n}e^{|r_u(x)|}\to 1
\qquad\text{and}\qquad
\sup_{|x|\le D_n}|1-e^{r_u(x)}|\to0.
\]
Moreover, we already proved that
\(
\sup_{|x|\le D_n}|A_u(x)+x|\to0.
\)
Combining these bounds, we conclude that
\[
\sup_{|x|\le D_n}|h_u'(x)-\phi'(x)|\to0.
\]
Since $h_u'(x)=(2u)\,g_u'(u+x\sqrt{2u})$, this proves
\[
\sup_{|x|\le D_n}\left|(2u)\,g_u'(u+x\sqrt{2u})-\phi'(x)\right|\to0.
\]
This proves \eqref{eq:chisq-clt-der}.
\end{proof}

Let $G,H\sim N(0,1)$ independent and $Y:=GH$. For an integer  $u\ge 1$, let $Y_1,Y_2,\dots$ be i.i.d.\ copies of $Y$ and define
\begin{equation}\label{eq:Bu-def}
B_u:=\sum_{k=1}^u Y_k.
\end{equation}
Let $f_u$ be the density of $B_u$ for $u\ge 1$. We now prove a Local CLT for $B_u$. Recall that by Lemma \ref{lem:innerprod-bessel}, $B_u=(X_u-Y_u)/2$ with $X_u,Y_u$ i.i.d.\ $\chi^2_u$. We state the Lemma here and we defer it's proof to Section \ref{sec:aux_for_applic}.

\begin{lemma}[Density CLT for $f_u$ and $f_u'$]\label{lem:Bu-clt}
Assume $D_n=o(u^{1/6})$. Uniformly for $|x|\le D_n$,
\begin{align}
\sup_{|x|\le D_n}\left| \sqrt{u} f_u(x\sqrt{u}) - \phi(x)\right| &\rightarrow 0,\label{eq:Bu-clt}\\
\sup_{|x|\le D_n}\left| u f_u'(x\sqrt{u}) - \phi'(x)\right| &\rightarrow 0.\label{eq:Bu-clt-der}
\end{align}
\end{lemma}

We also make use on the following bounds for the density and the derivative of the density of $B_u$. The proof of this is also deferred to Section \ref{sec:aux_for_applic}.

\begin{lemma}[Uniform bounds for $f_u$ and $f_u'$ away from zero]\label{lem:fu-away}
For every $u\ge 1$, the density $f_u$ of $B_u$ is differentiable on $\mathbb{R}\setminus\{0\}$. 
Moreover, there exists an absolute constant $C>0$ such that for all $u\ge 1$ and all $x\neq 0$,
\[
f_u(x)\le \frac{C}{|x|},
\qquad
|f_u'(x)|\le \frac{C}{x^2}.
\]
\end{lemma}

Now we are ready to move on to the density of the random variable $B$ that appears in the overlap of the sparse clustering model.

\begin{lemma}[Exact density for $B$ conditioning]\label{lem:B-mix}
Let $f_B$ denote the density of $B$ on $\mathbb{R}\setminus\{0\}$. Then for all $x\neq 0$,
\[
f_B(x)=\sum_{u=1}^p  \mathbb{P}(U=u) f_u(x),
\qquad
f_B'(x)=\sum_{u=1}^p  \mathbb{P}(U=u) f_u'(x).
\]
\end{lemma}

\begin{proof}
By definition, $B\mid(U=u)\stackrel{d}=B_u$. 
For any Borel set $E\subset\mathbb{R}\setminus\{0\}$,
\[
\mathbb{P}(B\in E)=\sum_{u=1}^p \mathbb{P}(U=u)\mathbb{P}(B_u\in E),
\]
hence by writing the equation with integrals and by the uniqueness of the PDF it should hold that $f_B(x)=\sum_{u=1}^p\mathbb{P}(U=u)f_u(x)$ for $x\ne 0$.
Since the sum is finite we differentiate with respect to $x\neq 0$ term by term to get the formula for $f_B'(x)$.
\end{proof}

\begin{lemma}[Density CLT for $B$ and $f_B'$ on the $\sqrt{\sigma_B}$ scale]\label{lem:B-clt}
Fix $0<\varepsilon<1$. Assume $D_n=o(\sigma_B^{1/6})$. Uniformly for $\varepsilon\le |x|\le D_n$,
\begin{align}
\sup_{\varepsilon\le |x|\le D_n}\left| \sqrt{\sigma_B} f_B(x\sqrt{\sigma_B})-\phi(x)\right| &\rightarrow 0,\label{eq:B-clt}\\
\sup_{\varepsilon\le |x|\le D_n}\left| \sigma_B f_B'(x\sqrt{\sigma_B})-\phi'(x)\right| &\rightarrow 0.\label{eq:B-clt-der}
\end{align}
\end{lemma}

\begin{proof}
Let $L=\log\sigma_B$ and $\mathcal{I}:=\{u:\ |u-\sigma_B|\le \sqrt{\sigma_B}L\}$.
We split the mixture from Lemma \ref{lem:B-mix} into $u\in\mathcal{I}$ and $u\notin\mathcal{I}$.

\smallskip\noindent
\textbf{Step 1:} 
By Chernoff inequality for the Binomial distribution, $ \mathbb{P}(U\notin\mathcal{I})\le 2e^{-cL^2}=o(1)$.
By Lemma~\ref{lem:fu-away}, for every $u\ge 1$ and every $z\neq 0$,
\[
f_u(z)\le \frac{C}{|z|},
\qquad
|f_u'(z)|\le \frac{C}{z^2}.
\]
Therefore for $\varepsilon\le |x|\le D_n$,
\[
|f_u(x\sqrt{\sigma_B})|\le \frac{C}{\varepsilon\sqrt{\sigma_B}},
\qquad
|f_u'(x\sqrt{\sigma_B})|\le \frac{C}{\varepsilon^2\sigma_B}.
\]
Define
\(
C_{\varepsilon}:=C/\varepsilon^2.
\)
Then using Lemma \ref{lem:B-mix} we obtain, uniformly over $\varepsilon\le |x|\le D_n$,
\[
\sqrt{\sigma_B}\cdot \sum_{u\notin\mathcal{I}} \P(U=u)\,|f_u(x\sqrt{\sigma_B})|
 \le
C_{\varepsilon}\,\P(U\notin\mathcal{I})
 = o(1),
\]
and similarly
\[
\sigma_B\cdot \sum_{u\notin\mathcal{I}} \P(U=u)\,|f_u'(x\sqrt{\sigma_B})|
 \le
C_{\varepsilon}\,\P(U\notin\mathcal{I})
 = o(1).
\]

\smallskip\noindent
\textbf{Step 2:} 
Fix $\varepsilon\le |x|\le D_n$ and $u\in\mathcal{I}$. Write
\[
z_u(x):=x\sqrt{\sigma_B/u}.
\]
Since $u=\sigma_B+O(\sqrt{\sigma_B}L)$, we have $z_u(x)=x+O(|x|L/\sqrt{\sigma_B})$, uniformly over $\varepsilon\le |x|\le D_n$.
Because $|x|\le D_n$ and $D_n L/\sqrt{\sigma_B}=o(1)$, we have $|z_u(x)|\le 2D_n$ for all large $(n,p)$.
Also, since $|x|\ge \varepsilon$, for all large $(n,p)$ we have $|z_u(x)|\ge \varepsilon/2$.
Apply Lemma \ref{lem:Bu-clt} uniformly for $\varepsilon/2\le |z|\le 2D_n$:
\[
f_u(x\sqrt{\sigma_B}) = f_u(z_u(x)\sqrt{u})
=
\frac{1}{\sqrt{u}}\phi(z_u(x)) + o\Big(\frac{1}{\sqrt{u}}\Big),
\]
uniformly over $\varepsilon\le |x|\le D_n$ and $u\in\mathcal{I}$. Next, since $\sup_{\varepsilon\le |x|\le D_n}\sup_{u\in\mathcal{I}}|z_u(x)-x|=o(1)$ and $\phi$ is $O(1)$ Lipschitz, we have
\[
\sup_{\varepsilon\le |x|\le D_n}\sup_{u\in\mathcal{I}}|\phi(z_u(x))-\phi(x)|=o(1).
\]
Also, uniformly for $u\in\mathcal{I}$, we have $1/\sqrt{u}=1/\sqrt{\sigma_B}+O(L/\sigma_B)$. Combining these two we get
\[
f_u(x\sqrt{\sigma_B})
=
\frac{1}{\sqrt{\sigma_B}}\phi(x) + o\Big(\frac{1}{\sqrt{\sigma_B}}\Big),
\]
uniformly over $u\in\mathcal{I}$ and $\varepsilon\le |x|\le D_n$.
Averaging with weights $ \mathbb{P}(U=u)$ over $u\in\mathcal{I}$ and using Step 1 (note that almost all the mass of $U$ lies inside $\mathcal{I}$) yields \eqref{eq:B-clt}.
The derivative statement \eqref{eq:B-clt-der} is identical, using Lemma \ref{lem:Bu-clt} and specifically \eqref{eq:Bu-clt-der}:
\[
f_u'(x\sqrt{\sigma_B}) = f_u'(z_u(x)\sqrt{u})
=
\frac{1}{u}\phi'(z_u(x)) + o\Big(\frac{1}{u}\Big),
\]
then $1/u=1/\sigma_B+O(L/\sigma_B^{3/2})$ and $\phi'(z_u(x))=\phi'(x)+o(1)$ again by using the fact that $\phi'$ is $O(1)$ Lipschitz. 
\end{proof}

\begin{lemma}[Local CLT for $A$ and tail]\label{lem:lcltA}
Let $T:=\sqrt{D_n}\,\log(n\vee p)$. Uniformly for integers $a\equiv n\pmod 2$ with $|a|\le \sqrt n T$,
\[
 \mathbb{P}(A=a)=\frac{2}{\sqrt n} \phi \Big(\frac{a}{\sqrt n}\Big)\Big(1+o(1)\Big).
\]
Also $\mathbb{P}(|A|\geq \sqrt n T)\le 2e^{-T^2/2}$.
\end{lemma}

\begin{proof}
The first part is a standard application of Lemma \ref{lem:rad-master-two-regimes}. For the tail, Hoeffding's inequality for sums of i.i.d.\ $\pm 1$ gives
\[
\mathbb{P}(|A|\geq\sqrt{n}T)\le 2e^{-T^2/2}.
\]
\end{proof}

\begin{lemma}[Mixture density and derivative]\label{lem:S-density}
Let $f_B$ be the density of $B$. Then for every $t\neq 0$,
\[
f_S(t)=\sum_{a\in\{-n,-n+2,\dots,n\}\setminus\{0\}}  \mathbb{P}(A=a) \frac{1}{|a|} f_B(t/a),
\]
and $f_S$ is differentiable for $t\neq 0$ with
\[
f_S'(t)=\sum_{a\neq 0} \mathbb{P}(A=a) \frac{1}{a|a|} f_B'(t/a).
\]
\end{lemma}

\begin{proof}
Since $S=A\cdot B$, we prove this by conditioning on $A$. By the law of total probability and independence of $A$ and $B$, for any $x\in\R$,
\[
\mathbb{P}(S\le x)
=
\sum_{a\in\{-n,-n+2,\dots,n\}}\mathbb{P}(A=a)\,\mathbb{P}(aB\le x).
\]
The sum is finite, so we may differentiate term by term at any $t\neq 0$ to get
\[
f_S(t)
=
\frac{d}{dt}\,\mathbb{P}(S\le t)
=
\sum_{a\neq 0}\mathbb{P}(A=a)\,\frac{1}{|a|}\,f_B(t/a),
\]
which is the claimed formula for $f_S(t)$ for $t\neq 0$.
For $f_S'$, the result follow directly by differentiating once more term by term with respect to $t\neq 0$. 
\end{proof}

Let $G,H\sim N(0,1)$ independent and set $W:=GH$.
Its density is
\begin{equation}\label{eq:finfty}
f_\infty(x)=\int_{\mathbb{R}}\frac{1}{|y|} \phi(y) \phi(x/y) dy,\qquad x\neq 0.
\end{equation}

\begin{lemma}[Density Local Limit Theorem (LLT) for $S$ at $t=D_n\sigma_S$]\label{thm:lltS}
Assume $D_n=o(\sigma_B^{1/6})$. Set $t:=D_n\sigma_S=D_n\sqrt{n\sigma_B}$, where the random variables $B, S$ where defined in \ref{sec:notspclust}. Then,
\begin{equation}\label{eq:lltS}
\sigma_S f_S(t)=f_\infty(D_n)\big(1+o(1)\big), \qquad 
\sigma_S^2 f_S'(t)=f_\infty'(D_n)\big(1+o(1)\big).
\end{equation}
\end{lemma}

\begin{proof}

Fix $t=D_n\sigma_S=D_n\sqrt{n\sigma_B}$ and let $T=\sqrt{D_n}\log(n\vee p)$ as in Lemma \ref{lem:lcltA}.
We will repeatedly use the elementary lower bound (valid for all $x\ge 1$)
\begin{equation}\label{eq:finfty-lower}
f_\infty(x)
=\int_{\R}\frac{1}{|u|}\phi(u)\phi(x/u)\,du
\ \ge\
\int_{\sqrt x}^{\sqrt x+1}\frac{1}{u}\phi(u)\phi(x/u)\,du
 \ge  \frac{c}{\sqrt x}\,e^{-Cx},
\end{equation}
for universal constants $c,C>0$ (since on $[\sqrt x,\sqrt x+1]$ we have $u=\Theta(\sqrt x)$,
$x/u=\sqrt x+O(1)$, hence $\phi(u)\phi(x/u)=\Theta( e^{-x})$).

Finally, note that for every fixed $x\neq 0$, one may differentiate \eqref{eq:finfty} under the integral sign
(because as $u\to 0$ the factor $\phi(x/u)$ decays super-exponentially, and as $|u|\to\infty$ the factor $\phi(u)$ decays),
obtaining for $x\neq 0$,
\begin{equation}\label{eq:finfty-der}
f_\infty'(x)=\int_{\R}\frac{1}{u|u|}\phi(u)\phi'(x/u)\,du.
\end{equation}
Furthermore, by integrating \eqref{eq:finfty-der} over $u\in[\sqrt x,\sqrt x+1]$ noting $\phi'(x/u)=\Theta( -\sqrt x\,\phi(\sqrt x))$ and then using once more \eqref{eq:finfty-lower})
\begin{equation}\label{eq:finfty-der_ineq}
f_\infty'(x)\ge C_1
\int_{\sqrt x}^{\sqrt x+1}\frac{1}{u}\phi(u)\phi(x/u)\,du
 \ge  \frac{c'}{\sqrt x}\,e^{-Cx},
\end{equation}
for some universal constants $C_1,c'>0$.

\smallskip\noindent
\textbf{Step 1:}
First we want to show that the tail contribution $|a|>\sqrt n\,T$ is very small. To do this, we will apply Lemma~\ref{lem:B-clt}
at $t/a$. First, notice that by Lemma \ref{lem:lcltA},
\[
\P(|A|\geq \sqrt n\,T)\le 2e^{-T^2/2}=2\exp\Big(-\frac12 D_n\log^2(n\vee p)\Big).
\]
For $|a|>\sqrt n\,T$, define
\[
x_a:=\frac{t/a}{\sqrt{\sigma_B}}
=\frac{D_n\sqrt{n\sigma_B}/a}{\sqrt{\sigma_B}}
=\frac{D_n}{a/\sqrt n},
\qquad\text{so}\qquad
|x_a|\le \frac{D_n}{T}=\frac{\sqrt{D_n}}{\log(n\vee p)}.
\]
In particular, $|x_a|\le D_n$, so Lemma~\ref{lem:B-clt} applies uniformly over the choice of $x_{\alpha}$ and yields
\[
|f_B(t/a)| \le \frac{C}{\sqrt{\sigma_B}},
\qquad
|f_B'(t/a)| \le \frac{C}{\sigma_B}\,|x_a|
\le \frac{C}{\sigma_B}\cdot \frac{D_n}{T},
\]
for an absolute constant $C>0$, using that $|\phi(x)|\le (2\pi)^{-1/2}$ and $|\phi'(x)|=|x|\phi(x)\le C|x|$.
Therefore, using Lemma~\ref{lem:S-density} and $|a|>\sqrt n\,T\Rightarrow |a|^{-1}\le (\sqrt n\,T)^{-1}$,
\begin{align}
\sigma_S\sum_{|a|>\sqrt n\,T}\P(A=a)\frac{1}{|a|}\,|f_B(t/a)|
&\le
\sigma_S\cdot \P(|A|>\sqrt n\,T)\cdot \frac{1}{\sqrt n\,T}\cdot \frac{C}{\sqrt{\sigma_B}}
\nonumber\\
&=
\frac{C}{T}\,\P(|A|>\sqrt n\,T)
=
o\Big(\frac{e^{-D_n}}{\sqrt{D_n}}\Big)
=
o\big(f_\infty(D_n)\big),
\label{eq:tail-density-bound-new}
\end{align}
where we used $\sigma_S=\sqrt{n\sigma_B}$ and \eqref{eq:finfty-lower}.
Similarly, since $|a|>\sqrt n\,T\Rightarrow |a|^{-2}\le (nT^2)^{-1}$,
\begin{align}
\sigma_S^2\sum_{|a|>\sqrt n\,T}\P(A=a)\frac{1}{|a|^2}\,|f_B'(t/a)|
&\le
\sigma_S^2\cdot \P(|A|>\sqrt n\,T)\cdot \frac{1}{nT^2}\cdot \frac{C}{\sigma_B}\cdot \frac{D_n}{T}
\nonumber\\
&=
C\,\frac{D_n}{T^3}\,\P(|A|>\sqrt n\,T)
=
o\Big(\frac{e^{-D_n}}{\sqrt{D_n}}\Big)
=
o\big(|f_\infty'(D_n)|\big),
\label{eq:tail-deriv-bound-new}
\end{align}

where the last step uses \eqref{eq:finfty-der_ineq}.
Thus, from now on we restrict the sums to $|a|\le \sqrt n\,T$.
% The same argument applies to the derivative sum for $S'$, using additionally that
% $|a|>\sqrt n T$ implies $1/a^2\le 1/(nT^2)$ and that $|f_B'(z)|$ grows at most like $C/|z|$ near $0$, yielding
% \[
% \sigma_S^2\sum_{|a|>\sqrt n\,T}\P(A=a)\frac{1}{|a|^2}|f_B'(t/a)|
% =o\big(|f_\infty'(D_n)|\big),
% \]
% where one may use a lower bound $|f_\infty'(x)|\ge c e^{-x}/\sqrt x$ for $x\ge 1$
% (obtained by integrating \eqref{eq:finfty-der} over $u\in[\sqrt x,\sqrt x+1]$ and noting $\phi'(x/u)=\Theta( -\sqrt x\,\phi(\sqrt x))$). Thus, from now on we restrict the sums to $|a|\le \sqrt n\,T$.

% \smallskip\noindent
% \textbf{Step 1:} 
% Restrict the sums in Lemma \ref{lem:S-density} to $0<|a|\le \sqrt{n}T$.
% The tail $|a|>\sqrt{n}T$ contributes $o(1/\sigma_S)$ to $f_S(t)$ and $o(1/\sigma_S^2)$ to $f_S'(t).$ Indeed,  by
% Lemma \ref{lem:lcltA} (tail bound) and boundedness of $f_B$ and $f_B'$ (Lemma~\ref{lem:B-mix} and Lemma~\ref{lem:fu-sup}).

\smallskip\noindent
\textbf{Step 2:} 
Write $y:=a/\sqrt{n}$ so $|y|\le T$ and $|a|=\sqrt{n}|y|$.
Then
\[
\frac{t}{a}=\frac{D_n\sqrt{n\sigma_B}}{\sqrt{n}y}=\frac{D_n\sqrt{\sigma_B}}{y}.
\]
Lemma \ref{lem:lcltA} gives $ \mathbb{P}(A=a)=(2/\sqrt{n})\phi(y) (1+o(1))$ uniformly on $|y|\le T$. Let
\[
\mathcal{A}:=\Big\{\alpha \in \Z : \tfrac12\sqrt{nD_n}\le |\alpha|\le 2\sqrt{nD_n}\Big\}.
\]
We will further truncate our sum to $\alpha \in \mathcal{A}$. We claim that the contribution of $\alpha\notin\mathcal{A}$ to $\sigma_S f_S(t)$ is still $o(f_\infty(D_n))$,
even after the truncation $|\alpha|\le \sqrt{n}T$. Indeed, using again the notation $y=a/\sqrt{n},$ then
\[
g_{D_n}(y):=\frac{1}{|y|}\phi(y)\phi(D_n/y).
\]
For $|y|\le \tfrac12\sqrt{D_n}$, we have $(D_n/y)^2\ge 4D_n$, hence
\begin{equation}\label{eq:gdninequality}
g_{D_n}(y)\le \frac{C}{|y|}\exp \Big(-\frac{1}{2}\cdot\frac{D_n^2}{y^2}\Big)
\le \frac{C}{|y|}e^{-2D_n},
\end{equation}
and integrating over $0<|y|\le \tfrac12\sqrt{D_n}$ gives
$\int_{0<|y|\le \frac12\sqrt{D_n}} g_{D_n}(y)\,dy \le Ce^{-2D_n}\log D_n$.
Similarly, for $|y|\ge 2\sqrt{D_n}$ we have $y^2\ge 4D_n$ so
$g_{D_n}(y)\le C e^{-2D_n}/|y|$, hence
$\int_{|y|\ge 2\sqrt{D_n}} g_{D_n}(y)\,dy \le Ce^{-2D_n}$.
Comparing with \eqref{eq:finfty-lower}, both tails are $o(f_\infty(D_n))$.
Fix $\alpha\in\mathcal{A}$ and the corresponding point $y=\alpha/\sqrt n$.
By Lemma~\ref{lem:lcltA} (using that $\mathcal{A}\subseteq\{|\alpha|\le \sqrt{n}T\}$ for all large $n$),
\begin{equation}\label{eq:A-lclt-used}
\P(A=a)=\frac{2}{\sqrt n}\phi(y)\big(1+o(1)\big),
\end{equation}
Next, note that  for $\alpha \in \mathcal{A}$ we have
\[
\frac{t}{a}=\frac{D_n\sqrt{n\sigma_B}}{\sqrt n\,y}=\frac{D_n\sqrt{\sigma_B}}{y},
\qquad\text{so}\qquad
\frac{t/a}{\sqrt{\sigma_B}}=\frac{D_n}{y}\in\Big[\tfrac12\sqrt{D_n},\,2\sqrt{D_n}\Big].
\]
In particular, this ratio is at most $ 2\sqrt{D_n}=o(\sigma_B^{1/6})$
since $D_n=o(\sigma_B^{1/6})$. Therefore, we may apply Lemma \ref{lem:B-clt}  and similarly for the derivative to get:
\begin{equation}\label{eq:B-clt-used}
f_B\Big(\frac{D_n\sqrt{\sigma_B}}{y}\Big)
=\frac{1}{\sqrt{\sigma_B}}\phi(D_n/y)\big(1+o(1)\big),
\qquad
f_B'\Big(\frac{D_n\sqrt{\sigma_B}}{y}\Big)
=\frac{1}{\sigma_B}\phi'(D_n/y)\big(1+ o(1)\big).
\end{equation}
Plugging \eqref{eq:A-lclt-used} and \eqref{eq:B-clt-used} into Lemma \ref{lem:S-density}, and using $|a|=\sqrt n\,|y|$, gives

% Define $L:=\log(n\vee p)$ and the ``main'' region
% \[
% \mathcal{Y}:=\Big\{y:\ \frac{\sqrt{D_n}}{L}\le |y|\le \sqrt{D_n}\,L\Big\}.
% \]
% For $|y|<\sqrt{D_n}/L$, we have $|D_n/y|>\sqrt{D_n}L$ so $\phi(D_n/y)\le e^{-cD_nL^2}$.
% For $|y|>\sqrt{D_n}L$, we have $\phi(y)\le e^{-cD_nL^2}$.
% Thus the contribution of $y\notin\mathcal{Y}$ is $o(1/\sigma_S)$ (and similarly $o(1/\sigma_S^2)$ for the derivative).
% Hence we may restrict to $y\in\mathcal{Y}$. 
% On $\mathcal Y$ we have $\left|D_n/y\right|\le \sqrt{D_n}\,L$.
% Thus, it suffices that Lemma~\ref{lem:B-clt} holds uniformly for $|x|\le \sqrt{D_n}\,L$,
% which is ensured by the condition $\sqrt{D_n}\,L=o(\sigma_B^{1/6})$.
% Now apply Lemma \ref{lem:B-clt}:
% \[
% f_B\Big(\frac{D_n\sqrt{\sigma_B}}{y}\Big)
% =\frac{1}{\sqrt{\sigma_B}}\phi(D_n/y)+o\Big(\frac{1}{\sqrt{\sigma_B}}\Big),
% \]
% uniformly for $y\in\mathcal{Y}$. Plugging into Lemma \ref{lem:S-density} yields
\[
f_S(t)
=
\frac{(1+o(1))}{\sigma_S}\cdot 
\sum_{\substack{0<|a|\le \sqrt{n}T\\ a\in\mathcal{A}}}
\Big(\frac{2}{\sqrt{n}}\Big)\frac{1}{|y|}\phi(y)\phi(D_n/y).
\]
The sum is a Riemann sum for \eqref{eq:finfty} evaluated at $x=D_n$, giving \eqref{eq:lltS}.

\noindent \textbf{Step 3:} For the derivative, we repeat this process using Lemma \ref{lem:S-density} for $f_S'$.

We first show that the contribution of $\alpha\notin\mathcal A$ is negligible also for the derivative.
For $y=a/\sqrt{n}$, $0<|y|\le \tfrac12\sqrt{D_n}$, using the bound  in \eqref{eq:gdninequality},
\[
\frac{D_n}{|y|^2}g_{D_n}(y)
\le
C\frac{D_n}{|y|^3}\exp\!\Big(-\frac{D_n^2}{2y^2}\Big).
\]
Therefore, by the change of variables $z=D_n/|y|$,
\[
\int_{0<|y|\le \frac12\sqrt{D_n}} \frac{D_n}{|y|^2}g_{D_n}(y)\,dy
\le
C\int_{2\sqrt{D_n}}^\infty \frac{z}{D_n}\phi(z)\,dz
\le
C\frac{e^{-2D_n}}{D_n}.
\]
Similarly, for $|y|\ge 2\sqrt{D_n}$ we have
\[
\frac{D_n}{|y|^2}g_{D_n}(y)\le C\frac{D_n}{|y|^3}e^{-y^2/2},
\]
and thus
\[
\int_{|y|\ge 2\sqrt{D_n}} \frac{D_n}{|y|^2}g_{D_n}(y)\,dy
\le
C D_n \int_{|y|\ge 2\sqrt{D_n}}\frac{e^{-y^2/2}}{|y|^3}\,dy
\le
C\frac{e^{-2D_n}}{D_n}.
\]
Comparing with \eqref{eq:finfty-der_ineq}, both tails are $o(|f_\infty'(D_n)|)$.

Now, for $\alpha\in\mathcal A$, by \eqref{eq:A-lclt-used} and \eqref{eq:B-clt-used},
\[
\P(A=a)=\frac{2}{\sqrt n}\phi(y)\big(1+o(1)\big),
\qquad
f_B'\Big(\frac{D_n\sqrt{\sigma_B}}{y}\Big)
=\frac{1}{\sigma_B}\phi'(D_n/y)\big(1+ o(1)\big),
\]
uniformly over $\alpha\in\mathcal A$.
Plugging this into Lemma \ref{lem:S-density} for $f_S'$ and using \(y=\alpha/\sqrt{n}\)
gives
\[
f_S'(t)
=
\frac{(1+o(1))}{\sigma_S^2}\cdot 
\sum_{\substack{0<|a|\le \sqrt{n}T\\ a\in\mathcal{A}}}
\Big(\frac{2}{\sqrt{n}}\Big)\frac{1}{y|y|}\phi(y)\phi'(D_n/y).
\]
The sum is a Riemann sum for \eqref{eq:finfty-der} evaluated at $x=D_n$, and by the tail bounds just proved the contribution of $y\notin [\frac12\sqrt{D_n},2\sqrt{D_n}]$ is negligible. Therefore,
\[
f_S'(t)=\frac{(1+o(1))}{\sigma_S^2}f_\infty'(D_n),
\]
which is the desired conclusion.

% For the derivative, we repeat this process using Lemma \ref{lem:S-density} for $f_S'$.
% % \[
% % f_B'\Big(\frac{D_n\sqrt{\sigma_B}}{y}\Big)
% % =\frac{1}{\sigma_B}\phi'(D_n/y)+o\Big(\frac{1}{\sigma_B}\Big),
% % \]
%  The resulting Riemann sum converges to $f_\infty'(D_n)$ (differentiate under the integral sign), yielding our result. \textcolor{red}{needs details} 
\end{proof}

We are now are ready to prove our main Lemma.

\begin{proof}[Proof of Lemma \ref{lem:derivative_clustering}]

By Lemma \ref{thm:lltS},
\[
f_S(D_n\sigma_S)=\frac{1}{\sigma_S}f_\infty(D_n) (1+o(1)),
\qquad
f_S'(D_n\sigma_S)=\frac{1}{\sigma_S^2}f_\infty'(D_n) (1+o(1)).
\]
Therefore
\[
\frac{d}{ds}\log f_S(s)\Big|_{s=D_n\sigma_S}
=
\frac{f_S'(D_n\sigma_S)}{f_S(D_n\sigma_S)}
=
\frac{1}{\sigma_S}\frac{f_\infty'(D_n)}{f_\infty(D_n)}+o\Big(\frac{1}{\sigma_S}\Big)
=
\frac{1}{\sigma_S}(\log f_\infty)'(D_n)+o\Big(\frac{1}{\sigma_S}\Big).
\]
Finally, by \eqref{eq:fd-bessel}, we have
\(
f_\infty(x)=\frac{1}{\pi}K_0(|x|),
\)
and for $x>0$,
\[
(\log f_\infty)'(x)=-\frac{K_1(x)}{K_0(x)}.
\]
Moreover, by Lemma~\ref{lem:bessel-asympt},
\[
\frac{K_1(x)}{K_0(x)}\to 1
\qquad\text{as }x\to\infty.
\]
Hence
\(
(\log f_\infty)'(x)\to -1
\ \text{as }\ x\to\infty,
\)
so after rescaling by $\sigma_S$, the leading order is $-1/\sigma_S$.
% Finally, the known identity $f_\infty(x)=\pi^{-1}K_0(|x|)$ implies $(\log f_\infty)'(x)=-K_1(x)/K_0(x)\to -1$ as $x\to\infty$, which can be found in \cite[(9.7.2)]{abramowitz1965handbook},
% so the leading order is $-1/\sigma_S$. % \textcolor{red}{Add ref!! the same as tensor pca}
\end{proof}

\subsection{Auxiliary lemmas}\label{sec:aux_for_applic}

In this section we present the proofs of Lemmas \ref{lem:innerprod-bessel}, \ref{lem:rad-master-two-regimes}, \ref{lem:Bu-clt} and \ref{lem:fu-away}.

% \begin{proof}[Proof of Lemma \ref{lem:innerprod-bessel}]
% Define
% \[
% U:=\frac{G+H}{\sqrt{2}},
% \qquad
% V:=\frac{G-H}{\sqrt{2}},
% \qquad
% X:=\|U\|^2,
% \qquad
% Y:=\|V\|^2.
% \]
% Then $U,V$ are independent $N(0,I_d)$ vectors, hence $X,Y$ are independent $\chi^2_d$ random variables. Next, notice that
% \[
% \|U\|^2-\|V\|^2
% =\frac{\|G+H\|^2-\|G-H\|^2}{2}
% =2\langle G,H\rangle,
% \]
% so \(
% W=\langle G,H\rangle=\frac{X-Y}{2},
% \) proving the first claim. To calculate the density of $W$ we do the following.
% Let $g_d$ denote the $\chi^2_d$ density:
% \[
% g_d(t)=\frac{1}{2^{d/2}\Gamma(d/2)}\,t^{d/2-1}\,e^{-t/2} 1\{t>0\}.
% \]
% Then for any $x\in\mathbb{R}$, the change-of-variables formula for $(X-Y)/2$ gives
% \begin{equation}\label{eq:fd-conv}
% f_d(x)=2\int_0^\infty g_d(t)\,g_d(t-2x) 1\{t-2x>0\}\,dt
% =2\int_{\max(0,2x)}^\infty g_d(t)\,g_d(t-2x) dt.
% \end{equation}
% A standard integral identity (equivalently, the classical representation of $K_\nu$)
% expresses the convolution integral \eqref{eq:fd-conv} in the closed form \eqref{eq:fd-bessel}. 
% \end{proof}

\begin{proof}[Proof of Lemma \ref{lem:innerprod-bessel}]
Set
\[
U=\frac{G+H}{\sqrt2},\qquad V=\frac{G-H}{\sqrt2}.
\]
Then $U,V$ are independent $N(0,I_d)$ vectors, and
\[
4\langle G,H\rangle=\|G+H\|^2-\|G-H\|^2
=2\|U\|^2-2\|V\|^2.
\]
Thus $W=(X-Y)/2$, where $X=\|U\|^2$ and $Y=\|V\|^2$ are i.i.d.\ $\chi^2_d$.

Next, for a single coordinate product $Z=G_1H_1$,
\[
\mathbb E[e^{itZ}\mid G_1]
=
\exp\!\left(-\frac{t^2G_1^2}{2}\right),
\]
so after averaging over $G_1\sim N(0,1)$,
\[
\phi_Z(t):=\mathbb E[e^{itZ}]=(1+t^2)^{-1/2}.
\]
Since $W=\sum_{i=1}^d G_iH_i$ is a sum of $d$ independent copies of $Z$,
\[
\phi_W(t)=\mathbb E[e^{itW}]=(1+t^2)^{-d/2}.
\]
Now use the standard cosine-transform identity
\begin{equation}\label{eq:cos-transform-bessel}
\int_0^\infty \frac{\cos(tx)}{(1+t^2)^\mu}\,dt
=
\frac{\sqrt{\pi}}{\Gamma(\mu)}
\left(\frac{|x|}{2}\right)^{\mu-\frac12}
K_{\mu-\frac12}(|x|),
\qquad \mu>0,\ x\neq0,
\end{equation}
see \cite[Chap.~9]{abramowitz1965handbook}. Since $\phi_W$ is even,
\[
f_d(x)
=
\frac{1}{2\pi}\int_{\mathbb R} e^{-itx}(1+t^2)^{-d/2}\,dt
=
\frac{1}{\pi}\int_0^\infty \frac{\cos(tx)}{(1+t^2)^{d/2}}\,dt.
\]
Applying \eqref{eq:cos-transform-bessel} with $\mu=d/2$ gives \eqref{eq:fd-bessel}.
The continuity statement at $0$ follows from \eqref{eq:bessel-small}.
\end{proof}

\begin{proof}[Proof of Lemma \ref{lem:rad-master-two-regimes}]
For $s\equiv n\pmod 2$ we have
\begin{equation}\label{eq:rad-master-exact}
\PP(S_n=s)=2^{-n}\binom{n}{j},\qquad j=\frac{n+s}{2},\quad p=\frac{j}{n}.
\end{equation}
Assume $|s|\le n/2$, so $p\in[1/4,3/4]$ and $j,n-j\ge n/4$.
Using Stirling's formula with remainder,
\[
\log(m!)=m\log m-m+\tfrac12\log(2\pi m)+O(1/m),
\]
one obtains uniformly over $p\in[1/4,3/4]$,
\begin{equation}\label{eq:rad-master-logbinom}
\log\binom{n}{j}
=
nH(p)-\tfrac12\log \big(2\pi n p(1-p)\big)+O(1/n),
\end{equation}
where $H(p)=-p\log p-(1-p)\log(1-p)$.
Combining \eqref{eq:rad-master-exact}-\eqref{eq:rad-master-logbinom} gives the form
\begin{equation}\label{eq:rad-master-form}
\PP(S_n=s)
=
\frac{1}{\sqrt{2\pi n p(1-p)}}\,
\exp\big(n(H(p)-\log 2)\big)\,
\exp\big(O(1/n)\big),
\qquad (|s|\le n/2).
\end{equation}
Write $p=\tfrac12+x$ with $x:=\tfrac{s}{2n}$ (so $|x|\le 1/4$ when $|s|\le n/2$).
A Taylor expansion of $H$ around $1/2$ yields
\begin{equation}\label{eq:rad-master-entropy-expansion}
n(H(p)-\log 2)
=
-\frac{s^2}{2n}
+O \Big(\frac{|s|^3}{n^2}\Big),
\qquad (|s|\le n/2),
\end{equation}
uniformly in $s$ in that range.
In particular, this also implies that
\begin{equation}\label{eq:rad-master-entropy-onesided}
n(H(p)-\log 2) \ge -\frac{s^2}{2n}-C\frac{|s|^3}{n^2},
\qquad (|s|\le n/2),
\end{equation}
for an absolute constant $C>0$.

\medskip
\noindent \textbf{Proof of first claim:}
Assume $|s|\le \sqrt n\,T$ with $T^3=o(\sqrt n)$. Then $|s|=o(n^{2/3})$ and in particular $|s|\le n/2$
for all large $n$, so \eqref{eq:rad-master-form} applies.
Also $p(1-p)=\tfrac14+O(s^2/n^2)$, hence
\[
\frac{1}{\sqrt{2\pi n p(1-p)}}=\frac{2}{\sqrt{2\pi n}}\exp \Big(O\Big(\frac{s^2}{n^2}\Big)\Big)
=\frac{2}{\sqrt{2\pi n}}\big(1+o(1)\big),
\]
uniformly on $|s|\le \sqrt n\,T$, since $s^2/n^2\le T^2/n=o(1)$.
By \eqref{eq:rad-master-entropy-expansion},
\[
\frac{|s|^3}{n^2}\le \frac{(\sqrt n\,T)^3}{n^2}=\frac{T^3}{\sqrt n}=o(1),
\]
so the exponential remainders in \eqref{eq:rad-master-form} are $1+o(1)$ uniformly on the same window.
Therefore,
\[
\PP(S_n=s)
=
\frac{2}{\sqrt{2\pi n}}e^{-s^2/(2n)}\big(1+o(1)\big)
=
\frac{2}{\sqrt n}\phi \Big(\frac{s}{\sqrt n}\Big)\big(1+o(1)\big),
\]
uniformly for $|s|\le \sqrt n\,T$, proving \eqref{eq:rad-master-lclt}.

\medskip
\noindent \textbf{Proof of second claim:}
Fix $n$ large and $0\le s\le n/2$. Then $p\in[1/4,3/4]$ and \eqref{eq:rad-master-form} holds.
Using the crude bound $p(1-p)\le 1/4$ gives
\[
\frac{1}{\sqrt{2\pi n p(1-p)}}\ge \frac{2}{\sqrt{2\pi n}}.
\]
Combine this with the one-sided entropy bound \eqref{eq:rad-master-entropy-onesided} and the factor
$\exp(O(1/n))\ge 1/2$ to conclude
\[
\PP(S_n=s)
\ge
\frac{c}{\sqrt n}\exp\Big(-\frac{s^2}{2n}-C\frac{s^3}{n^2}\Big),
\]
for absolute constants $c,C>0$ and all $n\ge n_\star$, which is \eqref{eq:rad-master-lb}.
\end{proof}

Before proving this we state two Auxiliary Lemmas. Their proofs are deferred to Section \ref{sec:twomorelemmas}.
\begin{lemma}[Gaussian product integral]\label{lem:gauss_product_integral}
Let $\phi(z):=(2\pi)^{-1/2}e^{-z^{2}/2}$ be the standard normal density. Then for every $x\in\R$,
\[
\int_{\R}\phi(z)\,\phi(z-\sqrt{2}\,x)\,dz=\frac{1}{\sqrt{2}}\,\phi(x).
\]
\end{lemma}

\begin{lemma}[Sup bounds for $\chi^2_u$ density and its derivative]\label{lem:chisq_sup_bounds}
Let $g_u$ be the density of a $\chi^2_u$ random variable:
\[
g_u(t)=\frac{1}{2^{u/2}\Gamma(u/2)}\,t^{u/2-1}e^{-t/2},\qquad t>0.
\]
Then there exists an absolute constant $C<\infty$ such that for all $u$,
\[
\sup_{t>0} g_u(t)\ \le\ \frac{C}{\sqrt{u}},
\qquad
\sup_{t>0} |g_u'(t)|\ \le\ \frac{C}{u}.
\]
\end{lemma}

\begin{proof}[Proof of Lemma~\ref{lem:Bu-clt}]
By Lemma~\ref{lem:innerprod-bessel}, $B_u=(X_u-Y_u)/2$ with $X_u,Y_u$ i.i.d. $\chi^2_u$ and density $g_u$.
Hence for any $b\in\R$,
\begin{equation}\label{eq:Bu-conv-proof}
f_u(b)=2\int_{0}^\infty g_u(t)\, g_u(t-2b)\,\mathbf{1}\{t-2b>0\}\,dt .
\end{equation}
Fix $|x|\le D_n$ and set $b=x\sqrt u$. Write $u=2\nu$. Let \(
t=u+z\sqrt{2u}, dt=\sqrt{2u}\,dz. \)
The constraint $t>0$ is equivalent to $z>-\sqrt{u/2}$.
The constraint $t-2b>0$ is equivalent to
\[
u+z\sqrt{2u}-2x\sqrt u>0
 \Leftrightarrow\
z>-\sqrt{u/2}+\sqrt2\,x.
\]
We denote from now on $z_{u,x}:=-\sqrt{u/2}+\sqrt2\,x$.
Since $|x|\le D_n=o(u^{1/6})$, we have $|x|=o(\sqrt u)$, hence
\[
\sup_{|x|\le D_n} z_{u,x} = -\sqrt{u/2}+o(\sqrt u)\to -\infty.
\]
Hence, with this change of variables, \eqref{eq:Bu-conv-proof} becomes
\begin{equation}\label{eq:Bu-zform}
f_u(x\sqrt u)=2\int_{z>z_{u,x}} g_u(u+z\sqrt{2u})\,g_u\big(u+(z-\sqrt2 x)\sqrt{2u}\big)\,\sqrt{2u}\,dz.
\end{equation}
Fix $M\ge 1$. Split the integral in \eqref{eq:Bu-zform} into
$|z|\le M$ and $|z|>M$.

For the tail $|z|>M$, we bound one factor by a uniform sup bound on $g_u$ and the other by its tail mass.
Using Lemma \ref{lem:chisq_sup_bounds}
\begin{equation}\label{eq:gu-sup}
\sup_{t>0} g_u(t)\le \frac{C}{\sqrt u}\qquad\text{for all large }u,
\end{equation}
for an absolute constant $C>0$.
Also,
\[
\int_{|z|>M} g_u(u+z\sqrt{2u})\,dz
=\frac{1}{\sqrt{2u}}\P\!\left(|X_u-u|>M\sqrt{2u}\right)
\le \frac{2}{\sqrt{2u}}e^{-cM^2},
\]
where we used a chi-square concentration bound and the last inequality holds for all $u$ large
and $M$.
Hence, uniformly over $|x|\le D_n$,
\begin{align}
& 2\sqrt{2u}\int_{|z|>M} g_u(u+z\sqrt{2u})\,g_u(u+(z-\sqrt2x)\sqrt{2u})\,dz \notag\\
&\le 2\sqrt{2u}\cdot \sup_{t>0}g_u(t)\cdot \int_{|z|>M} g_u(u+z\sqrt{2u})\,dz
 \le \frac{C}{\sqrt u}\,e^{-cM^2}. \label{eq:tail-bound-density}
\end{align}
After multiplying by $\sqrt u$, this tail is $\le C e^{-cM^2}$, which can be made arbitrarily small by choosing $M$ large. So it suffices to analyze \eqref{eq:Bu-zform} on $|z|\le M$.

On $|z|\le M$ and $|x|\le D_n$, we have $|z-\sqrt2x|\le M+\sqrt2 D_n$.
Since $D_n=o(u^{1/6})$, also $M+\sqrt2D_n=o(u^{1/6})$ for every fixed $M$, so Lemma~\ref{lem:chisq-density-clt} applies uniformly:
\[
g_u(u+z\sqrt{2u})=\frac{1}{\sqrt{2u}}\phi(z)\big(1+o(1)\big),
\qquad
g_u(u+(z-\sqrt2x)\sqrt{2u})=\frac{1}{\sqrt{2u}}\phi(z-\sqrt2x)\big(1+o(1)\big),
\]
uniformly over $|z|\le M$ and $|x|\le D_n$. Plugging into \eqref{eq:Bu-zform} (restricted to $|z|\le M$) yields
\begin{align}
f_u(x\sqrt u)
&=2\sqrt{2u}\int_{|z|\le M}\frac{1}{2u}\phi(z)\phi(z-\sqrt2x)\,dz \cdot (1+o(1))
 + O\!\left(\frac{1}{\sqrt u}e^{-cM^2}\right) \notag\\
&=\frac{\sqrt2}{\sqrt u}\int_{|z|\le M}\phi(z)\phi(z-\sqrt2x)\,dz\cdot(1+o(1))
 + O\!\left(\frac{1}{\sqrt u}e^{-cM^2}\right), \label{eq:density-main}
\end{align}
uniformly for $|x|\le D_n$.

Now let $u\to\infty$ first (with $M$ fixed), then let $M\to\infty$.
Using Lemma \ref{lem:gauss_product_integral}, $\int_{\R}\phi(z)\phi(z-\sqrt2x)\,dz=\frac{1}{\sqrt2}\phi(x)$ for all $x\in\R$,
and $\int_{|z|>M}\phi(z)\phi(z-\sqrt2x)\,dz\le \int_{|z|>M}\phi(z)\,dz\to 0$ uniformly in $x$,
from \eqref{eq:density-main} we obtain
\[
f_u(x\sqrt u)=\frac{1}{\sqrt u}\phi(x)\big(1+o(1)\big),
\]
uniformly for $|x|\le D_n$, proving \eqref{eq:Bu-clt}.

 For the second equation we want to prove for all $b\in \R$,
differentiate \eqref{eq:Bu-conv-proof} with respect to $b$ to get
\begin{equation}\label{eq:fu-der-conv}
f_u'(b)=-4\int_{0}^\infty g_u(t)\,g_u'(t-2b)\,\mathbf{1}\{t-2b>0\}\,dt.
\end{equation}
Fix $b=x\sqrt u$ with $|x|\le D_n$ and apply the same change of variables $t=u+z\sqrt{2u}$ to obtain
\[
f_u'(x\sqrt u)
=-4\sqrt{2u}\int_{z>z_{u,x}} g_u(u+z\sqrt{2u})\,
g_u'\big(u+(z-\sqrt2x)\sqrt{2u}\big)\,dz.
\]
Truncate to $|z|\le M$ exactly as before: using \eqref{eq:gu-sup} and 
$\sup_{t>0}|g_u'(t)|\le C/u$ from Lemma \ref{lem:chisq_sup_bounds},
together with the chi-square tail bound, the contribution of $|z|>M$ is
$O\big(u^{-1}e^{-cM^2}\big)$ uniformly in $|x|\le D_n$. On $|z|\le M$, Lemma~\ref{lem:chisq-density-clt} gives
\[
g_u(u+z\sqrt{2u})=\frac{1}{\sqrt{2u}}\phi(z)\big(1+o(1)\big),
\qquad
g_u'\big(u+(z-\sqrt2x)\sqrt{2u}\big)=\frac{1}{2u}\phi'(z-\sqrt2x)\big(1+o(1)\big),
\]
uniformly for $|z|\le M$ and $|x|\le D_n$.
Therefore
\[
f_u'(x\sqrt u)
=
-\frac{2}{u}\int_{|z|\le M}\phi(z)\phi'(z-\sqrt2x)\,dz\cdot(1+o(1))
 + O\!\left(\frac{1}{u}e^{-cM^2}\right),
\]
uniformly in $|x|\le D_n$.
Letting $u\to\infty$ then $M\to\infty$, and using
\[
\int_{\R}\phi(z)\phi'(z-\sqrt2x)\,dz=-\frac12\phi'(x),
\]
(which follows by differentiating $\int \phi(z)\phi(z-\sqrt2x)\,dz=\frac{1}{\sqrt2}\phi(x)$ with respect to $x$),
we conclude
\[
f_u'(x\sqrt u)=\frac{1}{u}\phi'(x)\big(1+o(1)\big),
\]
uniformly for $|x|\le D_n$, proving \eqref{eq:Bu-clt-der}.
\end{proof}

\begin{proof}[Proof of Lemma \ref{lem:fu-away}]
Recall that
\(
B_u=\langle G_u,H_u\rangle=\sum_{i=1}^u G_iH_i,
\)
where $G_u,H_u\sim N(0,I_u)$ are independent. Conditioning on $G_u$, we have
\[B_u\mid G_u\sim N(0,\|G_u\|^2).\] Writing $V:=\|G_u\|^2\sim\chi^2_u$, it follows that
\begin{equation}\label{eq:fu-mixture-normal-away}
f_u(x)=\E[\varphi_V(x)]
=\E\!\left[\frac{1}{\sqrt{2\pi V}}e^{-x^2/(2V)}\right],
\end{equation}
where $\varphi_v$ denotes the $N(0,v)$ density.

Fix $x\neq 0$. For every $v>0$,
\[
\varphi_v(x)
=\frac{1}{\sqrt{2\pi v}}e^{-x^2/(2v)}
=\frac{1}{|x|\sqrt{\pi}}
\left(\frac{x^2}{2v}\right)^{1/2}e^{-x^2/(2v)}.
\]
Since the function $t\mapsto t^{1/2}e^{-t}$ is bounded on $[0,\infty)$, we get
\[
\varphi_v(x)\le \frac{C}{|x|}\qquad\text{for all }v>0.
\]
Taking expectation in \eqref{eq:fu-mixture-normal-away} gives
\[
f_u(x)\le \frac{C}{|x|},
\]
uniformly over $u\ge 1$.

Next, for $x\neq 0$ and $v>0$,
\[
\varphi_v'(x)=-\frac{x}{v}\varphi_v(x),
\]
so
\[
|\varphi_v'(x)|
=\frac{|x|}{\sqrt{2\pi}}v^{-3/2}e^{-x^2/(2v)}
=\frac{2}{\sqrt{\pi}\,x^2}
\left(\frac{x^2}{2v}\right)^{3/2}e^{-x^2/(2v)}.
\]
Since $t\mapsto t^{3/2}e^{-t}$ is also bounded on $[0,\infty)$, it follows that
\[
|\varphi_v'(x)|\le \frac{C}{x^2}\qquad\text{for all }v>0.
\]
Thus for each fixed $x\neq 0$, the random variable $\varphi_V'(x)$ is dominated by the integrable constant $C/x^2$, so we can differentiate under the expectation and h ence $f_u$ is differentiable on $\R\setminus\{0\}$ with
\[
f_u'(x)=\E[\varphi_V'(x)],
\qquad\text{therefore}\qquad
|f_u'(x)|\le \frac{C}{x^2}.
\]
\end{proof}

\subsubsection{Proof of Lemmas \ref{lem:gauss_product_integral}, \ref{lem:chisq_sup_bounds}}\label{sec:twomorelemmas}

\begin{proof}[Proof of Lemma \ref{lem:gauss_product_integral}]
Using $\phi(z)=(2\pi)^{-1/2}e^{-z^{2}/2}$ we write
\[
\int_{\R}\phi(z)\phi(z-\sqrt2 x)\,dz
=\frac{1}{2\pi}\int_{\R}\exp\!\Big(-\frac{z^{2}+(z-\sqrt2 x)^{2}}{2}\Big)\,dz.
\]
Expand and complete the square:
\[
-\frac{z^{2}+(z-\sqrt2 x)^{2}}{2}=-(z^{2}-\sqrt2 x\,z+x^{2})
=-\Big(z-\frac{x}{\sqrt2}\Big)^{2}-\frac{x^{2}}{2}.
\]
With the change of variables $u=z-\frac{x}{\sqrt2}$ we have
\[
\int_{\R}\exp\!\Big(-\Big(z-\frac{x}{\sqrt2}\Big)^{2}\Big)\,dz
=\int_{\R}e^{-u^{2}}\,du=\sqrt{\pi}.
\]
Thus
\[
\int_{\R}\phi(z)\phi(z-\sqrt2 x)\,dz
=\frac{e^{-x^{2}/2}}{2\pi}\cdot \sqrt{\pi}
=\frac{1}{2\sqrt{\pi}}e^{-x^{2}/2}
=\frac{1}{\sqrt2}\cdot \frac{1}{\sqrt{2\pi}}e^{-x^{2}/2}
=\frac{1}{\sqrt2}\phi(x),
\]
as claimed.
\end{proof}

\begin{proof}[Proof of Lemma \ref{lem:chisq_sup_bounds}]
First,
\[
\frac{d}{dt}\log g_u(t)=\frac{u/2-1}{t}-\frac12,
\]
so this derivative is positive for $t<u-2$ and negative otherwise, implying: 
\[
\sup_{t>0}g_u(t)=g_u(u-2)=\frac{(u-2)^{u/2-1}e^{-(u-2)/2}}{2^{u/2}\Gamma(u/2)}.
\]
Using the Stirling's approximation lower bound $\Gamma(u/2)\ge \sqrt{2\pi}\,(u/2)^{u/2-1/2}e^{-u/2}$ gives
\[
\sup_{t>0}g_u(t)\le \frac{e}{2\sqrt{\pi}}\cdot \frac{1}{\sqrt{u}}\Big(1-\frac{2}{u}\Big)^{u/2-1}
\le \frac{e}{2\sqrt{\pi}}\cdot \frac{1}{\sqrt{u}}.
\]
Next,
\[
g_u'(t)=g_u(t)\Big(\frac{u/2-1}{t}-\frac12\Big)=g_u(t)\cdot \frac{u-2-t}{2t},
\]
hence
\[
|g_u'(t)|=g_u(t)\cdot \frac{|u-2-t|}{2t}.
\]
For \(t\neq u-2\), maximizing \(|g_u'(t)|\) is equivalent to maximizing
\(
h(t):=t^{u/2-2}e^{-t/2}|u-2-t|,\qquad t>0.
\) On each of the intervals \((0,u-2)\) and \((u-2,\infty)\), the function \(h\) is \(C^1\), and its critical points satisfy
\[
(t-(u-2))^2=2(u-2).
\]
Moreover, \(h(t)\to 0\) as \(t\to\infty\), and for \(u>2\) also \(h(t)\to 0\) as \(t\downarrow 0\). Since \(h\) is continuous and nonnegative on \((0,\infty)\), it follows that its global maximum is attained either at a boundary or at a critical point. In the boundary the limits are \(0\), so the maximum is attained at one of the critical points
\[
t=u-2\pm \sqrt{2(u-2)}.
\]
At either such point, \(|u-2-t|=\sqrt{2(u-2)}\), and for large \(u\) we have \(t\ge (u-2)/2\). Hence
\[
\sup_{t>0}|g_u'(t)|
\le \sup_{t>0} g_u(t)\cdot \frac{\sqrt{2(u-2)}}{2\cdot (u-2)/2}
\le \sup_{t>0} g_u(t)\cdot \frac{2}{\sqrt{u}}.
\]
Combining with the first bound $\sup_{t>0}g_u(t)\le C/\sqrt{u}$ yields
$\sup_{t>0}|g_u'(t)|\le C/u$ after adjusting the universal constant $C$.
\end{proof}

\section*{Acknowledgments}

The authors are thankful to Hugo Koubbi for interesting conversations during the early stages of this project.

\newpage
% Bibliography
% \clearpage
\bibliographystyle{alpha}
\bibliography{ref}

\appendix

\newpage

\section{A failure of the quenched FP potential to predict Low-Degree hardness}\label{sec:quenched}
It is natural to wonder whether the monotonicity of the original quenched FP potential, $\mathcal{F}_{\lambda}$ \eqref{eq:quenched}, is equivalent to the low-degree MMSE lower bounds for a family of GAMs. In other words:

\begin{quote}\centering \textit{Can one deduce results for the low-degree MMSE\\ directly from the monotonicity behavior of the \textbf{quenched} FP potential?}\end{quote}While our results establish this connection for the annealed FP potential, traditional physics intuition suggests that the same should hold for the quenched potential, as the annealed potential is typically defined merely as a tractable proxy for the quenched landscape. It remains possible that many of the results presented in this work transfer to the quenched setting—for instance, by directly establishing an agreement between the monotonicity of the annealed and quenched potentials. However, the purpose of this appendix is to present a thought-provoking counterexample, for which we show that while the monotonicity of the annealed FP potential aligns with the low-degree MMSE lower bounds, the monotonicity of the quenched FP potential does not. This discrepancy between the potentials is notable, and we leave the question of why the annealed potential can "outperform" the quenched potential in predicting algorithmic hardness as a compelling topic for future work.

We explicitly construct and analyze this counterexample. In this setting, the annealed FP is increasing "around $0$" (i.e., annealed physics-hard) precisely when the low-degree MMSE is trivial and it begins to decrease around $0$ (i.e., annealed physics-easy) exactly when the low-degree MMSE improves upon the trivial MSE. In contrast, the quenched FP remains non-decreasing (i.e., not quenched physics-easy) around $0$ well into the regime where the low-degree MMSE already strictly outperforms the trivial MSE.

\paragraph{The Model-Counterexample} The counterexample holds for the following \emph{``truncated" Rademacher" 3-tensor sparse PCA model}. For $k=n^{\beta+o(1)}$ where $\beta \in (0,1/2)$, we choose $X=\mathrm{vec}(v^{\otimes 3}),$ for $v\in \R^n$ generated as follows. Let $u \in \mathbb{R}^n$ with i.i.d. entries $\mathrm{Rad}(k/n),$ meaning that for all $i=1,\ldots,N$, $u_i=1$  with probability $k/(2n)$, $u_i=-1$ with probability $k/(2n)$ and $u_i=0$, otherwise. If $\|u\|_0  \in [k/2,2k]$, we set $v=u.$ Otherwise we set $v=1_{[k]}$ (the indicator of the first $k$-elements).

\begin{remark}[Explaining the ``truncation", and a roadmap] The model is almost identical to the Rademacher sparse tensor PCA model discussed in Section \ref{sec:statements_sptpca}, with the only difference being that we truncate the signal \(v\) to \(1_{[k]}\) whenever \(\|v\|_0 \notin [k/2,2k]\). A standard application of Bernstein's inequality shows that the event \(\|v\|_0 \in [k/2,2k]\) occurs with probability \(1 - e^{-\Theta(k)}\). For this reason, the truncated and original models share the same ``algorithmic'' thresholds (see Section \ref{sec:count_alg}), the same ``low-degree'' thresholds (see Section \ref{sec:count_equiv}), and the same behavior of the annealed Franz-Parisi potential (see Section \ref{sec:count_equiv}). In particular, combining our equivalence theorems in Section \ref{sec:statements_sptpca}, the monotonicity of the annealed FP potential remains equivalent to the low-degree MMSE bounds for the truncated model as well.

The truncation is introduced purely for technical convenience, as it simplifies the analysis of now the quenched Franz-Parisi potential. To complete the counterexample, we show that the monotonicity of the quenched FP potential is not in agreement with the low-degree MMSE lower bounds (see Section \ref{sec:quenched}). Specifically, the quenched FP potential remains non-monotonic over a large portion of the low-degree ``easy'' regime.

\end{remark}

\subsection{The algorithmic threshold}\label{sec:count_alg}
In this section, we prove that for $\lambda=\widetilde{\Theta}(1)$ one can achieve exact recovery with high probability (and hence  MMSE that is of smaller order to the trivial MMSE) with a (simple) diagonal thresholding polynomial-time method (see Algorithm \ref{alg:DT}). In particular, the ``truncation" in the prior doesn't affect the success of diagonal thresholding for the sparse Rademacher 3-tensor PCA model.

\begin{algorithm}[t]
\caption{Diagonal Thresholding for support recovery}
\label{alg:DT}
\begin{algorithmic}[1]
\Require Tensor $Y\in\R^{n^3}$, threshold $\tau>0$
\State Compute diagonal terms $d_i \gets Y_{iii}$ for $i=1,\dots,n$
\State Output $\widehat v \gets \mathrm{sign}(d_i)1\{ |d_i|\ge \tau\}$
\end{algorithmic}
\end{algorithm}

Our main result for this subsection is as follows.

\begin{lemma}[Diagonal thresholding recovers $v$ (including signs)]
\label{le:diag_threshold_single_lambda}
Let $Y$ be generated according to the truncated sparse 3-tensor PCA model. For each $i\in[n]$, let $d_i:=Y_{i i i}$. 
Define the diagonal-thresholding estimator $\widehat v\in\{-1,0,1\}^n$ by
\begin{equation}\label{eq:vhat_def}
\widehat v_i := \mathrm{sign}(d_i) 1\{|d_i|\ge \tau\},
\qquad i=1,\dots,n.
\end{equation}
 Set $\tau:=\sqrt{6\log n}$ and assume $\lambda \ge 2\tau$. Then
\begin{equation}\label{eq:success_prob_vhat}
\mathbb{P}(\widehat v \neq v) \le 2n^{-2} + 4k\,n^{-3} + 4k\,n^{-27}.
\end{equation}
In particular, if $k=n^{\beta+o(1)}$ for any $\beta\in(0,1/2)$, then
$\mathbb{P}(\widehat v\neq v)=o(1)$.
\end{lemma}

The proof of this Lemma is deferred to Section \ref{lem:MMSE_decomposition}.

\subsection{The equivalence between the annealead FP potential and the low-degree MMSE}\label{sec:count_equiv}

Our next claim is that the monotonicity of the annealed FP potential is equivalent to low-degree MMSE lower bounds for the \emph{truncated} sparse 3-tensor PCA model. We remind the reader that for the original (``untruncated") sparse 3-tensor PCA model from Section \ref{sec:statements_sptpca} the equivalence holds are already discussed in Theorem \ref{thm:sptpca}); specifically it follows from Theorems \ref{thm:corr-lower-bound_discrete} and \ref{thm:mainthm_disc} as the original model satisfies the Assumption \ref{assump:mainassumption_discrete}. 

Due to the technical complications of introducing the truncation to the prior, we prove the equivalence for the truncated sparse 3-tensor PCA model not by directly verifying Assumption \ref{assump:mainassumption_discrete} but proving that both (a) the monotonicity of the annealed FP potential and (b) the low-degree MMSE are up to $1+o(1)$-factors identical between the original and the truncated sparse 3-tensor PCA model. Then as the equivalence holds for the original model, the equivalence transfers to the truncated models as well.

\subsubsection{Equivalence of low-degree MMSE between the original and truncated models}\label{sec:mmse_trunc_untrunc}

We start with the almost equivalence between the low-degree MMSE of the truncated 3-sparse tensor PCA model and the original model. 
Let %$P_0$ denote the standard Gaussian law on $\R^N$ and 
    $ P$ be the untruncated prior on signals
$\mathrm{vec}(u^{\otimes 3})\in\R^N$ for the 3-sparse tensor case, $N=n^3$, and $\widetilde P$ denote the truncated prior.

\begin{proposition}\label{prop:MMSE-trunc}
   
%Let $P$ and $\widetilde P$ be the distributions defined above. 
%Let $\delta:=\P_{\widetilde P}(E(\widetilde X)^c)$.
There exists an absolute constant $C>0$, such that for any $D$ satisfying $D=O((\log n)^h)$ with some constant $0\le h<2$ and $\lambda = O(1)$,
\[
\big|\mathrm{MMSE}^{\le D}_{\widetilde P}(\lambda)-\mathrm{MMSE}^{\le D}_{ P}(\lambda)\big|
 \le 
e^{-\Theta(k)}.
\]
\end{proposition}
\begin{proof}
We proceed as follows to complete the proof.
\paragraph{Step 1: Coupling and reduction to the bad event.}
We employ the obvious coupling between the two priors. First, we draw $\widetilde X =\mathrm{vec}(u^{\otimes 3})\sim\widetilde P$.
Denote
$E_1(\widetilde X)=\{\|u\|_0\in[k/2,2k]\}$ as the measurable event depending on $\widetilde X$.
Fix the deterministic vector $x_0=\mathrm{vec}(1_{[k]}^{\otimes 3})\in\R^N$. Then, for the truncated signal
\[
X:=\widetilde X \mathbf 1_{E_1(\widetilde X)}+x_0 \mathbf 1_{E_1(\widetilde X)^c},
\]
we have $X\sim P$.
Let $Z\sim \mathcal N(0,I_N)$ be independent, and define
\[
\widetilde Y:=\sqrt{\lambda} \widetilde X+Z,
\qquad
Y:=\sqrt{\lambda} X+Z.
\]

By construction, $(X,Y)=(\widetilde X,\widetilde Y)$ on $E_1(\widetilde X)$. Hence for any
$f\in(\R[\cdot]_{\le D})^N$,
\begin{align}\label{eq:coupling-mse}
\Big|\E\|f(Y)-X\|^2-\E\|f(\widetilde Y)-\widetilde X\|^2\Big|
&=
\Big|\E\Big[\big(\|f(Y)-X\|^2-\|f(\widetilde Y)-\widetilde X\|^2\big)
\mathbf 1_{E_1(\widetilde X)^c}\Big]\Big|\nonumber\\
&\le
\E\Big[\big(\|f(Y)-X\|^2+\|f(\widetilde Y)-\widetilde X\|^2\big)
\mathbf 1_{E_1(\widetilde X)^c}\Big]\nonumber\\
&\le
2 \E\Big[\big(\|f(Y)\|^2+\|f(\widetilde Y)\|^2+\|X\|^2+\|\widetilde X\|^2\big)\mathbf 1_{E_1(\widetilde X)^c}\Big].
\end{align}

We now proceed with bounding the four terms:
\[
\E[\|f(Y)\|^2 \mathbf 1_{E_1^c}],\quad
\E[\|f(\widetilde Y)\|^2 \mathbf 1_{E_1^c}],\quad
\E[\|X\|^2 \mathbf 1_{E_1^c}],\quad
\E[\|\widetilde X\|^2 \mathbf 1_{E_1^c}].
\]

\paragraph{Step 2: Restricting the class of estimators $f$.}
First, we record the following elementary fact. For any random vectors $U,V$ and estimator $f(V)$, if
$\E\|f(V)-U\|^2\le \E\|U\|^2$, then by triangle inequality,
\[
\|f(V)\| \le \|f(V)-U\|+\|U\|,
\]
and hence
\[
\|f(V)\|^2 \le 2\|f(V)-U\|^2+2\|U\|^2.
\]
Taking expectations gives
\[
\E\|f(V)\|^2 \le 4 \E\|U\|^2.
\]

When bounding the low-degree MMSE, since the optimal estimator is always better than the zero estimator, we may restrict to estimators satisfying
\[
\E\|f(Y)-X\|^2\le \E\|X\|^2.
\]
Thus,
\[
\E\|f(Y)\|^2 \le 4\E\|X\|^2 \le C k^3.
\]

Similarly, under $\widetilde P$, we also have \(
\E\|f(\widetilde Y)\|^2 \le C k^3.\)

\paragraph{Step 3: Probability of the bad event.}
Let $\delta:=\P_{\widetilde P}(E_1(\widetilde X)^c)$. Since $\|u\|_0\sim\mathrm{Bin}(n,k/n)$ is subgaussian with mean $k$ and variance at most $k$, Bernstein's inequality implies that there exists an absolute constant $C>0$, such that
$\delta \le e^{-2ck}$.

\paragraph{Step 4: Bounds for the signal terms.}
On $E_1^c$, we have $X=x_0$, hence
\begin{equation}\label{eq:tail-1}
\E[\|X\|^2\mathbf 1_{E_1^c}] = \delta k^3.
\end{equation}

Also, since $\|\widetilde X\|^2=\|u\|_0^3$ and $\E[\|\widetilde X\|^4]\le C^2 k^6$ (see e.g.~\cite[Prop 2.5.2]{vershynin-HDP}), by Cauchy--Schwarz inequality,
\begin{equation}\label{eq:tail-2}
\E[\|\widetilde X\|^2\mathbf 1_{E_1^c}]
\le \sqrt{\E[\|\widetilde X\|^4]\P(E_1^c)}
\le C k^3 e^{-ck}.
\end{equation}

\paragraph{Step 5: Bound for $\E[\|f(Y)\|^2\mathbf 1_{E_1^c}]$.}
On $E_1^c$, $Y=\sqrt\lambda x_0+Z$, hence
\begin{equation}\label{eq:tail-3}
\E[\|f(Y)\|^2\mathbf 1_{E_1^c}]
=
\delta\,\E\|f(Y)\|^2
\le C k^3 e^{-2ck}.
\end{equation}

\paragraph{Step 6: Bound for $\E[\|f(\widetilde Y)\|^2\mathbf 1_{E_1^c}]$.}
It remains to control the only nontrivial term
$\E[\|f(\widetilde Y)\|^2 \mathbf 1_{E_1^c}]$.

\subparagraph{Step 6a: Upper bound via Hermite expansion.}
By conditioning on $u$, we first argue that
\begin{align}\label{eq:tail-4-1}
\mathbb{E}[\|f(\widetilde Y)\|^2 \mathbf{1}_{E_1^c}]
= \mathbb{E}_u \left[ \mathbf{1}_{E_1(\widetilde X)^c} \mathbb{E}_Z[\|f(Z + \sqrt{\lambda}\widetilde X)\|^2] \right] 
\le
\mathbb{E}_Z[\|f(Z)\|^2]
\sum_{j=0}^D \binom{D}{j} \frac{1}{j!} \lambda^j
\mathbb{E}[\|u\|_0^{3j} \mathbf{1}_{E_1^c}].
\end{align}

We conduct the proof of this identity by expanding $f$ in the multivariate Hermite basis as
$f(z) = \sum_{|\alpha| \le D} \widehat{f}_\alpha H_\alpha(z)$, and as a consequence,
$\mathbb{E}_Z[\|f(Z)\|^2] = \sum_{|\alpha| \le D} \widehat{f}_\alpha^2 \alpha!$.

Let $\mu = \sqrt{\lambda}\widetilde X$. Applying Proposition \ref{prop:translation}, we have
\[
f(Z+\mu) = \sum_{|\alpha| \le D} \widehat{f}_\alpha \sum_{\gamma \le \alpha} \binom{\alpha}{\gamma} \mu^{\alpha - \gamma} H_{\gamma}(Z);
\]
and by Lemma \ref{lem:shifted_product},
\[
\mathbb{E}_Z[\|f(Z+\mu)\|^2] = \sum_{|\gamma| \le D} \gamma! \left\| \sum_{\beta} \widehat{f}_{\gamma+\beta} \binom{\gamma+\beta}{\beta} \mu^\beta \right\|^2.
\]

We apply the Cauchy--Schwarz inequality to the inner sum over $\beta$:
\begin{equation*}
    \left\| \sum_{\beta} \widehat{f}_{\gamma+\beta} \sqrt{(\gamma+\beta)!} \frac{\binom{\gamma+\beta}{\beta}}{\sqrt{(\gamma+\beta)!}} \mu^\beta \right\|^2 \le \left( \sum_{\beta} \|\widehat{f}_{\gamma+\beta}\|^2 (\gamma+\beta)! \right) \left( \sum_{\beta} \frac{\binom{\gamma+\beta}{\beta}^2}{(\gamma+\beta)!} \mu^{2\beta} \right).
\end{equation*}
Multiplying both sides by $\gamma!$ and using the factorial identity $\gamma! \frac{\binom{\gamma+\beta}{\beta}^2}{(\gamma+\beta)!} = \binom{\gamma+\beta}{\beta} \frac{1}{\beta!}$ yields
\begin{equation*}
    \mathbb{E}_Z[\|f(Z+\mu)\|^2] \le \sum_{|\gamma| \le D} \left( \sum_{\beta} \|\widehat{f}_{\gamma+\beta}\|^2 (\gamma+\beta)! \right) \left( \sum_{\beta} \binom{\gamma+\beta}{\beta} \frac{\mu^{2\beta}}{\beta!} \right),
\end{equation*}
changing variables back to $\alpha = \gamma + \beta$, summing over $\gamma$ and $\beta$ with $|\gamma+\beta| \le D$ is equivalent to summing over $|\alpha| \le D$ and $\beta \le \alpha$:
\begin{equation*}
    \mathbb{E}_Z[\|f(Z+\mu)\|^2] \le \sum_{|\alpha| \le D} \|\widehat{f}_\alpha\|^2 \alpha! \sum_{\beta \le \alpha} \binom{\alpha}{\beta} \frac{\mu^{2\beta}}{\beta!}.
\end{equation*}
Combinatorially, $\binom{\alpha}{\beta} = \prod_{i=1}^N \binom{\alpha_i}{\beta_i} \le \binom{|\alpha|}{|\beta|} \le \binom{D}{|\beta|}$. Let $j = |\beta|$. Grouping the inner sum by the degree $j$, we get
\[
\sum_{\beta \le \alpha} \binom{\alpha}{\beta} \frac{\mu^{2\beta}}{\beta!} \le \sum_{j=0}^{|\alpha|} \binom{D}{j} \sum_{|\beta|=j} \frac{\mu^{2\beta}}{\beta!}
\]
By the multinomial theorem, $\sum_{|\beta|=j} \frac{j!}{\beta!} \mu^{2\beta} = (\sum_{i=1}^N \mu_i^2)^j = \|\mu\|^{2j}$. Therefore, the inner sum is uniformly bounded by $\sum_{j=0}^D \binom{D}{j} \frac{1}{j!} \|\mu\|^{2j}$ for any $\alpha$. Factoring this out, we extract the baseline norm:
\begin{equation}
    \mathbb{E}_Z[\|f(Z+\mu)\|^2] \le \left( \sum_{|\alpha| \le D} \|\widehat{f}_\alpha\|^2 \alpha! \right) \sum_{j=0}^D \binom{D}{j} \frac{1}{j!} \|\mu\|^{2j} = \mathbb{E}_Z[\|f(Z)\|^2] \sum_{j=0}^D C_j(D) \|\mu\|^{2j}
\end{equation}
where $C_j(D) = \binom{D}{j} \frac{1}{j!}$, which proves \eqref{eq:tail-4-1}.

Recall $\|\mu\|^2 = \lambda \|\widetilde X\|^2 = \lambda \|u\|_0^3$, where $\|u\|_0 \sim \mathrm{Binomial}(n, k/n)$. Similarly according to \cite[Prop 2.5.2]{vershynin-HDP},  for some absolute constant $C>0$ that for every $j\leq D$, $\E [\|u\|_0^{6j}] \leq (Ck)^{6j}$. Then by using Cauchy--Schwarz inequality,
\[
\mathbb{E}[\|u\|_0^{3j} \mathbf{1}_{E_1(\widetilde X)^c}] \le \sqrt{\E [\|u\|_0^{6j}] \P (E_1(\widetilde X)^c)} = \sqrt{\E [\|u\|_0^{6j}]} \sqrt{\delta} \le (Ck)^{3j} e^{-ck},
\]
\eqref{eq:tail-4-1} leads to 
\[
\mathbb{E}[\|f(\widetilde Y)\|^2 \mathbf{1}_{E_1^c}]
\le
\mathbb{E}_Z[\|f(Z)\|^2]\cdot \frac{1}{(D-1)!} (C\lambda k)^{3D} e^{-ck}.
\]

\subparagraph{Step 6b: Lower bound relating $\E\|f(\widetilde Y)\|^2$ and $\E_Z[\|f(Z)\|^2]$.}
To conclude, we need to upper bound $\mathbb{E}_Z[\|f(Z)\|^2]$ in terms of $\mathbb{E}\|f(\widetilde Y)\|^2$ and use Step 1.

For each multi-index $\gamma$, let $A_\gamma(\mu) = \sum_{\beta > 0} \widehat{f}_{\gamma+\beta} \binom{\gamma+\beta}{\beta} \mu^\beta$ represent the strictly higher-degree shift components. %We analyze the inner expectation for a fixed $\gamma$:
Then by Jensen's inequality,
\begin{equation} \label{eq:gamma_expectation}
    \mathbb{E}_\mu \left[ \|\widehat{f}_\gamma + A_\gamma(\mu)\|^2 \right] \ge \|\widehat{f}_\gamma + \mathbb{E}[A_\gamma(\mu)]\|^2. %+ \mathrm{Var}(A_\gamma(\mu)).
\end{equation}

Using Cauchy--Schwarz inequality, for any $\eta >0$, $(a + b)^2 \ge (1 + \eta)^{-1}a^2 - \eta^{-1}b^2$. Therefore, we obtain
\[
\E [\|f(Z + \mu)\|^2] \ge (1 + \eta)^{-1} \sum_{|\gamma| \le D} \gamma! \|\widehat f_\gamma\|^2 - \eta^{-1} \sum_{|\gamma| \le D} \gamma! \|\mathbb{E}[A_\gamma(\mu)]\|^2.
\]
Now define the moments $m_\beta = \E[\mu^\beta]$. By using Cauchy--Schwarz inequality in the $\beta$-sum, we have
\begin{align*}
    \|\mathbb{E}[A_\gamma(\mu)]\|^2 = \left\|\sum_{\beta > 0} \widehat{f}_{\gamma+\beta} \binom{\gamma+\beta}{\beta} m_\beta \right\|^2\le \left(\sum_{\beta>0}(\gamma+\beta)!\|\widehat f_{\gamma+\beta}\|^2\right) \left(\sum_{\beta>0}\frac{\binom{\gamma+\beta}{\beta}^2m_\beta^2}{(\gamma+\beta)!}\right).
\end{align*}
Multiplying by $\gamma!$ and using
$$
\gamma!\frac{\binom{\gamma+\beta}{\beta}^2}{(\gamma+\beta)!} = \binom{\gamma+\beta}{\beta}\frac1{\beta!},
$$
we have
\begin{align*}
\sum_{|\gamma| \le D} \gamma! \|\mathbb{E}[A_\gamma(\mu)]\|^2 &\le \left(\sum_{\beta>0}(\gamma+\beta)!\|\widehat f_{\gamma+\beta}\|^2\right) \left(\sum_{\beta>0}\binom{\gamma+\beta}{\beta}\frac{m_\beta^2}{\beta!}\right) \le \Bigg( \sup_{|\alpha|\le D}\sum_{0<\beta\le \alpha}\binom{\alpha}{\beta}\frac{m_\beta^2}{\beta!} \Bigg) \sum_{|\alpha|\le D}\alpha!\|\widehat f_\alpha\|^2 \\
&:= S_D \sum_{|\alpha| \le D} \alpha! \|\widehat f_\alpha\|^2 = S_D \cdot \mathbb{E}_Z[\|f(Z)\|^2],
\end{align*}
where we denote
\[
S_D = \sup_{|\alpha| \le D} \sum_{0 < \beta \le \alpha} \binom{\alpha}{\beta} \frac{m_\beta^2}{\beta!}.
\]
Combining this with the previous inequality, we get
\begin{equation}\label{eq:lower-bound-L2norm}
\E [\|f(Z + \mu)\|^2] \ge ((1 + \eta)^{-1} - \eta^{-1} S_D) \mathbb{E}_Z[\|f(Z)\|^2].
\end{equation}

Write $p = k/n$. The coordinates of $\mu$ are
$\mu_{abc} = \sqrt{\lambda} u_a u_b u_c$, for every $a,b,c\in [n]$.
Let $e_i(\beta)$ be the exponent of $u_i$ in $\mu^\beta$, and define
\[
r(\beta) = |\{i : e_i(\beta) > 0\}|.
\]
Since
\[
\E[u_i^r] = \begin{cases} 
1, & r = 0, \\
0, & r \text{ odd}, \\
p, & r \ge 2 \text{ even},
\end{cases}
\]
we obtain
\[
m_\beta = \lambda^{|\beta|/2} p^{r(\beta)} \mathbf{1}_{\{ e_i(\beta) \text{ all even} \}}.
\]
We split the sum defining $S_D$ according to whether $r(\beta)=1$ or
$r(\beta)\ge 2$:
\[
S_D\le S_D^{(1)}+S_D^{(\ge 2)},
\]
where
\[
S_D^{(1)}
:=
\sup_{|\alpha|\le D}
\sum_{\substack{0<\beta\le\alpha\\ r(\beta)=1}}
\binom{\alpha}{\beta}\frac{m_\beta^2}{\beta!},
\qquad
S_D^{(\ge 2)}
:=
\sup_{|\alpha|\le D}
\sum_{\substack{0<\beta\le\alpha\\ r(\beta)\ge 2}}
\binom{\alpha}{\beta}\frac{m_\beta^2}{\beta!}.
\]
Suppose $r(\beta)=1$ and $m_\beta\neq 0$.
Then there exists a unique $i\in[n]$ such that $e_i(\beta)>0$ and
$e_j(\beta)=0$ for all $j\neq i$.
This means that every tensor coordinate $(a,b,c)$ in the support of $\beta$
must satisfy $a=b=c=i$.
Therefore $\beta$ must be supported on a single pure diagonal coordinate
$(i,i,i)$, that is,
$\beta=b e_{(i,i,i)}$
for some integer $b\ge 1$, where $e_{(i,i,i)}$ denotes the corresponding
coordinate basis vector.
Moreover, for $m_\beta$ to be nonzero, all exponents $e_j(\beta)$ must be even;
here $e_i(\beta)=3b$, so necessarily $b$ is even.
Hence the nonzero terms with $r(\beta)=1$ are exactly the multi-indices
$\beta=2q e_{(i,i,i)}$, $q\ge 1$.
For such $\beta$,
$m_\beta^2
=
\lambda^{2q}p^2$.
Fix $\alpha$ with $|\alpha|\le D$.
Then
\begin{align*}
\sum_{\substack{0<\beta\le\alpha\\ r(\beta)=1}}
\binom{\alpha}{\beta}\frac{m_\beta^2}{\beta!}
&=
p^2
\sum_{i=1}^n
\sum_{\substack{2q\le \alpha_{(i,i,i)}\\ q\ge 1}}
\binom{\alpha_{(i,i,i)}}{2q}\frac{\lambda^{2q}}{(2q)!}.
\end{align*}
For each fixed $q\ge 1$, using
\[
\sum_{i=1}^n \binom{\alpha_{(i,i,i)}}{2q}
\le
\binom{\sum_{i=1}^n \alpha_{(i,i,i)}}{2q}
\le
\binom{|\alpha|}{2q}
\le
\binom{D}{2q},
\]
we get
\begin{align*}
\sum_{\substack{0<\beta\le\alpha\\ r(\beta)=1}}
\binom{\alpha}{\beta}\frac{m_\beta^2}{\beta!}
&\le
p^2
\sum_{q=1}^{\lfloor D/2\rfloor}
\binom{D}{2q}\frac{\lambda^{2q}}{(2q)!}.
\end{align*}
Taking the supremum over $|\alpha|\le D$ gives
\[
S_D^{(1)}
\le
p^2
\sum_{q=1}^{\lfloor D/2\rfloor}
\binom{D}{2q}\frac{\lambda^{2q}}{(2q)!}\le p^2 \sum_{q=1}^{\lfloor D/2\rfloor} \frac{(D\lambda)^{2q}}{((2q)!)^2} \le p^2 I_0(2\sqrt{D\lambda}),
\]
where $I_0(z)$ is the modified Bessel function of the first kind defined in Appendix \ref{app:bessel}.
Use the integral representation of the modified Bessel function, $I_0(z) = \frac{1}{\pi} \int_0^\pi \exp(z \cos \theta) d\theta$. Since $\cos \theta \leq 1$ over the entire integration interval, the integrand is bounded by $\exp(z)$. Therefore, $I_0(z) \leq \exp(z)$, and $S_D^{(1)} \le p^2 \exp(2\sqrt{D\lambda})$.

Since $p = k/n$, $k=O(\sqrt{n})$, $D=O((\log n)^h)$ with $h<2$, $\lambda = O(1)$, we can see $\exp(2\sqrt{D\lambda}) = n^{o(1)}$, so $S_D^{(1)}=o(1)$. 

If $r(\beta)\ge 2$ and $m_\beta\neq 0$, then
$m_\beta^2
=
\lambda^{|\beta|}p^{2r(\beta)}
\le
\lambda^{|\beta|}p^4$.
Hence for every $\alpha$ with $|\alpha|\le D$,
\begin{align*}
S_D^{(2)}=\sum_{\substack{0<\beta\le\alpha\\ r(\beta)\ge 2}}
\binom{\alpha}{\beta}\frac{m_\beta^2}{\beta!}
&\le
p^4
\sum_{0<\beta\le\alpha}
\binom{\alpha}{\beta}\frac{\lambda^{|\beta|}}{\beta!} = \widetilde O(p^2) = o(1).
\end{align*}
% This sum factorizes by the definition of multi-index combinatorial number:
% \begin{align*}
% \sum_{0\le\beta\le\alpha}
% \binom{\alpha}{\beta}\frac{\lambda^{|\beta|}}{\beta!}
% &=
% \prod_t
% \left(
% \sum_{b=0}^{\alpha_t}
% \binom{\alpha_t}{b}\frac{\lambda^b}{b!}
% \right).
% \end{align*}
% Therefore
% \begin{align*}
% \sum_{0<\beta\le\alpha}
% \binom{\alpha}{\beta}\frac{\lambda^{|\beta|}}{\beta!}
% &=
% \prod_t
% \left(
% \sum_{b=0}^{\alpha_t}
% \binom{\alpha_t}{b}\frac{\lambda^b}{b!}
% \right)-1.
% \end{align*}
The above argument shows $S_D = o(1)$, so for sufficiently large $n$, $S_D\le 0.4$.
% Notice that the entries of $\mu = \sqrt{\lambda} \mathrm{vec}(u^{\otimes 3})$ are identically zero with probability exactly $1 - (k/n)^3$. Denote $p = (k/n)^3$, then $p = o(1)$. 
% Furthermore, because the distribution of $v_i$ is symmetric, all odd moments of the entries of $\mu$ are strictly zero. Thus, $\mathbb{E}[\mu^\beta] \neq 0$ only if all non-zero entries of $\beta$ are strictly even.
% Combining these two facts, the main contributing term in calculating the expectation of $A_\gamma(\mu)$ comes from the second order term on $\mu$.
% Therefore, $\mathbb{E}[A_\gamma(\mu)] = \Theta(\lambda p)$ and $\mathrm{Var}(A_\gamma(\mu)) = \Theta(\lambda^2 p)$. 
% This implies there exists a sequence $\epsilon_n = o(1)$ such that $\mathbb{E}[A_\gamma(\mu)]^2 \le \epsilon_n \mathrm{Var}(A_\gamma(\mu))$.
% Back into \eqref{eq:gamma_expectation},
% \[
% (\widehat{f}_\gamma + \mathbb{E}[A_\gamma(\mu)])^2 + \mathrm{Var}(A_\gamma(\mu)) \ge \frac{1}{2}\widehat{f}_\gamma^2 + (1 - \epsilon_n) \mathrm{Var}(A_\gamma(\mu))\ge \frac{1}{2}\widehat{f}_\gamma^2.
% \]
By choosing $\eta = 0.99$ in \eqref{eq:lower-bound-L2norm},
\[\mathbb{E}_P\|f(\widetilde Y)\|^2  \ge 0.1 \mathbb{E}_Z[\|f(Z)\|^2].\]

Plugging back to \eqref{eq:tail-4-1} and combining with Step 1, we conclude the following:
\begin{equation}\label{eq:tail-4}
\mathbb{E}[\|f(\widetilde Y)\|^2 \mathbf{1}_{E_1^c}]
\le
\frac{10}{(D-1)!} k^3 (C\lambda k)^{3D} e^{-ck}
=
e^{-\Theta(k)}.
\end{equation}

\paragraph{Step 7: Conclusion.}
According to the inequalities \eqref{eq:coupling-mse}, \eqref{eq:tail-1}, \eqref{eq:tail-2}, \eqref{eq:tail-3} and \eqref{eq:tail-4}, we have
\begin{align*}
    &\big|\mathrm{MMSE}^{\le D}_{P}(\lambda)-\mathrm{MMSE}^{\le D}_{\widetilde P}(\lambda)\big| = \bigg|\inf_{f \in \mathbb{R}^N[Y]_{\le D}}
\mathbb{E}\!\left[\| f(Y)-X\|^2\right] - \inf_{f \in \mathbb{R}^N[\widetilde Y]_{\le D}}
\mathbb{E}\!\left[\| f(\widetilde Y)-\widetilde X\|^2\right]\bigg|\\
&\qquad\le \sup_{f\in\mathbb{R}^N_{\le D} , \mathbb{E}\|f(\widetilde Y)\|^2 \le C k^3, \mathbb{E}\|f( Y)\|^2 \le C k^3 }\Big|\E\|f(Y)-X\|^2-\E\|f(\widetilde Y)-\widetilde X\|^2\Big|\\
&\qquad\le 2\sup_{f\in\mathbb{R}^N_{\le D} , \mathbb{E}\|f(\widetilde Y)\|^2 \le C k^3, \mathbb{E}\|f( Y)\|^2 \le C k^3 }\Big| \E\Big[\big(\|f(Y)\|^2+\|f(\widetilde Y)\|^2+\|X\|^2+\|\widetilde X\|^2\big)\mathbf 1_{E_1(\widetilde X)^c}\Big]\Big|\\
&\qquad\le C k^3 (C\lambda k)^{3D} e^{-ck} = e^{-\Theta(k)}.
\end{align*}

\end{proof}

\subsubsection{Equivalence of the monotonicity of annealed FP potential between original and truncated models}
Next, we prove that the truncation to the prior doesn't change the monotonicity of the annealed FP potential. 

Indeed, we prove this by showing that the log-PMF of the overlap random variable $\<X,X'\>$ for two i.i.d. draws $X,X'$ from the truncated prior is up to a multiplicative $o(1)$ factor, the same as the log-PMF of the original prior.

We first introduce some notation. For $v,v' \in \{-1,0,1\}^n$ let \(S:=\langle v,v'\rangle,\) and define
\[
E_2=E_2(v,v'):=\Big\{\|v\|_0\in[k/2,2k]\Big\}\cap \Big\{\|v'\|_0\in[k/2,2k]\Big\}.\]Notice that under this notation the original prior on $v,v'$ is $\widetilde{P}$, where $v,v'$ are i.i.d.~with i.i.d.~$\text{Rad}(k/n)$
coordinates, relates with the truncated $ P$ by,
\[
{\mathbb{P}}_P(\,\cdot\,):= {\mathbb{P}}(\,\cdot \mid E_2).
\]
We then denote the annealed FP potentials of the  truncated and original models by
\[\mathcal F^{\widetilde{P}}_{\rm ann, \lambda}(q),q \in \mathbb{R} \text{, and, } \mathcal F^{P}_{\rm ann, \lambda}(q), q \in \mathbb{R}.\]

We prove the following statement about them.

\begin{theorem}\label{thm:annealed_FP_trunc}
Fix $\beta\in(0,1/2)$, let $k=n^{\beta+o(1)}$. For any $n^{-C}\leq \lambda \leq n^C$, for some constant $C>0$ and any integer $q=q_n$ with $1\le q=o(k)$, let $s=q^{1/3}$ and define the sequence $a_n(s)=(s+2)^3-s^3, n\in \N.$ Then, 
\begin{equation}\label{eq:log_robust_fann}
\Delta_{a_n(s)}\mathcal F^{\widetilde P}_{\rm ann,\lambda}(q)
=
(1+o(1))\Delta_{a_n(s)}\mathcal F^{P}_{\rm ann,\lambda}(q),
\end{equation}
uniformly over all such $q$.

% \begin{equation}\label{eq:log_robust_fann}
% \mathcal F^{\widetilde P}_{\rm ann, \lambda}(q)=(1+o(1))\mathcal F^{ P}_{\rm ann, \lambda}(q).
% \end{equation}

\end{theorem}

\begin{proof}
First note that for both models
\begin{equation}\label{eq:fannform}
\mathcal{F}_{\mathrm{ann},\lambda}(q)
=-\lambda q  - \log \mathbb{P}(\langle X,X' \rangle = q).
\end{equation} Notice also that for both models if $X,X'$ correspond to vectors $v,v',$ $\<X,X'\>=S^3,$ where $S=\<v,v'\>$. To prove the Theorem, we will start by proving the following for $q=o(k),$
\begin{equation}\label{eq:log_robust_only}
\log {\mathbb{P}}_P(S^3=q)=(1+o(1))\,\log \mathbb{P}(S^3=q).
\end{equation}
Let $L:=\|v\|_0$. Under the i.i.d.\ $\text{Rad}(k/n)$ prior, $L\sim \mathrm{Bin}(n,k/n)$ and $\E L=k$,
So by Chernoff's inequality, there exists $c_0>0$ such that
\[
\mathbb{P}\big(L\notin[k/2,2k]\big)\le 2e^{-c_0 k}.
\]
Since $v,v'$ are independent,
\begin{equation}\label{eq:E_tail_only}
\mathbb{P}(E_2^c) \leq 2\mathbb{P}\big(L\notin[k/2,2k]\big)\le 4e^{-c_0 k}=\exp\big(-\Omega(k)\big)=\exp\big(-\Omega(n^\beta)\big).
\end{equation}
Write $X_i:=v_i v'_i\in\{-1,0,1\}, i=1,2,\ldots,n$. Then $S=\sum_{i=1}^n X_i$ and
\[
\mathbb{P}(X_i=1)=\mathbb{P}(X_i=-1)=:p,\qquad \mathbb{P}(X_i=0)=1-2p,
\]
with $np=\Theta(k^2/n)=\Theta(n^{2\beta-1})\to 0$ since $\beta<1/2$.

Let $s=q^{1/3}$. Consider the event that exactly $s$ coordinates satisfy $X_i=1$
and the remaining $n-s$ satisfy $X_i=0$. This event implies $S=s$, hence
\[
\mathbb{P}(S=s)\ \ge\ \binom{n}{s}p^s(1-2p)^{n-s}.
\]
Since $np\to 0$, for all large $n$ we have $(1-2p)^{n-s}\ge 1/2$. Also,
\(
\binom{n}{s}\ge (n/s)^s.
\)
Therefore, for all large $n$,
\[
\mathbb{P}(S=s)
\ge
\frac12 \Big(\frac{n}{s}\Big)^s p^s
=
\frac12 \Big(\frac{np}{s}\Big)^s.
\]
Since $np=n^{-(1-2\beta)+o(1)}$, it follows that
\begin{equation}\label{eq:poly_lower_only}
\mathbb P(S=s)\ge \exp \big(-C(s\log n+s\log s)\big)
\end{equation}
for some constant $C<\infty$. By the assumption $q=o(k)$ it follows that  \(
s\log n+s\log s=o(k) \) and
we deduce that
\begin{equation}\label{eq:enough_lower}
\mathbb P(S=s)\ge \exp(-o(k)).
\end{equation}
For the upper bound, if $S=s$ then necessarily at least $s$ coordinates satisfy $X_i=1$, so
\[
\mathbb P(S=s)\le 
\sum_{j=s}^n \binom{n}{j}p^j(1-p)^{n-j}
\le
\sum_{j=s}^\infty \frac{(np)^j}{j!}
=
O((np)^s),
\]where we used $\binom{n}{j}\leq n^j/j!, 1-p\leq 1$ and that $np\to 0.$   Hence,
\begin{equation}\label{eq:log_diverges}
\log \mathbb P(S=s_n)\to -\infty.
\end{equation}
% Therefore
% \begin{equation}\label{eq:poly_lower_only}
% \mathbb P(S=s)=\Theta((np)^s).
% \end{equation}
% \textcolor{red}{IZ: Ok, first this is not implied, only the inequality is implied. Second set $r=3$. Third, many typos below, e.g., $E$ is not defined, $E_2$ is defined. Fourth, why sticking to $s$ constant is enough for the equivalence (its not, at least logarithmic is needed...)}

Now let $A_n:=\{S^3=s^3\}=\{S=s\}$. Then
\[
{\mathbb{P}}_P(A_n)
=\frac{\mathbb{P}(A_n\cap E_2)}{\mathbb{P}(E_2)}
=\mathbb{P}(A_n)\cdot \frac{1-\mathbb{P}(A_n\cap E_2^c)/\mathbb{P}(A_n)}{1-\mathbb{P}(E_2^c)}.
\]
Using \eqref{eq:E_tail_only} and \eqref{eq:poly_lower_only},
\[
0\le \frac{\mathbb{P}(A_n\cap E_2^c)}{\mathbb{P}(A_n)}
\le \frac{\mathbb{P}(E_2^c)}{\mathbb{P}(A_n)}
\le \exp(-\Omega(k)+o(k))=n^{-\omega(1)},
\]
and also $\mathbb{P}(E_2)=1-\mathbb{P}(E_2^c)=1-o(1)$. Hence
\begin{equation}\label{eq:ratio_only}
{\mathbb{P}}_P(A_n)=\mathbb{P}(A_n)\,(1+n^{-\omega(1)}).
\end{equation}
Taking logarithms gives, using $|\log (x+1)-x|\leq x^2/2$ for $x\to 0$,
\begin{equation}\label{eq:loq_trunc_untrun}
\log {\mathbb{P}}_P(A_n)=\log \mathbb{P}(A_n)+n^{-\omega(1)}.
\end{equation}
Finally, by \eqref{eq:log_diverges},
\begin{equation}\label{eq:logp_multip_error}
\log {\mathbb{P}}_P(A_n)=(1+n^{-\omega(1)})\,\log \mathbb{P}(A_n),
\end{equation}
which is exactly \eqref{eq:log_robust_only}. 
% From there, using \eqref{eq:fannform} Equation \eqref{eq:log_robust_fann} is immediate and the proof is complete. 
To obtain the discrete-derivative statement, apply Equation \eqref{eq:logp_multip_error}  both at $s$ and at $s+2$. Since for any GAM,
\[
\mathcal F_{\mathrm{ann},\lambda}(s^3)=-\lambda s^3-\log\PP(S=s),
\]
we obtain, using that $s\geq 1$,
\begin{align}
\Delta_{a_n(s)}\mathcal F^{\widetilde P}_{\rm ann,\lambda}(s^3)
&=
\frac{\mathcal F^{\widetilde P}_{\rm ann,\lambda}((s+2)^3)-\mathcal F^{\widetilde P}_{\rm ann,\lambda}(s^3)}{(s+2)^3-s^3}\\
&=
-\lambda-
\frac{\log\PP_{\widetilde P}(S=s+2)-\log\PP_{\widetilde P}(S=s)}{(s+2)^3-s^3}\\
&=
-\lambda-(1+n^{-\omega(1)})
\frac{\log\PP(S=s+2)-\log\PP(S=s)}{(s+2)^3-s^3}\\
&=(1+o(1))\Delta_{a_n(s)}\mathcal F^{P}_{\rm ann,\lambda}(s^3)+\lambda n^{-\omega(1)}.\label{eq:fannmanipul}
\end{align}
Next, using that $n^{-C}\leq \lambda \leq n^C$ combined with the fact that
\[
-\Delta_{a_n}\log \P(\<X,X'\>=q)
\ge n^{C'}\]for some $C'>0$ and for all $q=o(k)$ by Lemma \ref{le:fppred_sparserad_small_k}, we get that 
\begin{equation}\label{eq:fann_nottosmall}
\big|\Delta_{a_n(s)}\mathcal F^{P}_{\rm ann,\lambda}\big|\ge n^{-(C+C')}.
\end{equation}
Equation \eqref{eq:fann_nottosmall} together with \eqref{eq:fannmanipul} give \eqref{eq:log_robust_fann} and the proof is complete.

\end{proof}

\subsubsection{Small low-degree MMSE in the ``easy" regime}

In this section, we prove that the low-degree MMSE for the truncated prior does become small when $\lambda=\widetilde{\omega}(1),$ a crucial step for the equivalence.

We define
\[
\mathrm{MMSE}^{\leq D}(v):=\inf_{\widehat v}\E\|v-\widehat v(Y)\|_2^2,
\qquad
\mathrm{MMSE}^{\leq D}(X):=\inf_{\widehat X}\E\|X-\widehat X(Y)\|_F^2.
\]

\begin{lemma}[Reducing tensor estimation to vector estimation]\label{lem:reduction}
There exists a positive constant $C>0$ such that for all $D>0,$ 
\[
\mathrm{MMSE}^{\leq D}(X)\le C\,k^{2}\mathrm{MMSE}^{\leq D}(v).
\]
\end{lemma}

\begin{proof}
Let
\[
\mathcal C:=\{x\in\mathbb R^n:\ \|x\|_\infty\le 1,\ \|x\|_0\le 2k\}.
\]
Since the true vector $v$ belongs to $\mathcal C$ almost surely, for any estimator $\widehat v(Y)$ we may project it onto $\mathcal C$:
\[
\widetilde v(Y):=\Pi_{\mathcal C}(\widehat v(Y)).
\]
Since $C$ is a closed set,
\(
\|v-\widetilde v\|_2\le \|v-\widehat v\|_2\) almost surely. Also,
\(
\|v\|_2\le \sqrt{2k}, \ 
\|\widetilde v\|_2\le \sqrt{2k}.
\) Now define
\(
\widetilde X:=\widetilde v^{\otimes 3}.
\)
Using the telescoping identity
\[
v^{\otimes 3}-\widetilde v^{\otimes 3}
=
\sum_{j=1}^r
v^{\otimes(j-1)}\otimes (v-\widetilde v)\otimes \widetilde v^{\otimes(3-j)},
\]
together with the identity \(
\|v^{\otimes m}\|_F=\|v\|_2^m, \)
we obtain
\[
\|X-\widetilde X\|_F
\le
\sum_{j=1}^3
\|v\|_2^{j-1}\|v-\widetilde v\|_2\|\widetilde v\|_2^{3-j}
\le
6k\|v-\widetilde v\|_2.
\]
Squaring and taking expectations gives
\[
\E\|X-\widetilde X\|_F^2
\le
9k^2\E\|v-\widetilde v\|_2^2
\le
9k^2\E\|v-\widehat v\|_2^2.
\]
Finally, taking infimum over all estimators $\widehat v$ we get the desired result.
\end{proof}

\begin{lemma}[A diagonal odd-power estimator for $v$]\label{lem:diag_power_est}
Let $d\ge 1$ be an odd integer and define
\begin{equation}\label{eq:mse-v-bound}
\widehat v_i:=\lambda^{-d}Y_{i,\ldots,i}^d,
\ i=1,\dots,n. \quad \text{Then}, \ \ \E\|v-\widehat v\|_2^2
\le
\frac{8k\,d^3}{\lambda^2}
+
n\Big(\frac{2d}{\lambda^2}\Big)^d
\end{equation}
provided
\(
2d^2\le \lambda^2.
\)
\end{lemma}

\begin{proof}
Fix $i\in \{1, \dots ,n\}$.

\medskip

\noindent\emph{Case 1: $v_i=0$.}
Then $Y_{i,\ldots,i}=Z_{i, \dots ,i}$, hence
\(
\widehat v_i=\lambda^{-d}Z_{i, \dots ,i}^d,
\)
and therefore
\[
\E[\widehat v_i^2\mid v_i=0]
=
\lambda^{-2d}\E[Z_{i, \dots ,i}^{2d}]
=
\lambda^{-2d}(2d-1)!!.
\]
Since \(
(2d-1)!!=1\cdot 3\cdot 5\cdots (2d-1)\le (2d)^d,
\) we get
\[
\E[\widehat v_i^2\mid v_i=0]\le \Big(\frac{2d}{\lambda^2}\Big)^d.
\]

\medskip

\noindent\emph{Case 2: $v_i=1$.}
Now
\(
Y_{i,\ldots,i}=\lambda+Z_{i, \dots ,i},
\) so
\[
\widehat v_i-1
=
\lambda^{-d}(\lambda+Z_{i, \dots ,i})^d-1
=
\sum_{j=1}^d \binom{d}{j}\lambda^{-j}Z_{i, \dots ,i}^j.
\]
Hence, by using Cauchy--Schwarz inequality in the form
\(
\Big(\sum_{j=1}^d a_j\Big)^2\le d\sum_{j=1}^d a_j^2,
\) we obtain
\[
\E[(\widehat v_i-1)^2\mid v_i=1]
\le
d\sum_{j=1}^d \binom{d}{j}^2\lambda^{-2j}\E[Z_{i, \dots ,i}^{2j}].
\]
Using
\[
\binom{d}{j}\le \frac{d^j}{j!},
\qquad
\E[Z_{i, \dots ,i}^{2j}]=(2j-1)!!\le 2^j j!,
\]
we find
\[
\E[(\widehat v_i-1)^2\mid v_i=1]
\le
d\sum_{j=1}^d \frac{d^{2j}}{(j!)^2}\lambda^{-2j}\cdot 2^j j!
=
d\sum_{j=1}^d \frac{(2d^2/\lambda^2)^j}{j!}
\le
d\big(e^{2d^2/\lambda^2}-1\big).
\]
Under the assumption $2d^2\le \lambda^2$, we have $2d^2/\lambda^2\le 1$, and so using $e^x-1\le 2x, x\in [0,1]$ we get
\[
\E[(\widehat v_i-1)^2\mid v_i=1]
\le
\frac{4d^3}{\lambda^2}.
\]

\medskip

\noindent\emph{Case 3: $v_i=-1$.}
Since $d$ is odd and $Z_{i, \dots ,i}\stackrel{d}{=}-Z_{i, \dots ,i}$,
\[
\E[(\widehat v_i+1)^2\mid v_i=-1]
=
\E[(\widehat v_i-1)^2\mid v_i=1]
\le
\frac{4d^3}{\lambda^2}.
\]

\medskip

Now sum over all coordinates. Since $\|v\|_0\le 2k$ almost surely, there are at most $2k$ active coordinates and at most $n$ coordinates in total. Hence
\[
\E\|v-\widehat v\|_2^2
\le
2k\cdot \frac{4d^3}{\lambda^2}
+
n\Big(\frac{2d}{\lambda^2}\Big)^d,
\]
which is exactly \eqref{eq:mse-v-bound}.
\end{proof}

\begin{corollary}\label{cor:diag_d_logn}
Assume
\( 
\lambda\ge (\log n)^2.
\) Let $D$ be arbitrary with $D=\Theta(\log n)$. Then,
\[
\mathrm{MMSE}^{\leq D}(X)=o(k^3).
\]
\end{corollary}

\begin{proof}
Since $D\le \log n$ and $\lambda\ge (\log n)^2$, for all large $n$ we have \(
2D^2\le 2(\log n)^2\le \lambda^2,
\) so Lemma \ref{lem:diag_power_est} applies. Thus
\[
\E\|v-\widehat v\|_2^2
\le
\frac{8k\,D^3}{\lambda^2}
+
n\Big(\frac{2D}{\lambda^2}\Big)^D.
\]
For the first term,
\[
\frac{8k\,D^3}{\lambda^2}
\le
8k\frac{(\log n)^3}{(\log n)^4}
=
\frac{8k}{\log n}
=o(k).
\]
For the second term, since $D=\Theta( \log n)$ and $\lambda^2\ge (\log n)^4$,
\[
n\Big(\frac{2D}{\lambda^2}\Big)^D
\le
n\Big(\frac{2\log n}{(\log n)^4}\Big)^D
=
n\Big(\frac{2}{(\log n)^3}\Big)^D.
\]
Taking logarithms and using $D=\Theta( \log n)$,
\[
\log\!\left[n\Big(\frac{2}{(\log n)^3}\Big)^D\right]
=
\log n + D\big(\log 2-3\log\log n\big)
\to -\infty.
\]
Hence \(
n(2D/\lambda^2)^D=o(1), \)
and therefore it is certainly $o(k)$. This proves
\[
\E\|v-\widehat v\|_2^2=o(k),
\]
and so $\mathrm{MMSE}^{\leq D}(v)=o(k)$. The $\mathrm{MMSE}^{\leq D}(X)=o(k^3)$ conclusion then follows from Corollary \ref{cor:diag_d_logn}.
\end{proof}

\subsubsection{Putting it all together}

\begin{theorem}[Annealed FP monotonicity agreement with the low-degree MMSE for the truncated model]
\label{thm:truncated_final_equivalence}
Fix $C>0$ and suppose $n^{-C}\leq \lambda \leq n^C$. There exists $\lambda_1, \lambda_2 =\widetilde{\Theta}(1)$ with $\lambda_1<\lambda_2$ such that the following holds.

If $\lambda \leq \lambda_1$, then 

\begin{enumerate}
\item the annealed Franz--Parisi potential for the truncated model $\mathcal {F}^{P}_{\rm ann,\lambda}(q)$ is monotone decreasing on the interval $q=s^3 \in [0, \Theta((\log n)^3)].$ 

\item the low-degree MMSE for any $D =\Theta(\log n)$ is asymptotically trivial, namely
\[
\mathrm{MMSE}^{\le D}_{P}(\lambda)
=
(1-o(1))\mathrm{MMSE}^{\mathrm{trivial}}_{P}
\]
\end{enumerate}

If $\lambda \geq \lambda_2$, then 

\begin{enumerate}
\item the annealed Franz--Parisi potential for the truncated model $\mathcal {F}^{P}_{\rm ann,\lambda}(q)$ is monotone increasing on the interval $q=s^3 \in [0, n^{o(1)}].$

\item the low-degree MMSE for $D=\Theta(\log n)$ is asymptotically of lower order than the trivial, namely
\[
\mathrm{MMSE}^{\le D}_{P}(\lambda)
=
o( \mathrm{MMSE}^{\mathrm{trivial}}_{P}).
\]
\end{enumerate}
\end{theorem}

\begin{remark}
    Observe the agreement between the two different predictions for this task. When $\lambda$ is smaller than the algorithmic threshold, the problem is ``physics-hard" as the annealed FP is increasing and also low-degree hard as the low-degree MMSE is trivial. On the other hand, if $\lambda$ is larger than the algorithmic threshold the problem is ``physics-easy" as the annealed FP is decreasing, and also low-degree hard as the low-degree MMSE is beating the trivial MSE. 
\end{remark}
\begin{proof}
We use the notations $\widetilde P$ and $ P$ to denote the untruncated and truncated priors, respectively, as defined in the start of Section \ref{sec:mmse_trunc_untrunc}.

   First, by Theorem \ref{thm:sptpca} for the untruncated prior there exists $\lambda_1=\widetilde{\Theta}(1)$ such that if $\lambda \leq \lambda_1$  it holds for some $D=\Theta(\log n)$,
   \[\mathrm{MMSE}^{\le D}_{\widetilde P}(\lambda)=(1+o(1))\mathrm{MMSE}^{\textrm{trivial}}=(1+o(1))k^r.\]
By Proposition~\ref{prop:MMSE-trunc},
\[
\big|\mathrm{MMSE}^{\le D}_{P}(\lambda)-\mathrm{MMSE}^{\le D}_{\widetilde P}(\lambda)\big|
\le e^{-\Theta(k)}
\]and therefore
   \[\mathrm{MMSE}^{\le D}_{ P}(\lambda)=(1+o(1))\mathrm{MMSE}^{\textrm{trivial}}=(1+o(1))k^r.\]

Moreover, for any $1 \leq q=s^3 =n^{o(1)},$ and the sequence $a_n(s)=(s+2)^3-s^3, n\in \N$ we have from Lemma \ref{le:fppred_sparserad_small_k}, \begin{equation}\label{eq:log_robust_fann_2}
\Delta_{a_n(s)}\mathcal F^{P}_{\rm ann,\lambda}(q)=-\lambda+\Theta\left(\frac{\log n}{q^{2/3}}\right),
\end{equation}
where the $\Theta$ hides universal constants (independent of $q$). In particular, by shrinking further $\lambda_1$ if necessary, assuming $\lambda_1=o( 1/\log n)$, we have for any $\lambda \leq \lambda_1,$ that the annealed FP potential is decreasing in the interval $q=s^3 \in [0,\Theta((\log n)^3)].$ Now by Theorem \ref{thm:annealed_FP_trunc}, 
\begin{equation}\label{eq:log_robust_fann-2}
\Delta_{a_n(s)}\mathcal F^{\widetilde P}_{\rm ann,\lambda}(q)
=
(1+o(1))\Delta_{a_n(s)}\mathcal F^{P}_{\rm ann,\lambda}(q),
\end{equation} and therefore the same monotonic behavior holds for the annealed FP potential of the truncated prior.

   Now, for some $\lambda_2=\widetilde{\Theta}(1)$, by Corollary \ref{cor:diag_d_logn} it holds that if $\lambda \geq \lambda_2$ for any $D=\Theta(\log n)$ and the truncated prior that,
   \[\mathrm{MMSE}^{\le D}_{ P}(\lambda)=o(\mathrm{MMSE}^{\textrm{trivial}}).\]
   
   Moreover, fix any $1 \leq q=s^3 =n^{o(1)},$ and the sequence $a_n(s)=(s+2)^3-s^3, n\in \N$. We have from \eqref{eq:log_robust_fann_2} and assuming $\lambda_2=\omega(\log n)$, that for $\lambda \geq \lambda_2,$ the annealed FP potential for the untruncated prior is decreasing for all $1 \leq q=n^{o(1)}.$ Now by Theorem \ref{thm:annealed_FP_trunc}, the same monotonic behavior holds for the annealed FP potential of the truncated prior.

\end{proof}

\subsection{The quenched FP potential is not decreasing in the ``Low-degree easy" phase}\label{sec:quenched-fail}
We now state and prove the following theorem, proving that even in the ``low-degree" easy regime $\lambda=k^{\Omega(1)}$ where the low-degree MMSE is trivial, the quenched FP potential remains non-monotonic. In particular, the monotonocity of the quenched FP potential fails to capture the low-degree MMSE phase transition for this model. Strikingly, as we discussed in the previous section, the monotonicity of the annealed FP potential does track it accurately though.

\begin{theorem}\label{thm:quenched-fail}
   There exist constants $c,C,C'>0$ such that the following holds.  If $\beta>0$ is a small enough constant, $k=n^{\beta+o(1)},$ then for any $q \geq 1,$
\begin{align*}
    \mathcal{F}_{\lambda}(q)-\mathcal{F}_{\lambda}(0) &\geq -\lambda^2 q-Ck\log (n/k)+c\lambda  kq^{1/3}/\sqrt{ \log n}.
\end{align*}In particular if $\lambda=k^{\epsilon+o(1)}$ for any $\epsilon \in (0,1)$ (i.e., for values of $\lambda \geq n^{\Omega(1)}\lambda_{\mathrm{ALG}}$) then as long as $C (\log n)^3 \leq q \leq (\log n)^d$ for some $d>3,$ it holds
\begin{align*}
    \mathcal{F}_{\lambda}(q)-\mathcal{F}_{\lambda}(0) \geq k^{1+\Omega(1)}>0,
    \end{align*}and therefore, for all $d>3,$ the quenched FP potential $\mathcal{F}_{\lambda}(q)$ is not a decreasing function for $q \in [0,(\log n)^d].$
\end{theorem}
\subsubsection{Key Lemmas}To prove this we first need a few lemmas. Of crucial importance is the random curve \[\Gamma(q',m):=\max_{v' \in S_{n,m}: \langle v,v'\rangle=q'} \langle \mathrm{vec}(v'^{\otimes 3}),Z \rangle, q'\in \mathbb{R},m \in [n]\] where $S_{n,m}$ is the set of $m$-sparse vectors in $\{-1,0,1\}^n$ and as usual $Z$ has i.i.d. $N(0,1)$ entries.
\begin{lemma}\label{lem:bound-Gamma-q-s}
    Suppose $\beta>0$ is a small enough constant. Then there exists a constant $C>0$ such that for any $1\leq q'= o(k), k/2 \leq m \leq 2k$ and for any growing sequence $A_n$ we have with probability at least $1-C e^{-A_n/2},$
    \begin{align}\label{eq:1mm}\Gamma(q',m) \leq \sqrt{m^3}\sqrt{2\log\left(k \binom{k}{q'} \binom{n-k}{m-q'}\right)-\log (m\log(n/m))+A_n}. \end{align}
   In particular, for $C>0$ large enough, if $A_n \geq C\log n,$ we have that for any $1\leq q'= o(k)$, with probability $1-e^{-\Theta(A_n)}$, \eqref{eq:1mm} holds simultaneously for all $ k/2 \leq m \leq 2k$. 
   
   Moreover, for any $c>0$ if $\alpha>0$ is small enough then for all $m$ with $k/2 \leq m \leq 2k,$
    \[\mathbb{E}[\Gamma(q',m)1(\Gamma(q',m) \geq n^{c})]=n^{-\Omega(1)}.\]
\end{lemma}

\begin{proof} 
    For any $x \in \{-1,0,1\}^n$ which is $m$-sparse, observe that $\langle Z, \mathrm{vec}(x^{\otimes 3}) \rangle$ is a mean-zero Gaussian with variance $m^3$.

    Notice that the number of $x \in \{-1,0,1\}^n$ which is $m$-sparse and satisfies $\langle x,v \rangle=q',$ is \[\sum_{b \geq q'} \binom{b}{q'}\binom{k}{b}\binom{n-k}{m-b} \leq k \binom{k}{q'} \binom{n-k}{m-q'}.\] 
We first simplify the summand using the identity
\[
\binom{k}{b}\binom{b}{q'}=\binom{k}{q'}\binom{k-q'}{b-q'},
\]
Hence, writing $t=b-q'$,
\begin{align*}
\sum_{b\ge q'}\binom{b}{q'}\binom{k}{b}\binom{n-k}{m-b}
&=\binom{k}{q'}\sum_{t\ge 0}\binom{k-q'}{t}\binom{n-k}{(m-q')-t}.
\end{align*}
By Vandermonde's identity,
\[
\sum_{t\ge 0}\binom{k-q'}{t}\binom{n-k}{(m-q')-t}=\binom{(k-q')+(n-k)}{m-q'}=\binom{n-q'}{m-q'},
\]
which proves the exact equality
\[
\sum_{b\ge q'}\binom{b}{q'}\binom{k}{b}\binom{n-k}{m-b}
=\binom{k}{q'}\binom{n-q'}{m-q'}.
\]
For the inequality, it suffices to show
\[
\binom{n-q'}{m-q'}\le k\binom{n-k}{m-q'}.
\]
Let $r:=m-q'$. Since $m\le 2k$ and $q'\ge 1$, we have $0\le r\le 2k$.
Using the product representation of binomial coefficients,
\begin{align*}
\frac{\binom{n-q'}{r}}{\binom{n-k}{r}}
&=\prod_{i=0}^{r-1}\frac{n-q'-i}{n-k-i}
=\prod_{i=0}^{r-1}\Big(1+\frac{k-q'}{n-k-i}\Big)
\;\le\;\exp\!\Big(\sum_{i=0}^{r-1}\frac{k-q'}{n-k-i}\Big).
\end{align*}
Since $r\le 2k$ and $i\le r-1\le 2k-1$, we have $n-k-i\ge n-3k$. Thus
\[
\sum_{i=0}^{r-1}\frac{k-q'}{n-k-i}\le \frac{r(k-q')}{n-3k}\le \frac{2k^2}{n-3k}.
\]
Assume $k\le n^\beta$ with $\beta<\frac12$. Then $k^2/n=n^{2\beta-1}\to 0$, and also $3k=o(n)$,
so for all sufficiently large $n$ we have $\frac{2k^2}{n-3k}\le \log 2$. Therefore, for all large $n$,
\[
\frac{\binom{n-q'}{r}}{\binom{n-k}{r}}\le e^{\log 2}=2\le k,
\]
which implies $\binom{n-q'}{r}\le k\binom{n-k}{r}$ and completes the proof.
    
    Hence, by a union bound and Mill's ratio bound, the probability that \eqref{eq:1mm} does not hold  for some $x$ which is $m$-sparse and satisfies $\langle x,v \rangle=q',$ is at most 
    \begin{align*}
    & \frac{1+o(1)}{\sqrt{2\pi}} \frac{\exp\left( \frac{1}{2} \log \left(m\log(n/m)\right) - \frac{1}{2} A_n \right)}{\sqrt{2 \log \left( k\binom{k}{q'} \binom{n-k}{m-q'} \right) - \log \left(m\log(n/m)\right) + A_n}} \\
={}& \frac{1+o(1)}{\sqrt{2\pi}} \sqrt{ \frac{m\log(n/m)}{2 \log \left( k\binom{k}{q'} \binom{n-k}{m-q'} \right) - \log \left(m\log(m/s)\right) + A_n} } \exp\left( -\frac{1}{2} A_n \right)\\
={}& O\left( \exp\left( -\frac{1}{2} A_n \right) \right).
    \end{align*} For the second to last equality we used that since $q'=o(k)$ and hence $q'=o(m)$, it holds
    \[\log \left( \binom{k}{q'}\binom{n-k}{m-q'} \right) \geq \log \binom{n-k}{m-q'} \geq (1-o(1)) m\log (n/m)\]
    This shows \eqref{eq:1mm}.

    Fix $q'=o(k)$ and let 
\[
\mathcal S:=\{m\in\mathbb Z:\ k/2\le m\le 2k\},
\qquad |\mathcal S|\le 2k .
\]
    For each fixed $m\in\mathcal S$, \eqref{eq:1mm} yields
\[
\Pr\big(\eqref{eq:1mm}\ \text{fails for this }s\big)\le C e^{-A_n/2}.
\]
A union bound over $m\in\mathcal S$ gives
\[
\Pr\Big(\exists m\in\mathcal S:\ \eqref{eq:1mm}\ \text{fails}\Big)
\le \sum_{m\in\mathcal S} C e^{-A_n/2}
\le 2k\, C e^{-A_n/2}.
\]
In particular, if $A_n=\omega(\log n)$, with $k=n^{\beta}$ for some fixed $\beta>0$, then
$2k C e^{-A_n/2}=\exp(-\Theta(A_n))$, hence \eqref{eq:1mm} holds
simultaneously for all $m\in[k/2,2k]$ with probability $1-e^{-\Theta(A_n)}$.

Fix any $c>0$ and take $\beta>0$ small enough so that $2\beta<c$.
For $m\in[k/2,2k]$ define the threshold
\[
B_{q',m}:=\sqrt{m^3}\sqrt{2\log\!\Big(k\binom{k}{q'}\binom{n-k}{m-q'}\Big)
-\log\big(m\log(p/m)+A_n\big)}.
\]
Using $\log\binom{k}{q'}\le q'\log(ek/q')\le k\log(ek)$ (since $q'=o(k)$),
$\log\binom{n-k}{m-q'}\le m\log(en/m)$ (since $m\asymp k$), and $m\asymp k$,
we obtain for some constant $C>0$, that
\[
\log \Big(k\binom{k}{q'}\binom{n-k}{m-q'}\Big)\le C k\log n;
\]
hence for some absolute constant $C'>0$,
\[
B_{q',m}\le C' k^{3/2}\sqrt{k\log n}=C'k^2\sqrt{\log n}.
\]
Since $k=n^{\beta}$ and $2\beta<c$, for all large $n$ we have
$B_{q',m}\le \tfrac12 n^c$, uniformly over $m\in[k/2,2k]$. Therefore,
for all such $n$,
\[
\Pr\big(\Gamma(q',m)\ge n^c\big)
\le \Pr\big(\Gamma(q',m)>B_{q',m}\big)\le C e^{-A_n/2}.
\]
Moreover, $\Gamma(q',m)$ is the maximum of
$k\binom{k}{q'}\binom{n-k}{m-q'}$ centered subgaussian
variables with proxy variance $\le c m^3$ with some absolute constant $c>0$ (see the proof of
\eqref{eq:1mm}). By the standard maximal inequality for subgaussian
families (e.g. \cite[Prop.~2.5.2]{vershynin-HDP}), there are constants $C,C'>0$, such that
\[
\mathbb E[\Gamma(q',m)^2]\le Cm^3\log\!\Big(k\binom{k}{q'}\binom{n-k}{m-q'}\Big)
\le C'k^3\cdot(k\log n)=C'k^4\log n.
\]
Applying Cauchy--Schwarz inequality, it gives that for some absolute constant $\widetilde C>0$,
\[
\mathbb E\big[\Gamma(q',m)\mathbf 1\{\Gamma(q',m)\ge n^c\}\big]
\le \big(\mathbb E[\Gamma(q',m)^2]\big)^{1/2}\,
\Pr(\Gamma(q',m)\ge n^c)^{1/2}
\le \widetilde C k^2\sqrt{\log n}\, e^{-A_n/4}.
\]
Taking, e.g., $A_n=(\log n)^2$ (or any $A_n\to\infty$ fast enough) yields
$k^2\sqrt{\log n}\, e^{-A_n/4}=n^{-\Omega(1)}$, uniformly for all
$m\in[k/2,2k]$, completing the proof.
\end{proof}

We can also prove the following.

\begin{lemma}\label{lem:bound-Gamma-0-s}
    Suppose $0<\beta<1/5$.
      Then for any $k/2 \leq m \leq 2k$ and for any growing sequence $A_n$ we have with probability $1-n^{-\Omega(1)},$
    \[\Gamma(0,m) \geq \sqrt{m^3}\sqrt{2\log\binom{n}{m}-\log (m\log(n/m))-A_n}. \]
    
\end{lemma}

\begin{proof}
    This follows by applying \cite[Proposition 9.4]{chen2024low} for $t=3$, sparsity level $m$, and $\ell=1/2$, when $\beta<1/5$ and directly checking the with high probability guarantee $1-n^{-\Omega(1)}$ out of the second moment method (specifically, in \cite[Lemma 9.7]{chen2024low} all $o(1)$ terms can be straightforwardly checked to be $n^{-\Omega(1)}$). It should be noted that while formally \cite{chen2024low} prove \cite[Proposition 9.4]{chen2024low} only for integer $\ell$, the proof follows mutatis mutandis in the non-integer case for $\ell$. Moreover, in the case $\ell=1/2,$ $\Gamma_{[0,\ell]}$ from \cite[Proposition 9.4]{chen2024low} becomes equal to $\Gamma(0,m)/\sqrt{m^3}$ in our notation yielding our lemma.
\end{proof}

Based on our prior, let $p(m,q')=\mathbb{P}_v(\langle v,v' \rangle=q', \|v'\|_0=m).$

\begin{lemma}\label{lem:min-p-s-0}
    It holds \[\min_{k/2 \leq m \leq 2k} p(m,0) \geq e^{-\Theta(k)}.\]
\end{lemma}
\begin{proof}
Fix $m\in[k/2,2k]$ and consider the event
\[
E_m:=\{\|v'\|_0=m\}\cap\{\mathrm{supp}(v')\cap \mathrm{supp}(v)=\emptyset\}.
\]
On $E_m$ we have $\langle v,v'\rangle=0$, hence
\(
p(m,0) \ge \mathbb P(E_m).
\)
Condition on $v'$, then the support of $\{\|v'\|_0=m\}$ is some set
$S\subset[n]$ with $|S|=m$. By independence across coordinates,
\[
\mathbb P\big(\mathrm{supp}(v')\cap \mathrm{supp}(v)=\emptyset
  \big|\  \|v'\|_0=m,  \mathrm{supp}(v')=S\big)
=(1-k/n)^m.
\]
Therefore
\[
\mathbb P(E_m)=\mathbb P(\|v'\|_0=m) (1-k/n)^m.
\]

We have $\|v'\|_0\sim\mathrm{Bin}(n,k/n)$ with mean $k$.
For any $m\in[k/2,2k]$, the standard local lower bound for the binomial
distribution (see, e.g., \cite[Prop.~2.1.2]{vershynin-HDP}) yields
\(
\mathbb P(\|v'\|_0=m) \ge \exp(-C_1 k)
\)
for some absolute constant $C_1>0$ and all large $n$.

Since $m\le 2k$ and $\log(1-x)\ge -x-x^2$ for small $x$,
\[
(1-k/n)^m
=\exp \big(m\log(1-k/n)\big)
\ge \exp \Big(-\frac{km}{n}-O \Big(\frac{k^2 m}{n^2}\Big)\Big)
\ge \exp \Big(-O \Big(\frac{k^2}{n}\Big)\Big).
\]
Because $\beta<1/2$, we have $k^2/n=n^{2\beta-1+o(1)}=o(1)$, hence
\[
(1-k/n)^m=\exp(-o(k)).
\]
Combining the bounds,
\[
p(m,0) \ge \mathbb P(E_m)
 \ge \exp(-C_1 k)\cdot \exp(-o(k))
=\exp(-\Theta(k)),
\]
uniformly for all $m\in[k/2,2k]$. Taking the minimum over $m$
completes the proof.
\end{proof}

\subsubsection{Proof of Theorem \ref{thm:quenched-fail}}

By the definition of the quenched FP potential, for any $q =(q')^3> 0$, we have:
\begin{align*}
    \mathcal{F}_{\lambda}(q) = -\mathbb{E}_{v, Z} \left[ \log \mathbb{E}_{v'} \left[ \mathbf{1}_{\{\langle v,v' \rangle=q'\}} \exp\left( -\frac{1}{2}\|Y-\sqrt{\lambda}\mathrm{vec}((v')^{\otimes 3})\|_2^2 \right) \right] \right]
\end{align*}
where $Z \sim \mathcal{N}(0, I_N)$. Expanding the squared norm and dropping terms independent of $q=(q')^3$, the difference in the free energy simplifies to:
\begin{align*}
    \mathcal{F}_{\lambda}(q)-\mathcal{F}_{\lambda}(0) &= -\lambda^2 q - \mathbb{E}_{v,Z} \left[ \log\frac{\mathbb{E}_{v'} \left[ \mathbf{1}_{\{\langle v,v' \rangle=q'\}} e^{\lambda \langle \mathrm{vec}((v')^{\otimes 3}),Z\rangle} \right]}{\mathbb{E}_{v'} \left[ \mathbf{1}_{\{\langle v,v' \rangle=0\}} e^{\lambda \langle \mathrm{vec}((v')^{\otimes 3}),Z\rangle} \right]} \right].
\end{align*}

To bound the ratio of these partition functions, we stratify the inner expectations over the sparsity levels $m = \|v'\|_0 \in [\lfloor k/2 \rfloor, 2k]$. For the numerator, we bound the sum by the maximum over $m$ and the size of the support $S_{n,m}$:
\begin{align*}
    \sum_{m=\lfloor k/2 \rfloor}^{2k} \mathbb{E}_{v'} \left[ \mathbf{1}_{\{\langle v,v' \rangle=q', \|v'\|_0=m\}} e^{\lambda \langle \mathrm{vec}((v')^{\otimes 3}),Z\rangle} \right] &\leq \max_{m \in [\lfloor k/2 \rfloor, 2k]} \left( e^{\lambda \Gamma(q',m)} p(m,q') \right) \cdot \sum_{m=\lfloor k/2 \rfloor}^{2k} |S_{n,m}|.
\end{align*}
Since the total number of sparse supports is generously bounded by $\sum_{m=\lfloor k/2 \rfloor}^{2k}  |S_{n,m}| \leq 2k \binom{n}{2k} 2^{2k}$, its logarithm is $O(k \log(n/k))$.

For the denominator, we lower bound the sum by its maximum term:
\begin{align*}
    \sum_{m=\lfloor k/2 \rfloor}^{2k} \mathbb{E}_{v'} \left[ \mathbf{1}_{\{\langle v,v' \rangle=0, \|v'\|_0=m\}} e^{\lambda \langle \mathrm{vec}((v')^{\otimes 3}),Z\rangle} \right] &\geq \max_{m \in [\lfloor k/2 \rfloor, 2k]} \left( e^{\lambda \Gamma(0,m)} p(m,0) \right).
\end{align*}

Taking the ratio of these bounds and applying the logarithm, the free energy difference is bounded by:
\begin{align*}
    \mathcal{F}_{\lambda}(q)-\mathcal{F}_{\lambda}(0) &\geq -\lambda^2 q - O\left(k\log \frac{n}{k}\right) - \mathbb{E}_{Z} \left[ \max_{m \in [\lfloor k/2 \rfloor, 2k]} \lambda \big(\Gamma(q',m)-\Gamma(0,m)\big) \right] - \max_{m \in [\lfloor k/2 \rfloor, 2k]} \log \frac{1}{p(m,0)}.
\end{align*}

From Lemma \ref{lem:min-p-s-0}, we established that $\min_{k/2 \leq m \leq 2k} p(m,0) \ge e^{-\Theta(k)}$, which absorbs into the existing $O(k\log(n/k))$ term. Furthermore, with probability $1-n^{-\Omega(1)}$, the deviation between the maximums is bounded simultaneously for all $k/2 \leq m \leq 2k$ by Lemmas \ref{lem:bound-Gamma-q-s} and \ref{lem:bound-Gamma-0-s}:
\begin{align*}
    \Gamma(q',m)-\Gamma(0,m) &\leq -\frac{m(q'-A_n-\log k)}{\Theta(\sqrt{\log n})}.
\end{align*}
Applying Lemma \ref{lem:bound-Gamma-q-s} to control the tail expectation, we find that for $m > 2(A_n+\log k)$ and since $q=(q')^3,$
\begin{align*}
    \mathbb{E}_{Z} \left[ \max_{m \in [\lfloor k/2 \rfloor, 2k]} \lambda \big(\Gamma(q,s)-\Gamma(0,s)\big) \right] &\leq -\Theta\left( \frac{\lambda k q^{1/3}}{\sqrt{\log n}} \right) + o(1).
\end{align*}

Substituting this expectation bound back into the free energy inequality yields that for some constants $C, c > 0$:
\begin{align*}
    \mathcal{F}_{\lambda}(q)-\mathcal{F}_{\lambda}(0) &\geq -\lambda^2 q - Ck\log \frac{n}{k} + c\lambda \frac{k q^{1/3}}{\sqrt{\log n}}.
\end{align*}

Finally, we choose $A_n = C\log n$ for a large enough $C>0$. If we set $\lambda = k^{\epsilon+o(1)}$ for any $\epsilon \in (0,1)$ and evaluate at $q \geq (2C)^3 (\log n)^3$, the positive term strictly dominates. Thus, we conclude:
\begin{align*}
    \mathcal{F}_{\lambda}(q)-\mathcal{F}_{\lambda}(0) &\geq k^{1+\Omega(1)} > 0.
\end{align*}

\section{Proof of auxiliary lemmas }\label{lem:MMSE_decomposition}

In this section we prove Propositions \ref{prop:cumulant-linearity-one-arg} and \ref{prop:translation}, and Lemmas \ref{le:diag_threshold_single_lambda} and \ref{lem:vector-mmse-corr}.

\begin{proof}[Proof of Lemma \ref{le:diag_threshold_single_lambda}]
We use the standard Gaussian tail bound: for $Z\sim\mathcal{N}(0,1)$ and $t\ge 0$,
\begin{equation}\label{eq:gauss_tail}
\mathbb{P}(|Z|\ge t)\le 2e^{-t^2/2}.
\end{equation}
Recall that, $d_i=\lambda v_i + Z_i$ with $Z_i\sim\mathcal{N}(0,1)$ i.i.d. If $v_i=0$, then $d_i=Z_i$. The event $\widehat v_i\neq v_i$ is $\{|d_i|\ge\tau\}$, hence
by a union bound and \eqref{eq:gauss_tail},
\begin{equation}\label{eq:firstneeded}
\mathbb{P}\big(\exists i:\ v_i=0,\ \widehat v_i\neq 0\big)
\le \sum_{i:\,v_i=0}\mathbb{P}(|Z_i|\ge\tau)
\le n\cdot 2e^{-\tau^2/2}
=2n^{-2},
\end{equation}
since $\tau=\sqrt{6\log n}$.

Next, fix $i\in S=\supp(v)$. Then $v_i\in\{\pm1\}$ and $d_i\sim\mathcal{N}(\pm\lambda,1)$.
The event $\widehat v_i=0$ is $\{|d_i|<\tau\}$, so
\[
\mathbb{P}(|d_i|<\tau)
\le 2\exp\!\Big(-\frac12(\lambda-\tau)_+^2\Big)
\le 2e^{-\tau^2/2},
\]
where the first inequality is the same two-tail bound used in your original proof, and the
second uses $\lambda\ge 2\tau$ so that $\lambda-\tau\ge\tau$.
Union bounding over $i\in S$ and using $|S|\le 2k$ yields
\begin{equation}\label{eq:secondneeded}
\mathbb{P}\big(\exists i\in S:\ \widehat v_i=0\big)
\le |S|\cdot 2e^{-\tau^2/2}
\le 4k\,e^{-\tau^2/2}
=4k\,n^{-3}.
\end{equation}

Now, fix $i\in S$. If $v_i=1$ then $\widehat v_i=-1$ implies $d_i\le -\tau$;
if $v_i=-1$ then $\widehat v_i=1$ implies $d_i\ge \tau$.
Hence, for each $i\in S$,
\[
\mathbb{P}(\widehat v_i=-v_i)
\le \mathbb{P}\big(\mathcal{N}(\lambda,1)\le -\tau\big)
     +\mathbb{P}\big(\mathcal{N}(-\lambda,1)\ge \tau\big)
\le 2\exp\!\Big(-\frac12(\lambda+\tau)^2\Big),
\]
by \eqref{eq:gauss_tail}. Under $\lambda\ge 2\tau$, we have $\lambda+\tau\ge 3\tau$, so
\[
2\exp\!\Big(-\frac12(\lambda+\tau)^2\Big)
\le 2\exp\!\Big(-\frac12(3\tau)^2\Big)
=2e^{-27\log n}
=2n^{-27}.
\]
Union bounding over $i\in S$ and using $|S|\le 2k$ gives
\begin{equation}\label{eq:thirdneeded}
\mathbb{P}\big(\exists i\in S:\ \widehat v_i=-v_i\big)
\le |S|\cdot 2n^{-27}
\le 4k\,n^{-27}.
\end{equation}

Finally,
\[
\{\widehat v\neq v\}\subseteq
\big\{\exists i:\ v_i=0,\ \widehat v_i\neq 0\big\}
\ \cup\
\big\{\exists i\in S:\ \widehat v_i=0\big\}
\ \cup\
\big\{\exists i\in S:\ \widehat v_i=-v_i\big\}.
\]
Combining the bounds from \eqref{eq:firstneeded}, \eqref{eq:secondneeded}, \eqref{eq:thirdneeded} yields \eqref{eq:success_prob_vhat}.
\end{proof}

\begin{proof}[Proof of Proposition \ref{prop:cumulant-linearity-one-arg}]

Set $V_i=Z_i$ for $i\neq t$, and $V_t=aX+bY$.  Fix a partition $\pi\in\mathcal P([n])$ and let
$B_\pi$ be the (unique) block of $\pi$ that contains $t$. Then
\[
\E\Big[\prod_{i\in B_\pi} V_i\Big]
=
\E\Big[(aX+bY)\prod_{i\in B_\pi\setminus\{t\}} Z_i\Big]
=
a\,\E\Big[X\prod_{i\in B_\pi\setminus\{t\}} Z_i\Big]
+
b\,\E\Big[Y\prod_{i\in B_\pi\setminus\{t\}} Z_i\Big],
\]
by linearity of expectation. For every other block $B\in\pi$ with $t\notin B$, we have
$\E[\prod_{i\in B}V_i]=\E[\prod_{i\in B}Z_i]$, which does not depend on $X$ or $Y$.

Therefore, for this fixed $\pi$,
\[
\prod_{B\in\pi}\E\Big[\prod_{i\in B}V_i\Big]
=
a\Bigg(\E\Big[X\!\!\prod_{i\in B_\pi\setminus\{t\}} Z_i\Big]\prod_{B\in\pi:\,B\neq B_\pi}
\E\Big[\prod_{i\in B}Z_i\Big]\Bigg)
+
b\Bigg(\E\Big[Y\!\!\prod_{i\in B_\pi\setminus\{t\}} Z_i\Big]\prod_{B\in\pi:\,B\neq B_\pi}
\E\Big[\prod_{i\in B}Z_i\Big]\Bigg).
\]
Multiplying by the coefficient $(|\pi|-1)!\,(-1)^{|\pi|-1}$ and summing over all partitions $\pi$,
we can pull out the scalars $a$ and $b$ (the sum is finite), and using Proposition \ref{prop:moment-cumulant-partitions}, we obtain exactly the claimed linearity in the $t$-th argument.
\end{proof}

\begin{proof}[Proof of Lemma \ref{lem:vector-mmse-corr}]

By expanding the squared norm,
\[
\mathrm{MMSE}^{\le D}_X
=\inf_{f_1,\dots,f_n\in \mathbb{R}[Y]_{\le D}}
\sum_{i=1}^n  \mathbb{E}\big[(f_i(Y)-X_i)^2\big]
=\sum_{i=1}^n \inf_{f_i\in \mathbb{R}[Y]_{\le D}} \mathbb{E}\big[(f_i(Y)-X_i)^2\big].
\]
For each $i$ expand the square:
\[
\mathbb{E}[(f_i(Y)-X_i)^2]=\mathbb{E}[f_i(Y)^2]-2\mathbb{E}[f_i(Y)X_i]+\mathbb{E}[X_i^2].
\]
For any fixed $f_i\not\equiv 0$, minimize over the scalar $\alpha\in\mathbb R$ by considering
$g_i=\alpha f_i$:
\[
\inf_{\alpha\in\mathbb R}\mathbb{E}[(\alpha f_i(Y)-X_i)^2]
=\mathbb{E}[X_i^2]-\frac{\mathbb{E}[f_i(Y)X_i]^2}{\mathbb{E}[f_i(Y)^2]}.
\]
Indeed, the quadratic in $\alpha$ is
\[
\mathbb{E}[(\alpha f_i-X_i)^2]=\alpha^2\mathbb{E}[f_i^2]-2\alpha \mathbb{E}[f_iX_i]+\mathbb{E}[X_i^2],
\]
whose minimum occurs at $\alpha^\star=\mathbb{E}[f_iX_i]/\mathbb{E}[f_i^2]$. Now take the infimum over $f_i\in \mathbb{R}[Y]_{\le D}$:
\[
\mathrm{MMSE}^{\le D}_{X_i}
=\inf_{f_i\in\mathbb{R}_{\le D}} \inf_{\alpha\in\mathbb R}\mathbb{E}[(\alpha f_i(Y)-X_i)^2]
=\mathbb{E}[X_i^2]-\sup_{f\in\mathbb{R}_{\le D}}\frac{\mathbb{E}[f_i(Y)X_i]^2}{\mathbb{E}[f_i(Y)^2]}.
\]
Finally, normalize $f_i$ by setting $\tilde f_i := f_i/\sqrt{\mathbb{E}[f_i(Y)^2]}$ (when $\mathbb{E}[f(Y)^2]\neq 0$), so
\[
\sup_{f_i\in\mathbb{R}_{\le D}}\frac{\mathbb{E}[f_i(Y)X_i]^2}{\mathbb{E}[f_i(Y)^2]}
=\Big(\sup_{\substack{\tilde f_i\in\mathbb{R}_{\le D}\\ \mathbb{E}[\tilde f_i(Y)^2]=1}}\mathbb{E}[\tilde f_i(Y)X_i]\Big)^2
=\big(\mathrm{Corr}^{\le D}_{P_0,i}\big)^2.
\]
Substituting gives:
\[
\inf_{f_i\in \mathbb{R}[Y]_{\le D}} \mathbb{E}\big[(f_i(Y)-X_i)^2\big]
= \mathbb{E}[X_i^2]-\big(\mathrm{Corr}^{\le D}_{P_0,i}\big)^2.
\]
Summing over $i$ yields
\[
\mathrm{MMSE}^{\le D}_X
=\sum_{i=1}^n  \mathbb{E}[X_i^2]-\sum_{i=1}^n\big(\mathrm{Corr}^{\le D}_{P_0,i}\big)^2
= \mathbb{E}\|X\|^2-\sum_{i=1}^n\big(\mathrm{Corr}^{\le D}_{P_0,i}\big)^2.
\]

For the second claim consider arbitrary $(f_1,\dots,f_n)\in \mathbb{R}[Y]_{\le D}^n$. By the definition of
$\mathrm{Corr}^{\le D}_{P_0,i}$ and Cauchy--Schwarz inequality, for each $i$ we have
\[
\E[f_i(Y)X_i]
\le \sqrt{\E[f_i(Y)^2]} \mathrm{Corr}^{\le D}_{P_0,i}.
\]
Summing over $i$ and applying Cauchy--Schwarz inequality again gives
\[
\E\Big[\sum_{i=1}^n f_i(Y)X_i\Big]
\le \sum_{i=1}^n \sqrt{\EE[f_i(Y)^2]} \mathrm{Corr}^{\le D}_{P_0,i}
\le \sqrt{\EE\sum_{i=1}^n f_i(Y)^2} \sqrt{\sum_{i=1}^n\big(\mathrm{Corr}^{\le D}_{P_0,i}\big)^2}.
\]
Under the constraint $\E\sum_i f_i(Y)^2=1$, we obtain the upper bound
\[
\mathrm{Corr}^{\le D}_{P_0}
\le \sqrt{\sum_{i=1}^n\big(\mathrm{Corr}^{\le D}_{P_0,i}\big)^2}.
\]
For the matching lower bound, fix $\varepsilon>0$. For each $i$ choose
$g_i\in\mathbb{R}[Y]_{\le D}$ with $\E[g_i(Y)^2]=1$ and
$\E[g_i(Y)X_i]\ge \mathrm{Corr}^{\le D}_{P_0,i}-\varepsilon$
(which is possible by the definition of the supremum).
Define
\[
f_i(Y):=\frac{\mathrm{Corr}^{\le D}_{P_0,i}}{\sqrt{\sum_{j=1}^n(\mathrm{Corr}^{\le D}_{P_0,j})^2}} g_i(Y).
\]
Then $\E\sum_i f_i(Y)^2=1$, and hence
\begin{align*}
\E\Big[\sum_{i=1}^n f_i(Y)X_i\Big]
&=\frac{1}{\sqrt{\sum_{j}(\mathrm{Corr}^{\le D}_{P_0,j})^2}}
\sum_{i=1}^n \mathrm{Corr}^{\le D}_{P_0,i}\,\E[g_i(Y)X_i]\\
&\ge \frac{1}{\sqrt{\sum_{j}(\mathrm{Corr}^{\le D}_{P_0,j})^2}}
\sum_{i=1}^n \mathrm{Corr}^{\le D}_{P_0,i}\,(\mathrm{Corr}^{\le D}_{P_0,i}-\varepsilon)\\
&= \sqrt{\sum_{i=1}^n(\mathrm{Corr}^{\le D}_{P_0,i})^2}
-\varepsilon\,
\frac{\sum_{i=1}^n \mathrm{Corr}^{\le D}_{P_0,i}}{\sqrt{\sum_{j=1}^n(\mathrm{Corr}^{\le D}_{P_0,j})^2}}.
\end{align*}
Taking the supremum over  $f_i$ and then letting $\varepsilon\downarrow 0$
yields
\[
\mathrm{Corr}^{\le D}_{P_0}
\ge \sqrt{\sum_{i=1}^n\big(\mathrm{Corr}^{\le D}_{P_0,i}\big)^2}.
\]
Combining with the upper bound proves the second claim.
\end{proof}

\begin{proof}[Proof of Proposition \ref{prop:translation}]
The proof relies on the generating function characterization of the probabilist Hermite polynomials.
By \cite{magnus2013formulas}[Page 253], for any $x,t\in \R$ and $\alpha \in \N$,
$\sum_{\alpha=0}^\infty \frac{t^{\alpha}}{\alpha!} H_{\alpha}(x) = e^{t x - t^2/2}$, so the multivariate generating function is given by:
\begin{equation*}
\label{eq:gen_func}
G(t, x) \coloneqq \exp\left( \langle t, x \rangle - \frac{1}{2}|t|^2 \right) = \prod_{i=1}^N e^{t_i x_i - t_i^2/2} = \prod_{i=1}^N \left( \sum_{\alpha_i=0}^\infty \frac{t_i^{\alpha_i}}{\alpha_i!} H_{\alpha_i}(x_i) \right) = \sum_{\alpha \in \mathbb{N}^N} \frac{t^\alpha}{\alpha!} H_{\alpha}(x), \quad \forall t \in \mathbb{R}^N.
\end{equation*}
We evaluate the generating function at the shifted argument $x = z + \mu$. By the linearity of the inner product, $\langle t, z + \mu \rangle = \langle t, z \rangle + \langle t, \mu \rangle$, allowing us to factor the exponential term:
\begin{equation}
\label{eq:factorization}
\begin{aligned}
G(t, z + \mu) &= \exp\left( \langle t, z \rangle + \langle t, \mu \rangle - \frac{1}{2}|t|^2 \right) 
= \exp\left( \langle t, z \rangle - \frac{1}{2}|t|^2 \right) \cdot \exp\left( \langle t, \mu \rangle \right).
\end{aligned}
\end{equation}
We now expand both factors on the right-hand side of \eqref{eq:factorization} into their respective power series. The first factor is the generating function for $H_\gamma(z)$, and the second is the standard exponential series:
\begin{equation}
\label{eq:series_product}
G(t, z + \mu) = \left( \sum_{\gamma \in \mathbb{N}^N} \frac{t^\gamma}{\gamma!} H_{\gamma}(z) \right) \left( \sum_{\beta \in \mathbb{N}^N} \frac{t^\beta}{\beta!} \mu^\beta \right).
\end{equation}
Applying the Cauchy product formula for multivariate power series, we combine the two summations. We introduce the multi-index $\alpha = \gamma + \beta$, which implies $\beta = \alpha - \gamma$. The condition $\beta \in \mathbb{N}^N$ necessitates that $\gamma \le \alpha$ (component-wise). Thus, we rewrite \eqref{eq:series_product} as:
\begin{equation}
\label{eq:cauchy_res}
G(t, z + \mu) = \sum_{\alpha \in \mathbb{N}^N} t^\alpha \left( \sum_{\gamma \le \alpha} \frac{H_{\gamma}(z)}{\gamma!} \frac{\mu^{\alpha - \gamma}}{(\alpha - \gamma)!} \right).
\end{equation}
Separately, by the definition of the generating function in \eqref{eq:gen_func} applied to the argument $z+\mu$, the left-hand side is:
\begin{equation}
\label{eq:lhs_def}
G(t, z + \mu) = \sum_{\alpha \in \mathbb{N}^N} t^\alpha \left( \frac{H_{\alpha}(z + \mu)}{\alpha!} \right).
\end{equation}
Since the power series representation is unique, we equate the coefficients of $t^\alpha$ from \eqref{eq:cauchy_res} and \eqref{eq:lhs_def}:
\begin{equation}
\frac{H_{\alpha}(z + \mu)}{\alpha!} = \sum_{\gamma \le \alpha} \frac{H_{\gamma}(z)}{\gamma! (\alpha - \gamma)!} \mu^{\alpha - \gamma}.
\end{equation}
Multiplying both sides by $\alpha!$ yields:
\begin{equation}
H_{\alpha}(z + \mu) = \sum_{\gamma \le \alpha} \frac{\alpha!}{\gamma! (\alpha - \gamma)!} H_{\gamma}(z) \mu^{\alpha - \gamma}.
\end{equation}
Recognizing the multi-index binomial coefficient $\binom{\alpha}{\gamma} = \frac{\alpha!}{\gamma!(\alpha-\gamma)!}$, we obtain the stated identity.
\end{proof}

\end{document}